\documentclass[a4paper,12pt]{amsbook}

\usepackage{pdfpages, needspace}

\usepackage{amsmath,amssymb,latexsym}
\usepackage{setspace}

\def\usetodonotes{true}

\usepackage{tikz}
\usetikzlibrary{calc}

\usepackage{hyperref}

\usepackage{ifthen}
\newcounter{number_of_todos}
\ifthenelse{\equal{\usetodonotes}{true}}{
\usepackage{todonotes}

}{
\usepackage{marginnote}

}

\newtheorem{theorem}{Theorem}[section]
\newtheorem{lemma}[theorem]{Lemma}
\newtheorem{prop}[theorem]{Proposition}
\newtheorem{cor}[theorem]{Corollary}

\newtheorem*{prop*}{Proposition}
\newtheorem*{lemma*}{Lemma}
\theoremstyle{remark}
\newtheorem{definition}[theorem]{Definition}

\newtheorem{remark}[theorem]{Remark}
\newtheorem{example}[theorem]{Example}

\newenvironment{statement}{
\begin{enumerate}\renewcommand{\theenumi}{\theequation}
\addtocounter{equation}{1}\item
}{\end{enumerate}}

\usepackage[font=small,format=plain,labelfont=bf]{caption}

\numberwithin{equation}{section}

\def\CC{\mathbb C}
\def\NN{\mathbb N}
\def\PP{\mathbb P}

\def\TT{\mathbb T}
\def\ZZ{\mathbb Z}
\def\RR{\mathbb R}

\def\Cc{\mathcal C}
\def\Dd{\mathcal D}
\def\Ee{\mathcal E}
\def\Ff{\mathcal F}
\def\Kk{\mathcal K}

\def\Nn{\mathcal N}
\def\Hh{\mathcal H}
\def\id{\mathrm{id}}

\def\M{\mathrm{M}}
\def\U{\mathrm{U}}
\def\L{\mathrm{L}}

\providecommand{\braces}[1]{\{#1\}}
\providecommand{\uniqueinfinitepath}[1]{[ #1]^\infty}
\providecommand{\desourceclass}[1]{\lfloor #1\rfloor}
\providecommand{\bigdesourceclass}[1]{\big\lfloor #1\big\rfloor}

\def\tr{\mathrm{Tr}}

\makeindex
\newcommand{\notationindex}[1]{\index{00#1}}

\begin{document}
\pagenumbering{roman}

\title{Categorising the operator algebras of groupoids and higher-rank graphs}
\date{\today}
\author{Robert Hazlewood}

\thispagestyle{empty}
\begin{center}
\textsc{
\begin{Large}
Categorising the operator algebras of groupoids and higher-rank graphs
\end{Large}}\\\vspace{100pt}
\textsc{\begin{large}Robert Hazlewood\end{large}}\\\vspace{100pt}
A thesis in fulfilment of the requirements for the degree of\\\vspace{15pt}
Doctor of Philosophy\\\vspace{120pt}

The University of New South Wales\\\vspace{20pt}
School of Mathematics and Statistics\\\vspace{20pt}
Faculty of Science\\\vspace{80pt}
May 2013
\end{center}

\cleardoublepage

\raggedbottom
\begin{center}
    \begin{Large}
    \textbf{Acknowledgements}
    \end{Large}\\
\end{center}
\vspace{50pt}

First and foremost, I would like to thank my supervisor, Astrid an Huef. I'm not sure that I ever won her over to my `desperate hand waving' method of proof, but I will always feel very grateful for her guidance. She has always been very generous with her time and continually provided me with useful feedback. I highly recommend her as a supervisor.

I have been very fortunate to have been taught by a large number of teachers and lecturers who have taken a great deal of pride in their work. This has included staff from the universities of New South Wales, Otago, Wollongong, Victoria (Canada) and Newcastle, and from Merewether High School. In particular I would like to thank my honours supervisors Iain Raeburn and Aidan Sims.

Finally I would like to thank my family, Julie, David and Leia Hazlewood, and my girlfriend, Wei Xian Lim, for their tireless support and encouragement.

\tableofcontents

\chapter*{Abstract}
This dissertation concerns the classification of groupoid and higher-rank graph $C^*$-algebras and has two main components. Firstly, 
for a groupoid it is shown that the notions of strength of convergence in the orbit space and measure-theoretic accumulation along the orbits are equivalent ways of realising multiplicity numbers associated to a sequence of induced representations of the groupoid $C^*$-algebra. Examples of directed graphs are given, showing how to determine the multiplicity numbers associated to various sequences of induced representations of the directed graph $C^*$-algebras.

The second component of this dissertation uses path groupoids to develop various characterisations of the $C^*$-algebras of higher-rank graphs. Necessary and sufficient conditions are developed for the Cuntz-Krieger $C^*$-algebras of row-finite higher-rank graphs to be liminal and to be postliminal. When Kumjian and Pask's path groupoid is principal, it is shown precisely when these $C^*$-algebras have bounded trace, are Fell, and have continuous trace. Necessary and sufficient conditions are provided for the path groupoids of row-finite higher-rank graphs without sources to have closed orbits and to have locally closed orbits. When these path groupoids are principal, necessary and sufficient conditions are provided for them to be integrable, to be Cartan and to be proper.

\mainmatter

\chapter{Introduction}\label{chapter_introduction}

This thesis is broken up into 5 chapters. After this short introductory chapter, Chapter \ref{chapter_groupoid_intro} gives an overview of groupoids and groupoid $C^*$-algebras. Chapter \ref{chapter_graph_algebra_intro} gives an overview of the Cuntz-Krieger $C^*$-algebras of directed graphs and higher-rank graphs. The key research for this thesis is presented in Chapters \ref{chapter_strength_of_convergence} and \ref{chapter_categorising_higher-rank_graphs}, with the content of Chapter \ref{chapter_strength_of_convergence} published as \cite{Hazlewood-anHuef2011} and content of Chapter \ref{chapter_categorising_higher-rank_graphs} soon to be submitted for publication. The following Sections \ref{intro_strength_of_convergence} and \ref{intro_categorising_higher-rank_graphs} introduce the research Chapters \ref{chapter_strength_of_convergence} and \ref{chapter_categorising_higher-rank_graphs}, respectively. This chapter will end with the definition of various classes of $C^*$-algebras.

\section{Introduction: Strength of convergence in the orbit space of a groupoid}\label{intro_strength_of_convergence}

Suppose $H$ is a locally-compact Hausdorff group acting freely and continuously on a  locally-compact Hausdorff space $X$, so that $(H,X)$ is a free transformation group. In \cite[pp.~95--96]{Green1977} Green gives an example of a free non-proper action of $H=\RR$ on a subset $X$ of $\RR^3$; the non-properness comes down to the existence of $z\in X$, $\braces{x_n}\subset X$, and two sequences $\braces{s_n}$ and $\braces{t_n}$ in $H$ such that
\begin{enumerate}\renewcommand{\theenumi}{\roman{enumi}}\renewcommand{\labelenumi}{(\theenumi)}
\item $s_n^{-1}\cdot x_n\rightarrow z$ and $t_n^{-1}\cdot x_n\rightarrow z$; and 
\item $t_ns_n^{-1}\rightarrow\infty$ as $n\rightarrow\infty$, in the sense that $\braces{t_ns_n^{-1}}$ has no convergent subsequence.
\end{enumerate}
In  \cite[Definition~2.2]{Archbold-Deicke2005}), and subsequently in \cite[p.~2]{Archbold-anHuef2006}, the sequence $\{x_n\}$ is said to converge $2$-times in the orbit space to $z\in X$.  Each orbit $H\cdot x$  gives an induced  representation $\mathrm{Ind}\,\epsilon_{x}$ of the associated transformation-group $C^*$-algebra $C_0(X)\rtimes H$ which is irreducible, and the $k$-times convergence of $\{x_n\}$ in the orbit space to $z\in X$ translates into  statements about  various multiplicity numbers associated to $\mathrm{Ind}\,\epsilon_z$ in the spectrum of $C_0(X)\rtimes H$, as in  \cite[Theorem~2.5]{Archbold-Deicke2005}, \cite[Theorem~1.1]{Archbold-anHuef2006} and \cite[Theorem~2.1]{Archbold-anHuef2008}.

Upper and lower multiplicity numbers associated to irreducible representations $\pi$ of a $C^*$-algebra $A$ were introduced by Archbold \cite{Archbold1994} and extended to multiplicity numbers relative to a net of irreducible representations  by Archbold and Spielberg  \cite{Archbold-Spielberg1996}. 
The upper multiplicity $M_U(\pi)$ of $\pi$, for example, counts  `the number of nets of orthogonal equivalent pure states which can converge to a common pure state associated to $\pi$' \cite[p.~26]{aklss2001}.  The definition of $k$-times convergence and \cite[Theorem~2.5]{Archbold-Deicke2005} were very much motivated by a notion of $k$-times convergence in the dual space of a nilpotent Lie group \cite{Ludwig1990} and its connection with relative multiplicity numbers (see, for example,  \cite[Theorem~2.4]{aklss2001} and  \cite[Theorem~5.8]{Archbold-Ludwig-Schlichting2007}).

Theorem~1.1 of \cite{Archbold-anHuef2006} shows that the topological property of a  sequence $\{x_n\}$ converging $k$-times in the orbit space to $z\in X$ is equivalent to (1) a measure theoretic accumulation along the orbits $G\cdot x_n$ and (2) that the lower  multiplicity of $\mathrm{Ind}\,\epsilon_z$ relative to the sequence $\{\mathrm{Ind}\,\epsilon_{x_n}\}$ is at least $k$. In Chapter \ref{chapter_strength_of_convergence} we prove that the results of \cite{Archbold-anHuef2006} generalise to principal groupoids.  In our main arguments we have tried to preserve as much as possible the structure of those in  \cite{Archbold-anHuef2006}, although the arguments presented here are often more complicated in order to cope with the partially defined product in a groupoid and the set of measures that is a Haar system compared to the fixed Haar measure  used in the transformation-group case. Our theorems have led us to a new class of examples exhibiting $k$-times convergence in groupoids that are not based on transformation groups, thus justifying our level of generality. Given a row-finite directed graph $E$, Kumjian, Pask, Raeburn and Renault in \cite{kprr1997} used the set of all infinite paths in $E$ to construct an $r$-discrete groupoid $G_E$, called a {\em path groupoid}.  We prove that $G_E$ is principal if and only if $E$ contains no cycles (Proposition~\ref{prop_principal_iff_no_cycles}). We then exhibit principal $G_E$ with Hausdorff and non-Hausdorff orbits space, respectively, both with a $k$-times converging sequence  in the orbit space. In particular, our examples can be used to find a groupoid $G_E$ whose $C^*$-algebra has non-Hausdorff spectrum and distinct upper and lower multiplicity counts among its irreducible representations.

The directed graphs demonstrating $k$-times convergence in the orbit space of the associated path groupoids were carefully chosen so that the path groupoids would be principal, integrable and not Cartan. The search for examples of directed graphs with these properties provided the initial motivation for the research that is introduced in the following section. 

\section[Intro: Categorising higher-rank graphs using the path groupoid]{Introduction: Categorising the operator algebras of higher-rank graphs using the path groupoid}\label{intro_categorising_higher-rank_graphs}
Kumjian and Pask in \cite{Kumjian-Pask2000} introduced the notion of a higher-rank graph and, for row-finite higher-rank graph  $\Lambda$ without sources, defined the Cuntz-Krieger higher-rank graph $C^*$-algebra, denoted $C^*(\Lambda)$. These $C^*$-algebras were introduced as generalisations of the directed-graph $C^*$-algebras developed by Kumjian, Pask, Raeburn and Renault in \cite{kprr1997}. The class of higher-rank graph $C^*$-algebras is signifantly larger than the already large class of $C^*$-algebras associated to directed graphs (see Raeburn's book \cite{Raeburn2005} for an overview). As with the $C^*$-algebras associated to directed graphs, the study of the $C^*$-algebras of higher-rank graphs initially used the theory of groupoids. A more recent ``bare hands'' approach to studying higher-rank graphs has been developed that avoids the use of groupoids with the aim of simplifying arguments. These groupoid models have recently been seen to be a useful source of examples with their use to illustrate groupoid results in \cite{Clark-anHuef2012} and in Chapter \ref{chapter_strength_of_convergence}.

Suppose $\Lambda$ is a row-finite $k$-graph without sources and suppose $G_\Lambda$ is Kumjian and Pask's path groupoid associated to $\Lambda$; recall that $C^*(G_\Lambda)=C^*(\Lambda)$. In this part of the thesis we explore when $C^*(\Lambda)$ is postliminal, liminal and Fell, and when it has bounded- and countinuous-trace. Clark in \cite[Theorem~6.1]{Clark2007-2} showed that a groupoid $C^*$-algebra is liminal if and only if the orbits in the groupoid are closed and the $C^*$-algebras of each of the stability subgroups are liminal. In Section \ref{sec_liminal_postliminal} we establish a necessary and sufficient condition on $\Lambda$ for the orbits in $G_\Lambda$ to be closed; we use this to show exactly when $C^*(\Lambda)$ is liminal. Similarly we provide a condition on $\Lambda$ that is necessary and sufficient for the orbits in $G_\Lambda$ to be locally closed and use another of Clark's results, \cite[Theorem~7.1]{Clark2007-2}, to show precisely when $C^*(\Lambda)$ is postliminal. 

In Section \ref{sec_bded-trace_cts-trace_Fell} we will develop necessary and sufficient conditions on $\Lambda$ for $G_\Lambda$ to be integrable, to be Cartan and to be proper. With the assumption that $G_\Lambda$ is principal we will use results by Clark and an Huef in \cite{Clark-anHuef2008}, Clark in \cite{Clark2007} and Muhly and Williams in \cite{Muhly-Williams1990} to determine exactly when $C^*(\Lambda)$ has bounded trace, is Fell, and has continuous trace. In Section \ref{sec_boundary_paths_and_desourcification} we introduce the desourcification construction that Webster developed in \cite{Webster2011} based on Farthing's \cite{Farthing2008} and in Section \ref{sec_desourcification} we use the desourcification construction to provide similar characterisations of $C^*(\Lambda)$ while permitting the $k$-graphs to have sources. Finally in Section \ref{section_directed_graphs} it will be shown how these results can be simplified for directed graphs and how they relate to the directed graph characterisations developed by Ephrem in \cite{Ephrem2004} and by Deicke, Hong and Szyma{\'n}ski in \cite{dhs2003}.

It should be noted that Farthing, Muhly and Yeend in \cite{Farthing-Muhly-Yeend2005,Yeend2007} developed groupoid models of row-finite higher-rank graphs that may have sources. Developing characterisations for one of these groupoids would have made the section using the Farthing-Webster desourcification, Section \ref{sec_desourcification}, unnecessary. This extra level of groupoid generality does, however, come at a cost to their simplicity that makes them less attractive as a source of groupoid examples. By separating the groupoid arguments from the arguments dealing with row-finite higher-rank graphs that may have sources, Webster's desourcification also enables us to simplify our arguments.

\chapter{An introduction to groupoids}\label{chapter_groupoid_intro}	
This chapter provides a brief overview of the theory of groupoids, beginning with Renault's definition of a groupoid from \cite{Renault1980}. Examples are then given before introducing important technical results and defining the Haar system of a groupoid. A  fix for a mistake in Renault's \cite{Renault1980} is given (see Remark \ref{remark_problems_Renault_results}). Section \ref{section_representations_and_groupoid_algebras} describes induced representations and shows the construction of groupoid $C^*$-algebras. Sections \ref{section_transformation-group_groupoid} and \ref{sec_path_groupoid} define the groupoids that can be constructed from transformation groups and the infinite paths of directed graphs, respectively. Finally, Section \ref{sec_proper_Cartan_integrable} defines proper, Cartan and integrable groupoids, and shows the relationship between them and the classes of $C^*$-algebras from Section \ref{sec_algebra_classes}.

\section{Groupoids and Haar systems}
\begin{definition}[{\cite[Definition~1.1]{Renault1980}}]\label{def_groupoid}\notationindex{G@$G, G^{(2)}$}
A {\em groupoid} $G$ is a set endowed with a subset $G^{(2)}$ of $G\times G$, a map $(\alpha,\beta)\mapsto \alpha\beta$ from $G^{(2)}$ to $G$ and an involution $\alpha\mapsto \alpha^{-1}$ on $G$ such that:
\begin{enumerate}\renewcommand{\theenumi}{G\arabic{enumi}}
\item\label{G1} if $(\alpha,\beta)$ and $(\beta,\gamma)$ are in $G^{(2)}$, then so are $(\alpha\beta,\gamma)$ and $(\alpha,\beta\gamma)$, and $(\alpha\beta)\gamma=\alpha(\beta\gamma)$;
\item\label{G2} $(\alpha^{-1},\alpha)\in G^{(2)}$ for every $\alpha\in G$; and
\item\label{G3} $\alpha^{-1}(\alpha\beta)=\beta$ and $(\alpha\beta)\beta^{-1}=\alpha$ for every $(\alpha,\beta)\in G^{(2)}$.
\end{enumerate}
The elements of $G^{(2)}$ are called {\em composable pairs}\index{composable pairs}, the map $G^{(2)}\rightarrow G$ is called the {\em composition map}\index{composition map}, the involution is called the {\em inverse map}\index{inverse map} and $\alpha^{-1}$ is called the {\em inverse}\index{inverse} of $\alpha\in G$. The maps $r,s:G\rightarrow G$ defined by $r(\alpha)=\alpha\alpha^{-1}$ and $s(\alpha)=\alpha^{-1}\alpha$ for every $\alpha\in G$ are repectively called the {\em range}\index{range!in a groupoid} and {\em source}\index{source!in a groupoid} map. Then $r(\alpha)$ and $s(\alpha)$ are respectively called the {\em range} and {\em source} of $\alpha$. The range and source maps have a common image in $G$ which is called the {\em unit space}\index{unit space} and is denoted by $G^{(0)}$\notationindex{G0@$G^{(0)}$}.
\end{definition}

\begin{remark}\label{remark_composable_iff_range_source}
The point to the range and source maps is that they provide a simpler way of determining when composition makes sense: it follows from Definition \ref{def_groupoid} that $(\alpha,\beta)\in G^{(2)}$ if and only if $s(\alpha)=r(\beta)$.
\end{remark}

By the definition of the range and source maps we can see that $r(\alpha)\alpha=\alpha$ and $\alpha s(\alpha)=\alpha$ for every $\alpha\in G$, and that elements of $G^{(0)}$ are characterised by the equation $\alpha=\alpha^{-1}=\alpha^2$.

\begin{example}\label{example_group_is_groupoid}
Every group $G$ with identity $e$ can be considered to be a groupoid with $G^{(2)}=G\times G$ and $G^{(0)}=\braces{e}$.
\end{example}

\begin{example}\label{example_equivalence_relation_groupoid}
There is a natural groupoid that can be constructed from any set $X$ with an equivalence relation $\sim$. Let $G=\{(x,y)\in X\times X: x\sim y\}$ and 
\[G^{(2)}=\big\{\big((x,y), (y,z)\big): x,y,z\in X\text{ with }x\sim y\sim z\big\}.\]
Define composition by $(x,y)(y,z)=(x,z)$ and involution by $(x,y)^{-1}=(y,x)$. Then $G$ is a groupoid with $r(x,y)=(x,x)$, $s(x,y)=(y,y)$ and $G^{(0)}=\{(x,x):x\in X\}$.
\end{example}

Suppose $G$ is a groupoid. The subset $\{(r(\alpha),s(\alpha)):\alpha\in G\}$ of $G^{(0)}\times G^{(0)}$ is an equivalence relation on $G^{(0)}$ that is referred to as the {\em natural equivalence relation}\index{natural equivalence relation on $G^{(0)}$} on $G^{(0)}$. The {\em orbit}\index{orbit} of $u\in G^{(0)}$ is the equivalence class of $u$ with respect to the natural equivalence relation on $G^{(0)}$ and is denoted by $[u]$. A subset $A$ of $G^{(0)}$ is {\em $G$-invariant}\index{G-invariant@$G$-invariant} if it is saturated with respect to the natural orbit equivalence relation on $G^{(0)}$. In other words, $A$ is $G$-invariant if it is a union of orbits. A groupoid $G$ is {\em principal}\index{principal groupoid} if $\alpha,\beta\in G$ are equal whenever $r(\alpha)=r(\beta)$ and $s(\alpha)=s(\beta)$.

\begin{example}\label{example_equivalence_relation_groupoid_properties}
Suppose $X=\{1,2,3,4,5\}$ and $\sim$ is the equivalence relation on $X$ with equivalence classes $\{1,2,3\}$ and $\{4,5\}$. Let $G$ be the groupoid as in Example \ref{example_equivalence_relation_groupoid}. Then $[(1,1)]=\{(1,1),(2,2),(3,3)\}$ and $[(4,4)]=\{(4,4),(5,5)\}$. The subset $\{(4,4),(5,5)\}$ of $G^{(0)}$ is $G$-invariant whereas $\{(1,1)\}$ is not. The groupoid $G$ is principal since the range and source of each $\alpha\in G$ completely determines $\alpha$: If $r(\alpha)=(x,x)$ and $s(\alpha)=(y,y)$, then $\alpha=(x,y)$.
\end{example}

A {\em groupoid homomorphism}\index{groupoid!homomorphism} from a groupoid $G$ to a groupoid $F$ is a map $\phi:G\rightarrow F$ such that for every $(\alpha,\beta)\in G^{(2)}$, we have $\big(\phi(\alpha),\phi(\beta)\big)\in F^{(2)}$ and $\phi(\alpha\beta)=\phi(\alpha)\phi(\beta)$. If $A$ is a non-empty subset of $G^{(0)}$, we define $G^A=r^{-1}(A)$, $G_A=s^{-1}(A)$ and $G|_A=r^{-1}(A)\cap s^{-1}(A)$. If $u\in G^{(0)}$ then we define $G_u=G_{\{u\}}$, $G^u=G^{\{u\}}$ and $G|_u=G|_{\{u\}}$.

\begin{prop}[{\cite[Proposition~2.8]{Muhly-book}}]\label{prop_Muhly_2.8}
Suppose $G$ is a groupoid and $A$ is a non-empty subset of $G^{(0)}$. Then $G|_A$ is a groupoid with $G|_A^{(2)}=G^{(2)}\cap (G|_A\times G|_A)$ and operations that are the restrictions of those on $G$.
\end{prop}

\begin{example}
Suppose $G$ is the groupoid from Example \ref{example_equivalence_relation_groupoid_properties}. Let $A$ be the subset $\{(1,1),(2,2),(3,3),(4,4)\}$ of $G^{(0)}$ so that  $G|_A$ is a groupoid by Proposition \ref{prop_Muhly_2.8}. If $F$ is the groupoid constructed as in Example \ref{example_equivalence_relation_groupoid} with the set $\{1,2,3,4\}$ and equivalence classes $\{1,2,3\}$ and $\{4\}$, then $G|_A=F$.
\end{example}

The {\em stability subgroups}\index{stability subgroup} of a groupoid $G$ are the groups of the form $G|_u=r^{-1}(u)\cap s^{-1}(u)$ for $u\in G^{(0)}$; $G$ is principal if and only if each of the stability subgroups is trivial, that is, if $G|_u=\{u\}$ for every $u\in G^{(0)}$. Stability subgroups are also commonly known as {\em isotropy subgroups}\index{isotropy subgroup}.

\begin{definition}[{\cite[Definition~2.25]{Muhly-book}}]\label{def_topological_groupoid}
Suppose $G$ is a groupoid with a topology on the set $G$ and let $G^{(2)}$ have the subspace topology from $G\times G$ when equipped with the product topology. Then $G$ is called a {\em topological groupoid}\index{topological groupoid}\index{groupoid!topology} if the composition and inversion maps are continuous.
\end{definition}
\begin{remark}\label{remark_following_topological_groupoid_def}
It follows that the range and source maps of a topological groupoid are continuous. Then, since any continuous involution is open, the inversion map in a topological groupoid is a homeomorphism. It follows from the inverse map being open that the range map is open if and only if the source map is open. Whenever we refer to a topology on a groupoid, we assume that this is a topological groupoid unless otherwise stated.
\end{remark}
\begin{example}
A locally compact group is a locally compact groupoid in the sense of Example \ref{example_group_is_groupoid}.
\end{example}

To construct a $C^*$-algebra from a groupoid we need an analogue of the Haar measure on a group.
\needspace{6\baselineskip}
\begin{definition}[{\cite[Definition~2.28]{Muhly-book}}]\label{def_Haar_system}Suppose $G$ is a second countable, locally compact, Hausdorff groupoid. A {\em right Haar system}\index{Haar system} on $G$ is a set $\{\lambda_x:x\in G^{(0)}\}$ of non-negative Radon measures on $G$ such that
\begin{enumerate}\renewcommand{\theenumi}{H\arabic{enumi}}
\item\label{H1} $\mathrm{supp}\,\lambda_x=G_x$ $\big(=s^{-1}(\{x\})\big)$\quad for all $x\in G^{(0)}$;
\item\label{H2} for $f\in C_c(G)$, the function $x\mapsto \int f\,d\lambda_x$ on $G^{(0)}$ is in $C_c(G^{(0)})$; and
\item\label{H3} for $f\in C_c(G)$ and $\gamma\in G$, 
\[\int f(\alpha\gamma)\,d\lambda_{r(\gamma)}(\alpha)=\int f(\alpha)\,d\lambda_{s(\gamma)}(\alpha).\]
\end{enumerate}
We will refer to \eqref{H2} as the {\em continuity of the Haar system} and to \eqref{H3} as {\em Haar-system invariance}. The collection $\{\lambda^x:x\in G^{(0)}\}$ of measures where $\lambda^x(E):=\lambda_x(E^{-1})$ is a {\em left Haar system}, which is a system of measures such that $\mathrm{supp}\,\lambda^x=G^x$ and, for $f\in C_c(G)$, $x\mapsto \int f\,d\lambda^x$ is continuous and $\int f(\gamma\alpha)\,d\lambda^{s(\gamma)}(\alpha)=\int f(\alpha)\,d\lambda^{r(\gamma)}(\alpha)$. Given that we can easily convert a right Haar system $\{\lambda_x\}$ into a left Haar system $\{\lambda^x\}$ and vice versa, we will simply refer to a {\em Haar system} $\lambda$ and use subscripts to refer to elements of the right Haar system $\{\lambda_x\}$ and superscripts to refer to elements of the left Haar system $\{\lambda^x\}$.
\end{definition}

The next few lemmas establish some properties of Haar systems which we will use frequently.
\begin{lemma}
Let $G$ be a second countable, locally compact, Hausdorff groupoid with a Haar system $\lambda$. If $K\subset G$ is compact and $\gamma\in G$ with $s(\gamma)=x$ and $r(\gamma)=y$, then $\lambda_x(K\gamma)=\lambda_y(K)$ and $\lambda^x(\gamma^{-1} K)=\lambda^y(K)$.
\begin{proof}
Let $\braces{U_n}$ be a decreasing sequence of open neighbourhoods of $K$ so that if $\gamma\in G\backslash K$, then $\gamma\notin U_n$ eventually. $G$ is locally compact, so there is a compact neighbourhood of $K$, and we can assume that $\overline{U_0}$ is compact. By Urysohn's Lemma for each $n$ there exists $f_n\in C_c(G)$ with range in $[0,1]$ such that $f_n$ is identically $1$ on $K$ and $\mathrm{supp}\, f_n\subseteq U_n$. For each $n$, define $g_n:G\rightarrow \CC$ by $g_n(\alpha)=f_n(\alpha\gamma^{-1})$ for each $\alpha\in G$, so that each $g_n\in C_c(G)$ and $g_n(\alpha)\rightarrow \chi_K(\alpha\gamma^{-1})$ as $n\rightarrow\infty$ for every $\alpha\in G$. Let $h = \chi_{U_0\gamma}$, noting that $\mathrm{supp}\, h$ is compact since $U_0$ is relatively compact and the composition map is continuous. Then
\[
g_n(\alpha) = f_n(\alpha\gamma^{-1}) \le \chi_{U_n}(\alpha\gamma^{-1}) \le \chi_{U_0}(\alpha\gamma^{-1}) = \chi_{U_0\gamma}(\alpha) = h(\alpha)
\]
for every $\alpha\in G$ and $n\in\NN$.

Since $h$ and each $g_n$ have compact support with $g_n(\alpha)\le h(\alpha)$ for each $\alpha$, and since $g_n(\alpha)\rightarrow \chi_K(\alpha\gamma^{-1})$ for each $\alpha$, by the Dominated Convergence Theorem we have
\begin{align*}
\int f_n(\alpha&\gamma^{-1})\, d\lambda_x(\alpha)=\int g_n(\alpha)\, d\lambda_x(\alpha)\\
&\rightarrow\int \chi_K(\alpha\gamma^{-1})\, d\lambda_x(\alpha) = \int \chi_{K\gamma}(\alpha)\, d\lambda_x(\alpha)=\lambda_x(K\gamma).
\end{align*}
By Haar-system invariance \eqref{H3} we can see that
\[
\int f_n(\alpha\gamma^{-1})\, d\lambda_x(\alpha)=\int f_n(\alpha)\, d\lambda_y(\alpha)\rightarrow \int \chi_K(\alpha)\, d\lambda_y(\alpha)=\lambda_y(K),
\]
and it follows that $\lambda_x(K\gamma)=\lambda_y(K)$. Since $\lambda^x(E)=\lambda_x(E^{-1})$ for every measurable set $E$, we can deduce that
\[
\lambda^x(\gamma^{-1} K)=\lambda_x(K^{-1}\gamma)=\lambda_y(K^{-1})=\lambda^y(K).\qedhere
\]
\end{proof}
\end{lemma}

\begin{lemma}\label{astrids_lim_sup}
Let $G$ be a second countable, locally compact, Hausdorff groupoid with a Haar system $\lambda$ and let $K$ be a compact subset of $G$. If $\{x_n\}\subset G^{(0)}$ is a sequence that converges to $z\in G^{(0)}$, then
\[
\underset{\scriptstyle n}{\lim\,\sup}\,\lambda_{x_n}(K)\le\lambda_z(K).
\]
\begin{proof}
Fix $\epsilon>0$. By the outer regularity of $\lambda_z$, there exists an open neighbourhood $U$ of $K$ such that
\[
\lambda_z(K)\le \lambda_z(U)<\lambda_z(K)+\epsilon/2.
\]
By Urysohn's Lemma there exists $f\in C_c(G)$ with $0\le f\le 1$ such that $f$ is identically one on $K$ and zero off $U$. In particular we have
\begin{equation}\label{measure_of_f_near_measure_of_K}
\lambda_z(K)\le\int f\,d\lambda_z<\lambda_z(K)+\epsilon/2.
\end{equation}

The continuity of the Haar system \eqref{H2} implies $\int f\,d\lambda_{x_n}\rightarrow \int f\,d\lambda_z$, so there exists $n_0$ such that $n\ge n_0$ implies
\[
\int f\,d\lambda_z-\epsilon/2 < \int f\,d\lambda_{x_n}<\int f\,d\lambda_z +\epsilon/2.
\]
By our choice of $f$ we have $\lambda_{x_n}(K)\le \int f\,d\lambda_{x_n}$, so
\[
\lambda_{x_n}(K)\le\int f\,d\lambda_{x_n}<\int f\,d\lambda_z+\epsilon/2.
\]
Combining this with \eqref{measure_of_f_near_measure_of_K} enables us to observe that for $n\ge n_0$, we have $\lambda_{x_n}(K)<\lambda_z(K)+\epsilon$, completing the proof.
\end{proof}
\end{lemma}
\begin{lemma}\label{corollary_astrid_lemma}
Let $G$ be a second countable, locally compact, Hausdorff groupoid with a Haar system $\lambda$ and let $K$ be a compact subset of $G$. For every $\epsilon>0$ and $z\in G^{(0)}$ there exists a neighbourhood $U$ of $z$ in $G^{(0)}$ such that $x\in U$ implies $\lambda_x(K)<\lambda_z(K)+\epsilon$.
\begin{proof}
Fix $\epsilon>0$ and $z\in G^{(0)}$. Let $\{U_n\}$ be a decreasing neighbourhood basis for $z$ in $G^{(0)}$. If our claim is false, then each $U_n$ contains an element $x_n$ such that $\lambda_{x_n}(K)\ge\lambda_z(K)+\epsilon$. But since each $x_n\in U_n$, $x_n\rightarrow z$, and so by Lemma \ref{astrids_lim_sup} there exists $n_0$ such that $n\ge n_0$ implies $\lambda_{x_n}(K)<\lambda_z(K)+\epsilon$, a contradiction.
\end{proof}
\end{lemma}

Haar systems do not always exist. A following result, Proposition \ref{Renault_I_2_8i_sub}, can be used to show that some groupoids admit a Haar system. The next lemma is a small result that is frequently used in the groupoid literature.
\begin{lemma}\label{lemma_urysohn_corollary}
Let $G$ be a second countable, locally compact, Hausdorff groupoid with a Haar system $\lambda$. Then for every $x\in G^{(0)}$ there exists $f\in C_c(G)$ such that $\int_G f\, d\lambda_x> 0$ and the range of $f$ is contained in $[0,1]$.
\begin{proof}
Let $K$ be a compact neighbourhood of $x$ in $G$ and recall that the support of a Radon measure is the complement of the union of all open sets with measure zero. Then, since $x\in \mathrm{int}\, K\cap G_x=\mathrm{int}\, K\cap \mathrm{supp}\,\lambda_x$ and $\lambda_x$ is non-negative, we must have $\lambda_x(K)\ge\lambda_x(\mathrm{int}\, K)>0$. By Urysohn's Lemma (see, for example, \cite[Proposition~1.7.5]{Pedersen1989}) there exists $f\in C_c(G)$ such that $f$ is identically $1$ on $K$ and the range of $f$ is contained in $[0,1]$. Then $\int_G f\,d\lambda_x\ge \int_G \chi_K\, d\lambda_x>0$.
\end{proof}
\end{lemma}

\begin{prop}[{\cite[Proposition~I.2.4]{Renault1980}, \cite[Proposition~2.2.1]{Paterson1999}}]\label{prop_RenaultI1.2.4}
Suppose $G$ is a second countable, locally compact, Hausdorff groupoid with a Haar system. Then the range and source maps of $G$ are open.
\end{prop}

\begin{lemma}[{\cite[p. 361]{Ramsay1990}}]\label{lemma_composition_map_open}
If a locally compact, Hausdorff groupoid has open range and source maps, then the composition map is open.
\begin{proof}
Suppose $G$ is a locally compact, Hausdorff groupoid with open range map. Let $A$ and $B$ be open subsets of $G$. It suffices to show that $AB$ is open. If $AB$ is empty, then $AB$ is open and we are done. Suppose $AB$ is non-empty, so that there exist $\alpha\in A$ and $\beta\in B$ such that $s(\alpha)=r(\beta)$. Let $\circ:G^{(2)}\rightarrow G$ be the composition map. Since the composition map is continuous (Definition~\ref{def_topological_groupoid}), by the definition of the topology on $G^{(2)}$ there exist open neighbourhoods $H,K$ of $\alpha^{-1}, \alpha\beta$ respectively, such that $(H\times K)\cap G^{(2)}$ is contained in the open set $\circ^{-1}(B)$. Note that $HK$ is a subset of $B$.

Let $U=H^{-1}\cap A$, noting that $U$ is an open neighbourhood of $\alpha$ since the inverse map is open (Remark~\ref{remark_following_topological_groupoid_def}). Now let $V=r^{-1}\big(r(U)\big)\cap K$, noting that $V$ is an open neighbourhood of $\alpha\beta$ since the range map is open and continuous. The set $U^{-1}V$ is a subset of $B$ as $U^{-1}\subset H$, $V\subset K$ and $HK$ is a subset of $B$. Fix $\gamma\in V$. Since $r(V)\subset r(U)=s(U^{-1})$ and $U^{-1}V\subset B$, there exist $\delta\in U$ and $\epsilon\in B$ such that $\delta^{-1}\gamma = \epsilon$. Then $\gamma=\delta\epsilon$, so $V\subset UB\subset AB$. The set $V$ is open with $\alpha\beta\in V\subset AB$ so, since $\alpha\beta$ was chosen arbitrarily in $AB$, the set $AB$ is open and thus the composition map is open.
\end{proof}
\end{lemma}

\begin{lemma}\label{lemma_alpha_mapsto_gamma_alpha_homeomorphism}
Suppose $G$ is a topological groupoid. Then, for every $\gamma\in G$, 
\begin{enumerate}\renewcommand{\theenumi}{\roman{enumi}}
\item\label{lemma_alpha_mapsto_gamma_alpha_homeomorphism_1} the map from $G^{s(\gamma)}$ to $G^{r(\gamma)}$ such that $\alpha\mapsto\gamma\alpha$, and
\item\label{lemma_alpha_mapsto_gamma_alpha_homeomorphism_2} the map from $G_{r(\gamma)}$ to $G_{s(\gamma)}$ such that $\alpha\mapsto\alpha\gamma$
\end{enumerate}
are homeomorphisms.
\begin{proof}
Fix $\gamma\in G$ and note that it follows immediately from the definition of a groupoid (Definition \ref{def_groupoid}) that the map $\alpha\mapsto\gamma\alpha$ is injective and surjective. To see that the map $\alpha\mapsto\gamma\alpha$ is continuous, first suppose that $\{\alpha_i\}_{i\in I}$ is a net in $G^{s(\gamma)}$ that converges to some $\alpha\in G^{s(\gamma)}$. Since $\{\gamma\}_{i\in I}$ trivially converges to $\gamma$, it follows that $\{(\gamma,\alpha_i)\}_{i\in I}$ converges to $(\gamma,\alpha)$ in $G\times G$ with the product topology. Since $s(\gamma)=r(\alpha)$, by Remark \ref{remark_composable_iff_range_source} we can see that $(\gamma,\alpha)\in G^{(2)}$. Similarly each $(\gamma,\alpha_i)$ is in $G^{(2)}$, so $\{(\gamma,\alpha_i)\}_{i\in I}$ converges to $(\gamma,\alpha)$ in $G^{(2)}$. The composition map of a groupoid is continuous, so $\gamma\alpha_i$ converges to $\gamma\alpha$ in $G$, and the map $\alpha\mapsto\gamma\alpha$ is continuous. Since the inverse of $\alpha\mapsto\gamma\alpha$ is $\alpha\mapsto\gamma^{-1}\alpha$, the same argument shows that this inverse map is continous, thus showing that $\alpha\mapsto\gamma\alpha$ is a homeomorphism. A similar argument can be used to establish the second component of this lemma.
\end{proof}
\end{lemma}

A topological groupoid $G$ is {\em $r$-discrete}\index{$r$-discrete groupoid} if $G^{(0)}$ is an open subset of $G$.

\begin{lemma}[{\cite[Lemma~I.2.7(i)]{Renault1980}}]\label{lemma_Renault_I.2.7_1}
Let $G$ be an $r$-discrete groupoid. For every $u\in G^{(0)}$, the subspace topologies on $G_u$ and $G^u$ are discrete.
\begin{proof}
Fix $u\in G^{(0)}$ and $\gamma\in G^u$. Then, since $G^{(0)}$ is open, the set $G^{(0)}\cap G^{s(\gamma)}=\{s(\gamma)\}$ is an open subset of $G^{s(\gamma)}$. Since $\alpha\mapsto\gamma\alpha$ is a homeomorphism of $G^{s(\gamma)}$ onto $G^{r(\gamma)}$ (Lemma \ref{lemma_alpha_mapsto_gamma_alpha_homeomorphism}), it follows that $\{\gamma s(\gamma)\}=\{\gamma\}$ is an open subset of $G^{r(\gamma)}=G^u$. A similar argument can be used to show that $G_u$ is discrete.
\end{proof}
\end{lemma}

\begin{remark}\label{remark_problems_Renault_results}
The proofs of two useful results, \cite[Lemma~I.2.7, Proposition~I.2.8]{Renault1980}, omit details that appear to be important. In \cite[Lemma~I.2.7]{Renault1980} it is claimed that if an $r$-discrete groupoid has a Haar system, then that Haar system is `essentially the system of counting measures'. Component (ii) of the proof of this lemma constructs a system of counting measures from the given Haar system however does not establish the invariance \eqref{H3} of this system.

In \cite[Proposition~I.2.8]{Renault1980} it is claimed that if the range map of a groupoid is a local homeomorphism, then the groupoid is $r$-discrete and admits a Haar system. The (iv)$\implies$(i) component of the proof of this proposition constructs a system of measures, however the continuity \eqref{H2} of this system is not established, which is necessary to have a Haar system.
\end{remark}

We use the next result as an alternative to \cite[Lemma~I.2.7, Proposition~I.2.8]{Renault1980}; its proof uses ideas from \cite[Lemma~I.2.7, Proposition~I.2.8]{Renault1980}.
\begin{prop}[{based on \cite[Lemma~I.2.8]{Renault1980}}]\label{Renault_I_2_8i_sub}
Suppose $G$ is a second countable, locally compact, Hausdorff groupoid. Then the following are equivalent.
\begin{enumerate}
\item\label{Renault_I_2_8i_sub_i} $G$ is $r$-discrete and admits a Haar system of counting measures;
\item\label{Renault_I_2_8i_sub_ii} the source map is a local homeomorphism;
\item\label{Renault_I_2_8i_sub_iii} the range map is a local homeomorphism; and
\item\label{Renault_I_2_8i_sub_iv} $G$ is $r$-discrete and the composition map is a local homeomorphism.
\end{enumerate}
\end{prop}
The proof of this proposition will use the following technical lemma.

\begin{lemma}\label{lemma_used_for_Renault_I_2_8i_sub_ii_implies_i}
Suppose $G$ is a locally compact, Hausdorff groupoid and that $U$ is an open set in $G$ such that $s|_U$ is a homeomorphism onto the open set $s(U)$. Then for every $\gamma\in U$ there is an open neighbourhood $V$ of $\gamma$ in $G$ such that $\overline{V}\subset U$ and $s(\overline{V})$ is closed in $G^{(0)}$.
\begin{proof}
Fix $\gamma\in U$. Since $G$ is locally compact and Hausdorff, $\gamma$ has a base of compact neighbourhoods. Then there is a compact neighbourhood $N$ of $\gamma$ in $G$ such that $N\subset U$. Let $V=\mathrm{int}\,N$, so that $V$ is an open neighbourhood of $\gamma$ in $G$ with $\overline{V}\subset U$. The set $s(\overline{V})$ is compact since $s$ is continuous. Finally, since $G$ and $G^{(0)}$ are Hausdorff, the compact sets $\overline{V}$ and $s(\overline{V})$ are closed.
\end{proof}
\end{lemma}

To prove Proposition \ref{Renault_I_2_8i_sub}, it is easiest to deal with either left or right Haar systems; here we choose right Haar systems due to personal preference. In this light, condition \eqref{Renault_I_2_8i_sub_i} is most closely related to \eqref{Renault_I_2_8i_sub_ii}, so we will begin by establishing their equivalence. Conditions \eqref{Renault_I_2_8i_sub_ii} and \eqref{Renault_I_2_8i_sub_iii} are very closely related, so demonstrating their equivalence will follow. For the remainder of the proof we will show that \eqref{Renault_I_2_8i_sub_ii} and \eqref{Renault_I_2_8i_sub_iv} are equivalent.
\begin{proof}[Proof of Proposition \ref{Renault_I_2_8i_sub}]
\eqref{Renault_I_2_8i_sub_i}$\implies$\eqref{Renault_I_2_8i_sub_ii}. Suppose $G$ is $r$-discrete and $\lambda$ is a Haar system of counting measures on $G$. Since $s$ is continuous by the definition of a topological groupoid and open by Proposition \ref{prop_RenaultI1.2.4}, to show that $s$ is a local homeomorphism it suffices so show that it is locally injective. Fix $\gamma\in G$ and let $K$ be a compact neighbourhood of $\gamma$ in $G$. Since $G$ is $r$-discrete, $G_{s(\gamma)}$ is discrete by Lemma \ref{lemma_Renault_I.2.7_1}, so $K$ and $G_{s(\gamma)}$ must have finite a intersection; write $K\cap G_{s(\gamma)}=\{\gamma_1,\gamma_2,\ldots,\gamma_p\}$. For every $\gamma_i$ with $\gamma_i\ne \gamma$, there is a compact neighbourhood $K_i$ of $\gamma$ such that $\gamma_i\notin K_i$ and $K_i\subset K$. By replacing $K$ by $K_1\cap K_2\cap\ldots \cap K_p$, we can assume that $K\cap G_{s(\gamma)}=\{\gamma\}$, and so $\lambda_{s(\gamma)}(K)=1$. 

By Lemma \ref{corollary_astrid_lemma} there is an open neighbourhood $U$ of $s(\gamma)$ in $G^{(0)}$ such that $y\in U$ implies $\lambda_y(K)<\lambda_{s(\gamma)}(K)+0.1=1.1$. Let $K'$ be a compact neighbourhood of $\gamma$ with $K'\subset s^{-1}(U)\cap K$. Replace $K$ by $K'$ so that $\lambda_y(K)<1.1$ for every $y\in s(K)$. For every $y\in s(K)$ there exists $\delta\in K$ such that $s(\delta)=y$, so $\{\delta\}\subset G_y\cap K$ and $\lambda_y(K)\ge 1$. Thus $\lambda_y(K)=1$ for every $y\in s(K)$. So each $G_y\cap K$ has exactly one element and then the restriction $s|_K$ is injective. Thus $s$ is a local homeomorphism.

\eqref{Renault_I_2_8i_sub_ii}$\implies$\eqref{Renault_I_2_8i_sub_i}. Suppose \eqref{Renault_I_2_8i_sub_ii}. 
We will first show that $G$ is $r$-discrete. By \eqref{Renault_I_2_8i_sub_ii}, for any $x\in G^{(0)}$ there is an open neighbourhood $U_x$ of $x$ such that $s|_{U_x}$ is a homeomorphism. We claim that $U_xU_x^{-1}\subset G^{(0)}$ for each $x\in G^{(0)}$. Fix $y\in G^{(0)}$ and $\gamma\in U_yU_y^{-1}$. There exist $\alpha,\beta\in U_y$ with $s(\alpha)=s(\beta)$ such that $\gamma=\alpha\beta^{-1}$. Then $\alpha=\beta$ since $s|_{U_y}$ is injective, so $\gamma=\alpha\alpha^{-1}=r(\alpha)$ and so $U_yU_y^{-1}\subset G^{(0)}$. Note that $U_yU_y^{-1}$ is open since the composition and inverse maps are open (Lemma \ref{lemma_composition_map_open} and Remark \ref{remark_following_topological_groupoid_def}). Then $\cup_{x\in G^{(0)}}U_xU_x^{-1}$ is open and is all of $G^{(0)}$, so $G$ is $r$-discrete.

Since $G$ is $r$-discrete, the spaces $G_x$ for $x\in G^{(0)}$ are discrete by Lemma \ref{lemma_Renault_I.2.7_1}. For each $x\in G^{(0)}$ and every Borel set $E\subset G$, let $\lambda_x(E)$ be the possibly infinite number of elements in $G_x\cap E$. We claim that $\{\lambda_x:x\in G^{(0)}\}$ is a right Haar system for $G$. Since each $G_x$ is discrete, it is easy to see that each $\lambda_x$ is a Radon measure and it remains to show that \eqref{H1}, \eqref{H2} and \eqref{H3} are satisfied. Condition \eqref{H1} is an immediate consequence of the definition of the system $\{\lambda_x\}$.

Fix $f\in C_c(G)$. Before establishing \eqref{H2} and \eqref{H3} we use that $s$ is a local homeomorphism to make an assumption about $f$. Since $s$ is a local homeomorphism, for every $\gamma\in G$ there exists an open neighbourhood $U_\gamma$ of $\gamma$ in $G$ such that $s|_{U_\gamma}$ is a homeomorphism onto the open set $s(U_\gamma)$. By Lemma \ref{lemma_used_for_Renault_I_2_8i_sub_ii_implies_i} for each $\gamma$ there exists an open neighbourhood $V_\gamma$ of $\gamma$ in $G$ such that $\overline{V_\gamma}\subset U_\gamma$ and $s(\overline{V_\gamma})$ is closed in $G^{(0)}$. Now $\{V_\gamma:\gamma\in \mathrm{supp}\, f\}$ is an open cover of the compact space $\mathrm{supp}\, f$ so there exists a finite subset $\{\gamma_1,\gamma_2,\ldots,\gamma_p\}$ of $\mathrm{supp}\, f$ such that $\mathrm{supp}\, f\subset\cup_{i=1}^p V_{\gamma_i}$. By a partition of unity (see, for example, \cite[Proposition~1.7.12]{Pedersen1989}) there exist $f_1,f_2,\ldots,f_p\in C_c(G)$ such that $f(\gamma)=\sum_{i=1}^p f_i(\gamma)$ for every $\gamma\in G$ and $f_i(\gamma)=0$ for every $\gamma\notin V_{\gamma_i}$. To establish \eqref{H2} and \eqref{H3} we may thus assume that $f=f_j$ for some $1\le j\le p$. Let $V=V_{\gamma_j}$ and $U=U_{\gamma_j}$.

We will now establish \eqref{H2}. Suppose $\{x_i\}_{i\in I}$ is a net in $G^{(0)}$ that converges to some $x\in G^{(0)}$. For \eqref{H2} we need to show that $\int f\,d\lambda_{x_i}\rightarrow\int f\, d\lambda_x$. If $x$ is not in the open subset $s(U)$ of $G^{(0)}$, then $x_i$ is eventually not in the closed subset $s(\overline{V})$ of $G^{(0)}$ and thus not in $s(V)$. Since $f(\gamma)=0$ whenever $\gamma\notin V$, it follows that $\int f\, d\lambda_x=0=\int f\,d\lambda_{x_i}$ eventually. Now suppose $x$ is in the open set $s(U)$. Then $x_i$ is eventually in $s(U)$ so we may assume that $x_i\in s(U)$ for every $i\in I$. Let $\gamma_i=s|_{U}^{-1}(x_i)$ for $i\in I$ and let $\gamma=s|_{U}^{-1}(x)$. Since $x_i\rightarrow x$ in $s(U)$, the map $s|_{U}^{-1}$ is continuous and so $\{\gamma_i\}_{i\in I}$ converges to $\gamma$ in $U$. Furthermore, since $U$ is an open subset of $G$, $\{\gamma_i\}_{i\in I}$ converges to $\gamma$ in $G$. Now, for every $i\in I$,
\begin{align*}
\int f(\alpha)\, d\lambda_{x_i}(\alpha)&=\int f(\alpha)\chi_{\{\gamma_i\}}(\alpha)\, d\lambda_{x_i}(\alpha)\\
&=f(\gamma_i)\lambda_{x_i}(\{\gamma_i\})\\
&= f(\gamma_i)\quad(\lambda_{x_i}\text{ is the counting measure on }G_x)
\end{align*}
and similarly $\int f\,d\lambda_x=f(\gamma)$. Since $\gamma_i\rightarrow\gamma$ in $G$ and $f\in C_c(G)$,
\[
\int f\,d\lambda_{x_i}=f(\gamma_i)\rightarrow f(\gamma)=\int f\,d\lambda_x.
\]
Thus $x\mapsto \int f\,d\lambda_x$ is continuous. 

To show that the map $x\mapsto \int f\, d\lambda_x$ from \eqref{H2} has compact support, first suppose $x\in G^{(0)}$ satisfies $\int f\,d\lambda_x\ne 0$. Since $f\in C_c(G)$, the space $s(\mathrm{supp}\,f)$ is compact so it will suffice to show that $x\in s(\mathrm{supp}\, f)$. The intersection $G_x\cap U$ must be non-empty since $\mathrm{supp}\,\lambda_x=G_x$ and $f(\alpha)\ne 0$ requires $\alpha\in U$. Let $\eta=s|_{U}^{-1}(x)$. Then 
\[
\int f\,d\lambda_x=\int f(\alpha)\chi_{\{\eta\}}(\alpha)\,d\lambda_x(\alpha)=f(\eta)\lambda_x(\{\eta\})=f(\eta),\]
so $f(\eta)$ is non-zero and thus $\eta\in \mathrm{supp}(f)$. Then $s(\eta)=x\in s(\mathrm{supp}\,f)$ so the support of $x\mapsto \int f\,d\lambda_x$ is compact, establishing \eqref{H2}.

Fix $\alpha\in G$. To establish \eqref{H3} we will show that \[\int f(\beta\alpha)\,d\lambda_{r(\alpha)}(\beta)=\int f(\beta)\,d\lambda_{s(\alpha)}(\beta).\]
If $s(\alpha)\notin s(U)$, then $\beta\alpha\notin U$ for all $\beta\in G_{r(\alpha)}$ so $\int f(\beta\alpha)\,d\lambda_{r(\alpha)}(\beta)=0$ and similarly $\int f(\beta)\,d\lambda_{s(\alpha)}(\beta)=0$. Suppose $s(\alpha)\in s(U)$ and let $\delta=s|_{U}^{-1}(s(\alpha))$. The integrand of $\int f(\beta\alpha)\,d\lambda_{r(\alpha)}(\beta)$ is zero unless $\beta\alpha\in U$ and hence $\beta\alpha=\delta$. Now we have
{\allowdisplaybreaks\begin{align*}
\int f(\beta\alpha)\,d\lambda_{r(\alpha)}(\beta)&=\int f(\beta\alpha)\chi_{\{\delta\}}(\beta\alpha)\,d\lambda_{r(\alpha)}(\beta)\\
&=\int f(\delta)\chi_{\{\delta\alpha^{-1}\}}(\beta)\,d\lambda_{r(\alpha)}(\beta)\\
&=f(\delta)\lambda_{r(\alpha)}(\{\delta\alpha^{-1}\}) = f(\delta).
\end{align*}
}The integrand of $\int f(\beta)\,d\lambda_{s(\alpha)}(\beta)$ is zero unless $\beta\in U$ and hence $\beta=\delta$, so we have
{\allowdisplaybreaks
\begin{align*}
\int f(\beta)\,d\lambda_{s(\alpha)}(\beta)&=\int f(\beta)\chi_{\{\delta\}}(\beta)\,d\lambda_{s(\alpha)}(\beta)\\
&=f(\delta)\lambda_{s(\alpha)}(\{\delta\}) = f(\delta),
\end{align*}
}establishing \eqref{H3}. Thus $\{\lambda_x:x\in G^{(0)}\}$ is a right Haar system of counting measures for $G$, establishing \eqref{Renault_I_2_8i_sub_i}.

\eqref{Renault_I_2_8i_sub_ii}$\implies$\eqref{Renault_I_2_8i_sub_iii}.
Suppose \eqref{Renault_I_2_8i_sub_ii} and fix $\gamma\in G$. Since $s$ is a local homeomorphism there is an open neighbourhood $U$ of $\gamma^{-1}$ such that $s|_U$ is a homeomorphism onto the open set $s(U)$. Since the inverse map is open, $U^{-1}$ is an open neighbourhood of $\gamma$. Note that $r(U^{-1})$ is open since $r(U^{-1})=s(U)$. To show that $r|_{U^{-1}}$ is a homeomorphism onto $r(U^{-1})$, we need to show that $r|_{U^{-1}}$ is open and injective. Suppose $V\subset U^{-1}$ is open. Then $V^{-1}\subset U$ is open, so $s(V^{-1})$ is open since $s|_{U}$ is open.  It follows that $r|_{U^{-1}}$ is open since $r(V)=s(V^{-1})$. Now suppose $\alpha,\beta\in U^{-1}$ satisfy $r(\alpha)=r(\beta)$. Then $s(\alpha^{-1})=s(\beta^{-1})$. Since $\alpha^{-1},\beta^{-1}\in U$ and $s|_U$ is injective, we must have $\alpha^{-1}=\beta^{-1}$, so $\alpha=\beta$ and $r|_{U^{-1}}$ is injective, establishing \eqref{Renault_I_2_8i_sub_iii}. 

\eqref{Renault_I_2_8i_sub_iii}$\implies$\eqref{Renault_I_2_8i_sub_ii}. A proof establishing this claim is similar to the \eqref{Renault_I_2_8i_sub_ii}$\implies$\eqref{Renault_I_2_8i_sub_iii} proof.

\eqref{Renault_I_2_8i_sub_ii}$\implies$\eqref{Renault_I_2_8i_sub_iv}. Suppose \eqref{Renault_I_2_8i_sub_ii}. It was shown that $G$ is $r$-discrete in the \eqref{Renault_I_2_8i_sub_ii}$\implies$\eqref{Renault_I_2_8i_sub_i} component of this proof, so it remains to show that the composition map is a local homeomorphism. Let $\circ:G^{(2)}\rightarrow G$ be the composition map (i.e. the map such that $\circ(\alpha,\beta)=\alpha\beta$ for every $(\alpha,\beta)\in G^{(2)}$). Fix $(\alpha,\beta)\in G^{(2)}$. By \eqref{Renault_I_2_8i_sub_ii} and the previously proved claim that \eqref{Renault_I_2_8i_sub_ii}$\implies$\eqref{Renault_I_2_8i_sub_iii}, the maps $r$ and $s$ are local homeomorphisms. Thus there exist open neighbourhoods $U$ and $V$ of $\alpha$ and $\beta$, respectively, such that $r|_U$ and $s|_V$ are homeomorphisms. Then $(U\times V)\cap G^{(2)}$ is an open neighbourhood of $(\alpha,\beta)$ in $G^{(2)}$. To see that $\circ|_{(U\times V)\cap G^{(2)}}$ is injective, suppose $(\gamma_1,\delta_1),(\gamma_2,\delta_2)\in (U\times V)\cap G^{(2)}$ satisfy $\circ(\gamma_1,\delta_1)=\circ(\gamma_2,\delta_2)$. Then $\gamma_1\delta_1=\gamma_2\delta_2$, so $r(\gamma_1)=r(\gamma_2)$ and $s(\delta_1)=s(\delta_2)$. Since $\gamma_1,\gamma_2\in U$ and $\delta_1,\delta_2\in V$ with $r|_U$ and $s_V$ homeomorphisms, we must have $\gamma_1=\gamma_2$ and $\delta_1=\delta_2$. Thus $\circ|_{(U\times V)\cap G^{(2)}}$ is injective and it follows that $\circ$ is a local homeomorphism since $\circ$ is continuous (Definition \ref{def_topological_groupoid}) and open (Lemma \ref{lemma_composition_map_open}).

\eqref{Renault_I_2_8i_sub_iv}$\implies$\eqref{Renault_I_2_8i_sub_ii}. We begin by showing that $s$ is an open map. Let $U$ be an open subset of $G$. Then $U^{-1}$ is open (Remark \ref{remark_following_topological_groupoid_def}) and so $(U^{-1}\times U)\cap G^{(2)}$ is an open subset of $G^{(2)}$. Since every local homeomorphism is open and $\circ$ is a local homeomorphism, $\circ\big((U^{-1}\times U)\cap G^{(2)}\big)$ is an open subset of $G$.

We claim that $s(U)=\circ\big((U^{-1}\times U)\cap G^{(2)}\big)\cap G^{(0)}$. Fix $x\in \circ\big((U^{-1}\times U)\cap G^{(2)}\big)\cap G^{(0)}$. Then there exist $\alpha,\beta\in U$ with $r(\alpha)=r(\beta)$ and such that $x=\alpha^{-1}\beta$; note in particular that $x=s(\alpha)$. Then $\alpha x=\beta$ and $\alpha=\beta$ since $x=s(\alpha)$. We now have $x=\alpha^{-1}\alpha=s(\alpha)\in s(U)$. Now fix $y\in s(U)$. Then there exists $\eta\in U$ with $y=s(\eta)$. Now $(\eta^{-1},\eta)\in (U^{-1}\times U)\cap G^{(2)}$, so
\[
s(\eta)=\eta^{-1}\eta\in \circ\big((U^{-1}\times U)\cap G^{(2)}\big),
\]
and $s(U)=\circ\big((U^{-1}\times U)\cap G^{(2)}\big)\cap G^{(0)}$. Since $\circ\big((U^{-1}\times U)\cap G^{(2)}\big)$ is open and $G^{(0)}$ is open since $G$ is $r$-discrete, $s(U)$ is open and so $s$ is an open map.

Since $s$ is continuous (Remark \ref{remark_following_topological_groupoid_def}), to show that $s$ is a local homeomorphism it remains to show that it is locally injective. Fix $\gamma\in G$. Since $\circ$ is locally injective and $(\gamma^{-1},\gamma)\in G^{(2)}$, there is an open neighbourhood $A$ of $(\gamma^{-1},\gamma)$ in $G^{(2)}$ such that $\circ|_A$ is injective. By the definition of the topology on $G^{(2)}$ there exist open neighbourhoods $U$ and $V$ of $\gamma^{-1}$ and $\gamma$, respectively, in $G$ such that $(U\times V)\cap G^{(2)}\subset A$. Let $W=U^{-1}\cap V$, so that $W$ is an open neighbourhood of $\gamma$ in $G$ with
\[
(\gamma^{-1},\gamma)\in (W^{-1}\times W)\cap G^{(2)}\subset (U\times V)\cap G^{(2)}.
\]
Suppose $\eta,\zeta\in W$ satisfy $s(\eta)=s(\zeta)$. Then $\eta^{-1}\eta=\zeta^{-1}\zeta$ implies $(\eta^{-1},\eta)=(\zeta^{-1},\zeta)$ since  composition is injective on $(W^{-1}\times W)\cap G^{(2)}$. Then $\eta=\zeta$ and so $s$ is locally injective. Thus $s$ is a local homeomorphism.
\end{proof}

\section{Representations and the groupoid \texorpdfstring{$C^*$}{\it C*}-algebra}\label{section_representations_and_groupoid_algebras}

Suppose $G$ is a topological groupoid with Haar system $\lambda$. For $f,g\in C_c(G)$ and $\alpha\in G$, define
\[
f\ast g(\alpha):=\int f(\beta)g(\beta^{-1}\alpha)\,d\lambda^{r(\alpha)}(\beta)\notationindex{FG@$f\ast g$}
\]
and $f^*(\alpha)=\overline{f(\alpha^{-1})}$. Then $C_c(G)$ with respect to these operations and the inductive limit topology\footnote{An overview of the inductive limit topology can be found in Section D.2 of Raeburn and Williams' book \cite{Raeburn-Williams1998}} is a topological $\ast$-algebra. Since the Haar system is a key part of this structure, it is both common in the literature and indeed more precise, for this $\ast$-algebra to be written as $C_c(G,\lambda)$\notationindex{CGlambda@$C_c(G,\lambda)$}. In this thesis we usually write $C_c(G)$ when the Haar system is clear from the context.
\begin{remark}\label{remark_alternative_convolution}
By considering \eqref{H3} and the definition of a left Haar system (Definition \ref{def_Haar_system}), we can see that for $f,g\in C_c(G)$ and $\alpha\in G$,
\[
f\ast g(\alpha) = \int f(\alpha \beta^{-1})g(\beta)\,d\lambda_{s(\alpha)}(\beta).
\]
\end{remark}

\begin{definition}[{\cite[Definition~2.41]{Muhly-book}}]
A {\em representation} of $C_c(G,\lambda)$\index{representation!of $C_c(G,\lambda)$} on a Hilbert space $\Hh$ is a $\ast$-homomorphism $\pi$ from $C_c(G,\lambda)$ into $B(\Hh)$, that is continuous with respect to the inductive limit topology on $C_c(G,\lambda)$ and the weak operator topology on $B(\Hh)$, and that is nondegenerate in the sense that the span of $\{\pi(f)h:f\in C_c(G,\lambda), h\in H\}$ is a dense subset of $\Hh$.
\end{definition}

For $f\in C_c(G,\lambda)$, based on Hahn's definition of the $I$-norm in \cite[p. 38]{Hahn1978}, Renault in \cite[p. 50]{Renault1980} defined a quantity $\|f\|_I$\notationindex{0I@$\lVert\cdot\rVert_I,\lVert\cdot\rVert$} by
\[
\|f\|_I=\max\left\{\sup_{x\in G^{(0)}}\int_G\lvert f\rvert\,d\lambda^x,\sup_{x\in G^{(0)}}\int_G\lvert f\rvert\,d\lambda_x\right\}.
\]
Proposition~II.1.4 from \cite{Renault1980} shows that $\|\cdot\|_I$ is a $\ast$-algebra norm on $C_c(G,\lambda)$.
\begin{definition}[{\cite[Definition~II.1.5]{Renault1980}}]
Suppose $G$ is a second countable, locally compact, Hausdorff groupoid with Haar system $\lambda$. A representation $\pi$ of $C_c(G,\lambda)$ is {\em bounded}\index{bounded representation} if $\|\pi(f)\|\le\|f\|_I$ for every $f\in C_c(G,\lambda)$.
\end{definition}
Following the definition of a bounded representation, for each $f\in C_c(G,\lambda)$ Renault in \cite[p. 51]{Renault1980} defines a quantity
\[
\|f\|:=\sup\{\|\pi(f)\|:\pi\text{ is a bounded representation of }C_c(G,\lambda)\}.
\]
After noting that $\|\cdot\|$ is a $C^*$-semi-norm on $C_c(G,\lambda)$, Renault goes on to show that $\|\cdot\|$ is a $C^*$-norm on $C_c(G,\lambda)$.
\begin{definition}[{\cite[Definition~II.1.12]{Renault1980}}]
\notationindex{C*Glambda@$C^*(G,\lambda)$}\index{groupoid C*-algebra@groupoid $C^*$-algebra}Suppose $G$ is a second countable, locally compact, Hausdorff groupoid with Haar system $\lambda$. The {\em groupoid $C^*$-algebra} of $G$ with respect to $\lambda$, denoted $C^*(G,\lambda)$, is the completion of $C_c(G,\lambda)$ with the $C^*$-norm $\|\cdot\|$.
\end{definition}
A more recent definition of the groupoid $C^*$-algebra, as used by Muhly in \cite[Theorem~2.42]{Muhly-book}, by Paterson in \cite[p.~101]{Paterson1999} and by Anantharaman-Delaroche and Renault in \cite[p.~146]{Anantharaman-Renault2000}, defines $C^*(G,\lambda)$ to be the closure of $C_c(G,\lambda)$ with the norm such that
\[
\|f\|=\sup\{\|\pi(f)\|:\pi\text{ is a representation of }C_c(G,\lambda)\}
\]
for each $f\in C_c(G,\lambda)$. It follows from Renault's disintegration theorem \cite[Proposition~4.2]{Renault1987} (or \cite[Theorem~3.32]{Muhly-book}) that every representation of $C_c(G,\lambda)$ is bounded and that $\|\cdot\|$ is a $C^*$-norm on $C_c(G,\lambda)$, so this more recent definition is compatible with Renault's original definition.

It was shown by Muhly, Renault and Williams in \cite[Theorem~2.8]{Muhly-Renault-Williams1987} that if $\lambda$ and $\mu$ are Haar systems for a second countable, locally compact groupoid $G$, then $C^*(G,\lambda)$ is Morita equivalent to $C^*(G,\mu)$. When the Haar system is clear from the context we will usually write $C^*(G)$ instead of $C^*(G,\lambda)$. 

We now show how a Radon measure on the unit space of a groupoid $G$ with a Haar system $\lambda$ induces a representation of $C^*(G,\lambda)$. These induced representations will be used heavily in Chapter \ref{chapter_strength_of_convergence}.
The next definition is from \cite{Muhly-book}. Alternative descriptions may be found in \cite[p.~234]{Muhly-Williams1990} and \cite[pp.~81--82]{Renault1980}.
\begin{definition}[{\cite[Definition~2.45]{Muhly-book}}]\label{def_induced_representation}\index{induced representation}\index{representation!induced}
Suppose $G$ is a second countable, locally compact, Hausdorff groupoid with a Haar system $\lambda$ and let $\mu$ be a Radon measure on $G^{(0)}$.
\begin{enumerate}\renewcommand{\labelenumi}{(\roman{enumi})}
\item We write $\nu=\mu\circ\lambda=\int \lambda^x \, d\mu$ for the measure on $G$ defined for every Borel-measurable function $f:G\rightarrow\CC$ by 
\[\int_G f(\gamma)\, d\nu(\gamma)=\int_{G^{(0)}}\int_G f(\gamma)\, d\lambda^x(\gamma)\, d\mu(x).\]
We call $\nu$ the measure induced by $\mu$, and we write $\nu^{-1}$ for the image of $\nu$ under the homeomorphism $\gamma\mapsto \gamma^{-1}$.
\item For $f\in C_c(G)$, $\mathrm{Ind}\,\mu(f)$\notationindex{INDMU@$\mathrm{Ind}\, \mu$} is the operator on $L^2(G,\nu^{-1})$ defined by the formula
\begin{align*}
\big(\mathrm{Ind}\,\mu(f)\xi\big)(\gamma)&=\int_G f(\alpha)\xi(\alpha^{-1}\gamma)\, d\lambda^{r(\gamma)}(\alpha)\\
&=\int_G f(\gamma \alpha)\xi(\alpha^{-1})\, d\lambda^{s(\gamma)}(\alpha).
\end{align*}
\end{enumerate}
\end{definition}
For each $x\in G^{(0)}$, define $\L^x$\notationindex{Lx@$\L^x$} to be the representation induced by the point-mass measure $\delta_x$\notationindex{DELTAX@$\delta_x$} on $G^{(0)}$ as in \cite{Muhly-Williams1990,Clark-anHuef2008}.
\begin{remark}\label{measure_induced_epsilon_x} It follows from the definition of the induced measure that for $x\in G^{(0)}$, the measure $\nu$ induced by $\delta_x$ is equal to $\lambda^x$. In particular we have $\nu^{-1}=\lambda_x$, so $\L^x$ acts on $L^2(G,\lambda_x)$. The operator $\L^x(f)$ is then given by
\[
\big(\L^x(f)\xi\big)(\gamma)=\int_G f(\gamma \alpha^{-1})\xi(\alpha)\, d\lambda_x(\alpha)
\]
for all $\xi\in L^2(G,\lambda_x)$ and all $\gamma\in G$. There is a close relationship between the convolution on $C_c(G)$ and these induced representations: recall from Remark \ref{remark_alternative_convolution} that for $f,g\in C_c(G)$, the convolution $f\ast g\in C_c(G)$ is given by
\[
 f\ast g(\gamma)=\int_G f(\gamma\alpha^{-1})g(\alpha)\,d\lambda_{s(\gamma)}(\alpha)\quad\text{for all }\gamma\in G,
\]
so that
\[
\big(\L^x(f)g\big)(\gamma)=f\ast g(\gamma)\quad\text{for any }x\in G^{(0)}\text{ and }\gamma\in G_x.
\]
\end{remark}
We denote the norm in $L^2(G,\lambda_x)$ by $\|\cdot\|_x$\notationindex{0I2@$\lVert\cdot\rVert_x$}. Finally note that when $G$ is a second-countable locally-compact principal groupoid that admits a Haar system, each $\L^x$ is irreducible by \cite[Lemma~2.4]{Muhly-Williams1990}. The irreducible representations $\L^x$ will be a key focus of chapter \ref{chapter_strength_of_convergence}.

\section{The transformation-group groupoid}\label{section_transformation-group_groupoid}
Renault in \cite[Examples~1.2]{Renault1980} described a groupoid that can be constructed from a group acting on the right of a space. The following is a modification that describes a groupoid constructed from a left group action.
\begin{definition}[{\cite[Example~3.3]{Clark-anHuef2008}}]\label{def_transformation-group_groupoid}\index{transformation-group groupoid}
Let $(H,X)$ be a transformation group with $H$ acting on the left of the space $X$. Then $G=H\times X$ with
\[
G^{(2)}=\big\{\big((h,x),(k,y)\big)\in G\times G:y=h^{-1}\cdot x\big\}
\]
and operations $(h,x)(k,h^{-1}\cdot x)=(hk,x)$ and $(h,x)^{-1}=(h^{-1},h^{-1}\cdot x)$ is called the {\em transformation-group groupoid}. If $H$ and $X$ are topological spaces, then $G$ is considered to have the product topology.
\end{definition}
We identify the space $\{e\}\times X$ with $X$, so that the range and source maps $r,s:H\times X\rightarrow X$ are given by $r(h,x)=x$ and $s(h,x)=h^{-1}\cdot x$. A transformation group $(H,X)$ is {\em free}\index{free transformation group} if whenever $h\in H$ and $x\in X$ satisfy $h\cdot x=x$, then $h$ is the identity $e$ of $H$. 
\begin{lemma}
Suppose $(H,X)$ is a transformation group with $H$ acting on the left of $X$ and let $G$ be the transformation-group groupoid $H\times X$. Then $G$ is principal if and only if $(H,X)$ is free.
\begin{proof}
Suppose $G$ is principal and that $h\in H$ and $x\in X$ satisfy $h\cdot x=x$. Then $(h^{-1},x)$ and $(e,x)$ are elements of $G$ with $r(h^{-1},x)=x=r(e,x)$ and $s(h^{-1},x)=h\cdot x=x=s(e, x)$. Since $G$ is principal it follows that $(h^{-1},x)=(e,x)$, so $h=e$ and $(H,X)$ is free.

Conversely, suppose $(H,X)$ is free and that $(h,x),(k,y)\in G$ have equal range and source. Then $x=r(h,x)=r(k,y)=y$. Similarly $h^{-1}\cdot x=s(h,x)=s(k,y)=k^{-1}\cdot y$ so, since $x=y$, we have $kh^{-1}\cdot x=x$ and since $(H,X)$ is free we have $kh^{-1}=e$. Then $h=k$ and so $(h,x)=(k,y)$.
\end{proof}
\end{lemma}

\begin{remark}\label{remark_Haar_system_for_transformation_group}
Suppose $(H,X)$ is a locally compact, Hausdorff transformation group with $H$ acting on the left of the space $X$. If $\delta_x$ is the point-mass measure on $X$ and $\mu$ is a left Haar measure on $H$, then $\{\lambda^x:=\mu\times\delta_x:x\in X\}$ is a left Haar system for the transformation-group groupoid $H\times X$.
\end{remark}

We will now introduce the representations induced by evaluation maps on $C_0(X)$ as in \cite{Archbold-anHuef2006, anHuef2002, Williams1981, Williams1981-2}. Suppose that $H$ acts freely on $X$ and let $\nu$ be the right Haar measure such that $\nu(E)=\mu(E^{-1})$ for every Borel subset $E$ of $H$. For each $x\in X$, we define $\mathrm{Ind}\,\epsilon_x$ to be the representation $\tilde{\epsilon}_x\rtimes\lambda$ of the transformation group $C^*$-algebra $C_0(X)\rtimes H$ on the Hilbert space $L^2(H,\nu)$, where $\big(\tilde{\epsilon}_x(f)\xi\big)(h)=f(h\cdot x)\xi(h)$ and $(\lambda_k\xi)(h)=\Delta(k)^{1/2}\xi(k^{-1}h)$ for $f\in C_0(X)$, $\xi\in L^2(H,\nu)$ and $h,k\in H$ (see \cite[Lemma~4.14]{Williams1981}). It follows that
\[
\big(\mathrm{Ind}\,\epsilon_x(f)\xi\big)(h)=\int_H f(k,h\cdot x)\xi(k^{-1}h)\Delta(k)^{1/2}\,d\nu(k)
\]
for $f\in C_c(H,X)$ and $\xi\in L^2(H,\nu)$. 

\begin{remark}\label{remark_algebras_isomorphic} If $G=(H, X)$ is a second countable, free transformation group, then the representations $\L^x$ of the transformation-group groupoid $C^*$-algebra are unitarily equivalent to the representations $\mathrm{Ind}\,\epsilon_x$. Specifically, let $\nu$ be a choice of right Haar measure on $H$ and $\Delta$  the associated modular function. The map $\iota:C_c(H\times X)\to C_c(H\times X)$ defined by $\iota(f)(h,x)=f(h,x)\Delta(t)^{1/2}$ extends to an isomorphism $\iota$ of the groupoid $C^*$-algebra $C^*(H\times X)$ onto the transformation-group $C^*$-algebra $C_0(X)\rtimes H$ \cite[p.~58]{Renault1980}. Fix $x\in X$.  Then there is a unitary  $U_x:L^2(H,\nu)\to L^2(H\times X,\lambda_x)$,  characterised by $U(\xi)(h, y)=\xi(h)\delta_x(h^{-1}\cdot y)$ for $\xi\in C_c(H)$, and 
$U(\mathrm{Ind}\,\epsilon_x(\iota(f))U^*=\L^x(f)$ for $f\in C^*(H\times X)$.
\end{remark}

\section{The path groupoid of a directed graph}\label{sec_path_groupoid}
A {\em directed graph}\index{directed graph} $E=(E^0,E^1,r_E,s_E)$\notationindex{E@$E, E^0, E^1$}\notationindex{RE@$r_E$}\notationindex{SE@$s_E$} consists of two countable sets $E^0,E^1$ and functions $r_E,s_E:E^1\rightarrow E^0$. We call elements of $E^0$ and $E^1$ {\em vertices}\index{vertex!in a directed graph} and {\em edges}\index{edge!in a directed graph} respectively. For each edge $e$, we call $s_E(e)$ the {\em source}\index{source!of an edge in directed graph} of $e$ and call $r_E(e)$ the {\em range}\index{range!of an edge in a directed graph} of $e$. Where the graph is clear from the context, we usually write $s$ instead of $s_E$ and $r$ instead of $r_E$.

\begin{example}
Let $E$ be the directed graph such that $E^0=\{a,b\}$, $E^1=\{x,y\}$ and $r,s$ are given defined by: $r(x)=a$, $s(x)=b$, $r(y)=b$ and $s(y)=b$. This directed graph is fully described by the following picture.
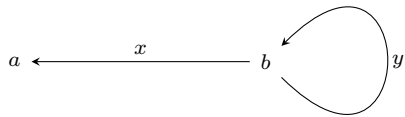
\begin{center}\begin{tikzpicture}[>=stealth,baseline=(current bounding box.center)]
\clip (-0.3em,-1.8em) rectangle (12.6em,1.9em);
\node (a) at (0em,0em) {$\scriptstyle a$};
\node (b) at (8em,0em) {$\scriptstyle b$};
\draw[->] (b) to node [midway,above=-0.1em] {$\scriptstyle x$} (a);
\draw [<-] (b) .. controls (13em,5em) and (13em,-5em) .. (b) node [midway,right=-0.2em] {$\scriptstyle y$};
\end{tikzpicture}
\end{center}
\end{example}


A {\em finite path}\index{finite path!in a directed graph} $\alpha$ in a directed graph $E$ is a finite sequence $\alpha_1\alpha_2\cdots\alpha_k$ of edges $\alpha_i$ with $s_E(\alpha_j)=r_E(\alpha_{j+1})$ for $1\le j\le k-1$, or a vertex $v\in E^0$; we write $s_E(\alpha)=s_E(\alpha_k)$\notationindex{SE@$s_E$}, $r_E(\alpha)=r_E(\alpha_1)$\notationindex{RE@$r_E$} and $s_E(v)=r_E(v)=v$. We define the {\em length}\index{length!of a finite path} $\lvert\alpha\rvert$\notationindex{ALPHAABS@$\lvert\alpha\rvert$} of $\alpha$ to be $k$ and the length $\lvert v\rvert$ of $v$ to be zero. An {\em infinite path} $x=x_1x_2\cdots$ is defined similarly (with $|x|=\infty$), although $s_E(x)$ remains undefined. Let $E^*$ and $E^\infty$\notationindex{E*@$E^*,E^\infty$} denote the set of all finite paths and infinite paths in $E$ respectively. If $\alpha=\alpha_1\cdots\alpha_k$ and $\beta=\beta_1\cdots\beta_j$ are finite paths then, provided $s_E(\alpha)=r_E(\beta)$, let $\alpha\beta$ be the path $\alpha_1\cdots\alpha_k\beta_1\cdots\beta_j$. When $x\in E^\infty$ with $s_E(\alpha)=r_E(x)$ define $\alpha x$ similarly. We of course simplify notation and write $s$ for $s_E$ and $r$ for $r_E$. Note that this notation is somewhat strange: a path $\alpha=\alpha_1\alpha_2\alpha_3$ is drawn as
\begin{center}
\begin{tikzpicture}[>=stealth,baseline=(current bounding box.center)] 
\def\cellwidth{5.5};
\clip (-0.2em,-0.2em) rectangle (16.7em, 0.8em);
\foreach \x in {0,1,2,3} {
\node  (v\x) at (\cellwidth*\x em,0em) {};
\fill[black] (v\x) circle (0.15em);
}
\foreach \x / \y in {0/1,1/2,2/3} \draw [<-] (v\x) to node[above=-0.2em] {$\alpha_{\y}$} (v\y);
\end{tikzpicture}
\end{center}
where the arrows are reversed from what seems natural. In Kumjian, Pask, Raeburn and Renault's original paper they used the seemingly more natural approach, however the definition of a path that we use is consistent with the later development of higher-rank graphs, where edges become morphisms in a category and the new convention ensures that ``composition of morphisms is compatible with multiplication of operators in $B(\Hh)$'' \cite[p.~2]{Raeburn2005}. This new convention is used elsewhere in work on directed graphs including in Raeburn's book \cite{Raeburn2005}. 

A directed graph $E$ is {\em row-finite}\index{row finite!directed graph} if $r^{-1}(v)$ is finite for every $v\in E^0$. For any finite path $\alpha$ in a directed graph $E$, let $\alpha E^\infty$\notationindex{ALPHAEINFTY@$\alpha E^\infty$}\label{def_alphaEinfty} be the set of all paths in $E^\infty$ of the form $\alpha y$ for some $y\in E^\infty$.
\begin{prop}[{\cite[Corollary~2.2]{kprr1997}}]\label{prop_topology_on_Einfty}
Suppose $E$ is a row-finite directed graph. The sets $\alpha E^\infty$ for $\alpha\in E^*$ form a basis of compact open sets for a locally compact, $\sigma$-compact, totally disconnected, Hausdorff topology on $E^\infty $
\end{prop}
Readers already familiar with graph algebras will recall that the key results in Kumjian, Pask, Raeburn and Renault's \cite{kprr1997} relate to directed graphs without sources (that is, vertices $v$ with $r^{-1}(v)=\emptyset$), Proposition \ref{prop_topology_on_Einfty} applies to any row-finite directed graph $E$ since the set $E^\infty$ only relates to the infinite paths of $E$.

Before presenting an example that illustrates convergence in the topology on the infinite path space $E^\infty$, we will introduce some notation that will only be used in this thesis for related examples. \notationindex{VFinfty@$[v,f]^\infty$}\label{def_vfinfty}When $v$ is a vertex, $f$ is an edge, and there is exactly one infinite path with range $v$ that includes the edge $f$, then we denote this infinite path by $[v,f]^\infty$. Define $\PP:=\NN\backslash\{0\}$\notationindex{P@$\PP\ (=\NN\char`\\\braces{0})$}.

\needspace{6\baselineskip}
\begin{example}
Let $E$ be the following row-finite directed graph.
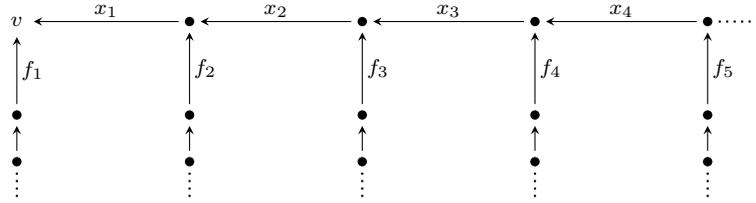
\begin{center}
\begin{tikzpicture}[>=stealth,baseline=(current bounding box.center)] 
\def\cellwidth{5.5};
\clip (-2em,-5.8em) rectangle (4*\cellwidth em + 2em,0.7em);

\node (x1y0) at (0 em,0 em) {$\scriptstyle v$};
\foreach \x in {2,3,4,5} \foreach \y in {0} \node (x\x y\y) at (\cellwidth*\x em-\cellwidth em,-3*\y em) {};
\foreach \x in {1,2,3,4,5} \node (x\x y1) at (\cellwidth*\x em-\cellwidth em,-3 em) {};
\foreach \x in {1,2,3,4,5} \node (x\x y2) at (\cellwidth*\x em-\cellwidth em,-4.5 em) {};
\foreach \x in {2,3,4,5} \foreach \y in {0} \fill[black] (x\x y\y) circle (0.15em);
\foreach \x in {1,2,3,4,5} \foreach \y in {1,2} \fill[black] (x\x y\y) circle (0.15em);

\foreach \x / \z in {1/2,2/3,3/4,4/5} \draw[black,<-] (x\x y0) to node[above=-0.2em] {$\scriptstyle x_{\x}$} (x\z y0);

\foreach \x in {1,2,3,4,5} \draw [<-] (x\x y0) to node[right=-0.2em] {$\scriptstyle f_{\x}$} (x\x y1);
\foreach \x in {1,2,3,4,5} \draw [<-] (x\x y1) -- (x\x y2);
\foreach \x in {1,2,3,4,5} {
	\node (endtail\x) at (\cellwidth*\x em-\cellwidth em, -6.0em) {};
	\draw [dotted,thick] (x\x y2) -- (endtail\x);
}

\node(endtailone) at (4*\cellwidth em + 2em,0em) {};
\draw[dotted,thick] (x5y0) -- (endtailone);
\end{tikzpicture}
\end{center}
The sequence $\{[v,f_n]^\infty\}_{n\in\PP}$ in $E^\infty$ converges to the infinite path $x=x_1x_2x_3\cdots$.
\begin{proof}
Let $U$ be a neighbourhood of $x$ in $E^\infty$. By Proposition \ref{prop_topology_on_Einfty} there exists $\alpha\in E^*$ such that $x\in \alpha E^\infty\subset U$. In order to have $x\in\alpha E^\infty$, we must have $\alpha=x_1x_2\cdots x_p$ for some $p\in\PP$. For every $n>p$ we have $[v,f_n]^\infty\in x_1x_2\cdots x_pE^\infty$, so $[v,f_n]^\infty$ is eventually in $\alpha E^\infty\subset U$, establishing the convergence.
\end{proof}
\end{example}

We say $x,y\in E^\infty$ are {\em shift equivalent}\index{shift equivalence!in a directed graph} with lag $n\in\ZZ$ (written $x\sim_n y$\notationindex{XSIMY@$x\sim_n y$}) if there exists $m\in\NN$ such that $x_i=y_{i-n}$ for all $i\ge m$.

\begin{example}
The following graph is row finite. In this graph $x\sim_1y$ and $y\sim_{-1} x$.
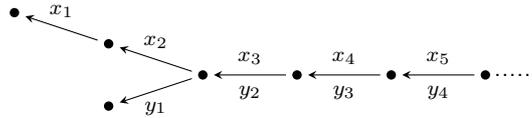
\begin{center}
\begin{tikzpicture}[>=stealth,baseline=(current bounding box.center)] 
\def\scale{3};
\clip (-0.3em,-1.2*\scale em) rectangle (5.6*\scale em,0.4em);
\node (x0) at (0,0) {};
\node (x1) at (\scale em,-1/3*\scale em) {};
\node (y0) at (\scale em, -\scale em) {};
\node (z0) at (2*\scale em, -2/3*\scale em) {};
\node (z1) at (3*\scale em, -2/3*\scale em) {};
\node (z2) at (4*\scale em, -2/3*\scale em) {};
\node (z3) at (5*\scale em, -2/3*\scale em) {};
\node (endtail) at (5.7*\scale em,-2/3*\scale em) {};

\draw [black,<-] (x0) to node[anchor=south] {$\scriptstyle x_1$} (x1);
\draw [black,<-] (x1) to node[anchor=south] {$\scriptstyle x_2$} (z0);
\draw [black,<-] (y0) to node[anchor=north] {$\scriptstyle y_1$} (z0);

\draw [black,<-] (z0) to node[anchor=north] {$\scriptstyle y_2$} (z1);
\draw [black,<-] (z1) to node[anchor=north] {$\scriptstyle y_3$} (z2);
\draw [black,<-] (z2) to node[anchor=north] {$\scriptstyle y_4$} (z3);

\node [anchor=south] at ($(z0)!0.5!(z1)$) {$\scriptstyle x_3$};
\node [anchor=south] at ($(z1)!0.5!(z2)$) {$\scriptstyle x_4$};
\node [anchor=south] at ($(z2)!0.5!(z3)$) {$\scriptstyle x_5$};

\fill[black] (x0) circle (0.15em);
\fill[black] (x1) circle (0.15em);
\fill[black] (y0) circle (0.15em);
\foreach \z in {0,1,2,3}  \fill[black] (z\z) circle (0.15em);

\draw [dotted,thick] (z3) -- (endtail);
\end{tikzpicture}
\end{center}
\end{example}
This definition of shift equivalence is a modification of the definition of shift equivalence in \cite[p.~509]{kprr1997} which required $x_i=y_{i+n}$ for all $i\ge m$. This change has been made because later work on higher-rank graphs by Kumjian and Pask in \cite[Definition~2.7]{Kumjian-Pask2000} involved constructing a groupoid in a way that is consistent with this modification. This doesn't matter a great deal because the groupoids that are constructed with these two different notions of shift equivalence are isomorphic.

Suppose $E$ is a row-finite directed graph. We refer to the groupoid constructed from $E$ by Kumjian, Pask, Raeburn and Renault in \cite{kprr1997} as the {\em path groupoid}\index{path groupoid} of $E$. Before describing this construction a reminder should be made to any reader going back to examine \cite{kprr1997}: the notion of shift equivalence has changed slightly and the definition of a path has changed (swapping range and source) to achieve consistency with the more recent higher-rank graphs. The path groupoid $G_E$\notationindex{GE@$G_E$} constructed from $E$ is defined as follows: let
\[
G_E:=\braces{(x,n,y)\in E^\infty\times \ZZ\times E^\infty :x\sim_n y}.
\]
For elements of 
\[
G_E^{(2)}:=\braces{\big((x,n,y),(y,m,z)\big):(x,n,y),(y,m,z)\in G_E},
\]
Kumjian, Pask, Raeburn, and Renault defined
\[
(x,n,y)\cdot (y,m,z):=(x,n+m,z),
\]
and for arbitrary $(x,n,y)\in G_E$, defined $(x,n,y)^{-1}:=(y,-n,x)$. For each $\alpha,\beta\in E^*$ with $s(\alpha)=s(\beta)$, let $Z(\alpha,\beta)$ be the set
\[
\braces{(x,n,y):x\in \alpha E^\infty, y\in \beta E^\infty, n=\lvert\alpha\rvert-\lvert\beta\rvert, x_i=y_{i-n}\text{ for } i>\lvert\alpha\rvert}.
\]
\begin{prop}[{\cite[Proposition 2.6]{kprr1997}}]
Let $E$ be a row-finite directed graph. The collection of sets
\[\braces{Z(\alpha,\beta):\alpha,\beta\in E^*, s(\alpha)=s(\beta)}\]
is a basis of compact open sets for a second countable, locally compact, Hausdorff topology on $G_E$ that makes $G_E$ $r$-discrete. The groupoid $G_E$ has a Haar system of counting measures and the map from $E^\infty$ onto $G_E^{(0)}$ defined by $x\mapsto (x,0,x)$ is a homeomorphism.
\end{prop}
In the proof of \cite[Proposition 2.6]{kprr1997} it is shown that $r$ is a local homeomorphism before using Renault's \cite[Lemma~2.7(ii), Proposition~2.8]{Renault1980} to show that the groupoid is $r$-discrete and admits a Haar system of counting measures. There are issues with the proofs of \cite[Lemma~2.7(ii), Proposition~2.8]{Renault1980} (see Remark \ref{remark_problems_Renault_results}), however Proposition \ref{Renault_I_2_8i_sub} can be used instead.

The point to examining the path groupoid is that properties of the groupoid are related to the combinatorial properties of the directed graph. In order to present our first result describing such a relationship, first define a {\em cycle}\index{cycle!of a directed graph} to be a finite path $\alpha$ of non-zero length with $r(\alpha)=s(\alpha)$ and $s(\alpha_i)\ne s(\alpha_j)$ for $i\ne j$. 
\begin{prop}\label{prop_principal_iff_no_cycles}
Suppose $E$ is a row-finite directed graph. The path groupoid $G_E$ is principal if and only if $E$ contains no cycles.
\begin{proof}
Let $G=G_E$. We first show that if $E$ contains no cycles, then $G$ is principal.  Suppose $G$ is not principal. Then there exists $x\in E^\infty$ such that the stability subgroup $G|_x$ is non-trivial and so there is a non-zero $n\in\ZZ$ of least absolute value with $x\sim_n x$. We may assume that $n>0$ since $x\sim_{-n}x$. As $x\sim_n x$ there exists $m\in\NN$ such that $x_i=x_{i-n}$ for all $i\ge m$ and so $x_{i-n}x_{i-n+1}\cdots x_i$ is a cycle in $E$.

Conversely, suppose $E$ contains the cycle $\alpha=\alpha_1\alpha_2\cdots \alpha_n$. Then $x:=\alpha\alpha\cdots$ is in $E^\infty $ with $x\sim_n x$, so the stability subgroup $G|_x$ is non-trivial and $G$ is not principal.
\end{proof}
\end{prop}

For a row-finite directed graph $E$, Section \ref{section_algebras_of_directed_graphs} describes the directed-graph Cuntz-Krieger $C^*$-algebra, $C^*(E)$, of $E$. Readers already familiar with the directed-graph Cuntz-Krieger $C^*$-algebras may be interested to know that the gauge-invariant uniqueness theorem (Theorem \ref{thm_gauge_invariant_uniqueness}) tells us that $C^*(G_E)$ is isomorphic to $C^*(E)$ when $E$ has no sources (i.e. no vertices $v$ with $r^{-1}(v)=\emptyset$). Example \ref{example_directed_graph_algebras_not_isomorphic} demonstrates why this `no sources' condition is necessary.
\section{Proper, Cartan and integrable groupoids and their \texorpdfstring{$C^*$}{\it C*}-algebras}\label{sec_proper_Cartan_integrable}
\begin{definition}
A groupoid $G$ is {\em proper}\index{proper groupoid} if the map $\Phi:G\rightarrow G^{(0)}\times G^{(0)}$ defined by $\Phi(\gamma)=\big(r(\gamma),s(\gamma)\big)$ for every $\gamma\in G$ is proper. In other words, if $\Phi^{-1}(K)$ is compact for every compact subset $K$ of $G^{(0)}\times G^{(0)}$, where $G^{(0)}\times G^{(0)}$ is equipped with the product topology.
\end{definition}
\begin{lemma}[{\cite[Lemma~7.2]{Clark2007}}]
A Hausdorff groupoid $G$ is proper if and only if $G|_N$ is compact for every compact subset $N$ of $G^{(0)}$.
\end{lemma}

\begin{definition}[{\cite[Definition~7.1]{Clark2007}}]
Suppose $G$ is a groupoid. A subset $N$ of $G^{(0)}$ is said to be {\em wandering}\index{wandering} if $G|_N=r^{-1}(N)\cap s^{-1}(N)$ is relatively compact. We say $G$ is {\em Cartan}\index{Cartan groupoid} if every $x\in G^{(0)}$ has a wandering neighbourhood.
\end{definition}
Note that it follows immediately that every proper groupoid is Cartan.
\begin{definition}[{\cite[Definition~3.1]{Clark-anHuef2008}}]
Suppose $G$ is a locally compact, Hausdorff groupoid with Haar system $\lambda$. We say $G$ is {\em integrable}\index{integrable groupoid} if for every compact subset $N$ of $G^{(0)}$, $\sup_{x\in N}\braces{\lambda_x(G^N)}<\infty$.
\end{definition}

\begin{lemma}\label{lemma_integrable_if_Cartan}
Every first countable, locally compact, Hausdorff, Cartan groupoid with a Haar system is integrable.
\begin{proof}
Let $G$ be a first countable, locally compact, Hausdorff, Cartan groupoid with a Haar system $\lambda$ and let $N$ be a compact subset of $G^{(0)}$. Since $G$ is Cartan for every $x\in N$ there is a compact neighbourhood $K_x$ of $x$ such that $G|_{K_x}$ is compact. Since each $\lambda_x$ is a Radon measure, every $\lambda_x(G|_{K_x})$ is finite. Since $G$ is locally compact and Hausdorff, $x$ has a neighbourhood base of compact sets. So by Lemma \ref{corollary_astrid_lemma}, for every $x\in N$ there exists a compact neighbourhood $W_x$ of $x$ such that $y\in W_x$ implies $\lambda_y(G|_{K_x})<\lambda_x(G|_{K_x})+1$. For every $x\in N$, let $N_x$ be the compact neighbourhood $K_x\cap W_x$ of $x$. Now $\braces{\mathrm{int}\,N_x:x\in N}$ is an open cover of the compact set $N$ so there exists $x_1,x_2,\ldots,x_p\in N$ such that $N\subset \cup_{i=1}^p N_{x_i}$.

We claim that $\lambda_x(G^N)\le\sum_{i=1}^p\lambda_{x_i}(G|_{K_{x_i}})+p$ for every $x\in N$. Fix $x\in N$ and observe that
\[
\lambda_x(G^N)\le\lambda_x\Big(\bigcup_{i=1}^p G^{N_{x_i}}\Big)\le\sum_{i=1}^p\lambda_x(G^{N_{x_i}}).
\]
It will suffice to show that $\lambda_x(G^{N_{x_i}})\le \lambda_{x_i}(G|_{K_{x_i}})+1$ for every $i$. Fix $i$ with $1\le i\le p$. Since the support of $\lambda_x$ is $s^{-1}(x)$, in order for $\lambda_x(G^{N_{x_i}})$ to be non-zero there must exist $\gamma\in G$ such that $s(\gamma)=x$ and $r(\gamma)\in N_{x_i}$. Then
\[
\lambda_x(G^{N_{x_i}})=\lambda_x(G^{N_{x_i}}\gamma^{-1}\gamma)=\lambda_{r(\gamma)}(G^{N_{x_i}}\gamma^{-1})=\lambda_{r(\gamma)}(G^{N_{x_i}})=\lambda_{r(\gamma)}(G|_{N_{x_i}}).
\]
Since $r(\gamma)\in N_{x_i}\subset W_{x_i}$ and $N_{x_i}\subset K_{x_i}$,
\[
\lambda_x(G^{N_{x_i}})=\lambda_{r(\gamma)}(G|_{N_{x_i}})\le \lambda_{r(\gamma)}(G|_{K_{x_i}}) \le \lambda_{x_i}(G|_{K_{x_i}})+1,
\]
and the result follows.
\end{proof}
\end{lemma}

\chapter{Strength of convergence in the orbit space of a groupoid}\label{chapter_strength_of_convergence}
In Section \ref{sec_upper_lower_multiplicities} we introduce the upper and lower multiplicity numbers that Archbold and Spielberg in \cite{Archbold1994,Archbold-Spielberg1996} associated to irreducible representations. In Section \ref{sec_AaH} we will introduce the notion of $k$-times convergence in the orbit space of a transformation group and state the theorems by Archbold and an Huef in \cite{Archbold-anHuef2006} that relate $k$-times convergence in a free transformation group to upper and lower multiplicity numbers of the transformation group $C^*$-algebra; we will provide an example based on the transformation group that Green described in \cite[pp. 95--96]{Green1977}. In Sections \ref{sec_lower_multiplicity_1} to \ref{sec_upper_multiplicity} we generalise the results in Archbold and an Huef's \cite{Archbold-anHuef2006} to results about principal groupoids. Finally in Section \ref{section_graph_algebra_examples} we will demonstrate these results with Kumjian, Pask, Raeburn and Renault's path groupoid of various directed graphs. These examples are new and interesting since they are not constructed from transformation groups. The contents of Sections \ref{sec_lower_multiplicity_1} to \ref{section_graph_algebra_examples} has been published as \cite{Hazlewood-anHuef2011}.

\section{Upper and lower multiplicities}\label{sec_upper_lower_multiplicities}
Let $A$ be a $C^*$-algebra. We  write $\theta$ for the canonical surjection from the space $P(A)$ of pure states of $A$ to the spectrum $\hat A$ of $A$.  We  frequently  identify an irreducible representation $\pi$ with its equivalence class in $\hat A$ and we write $\Hh_\pi$ for the Hilbert space on which $\pi(A)$ acts.

Let $\pi\in\hat A$ and let $\braces{\pi_\alpha}$ be a net in $\hat{A}$. \notationindex{MU@$M_\U,M_\L$}We now recall the definitions of {\em upper} and {\em lower multiplicity}\index{upper multiplicity}\index{lower multiplicity} $M_\U(\pi)$ and $M_\L(\pi)$ from \cite{Archbold1994}, and {\em relative upper} and {\em relative lower multiplicity}\index{relative upper multiplicity}\index{relative lower multiplicity} $\M_\U(\pi,\braces{\pi_\alpha})$ and  $\M_\L(\pi,\braces{\pi_\alpha})$ from \cite{Archbold-Spielberg1996}. Let $\Nn$ be the weak$^*$-neighbourhood base at zero in the dual $A^*$ of $A$ consisting of all open sets of the form
\[
N=\{\psi\in A^*:\lvert\psi(a_i)\rvert<\epsilon, 1\le i\le n\},
\]
where $\epsilon>0$ and $a_1,a_2,\ldots,a_n\in A$.
Suppose $\phi$ is a pure state of $A$ associated with $\pi$ and let $N\in\Nn$. Define
\[
V(\phi,N)=\theta\big((\phi+N)\cap P(A)\big),
\]
an open neighbourhood of $\pi$ in $\hat{A}$. For $\sigma\in \hat{A}$ let
\[
\mathrm{Vec}(\sigma,\phi,N)=\braces{\eta\in \Hh_\sigma:\lvert\eta\rvert=1,(\sigma(\cdot)\eta\, |\,\eta)\in \phi+N}.
\]
Note that $\mathrm{Vec}(\sigma,\phi,N)$ is non-empty if and only if $\sigma\in V(\phi,N)$. For any $\sigma\in V(\phi,N)$  define $d(\sigma,\phi,N)$ to be the supremum in $\PP\cup\braces{\infty}$ of the cardinalities of finite orthonormal subsets of $\mathrm{Vec}(\sigma,\phi,N)$. Write $d(\sigma,\phi,N)=0$ when $\mathrm{Vec}(\sigma,\phi,N)$ is empty.

Define
\[
\M_\U(\phi,N)=\underset{\sigma\in V(\phi,N)}{\sup} d(\sigma,\phi,N)\in\PP\cup\braces{\infty}.
\]
Note that if $N'\in \Nn$ and $N'\subset N$, then $\M_\U(\phi,N')\le\M_\U(\phi,N)$. Now define
\[
\M_\U(\phi)=\underset{N\in\Nn}{\inf}\M_\U(\phi,N)\in\PP\cup\braces{\infty}.
\]
By \cite[Lemma~2.1]{Archbold1994}, the value of $\M_\U(\phi)$ is independent of the pure state $\phi$ associated to $\pi$ and so $\M_\U(\pi):=\M_\U(\phi)$ is well-defined.  For lower multiplicity,  assume that $\braces{\pi}$ is not open in $\hat{A}$, and using \cite[Lemma~2.1]{Archbold1994} again,   define
\[
\M_\L(\pi):=\underset{N\in\Nn}\inf \Big(\underset{\sigma\rightarrow\pi, \sigma\ne\pi}{\lim\,\inf} d(\sigma,\phi,N)\Big)\in\PP\cup\braces{\infty}.
\]

Now suppose that $\{\pi_\alpha\}_{\alpha\in\Lambda}$ is a net in $\hat{A}$. For $N\in\Nn$ let
\[
\M_\U\big(\phi,N,\{\pi_\alpha\}\big)=\underset{\alpha\in\Lambda}{\lim\,\sup}\, d(\pi_\alpha,\phi,N)\in\NN\cup\braces{\infty}.
\]
Note that if $N'\in\Nn$ and $N'\subset N$ then $\M_\U\big(\phi,N',\{\pi_\alpha\}\big)\le \M_\U\big(\phi,N,\{\pi_\alpha\}\big)$. Then
\[
\M_\U\big(\pi,\{\pi_\alpha\}\big):=\underset{N\in\Nn}\inf \M_\U\big(\phi,N,\{\pi_\alpha\}\big)\in\NN\cup\braces{\infty},
\]
is well-defined because  the right-hand side  is independent of the choice of $\phi$ by an argument similar to the proof of \cite[Lemma~2.1]{Archbold1994}. Similarly define
\[
\M_\L\big(\phi,N,\{\pi_\alpha\}\big):=\underset{\alpha\in\Lambda}{\lim\,\inf}\, d(\pi_\alpha,\phi,N)\in\NN\cup\braces{\infty},
\]
and let
\[
\M_\L\big(\pi,\{\pi_\alpha\}\big)=\underset{N\in\Nn}\inf \M_\L\big(\phi,N,\{\pi_\alpha\}\big)\in\NN\cup\braces{\infty}.
\]
It follows that for any irreducible representation $\pi$ and any net $\{\pi_\alpha\}_{\alpha\in\Lambda}$ of irreducible representations,
\[
\M_\L\big(\pi,\{\pi_\alpha\}\big)\le\M_\U\big(\pi,\{\pi_\alpha\}\big)\le\M_\U(\pi)
\]
and, if $\{\pi_\alpha\}$ converges to $\pi$ with $\pi_\alpha\ne\pi$ eventually, $\M_\L(\pi)\le\M_\L\big(\pi,\{\pi_\alpha\}\big)$. Finally, if $\{\pi_\beta\}$ is a subnet of $\{\pi_\alpha\}$, then
\[
\M_\L\big(\pi,\{\pi_\alpha\}\big)\le \M_\L\big(\pi,\{\pi_\beta\}\big)\le \M_\U\big(\pi,\{\pi_\beta\}\big)\le \M_\U\big(\pi,\{\pi_\alpha\}\big).
\]

\begin{example}[{\cite[p. 121]{Archbold1994}}]
Let $A$ be the $C^*$-algebra of all continuous functions $f:[0,1]\rightarrow M_4(\CC)$ such that, for each $n\ge 1$,
\[
f\bigg(\frac{1}{n}\bigg)=\bigg(\begin{array}{cc}\sigma_n(f) & 0 \\0 & \sigma_n(f)\end{array}\bigg)\quad\text{where $\sigma_n(f)\in M_2(\CC)$}
\]
and $f(0)=\lambda(f)1_{M_4(\CC)}$ for some scalar $\lambda(f)$. We have
\[
\widehat{A}=\{\lambda\}\cup\{\sigma_n:n\ge 1\}\cup\{\epsilon_t:0<t<1,t^{-1}\notin\NN\}
\]
where $\epsilon_t(f)=f(t)$ for every $f\in A$. Archbold in \cite[p. 125]{Archbold1994} proves that $\M_\U(\lambda)=4$ and $\M_\L(\lambda)=2$. We now recall Archbold's proof.
\begin{proof}[Proof that $\M_\U(\lambda)=4$ and $\M_\L(\lambda)=2$]
We begin by showing that $\M_\U(\lambda)=4$. Let $\{e_1,e_2,e_3,e_4\}$ be the usual basis for $\CC^4$. For $0<t<1$ with $t^{-1}\notin\NN$, we define
\[
\phi_t^{(i)}=\big(\epsilon_t(\cdot)e_i\,|\,e_i\big)\quad (1\le i\le 4).
\]
Then, for each $i$, $\phi_t^{(i)}\rightarrow\lambda$ as $t\rightarrow 0$. Let $N\in\Nn$. There exists $t_0\in(0,1)$ such that
\[
\phi_t^{(i)}\in\lambda+N\quad(0<t<t_0,t^{-1}\notin\NN,1\le i\le 4).
\]
Hence $d(\epsilon_t,\lambda,N)=4$ for $0<t<t_0$ with $t^{-1}\notin\NN$ and so $\M_\U(\lambda,N)\ge 4$. Since $N$ was chosen arbitrarily, $\M_\U(\lambda)\ge 4$. The reverse inequality follows from the fact that $\dim(\pi)\le 4$ for all $\pi\in A^\wedge$.

To show that $\M_\L(\lambda)=2$, let $\{u_1,u_2\}$ be the usual orthonormal basis for $\CC^2$ and define
\[
\psi_n^{(i)}=\big(\sigma_n(\cdot)u_i\,|\,u_i\big)\quad\text{for }i=1,2.
\]
Then $\psi_n^{(i)}\rightarrow\lambda$ for each $i$. Fix $N\in\Nn$. There exists $t_i\in(0,1)$ such that if $n^{-1}<t_1$ then $\psi_n{(i)}\in\lambda+N$ for each $i$ and if $0<t<t_1$ with $t^{-1} \notin\NN$ then $\phi_t^{(i)}\in\lambda+N$ for each $i$. Let
\[
V=\{\lambda\}\cup\{\sigma_n:n^{-1}<t_1\}\cup\{\pi_t:0<t<t_1, t^{-1}\notin\NN\},
\]
an open neighbourhood of $\lambda$ contained in $V(\lambda,N)$. Then
\[
\mathrm{inf}_{\rho\in V\backslash\{\lambda\}} d(\rho,\lambda,N)=2
\]
and so $\M_\L(\lambda,N)\ge 2$. Since $N$ was fixed arbitrarily, $\M_\L(\lambda)\ge 2$. The reverse inequality can be obtained by noting that $\sigma_n\rightarrow\lambda$ and $\dim(\sigma_n)=2$ for all $n$.
\end{proof}
\end{example}

\section[Strength of convergence in a transformation group]{Strength of convergence in the orbit space of a transformation group}\label{sec_AaH}
Recall that a sequence $\{h_n\}\subset H$ tends to infinity if it admits no convergent subsequence.
\begin{definition}[{\cite[p. 92]{Archbold-anHuef2006}}]\label{def_k-times_convergence_transformation_group}
Suppose $(H,X)$ is a transformation group and let $k\in\PP$. A sequence $\{x_n\}$ in $X$ is \index{k-times convergence@$k$-times convergence!in the orbit space of a transformation group}{\em $k$-times convergent} in $X/H$ to $z\in X$ if there exist $k$ sequences $\{h_n^{(1)}\},\{h_n^{(2)}\},\ldots,\{h_n^{(k)}\}$ in $H$ such that
\begin{enumerate}
\item\label{def_k-times_convergence_transformation_group_1} $h_n^{(i)}\cdot x_n\rightarrow z$ as $n\rightarrow\infty$ for $1\le i\le k$; and
\item\label{def_k-times_convergence_transformation_group_2} if $1\le i<j\le k$ then $h_n^{(j)}(h_n^{(i)})^{-1}\rightarrow\infty$ as $n\rightarrow\infty$.
\end{enumerate}
\end{definition}
The following example describes a transformation group from Green's \cite{Green1977}. Following the description of the transformation group we will show that its orbit space exhibits $2$-times convergence.
\begin{example}\label{example_2-times_convergence_Green}\index{Green's transformation group example}
Define $\psi:\RR\rightarrow\RR^3$ by $\psi(a)=(0,a,0)$ and for each $n\in\PP$, define $\phi_n:\RR\rightarrow \RR^3$ by
\[
\phi_n(a)=\left\{\begin{tabular}{ll}
$(2^{-2n},a,0)$&$a\le n$\\
$\big(2^{-2n}-\frac{a-n}{\pi}2^{-2n-1},n\cos(a-n),n\sin(a-n)\big)$&$n<a<n+\pi$\\
$(2^{-2n-1},a-\pi-2n,0)$&$a\ge n+\pi$
\end{tabular} \right.
\]
For each positive $n$, the image of $\phi_n$ ``consists of two vertical half lines connected by a half circle" \cite[p. 96]{Green1977}. By projecting the first two components of $\RR^3$ into $\RR^2$ we can can sketch the images of $\phi_1$, $\phi_2$ and $\phi_3$ in $\RR^3$ as in Figure \ref{figure_sketch_phi012}. In this figure the dashed lines are the projections of the half circles and the image of $\psi$ is the vertical axis.
\begin{figure}
\begin{center}
\begin{tikzpicture}

\def\xscale{80}
\def\yscale{2}
\def\plotheight{3.5}
\def\verticalaxisheight{3.7}

\node (phi1-1) at (0.25*\xscale em,1*\yscale em) {};
\node (phi1-pi1) at (0.125*\xscale em,-1*\yscale em) {};
\node (phi2-2) at (0.0625*\xscale em,2*\yscale em) {};
\node (phi2-pi2) at (0.0312*\xscale em,-2*\yscale em) {};
\node (phi2-pi4) at (phi2-pi2 |- 0em,0em) {};

\node (phi2-2) at (0.0625*\xscale em,2*\yscale em) {};
\node (phi2-pi2) at (0.0312*\xscale em,-2*\yscale em) {};
\node (phi3-3) at (0.0156*\xscale em,3*\yscale em) {};
\node (phi3-pi3) at (0.0078*\xscale em,-3*\yscale em) {};

\draw[<-,thick,black] (phi1-1.center |- 0,-\plotheight*\yscale em) to (phi1-1.center);
\draw[->,thick,black] (phi1-pi1.center) to (phi1-pi1.center |- 0,\plotheight*\yscale em) node[above] {\scriptsize$\phi_1$};
\draw[thick,dashed,black] (phi1-1.center) to (phi1-pi1.center);

\draw[<-,thick,black] (phi2-2.center |- 0,-\plotheight*\yscale em) to (phi2-2.center);
\draw[->,thick,black] (phi2-pi2.center) to (phi2-pi2.center |- 0,\plotheight*\yscale em) node[above] {\scriptsize$\phi_2$};
\draw[thick,dashed,black] (phi2-2.center) to (phi2-pi2.center);

\draw[<-,thick,black] (phi3-3.center |- 0,-\plotheight*\yscale em) to (phi3-3.center);
\draw[->,thick,black] (phi3-pi3.center) to (phi3-pi3.center |- 0,\plotheight*\yscale em) node[above] {\scriptsize$\phi_3$};
\draw[thick,dashed,black] (phi3-3.center) to (phi3-pi3.center);

\fill[black] (phi1-1) circle (0.15em);
\node [above] at (phi1-1) {\scriptsize$\phi_1(1)$};
\fill[black] (phi1-pi1) circle (0.15em);
\node [below] at (phi1-pi1) {\scriptsize$\phi_1(1+\pi)$};

\node [right] at (phi2-2) {\scriptsize$\phi_2(2)$};
\fill[black] (phi2-2) circle (0.15em);
\fill[black] (phi2-pi2) circle (0.15em);
\node [right, rotate=-60] at (phi2-pi2) {\scriptsize$\phi_2(2+\pi)$};
\fill[black] (phi2-pi4) circle (0.15em);
\node [right, rotate=60] at (phi2-pi4) {\scriptsize$\phi_2(4+\pi)$};

\node (phi1-0) at (phi1-1|-0,0) {};
\fill[black] (phi1-0) circle (0.15em);
\node [above right] at (phi1-0) {\scriptsize$\phi_1(0)$};

\node (phi2-0) at (phi2-2|-0,0) {};
\fill[black] (phi2-0) circle (0.15em);
\node [above right] at (phi2-0) {\scriptsize$\phi_2(0)$};

\node (phi1-pi2) at (phi1-pi1|-0,0) {};
\fill[black] (phi1-pi2) circle (0.15em);
\node [above right] at (phi1-pi2) {\scriptsize$\phi_1(2+\pi)$};

\draw[->] (-2pt,0) -- (0.31*\xscale em,0);
\draw[<->] (0,-\verticalaxisheight*\yscale em) -- (0,\verticalaxisheight*\yscale em); 

\foreach \x in {0.1,0.2,0.3} \draw[shift={(\x*\xscale em,0)}] (0pt,2pt) -- (0pt,-2pt) node[below=-0.3em] {\scriptsize$\x$};
  
\node[left] at (-2pt,0) {\scriptsize$0$};
\foreach \y in {-3,-2,-1,1,2,3} \draw[shift={(0,\y*\yscale em)}] (2pt,0pt) -- (-2pt,0pt) node[left] {\scriptsize$\y$};
\end{tikzpicture}\end{center}
\vspace{-1em}
\addtocounter{theorem}{1}
\caption{A sketch of a projection of the images of $\phi_1$, $\phi_2$ and $\phi_3$ into $\RR^2$.}
\label{figure_sketch_phi012}
\end{figure}
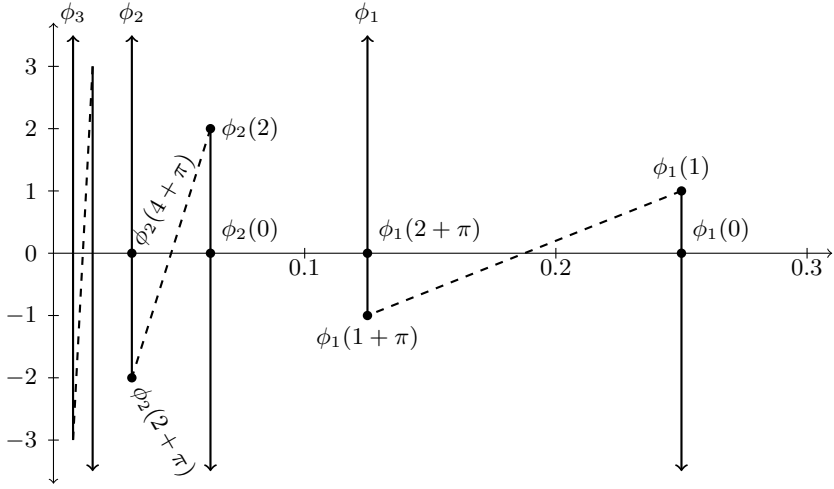

Let $X=\{\phi_n(a):n\in\NN, a\in\RR\}$ and define an action of $\RR$ on $X$ by $b\cdot\phi_n(a)=\phi_n(a+b)$, so that $(\RR,X)$ is a transformation group. Endow $X$ with the usual subspace topology of $\RR^3$ so that the action of $\RR$ on $X$ is continuous and $(\RR,X)$ is a second countable, locally compact, Hausdorff transformation group. The action is free since $b\cdot\phi_n(a)=\phi_n(a)$ only when $b=0$. Representatives of the orbits in $(\RR,X)$ are given by $z:=\psi(0)=(0,0,0)$ and $x_n:=\phi_n(0)=(2^{-2n},0,0)$ for $n\in\PP$. We claim that $\{x_n\}$ converges $2$-times in $X/\RR$ to $z$. To see this claim, let $b_n^{(1)}=0$ and $b_n^{(2)}=2n+\pi$ for each $n\in\PP$. The sequences $\{b_n^{(1)}\}$ and $\{b_n^{(2)}\}$ can be used as the sequences $\{h_n^{(1)}\}$ and $\{h_n^{(2)}\}$ in Definition \ref{def_k-times_convergence_transformation_group} to establish $2$-times convergence: note that
\begin{align*}
b_n^{(1)}\cdot x_n&=b_n^{(1)}\cdot\phi_n(0)=\phi_n(b_n^{(1)})=\phi_n(0)=(2^{-2n},0,0)\quad\text{and}\\
b_n^{(2)}\cdot x_n&=b_n^{(2)}\cdot\phi_n(0)=\phi_n(b_n^{(2)})=\phi_n(2n+\pi)=(2^{-2n-1},0,0),
\end{align*}
so that both $\{b_n^{(1)}\cdot x_n\}$ and $\{b_n^{(2)}\cdot x_n\}$ converge to $z=(0,0,0)$ in $X$ as $n\rightarrow\infty$. We also have $b_n^{(2)}(b_n^{(1)})^{-1}=(2n+\pi)-(0)=2n+\pi$, which tends to infinity as $n\rightarrow\infty$. It follows that $\{x_n\}$ converges $2$-times in $X/\RR$ to $z$.

In order for $\{x_n\}$ to converge $3$-times in $X/\RR$ to $z$, there would have to be a third sequence $\{b_n^{(3)}\}$ as in Definition \ref{def_k-times_convergence_transformation_group}. In order to have $b_n^{(3)}\cdot x_n\rightarrow z=(0,0,0)$ in this transformation group, we would eventually require $b_n^{(3)}$ to be close to $b_n^{(1)}$ or $b_n^{(2)}$, so the second condition Definition \ref{def_k-times_convergence_transformation_group} cannot be satisfied for the sequence $\{b_n^{(3)}\}$. Thus $\{x_n\}$ does not converge $3$-times in $X/\RR$ to $z$.
\end{example}
The next theorem is the first of the two main theorems in Archbold and an Huef's \cite{Archbold-anHuef2006}. After stating this theorem we will apply it to the transformation group from the previous example. A subset $S$ of a topological space $X$ is {\em locally closed}\index{locally closed subset} if there exist an open set $U$ of $X$ and a closed set $V$ of $X$ such that $S=U\cap V$; this is equivalent to $S$ being open in the closure of $S$ with the subspace topology by, for example, \cite[Lemma~1.25]{Williams2007}.
\needspace{6\baselineskip}
\begin{theorem}[{\cite[Theorem~1.1]{Archbold-anHuef2006}}]\label{AaH_thm1.1}
Suppose $(H,X)$ is a free second-countable transformation group. Let $k$ be a positive integer, let $z\in X$ and let $\{x_n\}$ be a sequence in $X$. Assume that $H\cdot z$ is locally closed in $X$. Then the following are equivalent:
\begin{enumerate}
\item\label{AaH_thm1.1_1} the sequence $\{x_n\}$ converges $k$-times in $X/H$ to $z$;
\item\label{AaH_thm1.1_2} $\M_\L(\mathrm{Ind}\,\epsilon_z,\{\mathrm{Ind}\,\epsilon_{x_n}\})\ge k$;
\item for every open neighbourhood $V$ of $z$ such that $\{h\in H: h\cdot z\in V\}$ is relatively compact we have
\[
\underset{n}{\lim\,\inf}\,\nu\big(\{h\in H:h\cdot x_n\in V\}\big)\ge k\nu\big(\{h\in H:h\cdot z\in V\}\big);
\]
\item there exists a real number $R>k-1$ such that for every open neighbourhood $V$ of $z$ with $\{h\in H:h\cdot z\in V\}$ relatively compact we have
\[
\underset{n}{\lim\,\inf}\,\nu\big(\{h\in H:h\cdot x_n\in V\}\big)\ge R\nu\big(\{h\in H:h\cdot z\in V\}\big);\quad\text{and}
\]
\item there exists a basic decreasing sequence of compact neighbourhoods $\{W_m\}$ of $z$ such that, for each $m\ge 1$,
\[
\underset{n}{\lim\,\inf}\,\nu\big(\{h\in H:h\cdot x_n\in W_m\}\big)> (k-1)\nu\big(\{h\in H:h\cdot z\in W_m\}\big).
\]
\end{enumerate}
\end{theorem}
In the next example we will use the equivalence of \eqref{AaH_thm1.1_1} and \eqref{AaH_thm1.1_2} to determine the lower multiplicity number of the representation induced by an evaluation map $\epsilon_z$.
\begin{example}\label{example_Green_lower_multiplicity}
In Example \ref{example_2-times_convergence_Green} we showed that in the free transformation group $X/\RR$, the sequence $\{x_n\}$ converges $2$-times, but not $3$-times, in $X/\RR$ to $z$. Since $\RR\cdot z=\RR\times\{0\}\times\{0\}$ is a closed subset of $X$, we can apply Theorem \ref{AaH_thm1.1} to see that $\M_\L(\mathrm{Ind}\,\epsilon_z,\{\mathrm{Ind}\,\epsilon_{x_n}\})=2$.
\end{example}

The other main theorem from Archbold and an Huef's \cite{Archbold-anHuef2006} follows and is the upper multiplicity analogue of Theorem \ref{AaH_thm1.1}.
\needspace{6\baselineskip}
\begin{theorem}[{\cite[Theorem~5.3]{Archbold-anHuef2006}}]\label{AaH_thm5.3}
Suppose $(H,X)$ is a free second-countable transformation group. Let $k$ be a positive integer, let $z\in X$ and let $\{x_n\}\subset X$ be a sequence such that $\{H\cdot x_n\}$ converges to $H\cdot z$ in $X/H$. Assume that $H\cdot z$ is locally closed in $X$. Then the following are equivalent:
\begin{enumerate}
\item there exists a subsequence $\{x_{n_i}\}$ of $\{x_n\}$ which converges $k$-times in $X/H$ to $z$;
\item $\M_\U(\mathrm{Ind}\,\epsilon_z,\{\mathrm{Ind}\,\epsilon_{x_n}\})\ge k$;
\item for every open neighbourhood $V$ of $z$ such that $\{h\in H: h\cdot z\in V\}$ is relatively compact we have
\[
\underset{n}{\lim\,\sup}\,\nu\big(\{h\in H:h\cdot x_n\in V\}\big)\ge k\nu\big(\{h\in H:h\cdot z\in V\}\big);
\]
\item there exists a real number $R>k-1$ such that for every open neighbourhood $V$ of $z$ with $\{h\in H:h\cdot z\in V\}$ relatively compact we have
\[
\underset{n}{\lim\,\sup}\,\nu\big(\{h\in H:h\cdot x_n\in V\}\big)\ge R\nu\big(\{h\in H:h\cdot z\in V\}\big);\quad\text{and}
\]
\item there exists a basic decreasing sequence of compact neighbourhoods $\{W_m\}$ of $z$ such that, for each $m\ge 1$,
\[
\underset{n}{\lim\,\sup}\,\nu\big(\{h\in H:h\cdot x_n\in W_m\}\big)> (k-1)\nu\big(\{h\in H:h\cdot z\in W_m\}\big).
\]
\end{enumerate}
\end{theorem}
\begin{example}
In Example \ref{example_2-times_convergence_Green} we showed that in the free transformation group $X/\RR$, the sequence $\{x_n\}$ converges $2$-times, but not $3$-times, in $X/\RR$ to $z$. By taking a subsequence $\{x_{n_i}\}$ of $\{x_n\}$ and considering the arguments in Example \ref{example_2-times_convergence_Green}, we can see that $\{x_{n_i}\}$ converges $2$-times, but not $3$-times, in $X/\RR$ to $z$. In Example \ref{example_Green_lower_multiplicity} we noted that $\RR\cdot z$ is closed so in order to be able to apply Theorem \ref{AaH_thm5.3} it remains to show that $\RR\cdot x_n\rightarrow \RR\cdot z$ in $X/\RR$. But $x_n\rightarrow z$ in $X$ and the quotient map $X\rightarrow X/\RR$ is continuous, so $\RR\cdot x_n\rightarrow \RR\cdot z$. We can now apply Theorem \ref{AaH_thm5.3} to see that $\M_\U(\mathrm{Ind}\,\epsilon_z,\{\mathrm{Ind}\,\epsilon_{x_n}\})=2$.
\end{example}

While Green's transformation group from Example \ref{example_2-times_convergence_Green} `folds' vertical lines with an arc of a helix, Rieffel in \cite[Example~1.18]{Rieffel2004} modified this to allow any fixed number of helix-arc folds for each orbit. This can be used to construct transformation groups with an irreducible representation that has any desired relative upper and lower multiplicity numbers. In particular, Archbold and an Huef in \cite[Example~6.4]{Archbold-anHuef2006} use this to describe a transformation group and an irreducible representation with relative upper multiplicity 3 and lower multiplicity 2. Rather than repeating this example, after generalising Theorems \ref{AaH_thm1.1} and \ref{AaH_thm5.3} to principal groupoid results, in Example \ref{ML2_MU3_example} we use the path groupoid of a directed graph to describe an irreducible representation with relative upper multiplicity 3 and lower multiplicity 2.
\section{Lower multiplicity and \texorpdfstring{$k$}{\it k}-times convergence I}\label{sec_lower_multiplicity_1}
The key goal of this chapter is to describe the relationship between multiplicities of induced representations and strength of convergence in the orbit space. We start  this section by recalling the definition of $k$-times convergence in a groupoid from \cite{Clark-anHuef2008}. We then show that if a sequence converges $k$-times in the orbit space of a principal groupoid, then the lower multiplicity of the associated sequence of representations is at least $k$; the converse will be shown in Section \ref{sec_lower_multiplicity_2}.

Recall that a sequence $\{\gamma_n\}\subset G$ tends to infinity if it admits no convergent subsequence.
\begin{definition}\label{def_k-times_convergence}
Suppose $G$ is a topological groupoid and let $k\in\PP$. A sequence $\{x_n\}$ in $G^{(0)}$ is {\em $k$-times convergent}\index{k-times convergence@$k$-times convergence!in the orbit space of a groupoid} in $G^{(0)}/G$ to $z\in G^{(0)}$ if there exist $k$ sequences $\{\gamma_n^{(1)}\},\{\gamma_n^{(2)}\},\ldots,\{\gamma_n^{(k)}\}\subset G$ such that
\begin{enumerate}\renewcommand{\theenumi}{\roman{enumi}}\renewcommand{\labelenumi}{(\theenumi)}
\item\label{def_k-times_convergence_1} $s(\gamma_n^{(i)})=x_n$ for all $n$ and $1\le i\le k$;
\item\label{def_k-times_convergence_2} $r(\gamma_n^{(i)})\rightarrow z$ as $n\rightarrow\infty$ for $1\le i\le k$; and
\item\label{def_k-times_convergence_3} if $1\le i<j\le k$ then $\gamma_n^{(j)}(\gamma_n^{(i)})^{-1}\rightarrow\infty$ as $n\rightarrow\infty$.
\end{enumerate}
\end{definition}

For each $x\in G^{(0)}$, let $\L^x$ be the induced representation $\mathrm{Ind}\, \delta_x$ as in Section \ref{section_representations_and_groupoid_algebras}. The proof of the following proposition is based on \cite[Theorem~2.3]{Archbold-Deicke2005} and a part of \cite[Theorem~1.1]{Archbold-anHuef2006}.
\begin{prop}\label{AaH_thm_1.1_1_implies_2}
Suppose $G$ is a second countable, locally compact, Hausdorff, principal groupoid with Haar system $\lambda$. Let $z\in G^{(0)}$ and suppose that $\{x_n\}$ is a sequence in $G^{(0)}$ that converges $k$-times to $z$ in $G^{(0)}/G$. Then $\M_\L(\L^z,\{\L^{x_n}\})\ge k$.
\begin{proof}
We will use a contradiction argument. Suppose that $\M_\L(\L^z,\{\L^{x_n}\})=r<k$. 
Fix a real-valued $g\in C_c(G)$ so that $\|g\|_z>0$. Define $\eta\in L^2(G,\lambda_z)$ by $\eta(\alpha)=\|g\|_z^{-1}g(\alpha)$ for all $\alpha\in G$. Then
\[
\|\eta\|_z^2=\|g\|_z^{-2}\int g(\alpha)^2\, d\lambda_z(\alpha)=\|g\|_z^{-2}\|g\|_z^2=1,
\]
so $\eta$ is a unit vector in $L^2(G,\lambda_z)$ and the GNS construction of $\phi:=(\L^z(\cdot)\eta\,|\,\eta)$ is unitarily equivalent to $\L^z$. By the definition of lower multiplicity we now have
\[
\M_\L(\L^z,\{\L^{x_n}\})=\inf_{N\in\Nn} \M_\L(\phi,N,\{\L^{x_n}\})=r,
\]
so there exists $N\in\Nn$ such that
\[
\M_\L(\phi,N,\{\L^{x_n}\})=\underset{\scriptstyle n}{\lim\,\inf}\, d(\L^{x_n},\phi,N)=r,
\]
and consequently there exists a subsequence $\{y_m\}$ of $\{x_n\}$ such that
\begin{equation}\label{observation_to_contradict}
d(\L^{y_m},\phi,N)=r\quad\text{for all }m.
\end{equation}
 Since any subsequence of a sequence that is $k$-times convergent is also $k$-times convergent, we know that $\{y_m\}$ converges $k$-times to $z$ in $G^{(0)}/G$.

We will now use the $k$-times convergence of $\{y_m\}$ to construct $k$ sequences of unit vectors with sufficient properties to establish our contradiction. By the $k$-times convergence of $\{y_m\}$ there exist $k$ sequences
\[
\{\gamma_m^{(1)}\},\{\gamma_m^{(2)}\},\ldots,\{\gamma_m^{(k)}\}\subset G
\]
satisfying
{\allowdisplaybreaks\begin{enumerate}\renewcommand{\labelenumi}{(\roman{enumi})}
\item $s(\gamma_m^{(i)})=y_m$ for all $m$ and $1\le i\le k$;
\item $r(\gamma_m^{(i)})\rightarrow z$ as $m\rightarrow\infty$ for $1\le i\le k$; and
\item if $1\le i<j\le k$ then $\gamma_m^{(j)}(\gamma_m^{(i)})^{-1}\rightarrow\infty$ as $m\rightarrow\infty$.
\end{enumerate}}
\noindent For each $1\le i\le k$ and $m\ge 1$, define $\eta_m^{(i)}$ by
\[
\eta_m^{(i)}(\alpha)=\|g\|_{r(\gamma_m^{(i)})}^{-1}g\big(\alpha(\gamma_m^{(i)})^{-1}\big)\quad\text{for all }\alpha\in G.
\]
It follows from Haar-system invariance that
\begin{align*}
\|\eta_m^{(i)}\|_{y_m}^2&=\|g\|_{r(\gamma_m^{(i)})}^{-2}\int g\big(\alpha(\gamma_m^{(i)})^{-1}\big)^2\,d\lambda_{y_m}(\alpha)\\
&=\|g\|_{r(\gamma_m^{(i)})}^{-2}\int g(\alpha)^2\,d\lambda_{r(\gamma_m^{(i)})}(\alpha)\\
&=\|g\|_{r(\gamma_m^{(i)})}^{-2}\|g\|_{r(\gamma_m^{(i)})}^2=1,
\end{align*}
so the $\eta_m^{(i)}$ are unit vectors in $L^2(G,\lambda_{y_m})$. 

Now suppose that $1\le i<j\le k$. Then
\begin{equation}\label{working_the_etas}
(\eta_m^{(i)}\,|\,\eta_m^{(j)})_{y_m}= \|g\|_{r(\gamma_m^{(i)})}^{-1}\|g\|_{r(\gamma_m^{(j)})}^{-1}\int g\big(\alpha(\gamma_m^{(i)})^{-1}\big)g\big(\alpha(\gamma_m^{(j)})^{-1}\big)\, d\lambda_{y_m}(\alpha)\\
\end{equation}
Since $\gamma_m^{(i)}(\gamma_m^{(j)})^{-1}\rightarrow\infty$, the sequence $\{\gamma_m^{(i)}(\gamma_m^{(j)})^{-1}\}$ is eventually not in the compact set $(\mathrm{supp}\, g)^{-1}(\mathrm{supp}\, g)$, and so there exists $m_0$ such that if $m\ge m_0$, then
\[
(\mathrm{supp}\, g)\gamma_m^{(i)}\cap(\mathrm{supp}\, g)\gamma_m^{(j)}=\emptyset.
\]
(To see this claim, note that if $(\mathrm{supp}\, g)\gamma_m^{(i)}\cap (\mathrm{supp}\, g)\gamma_m^{(j)}\ne\emptyset$ then there exist $\alpha,\beta\in \mathrm{supp}\, g$ such that $\alpha\gamma_m^{(i)}=\beta\gamma_m^{(j)}$, and so $\gamma_m^{(i)}(\gamma_m^{(j)})^{-1}=\alpha^{-1}\beta\in (\mathrm{supp}\, g)^{-1}(\mathrm{supp}\, g)$.) For the integrand of \eqref{working_the_etas} to be non-zero, both $\alpha(\gamma_m^{(i)})^{-1}$ and $\alpha(\gamma_m^{(j)})^{-1}$ must be in $\mathrm{supp}\, g$, so $\alpha$ must be in $(\mathrm{supp}\, g)\gamma_m^{(i)}\cap(\mathrm{supp}\, g)\gamma_m^{(j)}$. But this is not possible if $m\ge m_0$. Thus, for any distinct $i,j$, we will eventually have $\eta_m^{(i)}\perp\eta_m^{(j)}$.

For the last main component of this proof we will establish that
\[
\big(\L^{y_m}(\cdot)\eta_m^{(i)}\,\big|\,\eta_m^{(i)}\big)\rightarrow \big(\L^z(\cdot)\eta\,\big|\,\eta\big)=\phi\quad\text{as }m\rightarrow\infty
\]
in the dual of $C^*(G)$ with the weak$^*$ topology for each $i$. Fix $f\in C_c(G)$. We have
{\allowdisplaybreaks\begin{align}
\big(\L^{z}(f)\eta\,\big|\,\eta\big)&=
\int_G\big(\L^z(f)\eta\big)(\alpha)\eta(\alpha)\, d\lambda_z(\alpha)\notag\\
&=\int_G\int_G f(\alpha\beta^{-1})\eta(\beta)\eta(\alpha)\,d\lambda_z(\beta)\,d\lambda_z(\alpha)\notag\\
\label{dealing_with_etas} &=\|g\|_z^{-2}\int_G\int_G f(\alpha\beta^{-1})g(\beta)g(\alpha)\,d\lambda_z(\beta)\,d\lambda_z(\alpha)
\end{align}
}Now fix $1\le i\le k$. By the invariance of the Haar system we have
{\allowdisplaybreaks\begin{align}
\big(\L^{y_m}&(f)\eta_m^{(i)}\,\big|\,\eta_m^{(i)}\big)=\int_G\int_G f(\alpha\beta^{-1})\eta_m^{(i)}(\beta)\eta_m^{(i)}(\alpha)\,d\lambda_{y_m}(\beta)\,d\lambda_{y_m}(\alpha)\notag\\
&=\|g\|_{r(\gamma_m^{(i)})}^{-2}\int_G\int_G f(\alpha\beta^{-1})g\big(\alpha(\gamma_m^{(i)})^{-1}\big)g\big(\beta(\gamma_m^{(i)})^{-1}\big)\, d\lambda_{y_m}(\beta)\,d\lambda_{y_m}(\alpha)\notag\\
&=\|g\|_{r(\gamma_m^{(i)})}^{-2}\int_G\int_G f(\alpha\beta^{-1})g(\alpha)g(\beta)\, d\lambda_{r(\gamma_m^{(i)})}(\beta)\, d\lambda_{r(\gamma_m^{(i)})}(\alpha)\notag\\
&=\|g\|_{r(\gamma_m^{(i)})}^{-2}\int_G f\ast g(\alpha) g(\alpha)\, d\lambda_{r(\gamma_m^{(i)})}(\alpha).\label{still_working_etas}
\end{align}
}We know that $r(\gamma_m^{(i)})\rightarrow z$ as $m\rightarrow\infty$ so, by the continuity of the Haar system, $\|g\|_{r(\gamma_m^{(i)})}\rightarrow \|g\|_z$ as $m\rightarrow\infty$. Since $f\ast g\in C_c(G)$ we can apply the continuity of the Haar system with \eqref{dealing_with_etas} and \eqref{still_working_etas} to see that
\begin{align*}
\big(\L^{y_m}(f)\eta_m^{(i)}\,\big|\,\eta_m^{(i)}\big)&=\|g\|_{r(\gamma_m^{(i)})}^{-2}\int_G f\ast g(\alpha) g(\alpha)\, d\lambda_{r(\gamma_m^{(i)})}(\alpha)\\
&\rightarrow \|g\|_{z}^{-2}\int_G f\ast g(\alpha) g(\alpha)\, d\lambda_{z}(\alpha)=\big(\L^{z}(f)\eta\,\big|\,\eta\big)
\end{align*}
as $m\rightarrow\infty$.

We have thus shown that, for each $i$,
\[
\big(\L^{y_m}(\cdot)\eta_m^{(i)}\,\big|\,\eta_m^{(i)}\big)\rightarrow \big(\L^z(\cdot)\eta\,\big|\,\eta\big)=\phi
\]
in the dual of $C^*(G)$ equipped with the weak$^*$ topology. Thus there exists $m_1$ such that for any $m\ge m_1$ and any $1\le i\le k$, the pure state $\big(\L^{y_m}(\cdot)\eta_m^{(i)}\,\big|\,\eta_m^{(i)}\big)$ is in $\phi+N$. We have now established that every $\eta_m^{(i)}$ with $m\ge\max\{m_0,m_1\}$ is in $\mathrm{Vec}(\L^{y_m},\phi,N)$ with $\eta_m^{(i)}\perp\eta_m^{(j)}$ for $i\ne j$, so $d(\L^{y_m},\phi,N)\ge k$ for all $m\ge\max\{m_0,m_1\}$, contradicting our choice of $\{y_m\}$ that in \eqref{observation_to_contradict} had $d(\L^{y_m},\phi,N)=r<k$ for all $m$.
\end{proof}
\end{prop}

\section{Measure ratios and \texorpdfstring{$k$}{\it k}-times convergence}\label{sec_measure_ratios_k-times_convergence}
In this section we show that lower bounds on measure ratios along orbits give strength of convergence in the orbit space. We begin with a lemma before generalising \cite[Proposition~4.1]{Archbold-anHuef2006}.  

\begin{lemma}\label{the_unbroken_lemma}
Suppose $G$ is a second countable, locally compact, Hausdorff groupoid with Haar system $\lambda$. Let $W$ be a compact neighbourhood of $z\in G^{(0)}$ and let $K$ be a compact subset of $G$. Let $\{x_n\}$ be a sequence in $G^{(0)}$ such that $[x_n]\rightarrow [z]$ uniquely in $G^{(0)}/G$. Then for every $\delta>0$ there exists $n_0$ such that, for every $n\ge n_0$ and every $\gamma\in G_{x_n}^W$,
\[
\lambda_{x_n}(K\gamma\cap G^W)<\lambda_z(G^W)+\delta.
\]
\begin{proof}
Suppose not. Then, by passing to a subsequence if necessary, there exists $\delta>0$ such that for each $n$ there exists $\gamma_n\in G_{x_n}^W$ with
\begin{equation}\label{assumption_in_unbroken_lemma}
\lambda_{x_n}(K\gamma_n\cap G^W)\ge \lambda_z(G^W)+\delta.
\end{equation}
Since each $r(\gamma_n)$ is in the compact set $W$, we can pass to a subsequence so that $r(\gamma_n)\rightarrow y$ for some $y\in G^{(0)}$. This implies $[r(\gamma_n)]\rightarrow [y]$, but $[r(\gamma_n)]=[s(\gamma_n)]=[x_n]$ and $[x_n]\rightarrow [z]$ uniquely, so $[y]=[z]$. Choose $\psi\in G$ with $s(\psi)=z$ and $r(\psi)=y$. By Haar-system invariance
\[
\lambda_{x_n}(K\gamma_n\cap G^W)=\lambda_{r(\gamma_n)}(K\cap G^W),
\]
so by applying Lemma \ref{astrids_lim_sup} with the compact space $K\cap G^W$ and $\{r(\gamma_n)\}$ converging to $y$,
\begin{align*}
\underset{\scriptstyle n}{\lim\,\sup}\,\lambda_{x_n}(K\gamma_n\cap G^W)
&=\underset{\scriptstyle n}{\lim\,\sup}\,\lambda_{r(\gamma_n)}(K\cap G^W)\\
&\le\lambda_y(K\cap G^W)\quad\text{(by Lemma \ref{astrids_lim_sup})}\\
&=\lambda_z(K\psi\cap G^W)\quad\text{(Haar-system invariance)}\\
&\le\lambda_z(G^W).
\end{align*}
This contradicts our assertion \eqref{assumption_in_unbroken_lemma}.
\end{proof}
\end{lemma}

\begin{prop}\label{AaH_prop4_1_1}
Suppose $G$ is a second countable, locally compact, Hausdorff groupoid with Haar system $\lambda$. Let $k\in\PP$ and $z\in G^{(0)}$ with $[z]$ locally closed in $G^{(0)}$. Assume that $\{x_n\}$ is a sequence in $G^{(0)}$ such that $[x_n]\rightarrow [z]$ uniquely in $G^{(0)}/G$. Suppose $\{W_m\}$ is a basic decreasing sequence of compact neighbourhoods of $z$ such that each $m$ satisfies 
\[
\underset{\scriptstyle n}{\lim\,\inf}\,\lambda_{x_n}(G^{W_m})>(k-1)\lambda_z(G^{W_m}).
\]
Then $\{x_n\}$ converges $k$-times in $G^{(0)}/G$ to $z$.
\begin{proof}
Let $\{K_m\}$ be an increasing sequence of compact subsets of $G$ such that $G=\bigcup_{m\ge 1}\mathrm{Int}\,K_m$. By the regularity of $\lambda_z$, for each $m\ge 1$ there exist $\delta_m>0$ and an open neighbourhood $U_m$ of $G_z^{W_m}$ such that
\begin{equation}\label{AaH4.1}
\underset{\scriptstyle n}{\lim\,\inf}\,\lambda_{x_n}(G^{W_m})>(k-1)\lambda_z(U_m)+\delta_m.
\end{equation}
We will construct, by induction, a strictly increasing sequence of positive integers $\{n_m\}$ such that, for all $n\ge n_m$,
\begin{align}
&\lambda_{x_n}(K_m\alpha\cap G^{W_m})<\lambda_z(U_m)+\delta_m/k\quad\text{for all }\alpha\in G_{x_n}^{W_m},\quad\text{and}\label{AaH4.2}\\
&\lambda_{x_n}(G^{W_m})>(k-1)\lambda_z(U_m)+\delta_m.\label{AaH4.3}
\end{align}
By applying Lemma \ref{the_unbroken_lemma} with $\delta=\lambda_z(U_1)-\lambda_z(G^{W_1})+\delta_1/k$ there exists $n_1$ such that $n\ge n_1$ implies
\[\lambda_{x_n}(K_1\alpha\cap G^{W_1})<\lambda_z(U_1)+\delta_1/k \quad\text{for all }\alpha\in G_{x_n}^{W_1},\] establishing \eqref{AaH4.2} for $m=1$. If necessary we can increase $n_1$ to ensure \eqref{AaH4.3} holds for $m=1$ by considering \eqref{AaH4.1}. Assuming that we have constructed $n_1<n_2<\cdots<n_{m-1}$, we apply Lemma \ref{the_unbroken_lemma} with $\delta=\lambda_z(U_m)-\lambda_z(G^{W_m})+\delta_m/k$ to obtain $n_m>n_{m-1}$ such that \eqref{AaH4.2} holds, and again, if necessary, increase $n_m$ to obtain \eqref{AaH4.3}.

If $n_1>1$ then, for each $1\le n<n_1$ and $1\le i\le k$, let $\gamma_n^{(i)}=x_n$. For each $n\ge n_1$ there is a unique $m$ such that $n_m\le n<n_{m+1}$. For every such $n$ and $m$ choose $\gamma_n^{(1)}\in G_{x_n}^{W_m}$ (which is always non-empty by \eqref{AaH4.3}). Using \eqref{AaH4.2} and \eqref{AaH4.3} we have
\begin{align*}
\lambda_{x_n}(G^{W_m}\backslash K_m\gamma_n^{(1)})&=\lambda_{x_n}(G^{W_m})-\lambda_{x_n}(G^{W_m}\cap K_m\gamma_n^{(1)})\\
&>\big((k-1)\lambda_z(U_m)+\delta_m\big)-\big(\lambda_z(U_m)+\delta_m/k\big)\\
&=(k-2)\lambda_z(U_m)+\frac{(k-1)}k\delta_m.
\end{align*}
So for each $n\ge n_1$ and its associated $m$ we can choose $\gamma_n^{(2)}\in G_{x_n}^{W_m}\backslash K_m\gamma_n^{(1)}$. We now have
{\allowdisplaybreaks\begin{align*}
\lambda_{x_n}&\big(G^{W_m}\backslash (K_m\gamma_n^{(1)}\cup K_m\gamma_n^{(2)})\big)\\
&=\lambda_{x_n}(G^{W_m}\backslash K_m\gamma_n^{(1)})-\lambda_{x_n}\big((G^{W_m}\backslash K_m\gamma_n^{(1)})\cap K_m\gamma_n^{(2)}\big)\\
&\ge\lambda_{x_n}(G^{W_m}\backslash K_m\gamma_n^{(1)})-\lambda_{x_n}(G^{W_m}\cap K_m\gamma_n^{(2)})\\
&>\Big((k-2)\lambda_z(U_m)+\frac{(k-1)}k\delta_m\Big)-\Big(\lambda_z(U_m)+\delta_m/k\Big)\\
&=(k-3)\lambda_z(U_m)+\frac{(k-2)}k\delta_m,
\end{align*}}enabling us to choose $\gamma_n^{(3)}\in G_{x_n}^{W_m}\backslash (K_m\gamma_n^{(1)}\cup K_m\gamma_n^{(2)})$. By continuing this process, for each $j=3,\ldots,k$ and each $n\ge n_1$ we have
\[
\lambda_{x_n}\Bigg(G^{W_m}\backslash \bigg(\bigcup_{i=1}^{j-1}K_m\gamma_n^{(i)}\bigg)\Bigg)>(k-j)\lambda_z(U_m)+\frac{(k-j-1)\delta_m}k,
\]
enabling us to choose
\begin{equation}\label{eqn_choosing_gammas}
\gamma_n^{(j)}\in G^{W_m}_{x_n}\backslash \bigg(\bigcup_{i=1}^{j-1}K_m\gamma_n^{(i)}\bigg).
\end{equation}
Note that for $n_m\le n<n_{m+1}$ we have $\gamma_n^{(j)}\notin K_m\gamma_n^{(i)}$ for $1\le i<j\le k$.

We will now establish that $x_n$ converges $k$-times to $z$ in $G^{(0)}/G$ by considering the $\gamma_n^{(i)}$. Note that $s(\gamma_n^{(i)})=x_n$ for all $n$ and $i$ by our choice of the $\gamma_n^{(i)}$. To see that $r(\gamma_n^{(i)})\rightarrow z$ as $n\rightarrow\infty$ for $1\le i\le k$, first fix $i$ and let $V$ be an open neighbourhood of $z$. Since $W_m\rightarrow \{z\}$ there exists $m_0$ such that $m\ge m_0$ implies $W_m\subset V$. For each $n\ge n_{m_0}$ there exists a $m\ge m_0$ such that $n_m\le n <n_{m+1}$, and so $r(\gamma_n^{(i)})\in W_m\subset V$.

Finally we claim that $\gamma_n^{(j)}(\gamma_n^{(i)})^{-1}\rightarrow\infty$ as $n\rightarrow\infty$ for $1\le i<j\le k$. Fix $i<j$ and let $K$ be a compact subset of $G$. There exists $m_0$ such that $K\subset K_m$ for all $m\ge m_0$. Then for $n\ge n_{m_0}$ there exists $m\ge m_0$ such that $n_m\le n<n_{m+1}$. By \eqref{eqn_choosing_gammas} we know
\begin{align*}
\gamma_n^{(j)}&\in G_{x_n}^{W_m}\backslash (K_m\gamma_n^{(i)})\\
&=\big(G_{x_n}^{W_m}(\gamma_n^{(i)})^{-1}\gamma_n^{(i)}\big)\backslash (K_m\gamma_n^{(i)})\\
&=\big((G_{x_n}^{W_m}(\gamma_n^{(i)})^{-1})\backslash K_m\big)\gamma_n^{(i)},
\end{align*}
and so $\gamma_n^{(j)}(\gamma_n^{(i)})^{-1}\in \big(G_{x_n}^{W_m}(\gamma_n^{(i)})^{-1}\big)\backslash K_m\subset G\backslash K_m\subset G\backslash K$, enabling us to conclude that $\{x_n\}$ converges $k$-times in $G^{(0)}/G$ to $z$.
\end{proof}
\end{prop}

In Proposition~\ref{part_of_AaH4.2_generalisation} we prove a generalisation of a part of \cite[Proposition~4.2]{Archbold-anHuef2006}.

\begin{prop}\label{part_of_AaH4.2_generalisation}
Suppose $G$ is a second countable, locally compact, Hausdorff, principal groupoid with Haar system $\lambda$. Suppose that $z\in G^{(0)}$ with $[z]$ locally closed in $G^{(0)}$ and suppose $\{x_n\}$ is a sequence in $G^{(0)}$. Assume that for every open neighbourhood $V$ of $z$ in $G^{(0)}$ such that $G_z^V$ is relatively compact, $\lambda_{x_n}(G^V)\rightarrow \infty$ as $n\rightarrow\infty$. Then, for every $k\ge 1$, the sequence $\{x_n\}$ converges $k$-times in $G^{(0)}/G$ to $z$.
\begin{proof}
Let $\{ K_m\}$ be an increasing sequence of compact subsets of $G$ such that $G=\cup_{m\ge 1}\,\mathrm{Int}\, K_m$. By Lemma \ref{corollary_astrid_lemma}, for each $K_m$ there exists an open neighbourhood $V_m$ of $z$ such that $x\in V_m$ implies $\lambda_x(K_m)<\lambda_z(K_m)+1$. Since $[z]$ is locally closed, by Lemma~4.1(1) in \cite{Clark-anHuef2010-preprint} we can crop $V_1$ if necessary to ensure that $G_z^{V_1}$ is relatively compact. By further cropping each $V_m$ we may assume that $\{V_m\}$ is a decreasing neighbourhood basis of $z$. By our hypothesis, for each $m$ there exists $n_m$ such that 
\begin{equation}\label{lower_bound}
n\ge n_m\quad\text{implies}\quad\lambda_{x_n}(G^{V_m})>k\big(\lambda_z(K_m)+1\big).
\end{equation}
Note that for any $\gamma\in G_{x_n}^{V_m}$ with $n\ge n_m$, we have $r(\gamma)\in V_m$, and so $\lambda_{r(\gamma)}(K_m)<\lambda_z(K_m)+1$. By Haar-system invariance we know that $\lambda_{r(\gamma)}(K_m)=\lambda_{x_n}(K_m\gamma)$, which shows us that
\begin{equation}\label{upper_bound}
\lambda_{x_n}(K_m\gamma)<\lambda_z(K_m)+1.
\end{equation}
If necessary we can increase the elements of $\{n_m\}$ so that it is a strictly increasing sequence.

We now proceed as in the proof of Proposition \ref{AaH_prop4_1_1}. For all $n<n_1$ and $1\le i\le k$ let $\gamma_n^{(i)}=x_n$. For each $n\ge n_1$ there exists a unique number $m(n)$ such that $n_{m(n)}\le n<n_{m(n)+1}$. For the remainder of this proof we will write $m$ instead of $m(n)$ because the specific $n$ will be clear from the context. For each $n\ge n_1$ choose $\gamma_n^{(1)}\in G_{x_n}^{V_{m}}$. Then by \eqref{lower_bound} and \eqref{upper_bound} we have
\begin{align*}
\lambda_{x_n}(G^{V_m}\backslash K_m\gamma_n^{(1)})&=\lambda_{x_n}(G^{V_m})-\lambda_{x_n}(G^{V_m}\cap K_m\gamma_n^{(1)})\\
&\ge \lambda_{x_n}(G^{V_m})-\lambda_{x_n}(K_m\gamma_n^{(1)})\\
&>k\big(\lambda_z(K_m)+1\big) - \big(\lambda_z(K_m)+1\big)\\
&=(k-1)\big(\lambda_z(K_m)+1\big).
\end{align*}
We can thus choose $\gamma_n^{(2)}\in G_{x_x}^{V_m}\backslash K_m\gamma_n^{(1)}$ for each $n\ge n_1$. This now gives us
\begin{align*}
\lambda_{x_n}&(G^{V_m}\backslash (K_m\gamma_n^{(1)}\cup K_m\gamma_n^{(2)}))\\
&=\lambda_{x_n}(G^{V_m}\backslash K_m\gamma_n^{(1)})-\lambda_{x_n}\big( (G^{V_m}\backslash K_m\gamma_n^{(1)})\cap K_m\gamma_n^{(1)}\big)\\
&\ge \lambda_{x_n}(G^{V_m}\backslash K_m\gamma_n^{(1)})-\lambda_{x_n}(K_m\gamma_n^{(2)})\\
&>(k-1)\big(\lambda_z(K_m)+1\big) - \big(\lambda_z(K_m)+1\big)\\
&=(k-2)\big(\lambda_z(K_m)+1\big).
\end{align*}
Continuing in this manner we can choose 
\[
\gamma_n^{(j)}\in G^{V_m}_{x_n}\backslash \bigg(\bigcup_{i=1}^{j-1}K_m\gamma_n^{(i)}\bigg)
\]
for every $n\ge n_1$ and $j=3,\ldots,k$. The tail of the proof of Proposition \ref{AaH_prop4_1_1} establishes our desired result.
\end{proof}
\end{prop}

\section{Measure ratios and bounds on lower multiplicity}\label{sec_measure_ratios_bounds_lower_multiplicity}
In this section we show that upper bounds on measure ratios\index{measure ratio} along orbits give upper bounds on  multiplicities. 

\begin{lemma}\label{based_on_Ramsay}
Let $G$ be a second countable, locally compact, Hausdorff groupoid. Suppose $z\in G^{(0)}$ and $[z]$ is locally closed. Then the restriction of $r$ to $G_z/(G|_{\{z\}})$ is a homeomorphism onto $[z]$. If in addition $G$ is principal, then the restriction of $r$ to $G_z$ is a homeomorphism onto $[z]$.
\begin{proof}
We consider the transitive groupoid $G|_{[z]}$. Since $[z]$ is locally closed, $G|_{[z]}$ is a second countable, locally compact, Hausdorff groupoid. Thus $G|_{[z]}$ is Polish by, for example, \cite[Lemma~6.5]{Williams2007}. Now \cite[Theorem~2.1]{Ramsay1990} applies to give the result.
\end{proof}
\end{lemma}

Theorem~\ref{M2_thm} is based on \cite[Theorem~3.1]{Archbold-anHuef2006}; it is only an intermediary result which will be used to prove a sharper bound in Theorem \ref{M_thm}.
\needspace{6\baselineskip}
\begin{theorem}\label{M2_thm}
Suppose $G$ is a second countable, locally compact, Hausdorff principal groupoid with Haar system $\lambda$. Let $M\in\RR$ with $M\ge 1$, suppose $z\in G^{(0)}$ such that $[z]$ is locally closed and let $\{x_n\}$ be a sequence in $G^{(0)}$. Suppose there exists an open neighbourhood $V$ of $z$ in $G^{(0)}$ such that $G_z^V$ is relatively compact and 
\[
\lambda_{x_n}(G^V)\le M\lambda_z(G^V)
\]
frequently (in the sense that there is a subsequence
$\{x_{n_i}\}$ of $\{x_n\}$ with $\lambda_{x_{n_i}}(G^V)\leq
M\lambda_z(G^V)$ for all $i$). Then $\M_\L(\L^z,\{\L^{x_n}\})\le\lfloor M^2\rfloor$.
\begin{proof}
Fix $\epsilon>0$ such that $M^2(1+\epsilon)^2<\lfloor M^2\rfloor +1$. We will build a function $D\in C_c(G)$ such that $\L^z(D^\ast\ast D)$ is a rank-one projection and
\[
\tr\big(\L^{x_n}(D^\ast\ast D)\big)<M^2(1+\epsilon)^2<\lfloor M^2\rfloor +1
\]
frequently. By the generalised lower semi-continuity result of \cite[Theorem~4.3]{Archbold-Spielberg1996} we will have
\begin{align*}
\liminf\,\tr\big(\L^{x_n}(D^\ast\ast D)\big)&\ge \M_\L(\L^z,\{\L^{x_n}\})\, \tr \big(\L^z(D^*\ast D)\big)\\
&=\M_\L(\L^z,\{\L^{x_n}\}),
\end{align*}
and the result will follow.

For the next few paragraphs we will be working with $G_z$ equipped with the subspace topology. Note that $\lambda_z$ can be thought of as a Radon measure on $G_z$ with $\lambda_z(S\cap G_z)=\lambda_z(S)$ for any $\lambda_z$-measurable subset $S$ of $G$. Fix $\delta>0$ such that
\[
\delta<\frac{\epsilon\lambda_z(G^V)}{1+\epsilon}<\lambda_z(G^V).
\]
Since $G_z$ is a second countable, locally compact, Hausdorff space, the Radon measure $\lambda_z$ is regular (see, for example, \cite[Proposition~6.3.6]{Pedersen1989}). It follows that there exists a $G_z$-compact subset $W$ of $G_z^V$ such that
\[
0<\lambda_z(G_z^V)-\delta<\lambda_z(W).
\]
Since $W$ is $G_z$-compact there exists a $G_z$-compact neighbourhood $W_1$ of $W$ that is contained in $G_z^V$ and there exists a continuous function $g:G_z\rightarrow [0,1]$ that is identically one on $W$ and zero off the interior of $W_1$. We have
\[
\lambda_z(G^V)-\delta=\lambda_z(G_z^V)-\delta<\lambda_z(W)\le\int_{G_z} g(t)^2\, d\lambda_z(t)=\|g\|_z^2,
\]
and hence
\begin{equation}\label{AaH3.1}
\frac{\lambda_z(G^V)}{\|g\|_z^2}<1+\frac{\delta}{\|g\|_z^2}<1+\frac{\delta}{\lambda_z(G^V)-\delta}<1+\epsilon.
\end{equation}

By Lemma \ref{based_on_Ramsay} the restriction $\tilde{r}$ of $r$ to $G_z$ is a homeomorphism onto $[z]$. So there exists a continuous function $g_1:\tilde{r}(W_1)\rightarrow [0,1]$ such that $g_1\big(\tilde{r}(\gamma)\big)=g(\gamma)$ for all $\gamma\in W_1$. Thus $\tilde{r}(W_1)$ is $[z]$-compact, which implies that $\tilde{r}(W_1)$ is $G^{(0)}$-compact. Since we know that $G^{(0)}$ is second countable and Hausdorff, Tietze's Extension Theorem can be applied to extend $g_1$ to a continuous map $g_2:G^{(0)}\rightarrow [0,1]$. Because $\tilde{r}(W_1)$ is a compact subset of the open set $V$, there exist a compact neighbourhood $P$ of $\tilde{r}(W_1)$ contained in $V$ and a continuous function $h:G^{(0)}\rightarrow [0,1]$ that is identically one on $\tilde{r}(W_1)$ and zero off the interior of $P$. Note that $h$ has compact support that is contained in $P$.

We set $f(x)=h(x)g_2(x)$. Then $f\in C_c(G^{(0)})$ with $0\le f\le 1$ and 
\begin{equation}\label{supp_f_contained_in_V}
\mathrm{supp}\,f\subset \mathrm{supp}\,h\subset P\subset V.
\end{equation}
Note that
{\allowdisplaybreaks
\begin{align}
\|f\circ r\|_z^2&=\int_{G_z} f\big(\tilde{r}(\gamma)\big)^2\,d\lambda_z(\gamma)\notag\\
&=\int_{G_z} h\big(\tilde{r}(\gamma)\big)^2g_2\big(\tilde{r}(\gamma)\big)^2\, d\lambda_z(\gamma)\notag\\
&\ge\int_{W_1} h\big(\tilde{r}(\gamma)\big)^2 g(\gamma)^2\, d\lambda_z(\gamma)\notag\\
&=\int_{W_1}g(\gamma)^2\,d\lambda_z(\gamma)\notag\\
&=\|g\|_z^2\label{AaH3.2}
\end{align}}since $\mathrm{supp}\, g\subset W_1$ and $h$ is identically one on $\tilde{r}(W_1)$. We now define $F\in C_c(G^{(0)})$ by
\begin{equation}\label{definition_F}
F(x)=\frac{f(x)}{\|f\circ r\|_z}.
\end{equation}
Then $\|F\circ r\|_z=1$ and
\begin{equation}\label{unmotivated_ref}
F\circ r(\gamma)\ne 0\implies h\big(r(\gamma)\big)\ne 0 \implies r(\gamma)\in V \implies \gamma\in G^V.
\end{equation}

Let $N=\mathrm{supp}\, F$ so that $N=\mathrm{supp}\, f\subset V$ by \eqref{supp_f_contained_in_V} and \eqref{definition_F}. Since $G_z^V$ is relatively compact by our hypothesis, the set $\overline{G_z^N}$ is compact. Let $b\in C_c(G)$ be a function that is identically one on $(\overline{G_z^N})(\overline{G_z^N})^{-1}$ and has range contained in $[0,1]$. We can assume that $b$ is self-adjoint by considering $\frac12(b+b^*)$ if necessary. Define $D\in C_c(G)$ by
\[
D(\gamma):=F\big(r(\gamma)\big)F\big(s(\gamma)\big)b(\gamma).
\]
For $\xi\in L^2(G,\lambda_u)$ and $\gamma\in G$ we have
\begin{align*}
\big(\L^{u}(D)\xi\big)(\gamma)&=\int_GD(\gamma\alpha^{-1})\xi(\alpha)\,d\lambda_u(\alpha)\\
&=\int_G F\big(r(\gamma)\big)F\big(s(\alpha^{-1})\big) b(\gamma\alpha^{-1})\xi(\alpha)\, d\lambda_u(\alpha)\\
&=F\big(r(\gamma)\big) \int_G F\big(r(\alpha)\big) b(\gamma\alpha^{-1})\xi(\alpha)\, d\lambda_u(\alpha).
\end{align*}

In the case where $u=z$, if $\alpha,\gamma\in \mathrm{supp}\, F\circ r\cap s^{-1}(z)$, then $r(\alpha),r(\gamma)\in \mathrm{supp}\, F=N$ and $\gamma,\alpha\in G_z^N$. This implies $b(\gamma\alpha^{-1})=1$, so
\begin{align*}
\big(\L^z(D)\xi\big)(\gamma)&=F\big(r(\gamma)\big)\int_G F\big(r(\alpha)\big) \xi(\alpha)\, d\lambda_z(\alpha)\\
&=(\xi\,|\,F\circ r)_z F\circ r(\gamma),
\end{align*}
and $L^z(D)$ is a rank-one projection.

By the hypothesis on $V$ there exists a subsequence $\{x_{n_i}\}$ of $\{x_n\}$ such that
\[
\lambda_{x_{n_i}}(G^V)\le M\lambda_z(G^V)
\]
for all $i\ge 1$. If we define $E:=\{\gamma\in G:F\big(r(\gamma)\big)\ne 0\}$ then $E$ is open with
\begin{equation}\label{measure_of_E_finite}
\lambda_{x_{n_i}}(E)\le\lambda_{x_{n_i}}(G^V)\le M\lambda_z(G^V)
\end{equation}
by \eqref{unmotivated_ref} and
\begin{equation}\label{AaH3.3}
\int_G\big(F\circ r(\gamma)\big)^2\, d\lambda_{x_{n_i}}(\gamma)\le \frac{\lambda_{x_{n_i}}(E)}{\|f\circ r\|_z^2}
\le \frac{M\lambda_z(G^V)}{\|g\|_z^2}.
\end{equation}
by \eqref{AaH3.2}. Consider the continuous function
\[
T(\alpha,\beta):=F\big(r(\alpha)\big) F\big(r(\beta)\big) b(\alpha\beta^{-1}).
\]
Note that
\begin{align*}
\int_G &T(\alpha,\beta)^2\, d(\lambda_{x_{n_i}}\times\lambda_{x_{n_i}})(\alpha,\beta)\\
&=\int_G F\big(r(\alpha)\big)^2F\big(r(\beta)\big)^2b(\alpha\beta^{-1})^2\, d(\lambda_{x_{n_i}}\times\lambda_{x_{n_i}})(\alpha,\beta)\\
&\le\|F\|_\infty^4\int_G \chi_{E\times E}(\alpha,\beta)\,d(\lambda_{x_{n_i}}\times\lambda_{x_{n_i}})(\alpha,\beta)\\
&=\|F\|_\infty^4\lambda_{x_{n_i}}(E)^2,
\end{align*}
which is finite by \eqref{measure_of_E_finite}. Thus
\[T\in L^2(G\times G,\lambda_{x_{n_i}}\times \lambda_{x_{n_i}}),\]
and since $T$ is conjugate symmetric, \cite[Proposition~3.4.16]{Pedersen1989} implies that $\L^{x_{n_i}}(D)$ is the self-adjoint Hilbert-Schmidt operator on $L^2(G,\lambda_{x_{n_i}})$ with kernel $T$. It follows that $\L^{x_{n_i}}(D^*\ast D)$ is a trace-class operator, and since we equip the Hilbert-Schmidt operators with the trace norm, we have
\[
\tr\,\L^{x_{n_i}}(D^*\ast D)=\|T\|_{L^2(\lambda_{x_{n_i}}\times\lambda_{x_{n_i}})}^2.
\]
Applying Fubini's Theorem to $T$ now gives
{\allowdisplaybreaks\begin{align}
\tr\,\L^{x_{n_i}}&(D^*\ast D)\notag\\
&=\int_G\int_G F\big(r(\alpha)\big)^2 F\big(r(\beta)\big)^2b(\alpha\beta^{-1})^2\, d\lambda_{x_{n_i}}(\alpha)\, d\lambda_{x_{n_i}}(\beta)\notag\\
&\le\bigg(\int_G F\big(r(\alpha)\big)^2\, d\lambda_{x_{n_i}}(\alpha)\bigg)^2\notag\\
&\le\frac{M^2\lambda_z(G^V)^2}{\|g\|_z^4}\text{\quad (using \eqref{AaH3.3})}\notag\\
&<M^2(1+\epsilon)^2\text{\quad (using \eqref{AaH3.1})}\label{Aah3.4}.
\end{align}
}Now
\begin{align*}
\M_\L(\L^z,\{\L^{x_n}\})&\le \underset{\scriptstyle n}{\lim\,\inf}\, \tr\big(\L^{x_n}(D^*\ast D)\big)\\
&\le M^2(1+\epsilon)^2\\
&<\lfloor M^2\rfloor+1,
\end{align*}
and hence $\M_\L(\L^z,\{\L^{x_n}\})\le\lfloor M^2\rfloor$, completing the proof.
\end{proof}
\end{theorem}

The following proposition is an immediate consequence of Theorem \ref{M2_thm} and Proposition \ref{part_of_AaH4.2_generalisation}. This result will be strengthened later in Corollary \ref{AaH_cor5.6}, where we will show that these three items are in fact equivalent.
\begin{prop}\label{AaH_prop_4.2}
Suppose $G$ is a second countable, locally compact, Hausdorff, principal groupoid with Haar system $\lambda$. Let $z\in G^{(0)}$ and let $\{x_n\}$ be a sequence in $G^{(0)}$. Assume that $[z]$ is locally closed in $G^{(0)}$. Consider the following properties.
\begin{enumerate}\renewcommand{\labelenumi}{(\arabic{enumi})}
\item\label{AaH_prop_4_2_1} $\M_\L(\L^z,\{\L^{x_n}\})=\infty$.
\item\label{AaH_prop_4_2_2} For every open neighbourhood $V$ of $z$ such that $G_z^V$ is relatively compact, $\lambda_{x_n}(G^V)\rightarrow\infty$ as $n\rightarrow\infty$.
\item\label{AaH_prop_4_2_3} For each $k\ge 1$, the sequence $\{x_n\}$ converges $k$-times in $G^{(0)}/G$ to $z$.
\end{enumerate}
Then {\rm\eqref{AaH_prop_4_2_1}} implies {\rm\eqref{AaH_prop_4_2_2}} and {\rm\eqref{AaH_prop_4_2_2}} implies {\rm\eqref{AaH_prop_4_2_3}}.
\end{prop}

Our next goal is to sharpen the $\lfloor M^2\rfloor$ bound in Theorem \ref{M2_thm}. This strengthened theorem appears later on as Theorem \ref{M_thm}. We will first establish several results to assist in strengthening this bound.

\begin{lemma}\label{lemma_orbits_equal}
Suppose $G$ is a second countable groupoid and $x,y\in G^{(0)}$. If $\overline{[x]}=\overline{[y]}$ and $[x]$ is locally closed, then $[x]=[y]$.
\begin{proof}
We have $x\in \overline{[y]}$, so there exists $\{\gamma_n\}\subset G$ such that $s(\gamma_n)=y$ and $r(\gamma_n)\rightarrow x$. Since $[x]$ is locally closed, there exists an open subset $U$ of $G$ such that $[x]=U\cap\overline{[x]}$. Then $r(\gamma_n)$ is eventually in $U$, so eventually $r(\gamma_n)\in U\cap\overline{[y]}=U\cap\overline{[x]}=[x]$. Thus there exists $\gamma\in G$ with $s(\gamma)=y$ and $r(\gamma)\in [x]$, as required.
\end{proof}
\end{lemma}

The following is a generalisation of \cite[Lemma~3.3]{Archbold-anHuef2006}.
\begin{lemma}\label{AaH_Lemma3.3}
Suppose $G$ is a second countable, locally compact, Hausdorff groupoid with Haar system $\lambda$. Fix $\epsilon>0$, $z\in G^{(0)}$ and let $V$ be an open neighbourhood of $z\in G^{(0)}$ such that $\lambda_z(G^V)<\infty$. Then there exists an open relatively-compact neighbourhood $V_1$ of $z$ such that $\overline{V_1}\subset V$ and
\[
\lambda_z(G^V)-\epsilon<\lambda_z(G^{V_1})\le\lambda_z(G^{\overline{V_1}})\le\lambda_z(G^V)<\lambda_z(G^{V_1})+\epsilon.
\]
\begin{proof}
We use $G_z$ equipped with the subspace topology to find a compact subset $\lambda_z$-estimate of $V$. This estimate is then used to obtain the required open set $V_1$. Since $G_z^V$ is $G_z$-open, by the regularity of $\lambda_z$ there exists a compact subset $W$ of $G_z^V$ such that $\lambda_z(W)>\lambda_z(G_z^V)-\epsilon$. Then $r(W)$ is compact and contained in $V$, so there exists an open relatively-compact neighbourhood $V_1$ of $r(W)$ such that $\overline{V_1}\subset V$. Then
\begin{align*}
\lambda_z(G^V)-\epsilon<\lambda_z(W)&\le \lambda_z(G^{V_1}) \le \lambda_z(G^{\overline{V_1}})\le \lambda_z(G^V)\\&<\lambda_z(W)+\epsilon\le\lambda_z(G^{V_1})+\epsilon,
\end{align*}
as required.
\end{proof}
\end{lemma}

The following lemma is equivalent to the claim in \cite[Proposition~3.6]{Clark2007} that $[x]\mapsto [L^x]$ from $G^{(0)}/G$ to the spectrum of $C^*(G)$ is open.
\begin{lemma}\label{lemma_ind_reps_converge_imply_orbits_converge}
Suppose $G$ is a second countable, locally compact, Hausdorff, principal groupoid with Haar system $\lambda$. If $\{x_n\}$ is a sequence in $G^{(0)}$ with $\L^{x_n}\rightarrow \L^{z}$, then $[x_n]\rightarrow [z]$.
\begin{proof}
We prove the contrapositive. Suppose $[x_n]\nrightarrow [z]$. Then there exists an open neighbourhood $U_0$ of $[z]$ in $G^{(0)}/G$ such that $[x_n]$ is frequently not in $U_0$. Let $q:G^{(0)}\rightarrow G^{(0)}/G$ be the quotient map $x\mapsto [x]$. Then $U_1:=q^{-1}(U_0)$ is an open invariant neighbourhood of $z$ and $x_n\notin U_1$ frequently. Note that $C^*(G|_{U_1})$ is isomorphic to a closed two-sided ideal $I$ of $C^*(G)$ (see \cite[Lemma~2.10]{Muhly-Renault-Williams1996}). 

We now claim that $I\subset\ker\,\L^{x_n}$ whenever $x_n\notin U_1$. Suppose $x_n\notin U_1$ and recall from Remark \ref{measure_induced_epsilon_x} that $\L^{x_n}$ acts on $L^2(G,\lambda_{x_n})$. Fix $f\in C_c(G)$ such that $f(\gamma)=0$ whenever $\gamma\notin G|_{U_1}$ and fix $\xi\in L^2(G,\lambda_{x_n})$. Then by Remark \ref{measure_induced_epsilon_x} we have
\[
\|\L^{x_n}(f)\xi\|_{x_n}^2=\int_G\bigg(\int_G f(\gamma\alpha^{-1})\xi(\alpha)\, d\lambda_{x_n}(\alpha)\bigg)^2\, d\lambda_{x_n}(\gamma).
\]
When evaluating the inner integrand, we have $s(\alpha)=s(\gamma)=x_n$, so $\gamma\alpha^{-1}\in G|_{[x_n]}$. Since $U_1$ is invariant with $x_n\notin U_1$, it follows that $\gamma\alpha^{-1}\notin G|_{U_1}$, and so $f(\gamma\alpha^{-1})=0$. Thus $\|\L^{x_n}(f)\xi\|_{x_n}=0$ and since $\xi$ was fixed arbitrarily, $\L^{x_n}(f)=0$. This implies that $I\subset \ker\,\L^{x_n}$.

We now conclude by observing that since $I\subset\ker\,\L^{x_n}$ frequently, $\L^{x_n}\notin \hat{I}$ frequently. But $\hat{I}$ is an open neighbourhood of $\L^z$, so $\L^{x_n}\nrightarrow \L^z$.
\end{proof}
\end{lemma}

We may now proceed to strengthening the $\lfloor M^2\rfloor$ bound in Theorem \ref{M2_thm}. This theorem is a generalisation of \cite[Theorem~3.5]{Archbold-anHuef2006}.
\begin{theorem}\label{M_thm}
Suppose $G$ is a second countable, locally compact, Hausdorff, principal groupoid with Haar system $\lambda$. Let $M\in\RR$ with $M\ge 1$, suppose $z\in G^{(0)}$ such that $[z]$ is locally closed and let $\{x_n\}$ be a sequence in $G^{(0)}$. Suppose there exists an open neighbourhood $V$ of $z$ in $G^{(0)}$ such that $G_z^V$ is relatively compact and 
\[
\lambda_{x_n}(G^V)\le M\lambda_z(G^V)
\]
frequently. Then $\M_\L(\L^z,\{\L^{x_n}\})\le\lfloor M\rfloor$.
\begin{proof}
If $\L^{x_n}$ does not converge to $\L^z$, then $\M_\L(\L^z,\{\L^{x_n}\})=0<\lfloor M \rfloor$. So we assume from now on that $\L^{x_n}\rightarrow \L^z$. Lemma \ref{lemma_ind_reps_converge_imply_orbits_converge} now shows that $[x_n]\rightarrow [z]$. Next we claim that we may assume, without loss of generality, that $[z]$ is the unique limit of $\{[x_n]\}$ in $G^{(0)}/G$. To see this, note that $\M_\L(\L^z,\{\L^{x_n}\})\le\lfloor M^2\rfloor<\infty$ by Theorem \ref{M2_thm}. Hence, by \cite[Proposition~3.4]{Archbold-anHuef2006}, $\{\L^z\}$ is open in the set of limits of $\{\L^{x_n}\}$. So there exists an open neighbourhood $U_2$ of $\L^z$ in $C^*(G)^\wedge$ such that $\L^z$ is the unique limit of $\{\L^{x_n}\}$ in $U_2$. By \cite[Proposition~2.5]{Muhly-Williams1990} there is a continuous function $\L:G^{(0)}/G\rightarrow C^*(G)^\wedge$ such that $[x]\mapsto\L^x$ for all $x\in G^{(0)}$. Define $p:G^{(0)}\rightarrow G^{(0)}/G$ by $p(x)=[x]$ for all $x\in G^{(0)}$. Then $p$ is continuous, and $Y:=(\L\circ p)^{-1}(U_2)$ is an open $G$-saturated neighbourhood of $z$ in $G^{(0)}$. Note that $x_n\in Y$ eventually.

Now suppose that, for some $y\in Y$, $[x_n]\rightarrow [y]$ in $Y/G$ and hence in $G^{(0)}/G$. Then $\L^{x_n}\rightarrow \L^y$ by \cite[Proposition~2.5]{Muhly-Williams1990}, and $\L^y\in U_2$ since $y\in(\L\circ p)^{-1}(U_2)$. But $\{\L^{x_n}\}$ has the unique limit $\L^z$ in $U_2$, so $\L^z=\L^y$ and hence $\overline{[z]}=\overline{[y]}$. Since $[z]$ is locally closed, Lemma \ref{lemma_orbits_equal} shows that $[z]=[y]$ in $G^{(0)}$ and hence in $Y$.

We know $Y$ is an open saturated subset of $G^{(0)}$, so $C^*(G|_Y)$ is isomorphic to a closed two-sided ideal $J$ of $C^*(G)$. We can apply \cite[Proposition~5.3]{Archbold-Somerset-Spielberg1997} with the $C^*$-subalgebra $J$ to see that $\M_\L(\L^z,\{\L^{x_n}\})$ is the same whether we compute it in the ideal $J$ or in $C^*(G)$. Since $Y$ is $G$-invariant, $G_z^V=G_z^{V\cap Y}$ and eventually $G_{x_n}^V=G_{x_n}^{V\cap Y}$. We may thus consider $G|_Y$ instead of $G$ and therefore assume that $[z]$ is the unique limit of $[x_n]$ in $G^{(0)}/G$ as claimed.

As in \cite{Archbold-anHuef2006}, the idea for the rest of the proof is the same as in Theorem \ref{M2_thm}, although more precise estimates are used. Fix $\epsilon>0$ such that $M(1+\epsilon)^2<\lfloor M\rfloor +1$ and choose $\kappa>0$ such that
\begin{equation}\label{choice_of_kappa}
\kappa<\frac{\epsilon\lambda_z(G^V)}{1+\epsilon}<\lambda_z(G^V).
\end{equation}
By Lemma \ref{AaH_Lemma3.3} there exists an open relatively compact neighbourhood $V_1$ of $z$ such that $\overline{V_1}\subset V$ and
\[
0<\lambda_z(G^V)-\kappa<\lambda_z(G^{V_1})\le\lambda_z(G^{\overline{V_1}})\le\lambda_z(G^V)<\lambda_z(G^{V_1})+\kappa.
\]
Choose a subsequence $\{x_{n_i}\}$ of $\{x_n\}$ such that
\[
\lambda_{x_{n_i}}(G^V)\le M\lambda_z(G^V)
\]
for all $i\ge 1$. Then
\begin{align}
\lambda_{x_{n_i}}(G^{V_1})&\le\lambda_{x_{n_i}}(G^V)\notag\\
&\le M\lambda_z(G^V)\notag\\
&<M\big(\lambda_z(G^{V_1})+\kappa\big)\notag\\
&<M\lambda_z(G^{V_1})+M\epsilon\big(\lambda_z(G^V)-\kappa\big)\quad\text{(by \eqref{choice_of_kappa})}\notag\\
&<M\lambda_z(G^{V_1})+M\epsilon\lambda_z(G^{V_1})\notag\\
&=M(1+\epsilon)\lambda_z(G^{V_1})\label{AaH3.8}
\end{align}
for all $i$. Since
\[
\frac{\lambda_z(G^{V_1})\big(\lambda_z(G^{V_1})+\kappa+1/j\big)}{\big(\lambda_z(G^{V_1})-1/j\big)^2}
\rightarrow
1+\frac{\kappa}{\lambda_z(G^{V_1})}
<1+\epsilon
\]
as $j\rightarrow\infty$, there exists $\delta>0$ such that $\delta<\lambda_z(G^{V_1})$ and
\begin{equation}\label{AaH3.9}
\frac{\lambda_z(G^{V_1})\big(\lambda_z(G^{\overline{V_1}})+\delta\big)}{\big(\lambda_z(G^{V_1})-\delta\big)^2}
<
\frac{\lambda_z(G^{V_1})\big(\lambda_z(G^{V_1})+\kappa+\delta\big)}{\big(\lambda_z(G^{V_1})-\delta\big)^2}<1+\epsilon.
\end{equation}

We will now construct a function $F\in C_c(G^{(0)})$ with support inside $V_1$. Since $\lambda_z$ is inner regular on open sets and $G_z^{V_1}$ is $G_z$-open, there exists a $G_z$-compact subset $W$ of $G_z^{V_1}$ such that
\[
0<\lambda_z(G_z^{V_1})-\delta<\lambda_z(W).
\]
Since $W$ is $G_z$-compact there exists a $G_z$-compact neighbourhood $W_1$ of $W$ that is contained in $G_z^{V_1}$ and there exists a continuous function $g:G_z\rightarrow [0,1]$ that is identically one on $W$ and zero off the interior of $W_1$. We have
\begin{equation}\label{AaH3.10}
\lambda_z(G^{V_1})-\delta<\lambda_z(W)
\le\int_{G_z} g(t)^2\, d\lambda_z(t)
=\|g\|_z^2,
\end{equation}
By Lemma \ref{based_on_Ramsay} the restriction $\tilde{r}$ of $r$ to $G_z$ is a homeomorphism onto $[z]$. So there exists a continuous function $g_1:\tilde{r}(W_1)\rightarrow [0,1]$ such that $g_1\big(\tilde{r}(\gamma)\big)=g(\gamma)$ for all $\gamma\in W_1$. Thus $\tilde{r}(W_1)$ is $[z]$-compact, which implies that $\tilde{r}(W_1)$ is $G^{(0)}$-compact. Since we know that $G^{(0)}$ is second countable and Hausdorff, Tietze's Extension Theorem can be applied to show that $g_1$ can be extended to a continuous map $g_2:G^{(0)}\rightarrow [0,1]$. Because $\tilde{r}(W_1)$ is a compact subset of the open set $V_1$, there exist a compact neighbourhood $P$ of $\tilde{r}(W_1)$ contained in $V_1$ and a continuous function $h:G^{(0)}\rightarrow [0,1]$ that is identically one on $\tilde{r}(W_1)$ and zero off the interior of $P$. Note that $h$ has compact support that is contained in $P$.

We set $f(x)=h(x)g_2(x)$. Then $f\in C_c(G^{(0)})$ with $0\le f\le 1$ and 
\begin{equation}\label{supp_f_contained_in_V_2}
\mathrm{supp} f\subset \mathrm{supp} h\subset P\subset V_1.
\end{equation}
Note that
\begin{align}
\|f\circ r\|_z^2&=\int_{G_z} f\big(\tilde{r}(\gamma)\big)^2\,d\lambda_z(\gamma)\notag\\
&=\int_{G_z} h\big(\tilde{r}(\gamma)\big)^2g_2\big(\tilde{r}(\gamma)\big)^2\, d\lambda_z(\gamma)\notag\\
&\ge\int_{W_1} h\big(\tilde{r}(\gamma)\big)^2 g(\gamma)^2\, d\lambda_z(\gamma)\notag\\
&=\int_{W_1}g(\gamma)^2\,d\lambda_z(\gamma)\notag\\
&=\|g\|_z^2\label{AaH3.11}
\end{align}
since $\mathrm{supp}\, g\subset W_1$ and $h$ is identically one on $\tilde{r}(W_1)$. We now define $F\in C_c(G^{(0)})$ by
\begin{equation}\label{definition_F_2}
F(x)=\frac{f(x)}{\|f\circ r\|_z}.
\end{equation}
Then $\|F\circ r\|_z=1$ and
\begin{equation}\label{unmotivated_ref_2}
F\circ r(\gamma)\ne 0\implies h\big(r(\gamma)\big)\ne 0 \implies r(\gamma)\in V_1 \implies \gamma\in G^{V_1}.
\end{equation}

Let $N=\mathrm{supp}\, F$. Suppose $K$ is an open relatively compact symmetric neighbourhood of $(\overline{G_z^N})(\overline{G_z^N})^{-1}$ in $G$ and choose $b\in C_c(G)$ such that $b$ is identically one on $(\overline{G_z^N})(\overline{G_z^N})^{-1}$ and identically zero off $K$. As in Theorem \ref{M2_thm} we may assume that $b$ is self-adjoint by considering $\frac12(b+b^*)$. Define $D\in C_c(G)$ by $D(\gamma):=F\big(r(\gamma)\big)F\big(s(\gamma)\big)b(\gamma)$. By the same argument as in Theorem \ref{M2_thm}, $\L^z(D)$, and hence $\L^z(D^*\ast D)$, is the rank one projection determined by the unit vector $F\circ r\in L^2(G,\lambda_z)$. From \eqref{Aah3.4} we have
\begin{align*}
\tr\big(\L^{x_{n_i}}&(D^*\ast D)\big)\\
&=\int_G F\big(r(\beta)\big)^2\bigg(\int_G F\big(r(\alpha)\big)^2 b(\alpha\beta^{-1})^2\, d\lambda_{x_{n_i}}(\alpha)\bigg)\, d\lambda_{x_{n_i}}(\beta).
\end{align*}
Since $b$ is identically zero off $K$, the inner integrand is zero unless $\alpha\beta^{-1}\in K$. Combining this with \eqref{supp_f_contained_in_V_2} and the fact that $\mathrm{supp}\,\lambda_{x_{n_i}}\subset G_{x_{n_i}}$ enables us to see that this inner integrand is zero unless $\alpha\in G_{x_{n_i}}^{V_1}\cap K\beta$. Thus
\begin{align*}
\tr\big(\L^{x_{n_i}}&(D^*\ast D)\big)\\
&\le\int_{\beta\in G_{x_{n_i}}^{V_1}} F\big(r(\beta)\big)^2\bigg(\int_{\alpha\in G_{x_{n_i}}^{V_1}\cap K\beta} F\big(r(\alpha)\big)^2\, d\lambda_{x_{n_i}}(\alpha)\bigg)\, d\lambda_{x_{n_i}}(\beta).\\
&\le\frac{1}{\|f\circ r\|_z^4}\int_{\beta\in G_{x_{n_i}}^{V_1}} 1\bigg(\int_{\alpha\in G_{x_{n_i}}^{V_1}\cap K\beta} 1\, d\lambda_{x_{n_i}}(\alpha)\bigg)\, d\lambda_{x_{n_i}}(\beta).
\end{align*}

Since $\overline{V_1}$ and $\overline{K}$ are compact, by Lemma \ref{the_unbroken_lemma} there exists $i_0$ such that for every $i\ge i_0$ and any $\beta\in G_{x_{n_i}}^{\overline{V_1}}$,
\[
\lambda_{x_{n_i}}(K\beta\cap G^{\overline{V_1}})<\lambda_z(G^{\overline{V_1}})+\delta.
\]
So, provided $i\ge i_0$,
{\allowdisplaybreaks\begin{align*}
\tr\big(\L^{x_{n_i}}(D^*\ast D)\big)&\le \frac{1}{\|f\circ r\|_z^4}\int_{\beta\in G_{x_{n_i}}^{V_1}}\lambda_{x_{n_i}}(K\beta\cap G_{x_{n_i}}^{V_1})\, d\lambda_{x_{n_i}}(\beta)\\
&\le\frac{1}{\|f\circ r\|_z^4}\int_{\beta\in G_{x_{n_i}}^{V_1}}\big(\lambda_z(G_z^{\overline{V_1}})+\delta\big)\, d\lambda_{x_{n_i}}(\beta)\\
&<\frac{\big(\lambda_z(G^{\overline{V_1}})+\delta\big)\lambda_{x_{n_i}}(G^{V_1})}{\|f\circ r\|_z^4}\\
&<\frac{M(1+\epsilon)\big(\lambda_z(G^{\overline{V_1}})+\delta\big)\lambda_z(G^{V_1})}{\|g\|_z^4}\quad\text{(by \eqref{AaH3.8} and \eqref{AaH3.11})}\\
&<\frac{M(1+\epsilon)\big(\lambda_z(G^{\overline{V_1}})+\delta\big)\lambda_z(G^{V_1})}{(\lambda_z(G^{V_1})-\delta)^2}\quad\text{(by \eqref{AaH3.10})}\\
&<M(1+\epsilon)^2\quad\text{(by \eqref{AaH3.9})}.
\end{align*}}

We can now make our conclusion as in \cite[Theorem~3.5]{Archbold-anHuef2006}: by generalised lower semi-continuity \cite[Theorem~4.3]{Archbold-Spielberg1996},
{\allowdisplaybreaks\begin{align*}
\underset{\scriptstyle n}{\lim\,\inf}\,\tr\big(\L^{x_n}(D^\ast\ast D)\big)&\ge \M_\L(\L^z,\{\L^{x_n}\})\, \tr \big(\L^z(D^*\ast D)\big)\\
&=\M_\L(\L^z,\{\L^{x_n}\}).
\end{align*}
}We now have
{\allowdisplaybreaks\begin{align*}
\M_\L(\L^z,\{\L^{x_n}\})&\le \underset{\scriptstyle n}{\lim\,\inf}\, \tr\big(\L^{x_n}(D^*\ast D)\big)\\
&\le M(1+\epsilon)^2\\
&<\lfloor M\rfloor+1,
\end{align*}
}and so $\M_\L(\L^z,\{\L^{x_n}\})\le\lfloor M\rfloor$, as required.
\end{proof}
\end{theorem}

\section{Lower multiplicity and \texorpdfstring{$k$}{\it k}-times convergence II}\label{sec_lower_multiplicity_2}
We proved in Proposition \ref{AaH_thm_1.1_1_implies_2} that if a sequence converges $k$-times in the orbit space of a principal groupoid, then the lower multiplicity of the associated sequence of representations is at least $k$. In this section we will prove the converse. The first result in this section generalises \cite[Lemma~5.1]{Archbold-anHuef2006}; with the exception of notation changes, the proof is the same as the proof in \cite{Archbold-anHuef2006}.
\begin{lemma}\label{AaH_lemma5.1}
Suppose $G$ is a second countable, locally compact, Hausdorff, principal groupoid. Let $k\in\PP$, $z\in G^{(0)}$, and $\{x_n\}$ be a sequence in $G^{(0)}$. Assume that $[z]$ is locally closed in $G^{(0)}$ and that there exists $R>k-1$ such that for every open neighbourhood $U$ of $z$ with $G_z^U$ relatively compact we have
\[
\underset{\scriptstyle n}{\lim\,\inf}\,\lambda_{x_n}(G^U)\ge R\lambda_z(G^U).
\]
Given an open neighbourhood $V$ of $z$ such that $G_z^V$ is relatively compact, there exists a compact neighborhood $N$ of $z$ with $N\subset V$ such that
\[
\underset{\scriptstyle n}{\lim\,\inf}\,\lambda_{x_n}(G^N)>(k-1)\lambda_z(G^N).
\]
\begin{proof}
Apply Lemma \ref{AaH_Lemma3.3} to $V$ with $0<\epsilon<\frac{R-k+1}R \lambda_z(G^V)$ to get an open relatively-compact neighbourhood $V_1$ of $z$ with $\overline{V_1}\subset V$ and
\[
\lambda_z(G^V)-\epsilon<\lambda_z(G^{V_1})\le\lambda_z(G^{\overline{V_1}})\le\lambda_z(G^V)<\lambda_z(G^{V_1})+\epsilon.
\]
Since $G_z^{V_1}$ is relatively compact we have
\begin{align*}
\underset{\scriptstyle n}{\lim\,\inf}\,\lambda_{x_n}(G^{\overline{V_1}})&\ge\underset{\scriptstyle n}{\lim\,\inf}\,\lambda_{x_n}(G^{V_1})\\
&\ge R\lambda_z(G^{V_1})&\text{(by hypothesis)}\\
&>R\big(\lambda_z(G^V)-\epsilon\big)\\
&>(k-1)\lambda_z(G^V)&\text{(by our choice of }\epsilon\text{)}\\
&\ge (k-1)\lambda_z(G^{\overline{V_1}}).
\end{align*}
So we may take $N=\overline{V_1}$.
\end{proof}
\end{lemma}
\begin{remark}\label{AaH_lemma5.1_variant}
The preceding lemma also holds when $\lim\,\inf$ is replaced by $\lim\,\sup$.  No modification of the proof is needed beyond replacing the two occurrences of $\lim\,\inf$ with $\lim\,\sup$.
\end{remark}

We may now proceed to our main theorem.
\needspace{6\baselineskip}
\begin{theorem}\label{circle_thm}
Suppose $G$ is a second countable, locally compact, Hausdorff principal groupoid that admits a Haar system $\lambda$. Let $k$ be a positive integer, let $z\in G^{(0)}$ and let $\{x_n\}$ be a sequence in $G^{(0)}$. Assume that $[z]$ is locally closed in $G^{(0)}$. Then the following are equivalent:
\begin{enumerate}\renewcommand{\labelenumi}{(\arabic{enumi})}
\item\label{circle_thm_1} the sequence $\{x_n\}$ converges $k$-times in $G^{(0)}/G$ to $z$;
\item\label{circle_thm_2} $\M_\L(\L^z,\{\L^{x_n}\})\ge k$;
\item\label{circle_thm_3} for every open neighbourhood $V$ of $z$ in $G^{(0)}$ such that $G_z^V$ is relatively compact we have
\[
\underset{\scriptstyle n}{\lim\,\inf}\,\lambda_{x_n}(G^V)\ge k\lambda_z(G^V);
\]
\item\label{circle_thm_4} there exists a real number $R>k-1$ such that for every open neighbourhood $V$ of $z$ in $G^{(0)}$ with $G_z^V$ relatively compact we have
\[
\underset{\scriptstyle n}{\lim\,\inf}\,\lambda_{x_n}(G^V)\ge R\lambda_z(G^V);\quad\text{and}
\]
\item\label{circle_thm_5} there exists a basic decreasing sequence of compact neighbourhoods $\{W_m\}$ of $z$ in $G^{(0)}$ such that, for each $m\ge 1$,
\[
\underset{\scriptstyle n}{\lim\,\inf}\,\lambda_{x_n}(G^{W_m})>(k-1)\lambda_z(G^{W_m}).
\]
\end{enumerate}
\begin{proof}
We know that \eqref{circle_thm_1} implies \eqref{circle_thm_2} by Proposition \ref{AaH_thm_1.1_1_implies_2}.

Suppose \eqref{circle_thm_2}. If $\M_\L(\L^z,\{\L^{x_n}\})\ge k$, then $\M_\L(\L^z,\{\L^{x_n}\})>\lfloor k-\epsilon\rfloor$ for all $\epsilon>0$. By Theorem \ref{M_thm}, for every $G^{(0)}$-open neighborhood $V$ of $z$ such that $G_z^V$ is relatively compact, $\lambda_{x_n}(G^V)>(k-\epsilon)\lambda_z(G^V)$ eventually, and hence \eqref{circle_thm_3} holds.

It is immediately true that \eqref{circle_thm_3} implies \eqref{circle_thm_4}.

Suppose \eqref{circle_thm_4}. We will construct the sequence $\{W_m\}$ of compact neighbourhoods inductively. Let $\{V_j\}$ be a basic decreasing sequence of open neighborhoods of $z$ such that $G_z^{V_1}$ is relatively compact (such neighborhoods exist by \cite[Lemma~4.1(1)]{Clark-anHuef2010-preprint}). By Lemma \ref{AaH_lemma5.1} there exists a compact neighbourhood $W_1$ of $z$ such that $W_1\subset V_1$ and $\lambda_{x_n}(G^{W_1})>(k-1)\lambda_z(G^{W_1})$. Now assume there are compact neighbourhoods $W_1,W_2,\ldots,W_m$ of $z$ with $W_1\supset W_2\supset\cdots\supset W_m$ such that
\begin{equation}\label{circle_thm_4_imply_5_eqn}
W_i\subset V_i\quad\text{and}\quad\lambda_{x_n}(G^{W_i})>(k-1)\lambda_z(G^{W_i})
\end{equation}
for all $1\le i\le m$. Apply Lemma \ref{AaH_lemma5.1} to $(\mathrm{Int}\, m)\cap V_{m+1}$ to obtain a compact neighbourhood $W_{m+1}$ of $z$ such that $W_{m+1}\subset(\mathrm{Int}\, W_m)\cap V_{m+1}$ and \eqref{circle_thm_4_imply_5_eqn} holds for $i=m+1$, establishing \eqref{circle_thm_5}.

Suppose \eqref{circle_thm_5}. We begin by showing that $[x_n]\rightarrow [z]$ in $G^{(0)}/G$. Let $q:G^{(0)}\rightarrow G^{(0)}/G$ be the quotient map. Let $U$ be a neighbourhood of $[z]$ in $G^{(0)}/G$ and $V=q^{-1}(U)$. There exists $m$ such that $W_m\subset V$. Since $\lim\,\inf_n\,\lambda_{x_n}(G^{W_m})>0$ there exists $n_0$ such that $G_{x_n}^{W_m}\ne\emptyset$ for all $n\ge n_0$. Thus, for $n\ge n_0$, $[x_n]=q(x_n)\in q(W_m)\subset q(V)=U$, so $[x_n]$ is eventually in every neighbourhood of $[z]$ in $G^{(0)}/G$.

Now suppose that $\M_\L(\L^z,\{\L^{x_n}\})<\infty$. Then, as in the proof of Theorem \ref{M_thm}, we may localise to an open invariant neighbourhood $Y$ of $z$ such that $[z]$ is the unique limit in $Y/G$ of $[x_n]$. Eventually $W_m\subset Y$, and so the sequence $\{x_n\}$ converges $k$-times in $Y/(G|_Y)=Y/G$ to $z$ by Proposition \ref{AaH_prop4_1_1} applied to the groupoid $G|_Y$. This implies that the sequence $\{x_n\}$ converges $k$-times in $G^{(0)}/G$.

Finally, if $\M_\L(\L^z,\{\L^{x_n}\})=\infty$, then $\{x_n\}$ converges $k$-times in $G^{(0)}/G$ to $z$ by Proposition \ref{AaH_prop_4.2}, establishing \eqref{2nd_circle_thm_1} and completing the proof.
\end{proof}
\end{theorem}

\begin{cor}
Suppose that $G$ is a second countable, locally compact, Hausdorff, principal groupoid such that all the orbits are locally closed. Let $k\in\PP$ and let $z\in G^{(0)}$ such that $[z]$ is not open in $G^{(0)}$. Then the following are equivalent:
\begin{enumerate}\renewcommand{\labelenumi}{(\arabic{enumi})}
\item\label{AaH_cor5.5_1} whenever $\braces{x_n}$ is a sequence in $G^{(0)}$ which converges to $z$ with $[x_n]\ne [z]$ eventually, then $\braces{x_n}$ is $k$-times convergent in $G^{(0)}/G$ to $z$;
\item\label{AaH_cor5.5_2} $\M_\L(\L^z)\ge k$.
\end{enumerate}
\begin{proof}
Assume \eqref{AaH_cor5.5_1}. We must first establish that $\braces{\L^z}$ is not open in $C^*(G)^\wedge$. If this is not the case, then $\braces{\L^z}$ is open and we can apply \cite[Proposition~3.6]{Clark2007} to see that $\braces{[z]}$ is open in $G^{(0)}/G$, and so $[z]$ is open in $G^{(0)}$, contradicting our assumption. Since $\braces{\L^z}$ is not open in $C^*(G)^\wedge$, we can apply \cite[Lemma~A.2]{Archbold-anHuef2006} to see that there exists a sequence $\braces{\pi_i}$ of irreducible representations of $C^*(G)$ such that each $\pi_i$ is not unitarily equivalent to $\L^z$, $\pi_i\rightarrow \L^z$ in $C^*(G)^\wedge$, and
\begin{equation}\label{AaH_5.3}
\M_\L(\L^z)=\M_\L(\L^z,\braces{\pi_i})=\M_\U(\L^z,\braces{\pi_i}).
\end{equation}
Since the orbits are locally closed, the map $G^{(0)}/G\rightarrow C^*(G)^\wedge$ such that $[x]\mapsto \L^x$ is a homeomorphism by \cite[Proposition~5.1]{Clark2007}\footnote{Proposition~5.1 in \cite{Clark2007} states that if a principal groupoid has locally closed orbits, then the map from $G^{(0)}/G$ to $C^*(G)^\wedge$ where $[x]\mapsto \L^x$ is a `homeomorphism from $G^{(0)}/G$ into $C^*(G)^\wedge$'. The proof of \cite[Proposition~5.1]{Clark2007} explicitly shows that this map is a surjection.}. It follows that the mapping $G^{(0)}\rightarrow C^*(G)^\wedge$ such that $x\mapsto \L^x$ is an open surjection, so by \cite[Proposition~1.15]{Williams2007} there is a sequence $\braces{x_n}$ in $G^{(0)}$ such that $x_n \rightarrow z$ and $\braces{\L^{x_n}}$ is unitarily equivalent to a subsequence of $\braces{\pi_i}$.
By \eqref{AaH_5.3}, 
\[
\M_\L(\L^z)=\M_\U(\L^z,\braces{\pi_i})\ge \M_\U(\L^z,\braces{\L^{x_n}})\ge \M_\L(\L^z,\braces{\L^{x_n}}).
\]
We know by \eqref{AaH_cor5.5_1} that $\braces{x_n}$ converges $k$-times to $z$ in $G^{(0)}/G$, so it follows from Theorem \ref{circle_thm} that $\M_\L(\L^z)\ge\M_\L(\L^z,\braces{\L^{x_n}})\ge k$.

Assume \eqref{AaH_cor5.5_2}. If $\braces{x_n}$ is a sequence in $G^{(0)}$ which converges to $z$ such that $[x_n]\ne [z]$ eventually, then
\[
\M_\L(\L^z,\braces{\L^{x_n}})\ge\M_\L(\L^z)\ge k.
\]
By Theorem \ref{circle_thm}, $\braces{x_n}$ is $k$-times convergent to $z$ in $G^{(0)}/G$.
\end{proof}
\end{cor}

The next corollary improves Proposition \ref{AaH_prop_4.2} and is an immediate consequence of Proposition \ref{AaH_prop_4.2} and Theorem \ref{circle_thm}.
\begin{cor}\label{AaH_cor5.6}
Suppose that $G$ is a second countable, locally compact, Hausdorff, principal groupoid with Haar system $\lambda$. Let $z\in G^{(0)}$ and let $\braces{x_n}$ be a sequence in $G^{(0)}$. Assume that $[z]$ is locally closed. Then the following are equivalent:
\begin{enumerate}\renewcommand{\labelenumi}{(\arabic{enumi})}
\item\label{AaH_cor5.6_1} $\M_\L(\L^z,\{\L^{x_n}\})=\infty$.
\item\label{AaH_cor5.6_2} For every open neighbourhood $V$ of $z$ such that $G_z^V$ is relatively compact, $\lambda_{x_n}(G^V)\rightarrow\infty$ as $n\rightarrow\infty$.
\item\label{AaH_cor5.6_3} For each $k\ge 1$, the sequence $\{x_n\}$ converges $k$-times in $G^{(0)}/G$ to $z$.
\end{enumerate}
\end{cor}

\section{Upper multiplicity and \texorpdfstring{$k$}{\it k}-times convergence}\label{sec_upper_multiplicity}
The results in this section  are corollaries of Theorems~\ref{M_thm} and~\ref{circle_thm}: they relate $k$-times convergence, measure ratios and upper multiplicity numbers, generalising all the upper-multiplicity results of \cite{Archbold-anHuef2006}. We begin with the upper-multiplicity analogue of Theorem~\ref{M_thm}.
\begin{theorem}\label{thm_AaH3.6}
Suppose that $G$ is a second countable, locally compact, Hausdorff, principal groupoid with Haar system $\lambda$. Let $M\in \RR$ with $M\ge 1$, let $z\in G^{(0)}$ and let $\braces{x_n}$ be a sequence in $G^{(0)}$. Assume that $[z]$ is locally closed. Suppose that there exists an open neighbourhood $V$ of $z$ in $G^{(0)}$ such that $G_z^V$ is relatively compact and
\[
\lambda_{x_n}(G^V)\le M\lambda_z(G^V)<\infty
\]
eventually. Then $\M_\U(\L^z,\braces{\L^{x_n}})\le \lfloor M\rfloor$.
\begin{proof}
Since $G$ is second countable, $C^*(G)$ is separable. By \cite[Lemma~A.1]{Archbold-anHuef2006} there exists a sequence $\braces{\L^{x_{n_i}}}$ such that
\[
\M_\U(\L^z,\braces{\L^{x_n}})=\M_\U(\L^z,\braces{\L^{x_{n_i}}})=\M_\L(\L^z,\braces{\L^{x_{n_i}}}).
\]
By Theorem \ref{M_thm}, $\M_\L(\L^z,\braces{\L^{x_{n_i}}})\le \lfloor M\rfloor$, so $\M_\U(\L^z,\braces{\L^{x_n}})\le\lfloor M\rfloor$.
\end{proof}
\end{theorem}

\needspace{6\baselineskip}
\begin{cor}\label{AaH_cor3.7}
Suppose that $G$ is a second countable, locally compact, Hausdorff, principal groupoid with Haar system $\lambda$ such that all the orbits are locally closed. Let $M\in \RR$ with $M\ge 1$ and let $z\in G^{(0)}$. If for every sequence $\braces{x_n}$ in $G^{(0)}$ which converges to $z$ there exists an open neighbourhood $V$ of $z$ in $G^{(0)}$ such that $G_z^V$ is relatively compact and 
\[
\lambda_{x_n}(G^V)\le M\lambda_z(G^V)<\infty
\]
frequently, then $\M_\U(\L^z)\le\lfloor M\rfloor$.
\begin{proof}
Since $G$ is second countable, $C^*(G)$ is separable, and so we can apply \cite[Lemma~1.2]{Archbold-Kaniuth1999} to see that there exists a sequence $\braces{\pi_n}$ in $C^*(G)^\wedge$ that converges to $\L^z$ such that
\[
\M_\L(\L^z,\braces{\pi_n})=\M_\U(\L^z,\braces{\pi_n})=\M_\U(\L^z).
\]
Since the orbits are locally closed, the map $G^{(0)}/G\rightarrow C^*(G)^\wedge$ such that $[x]\mapsto \L^x$ is a homeomorphism by \cite[Proposition~5.1]{Clark2007}. In particular, the mapping $G^{(0)}\rightarrow C^*(G)^\wedge$ such that $x\mapsto \L^x$ is an open surjection, so by \cite[Proposition~1.15]{Williams2007} there exists a sequence $\braces{x_i}$ in $G^{(0)}$ converging to $z$ such that $\braces{[\L^{x_i}]}$ is a subsequence of $\braces{[\pi_n]}$. By Theorem \ref{M_thm}, $\M_\L(\L^z,\braces{\L^{x_n}})\le\lfloor M\rfloor$. Since 
\begin{align*}
\M_\U(\L^z)&=\M_\L(\L^z,\braces{\pi_n})\le\M_\L(\L^z,\braces{\L^{x_i}})\le\M_\U(\L^z,\braces{\L^{x_i}})\\
&\le\M_\U(\L^z,\braces{\pi_n})=\M_\U(\L^z),
\end{align*}
we obtain $\M_\U(\L^z)\le\lfloor M\rfloor$, as required.
\end{proof}
\end{cor}

In Proposition~\ref{AaH_prop4_1_1} we generalised the first part of \cite[Proposition~4.1]{Archbold-anHuef2006}. We will now generalise the second part. The argument we use is similar to that used in Proposition \ref{AaH_prop4_1_1}.
\begin{prop}\label{AaH_prop4_1_2}
Let $G$ be a second countable, locally compact, Hausdorff, principal groupoid with Haar system $\lambda$. Let $k\in\PP$ and $z\in G^{(0)}$ with $[z]$ locally closed in $G^{(0)}$. Assume that $\{x_n\}$ is a sequence in $G^{(0)}$ such that $[x_n]\rightarrow [z]$ uniquely in $G^{(0)}/G$. Suppose $\{W_m\}$ is a basic decreasing sequence of compact neighbourhoods of $z$ such that each $m$ satisfies
\[
\underset{\scriptstyle n}{\lim\,\sup}\,\lambda_{x_n}(G^{W_m})>(k-1)\lambda_z(G^{W_m}).
\]
Then there exists a subsequence of $\{x_n\}$ which converges $k$-times in $G^{(0)}/G$ to $z$. 

\begin{proof}
Let $\{K_m\}$ be an increasing sequence of compact subsets of $G$ such that $G=\bigcup_{m\ge 1}\mathrm{Int}\,K_m$. By the regularity of $\lambda_z$, for each $m\ge 1$ there exist $\delta_m>0$ and an open neighbourhood $U_m$ of $G_z^{W_m}$ such that
\begin{equation}\label{AaH4.4}
\underset{\scriptstyle n}{\lim\,\sup}\,\lambda_{x_n}(G^{W_m})>(k-1)\lambda_z(U_m) + \delta_m.
\end{equation}
We will construct, by induction, a strictly increasing sequence of positive integers $\{i_m\}$ such that, for all $m$,
\begin{align}
&\lambda_{x_{i_m}}(K_m\alpha\cap G^{W_m})<\lambda_z(U_m)+\delta_m/k\quad\text{for all }\alpha\in G_{x_{i_m}}^{W_m},\quad\text{and}\label{AaH4.5}\\
&\lambda_{x_{i_m}}(G^{W_m})>(k-1)\lambda_z(U_m)+\delta_m.\label{AaH4.6}
\end{align}

By Lemma \ref{the_unbroken_lemma} with $\delta=\lambda_z(U_1)-\lambda_z(G^{W_1})+\delta_1/k$, there exists $n_1$ such that $n\ge n_1$ implies
\[
\lambda_{x_n}(K_1\alpha\cap G^{W_1})<\lambda_z(U_1)+\delta_1/k\quad\text{for all }\alpha\in G_{x_n}^{W_m}.
\]
By considering \eqref{AaH4.4} with $m=1$ we can choose $i_1\ge n_1$ such that
\[
\lambda_{x_{i_1}}(G^{W_1})>(k-1)\lambda_z(U_1)+\delta_1.
\]
Assuming that $i_1<i_2<\cdots<i_{m-1}$ have been chosen, we can apply Lemma \ref{the_unbroken_lemma} with $\delta=\lambda_z(U_m)-\lambda_z(G^{W_m})+\delta_m/k$ to obtain $n_m>i_{m-1}$ such that
\[
n\ge n_m\quad\text{implies}\quad\lambda_{x_n}(K_m\alpha\cap G^{W_m})<\lambda_z(U_m)+\delta_m/k\quad\text{for all }\alpha\in G_{x_n}^{W_m},
\]
and then by \eqref{AaH4.4} we can choose $i_m\ge n_m$ such that
\[
\lambda_{x_{i_m}}(G^{W_m})>(k-1)\lambda_z(U_m)+\delta_m.
\]

For each $m\in\PP$ choose $\gamma_{i_m}^{(1)}\in G_{x_{i_m}}^{W_m}$ (which is non-empty by \eqref{AaH4.6}). By \eqref{AaH4.5} and \eqref{AaH4.6} we have
\begin{align*}
\lambda_{x_{i_m}}(G^{W_m}\backslash K_m\gamma_{i_m}^{(1)})&=\lambda_{x_{i_m}}(G^{W_m})-\lambda_{x_{i_m}}(G^{W_m}\cap K_m\gamma_{i_m}^{(1)})\\
&>(k-1)\lambda_z(U_m)+\delta_m-\big(\lambda_z(U_m)+\delta_m/k\big)\\
&=(k-2)\lambda_z(U_m)+\frac{k-1}k\delta_m.
\end{align*}
So we can choose $\gamma_{i_m}^{(2)}\in G_{x_{i_m}}^{W_m}\backslash K_m\gamma_{i_m}^{(1)}$. This implies, as in the proof of Proposition \ref{AaH_prop4_1_1}, that
\[
\lambda_{x_{i_m}}\big(G^{W_m}\backslash (K_m\gamma_{i_m}^{(1)}\cup K_m\gamma_{i_m}^{(2)})\big)>(k-3)\lambda_z(U_m)+\frac{(k-2)}k\delta_m,
\]
enabling us to choose $\gamma_{i_m}^{(3)}\in G_{x_{i_m}}^{W_m}\backslash (K_m\gamma_{i_m}^{(1)}\cap K_m\gamma_{i_m}^{(2)})$. Continuing in this way for  $j=3,\ldots,k$, for each $i_m$ we choose
\begin{equation}\label{eqn_choosing_gammas-2}
\gamma_{i_m}^{(j)}\in G_{x_{i_m}}^{W_m}\backslash\bigg(\bigcup_{l=1}^{j-1}K_m\gamma_{i_m}^{(l)}\bigg).
\end{equation}
Note that $\gamma_{i_m}^{(j)}\notin K_m\gamma_{i_m}^{(l)}$ for $1\le l<j\le k$.

We claim that $r(\gamma_{i_m}^{(l)})\rightarrow z$ as $m\rightarrow\infty$ for $1\le l\le k$. To see this, fix $l$ and let $V$ be an open neighbourhood of $z$. Since $\{W_m\}$ is a decreasing neighbourhood basis for $z$ there exists $m_0$ such that $m\ge m_0$ implies $W_m\subset V$, and so $r(\gamma_{i_m}^{(l)})\in W_m\subset V$.

Finally we claim that $\gamma_{i_m}^{(j)}(\gamma_{i_m}^{(l)})^{-1}\rightarrow\infty$ as $m\rightarrow\infty$ for $1\le l<j\le k$. Fix $l<j$ and let $K$ be a compact subset of $G$. There exists $m_0$ such that $K\subset K_m$ for all $m\ge m_0$. By \eqref{eqn_choosing_gammas-2} we know
\begin{align*}
\gamma_{i_m}^{(j)}&\in G_{x_{i_m}}^{W_m}\backslash (K_m\gamma_{i_m}^{(l)})\\
&=\big(G_{x_{i_m}}^{W_m}(\gamma_{i_m}^{(l)})^{-1}\gamma_{i_m}^{(l)}\big)\backslash (K_m\gamma_{i_m}^{(l)})\\
&=\big((G_{x_{i_m}}^{W_m}(\gamma_{i_m}^{(l)})^{-1})\backslash K_m\big)\gamma_{i_m}^{(l)}.
\end{align*}
So provided $m\ge m_0$, $\gamma_{i_m}^{(j)}(\gamma_{i_m}^{(l)})^{-1}\in \big(G_{x_{i_m}}^{W_m}(\gamma_{i_m}^{(l)})^{-1}\big)\backslash K_m\subset G\backslash K_m\subset G\backslash K$, enabling us to conclude that $\{x_{i_m}\}$ converges $k$-times in $G^{(0)}/G$ to $z$.
\end{proof}
\end{prop}

\needspace{6\baselineskip}
\begin{theorem}\label{2nd_circle_thm}
Suppose that $G$ is a second countable, locally compact, Hausdorff, principal groupoid with Haar system $\lambda$. Let $k\in\PP$, let $z\in G^{(0)}$, and let $\braces{x_n}$ be a sequence in $G^{(0)}$ such that $[x_n]$ converges to $[z]$ in $G^{(0)}/G$. Assume that $[z]$ is locally closed. Then the following are equivalent:
\begin{enumerate}\renewcommand{\labelenumi}{(\arabic{enumi})}
\item\label{2nd_circle_thm_1} there exists a subsequence $\braces{x_{n_i}}$ of $\braces{x_n}$ which converges $k$-times in $G^{(0)}/G$ to $z$;
\item\label{2nd_circle_thm_2} $\M_\U(\L^z,\braces{\L^{x_n}})\ge k$;
\item\label{2nd_circle_thm_3} for every open neighbourhood $V$ of $z$ such that $G_z^V$ is relatively compact we have
\[
\underset{\scriptstyle n}{\lim\,\sup}\,\lambda_{x_n}(G^V)\ge k\lambda_z(G^V);
\]
\item\label{2nd_circle_thm_4}
there exists a real number $R>k-1$ such that for every open neighbourhood $V$ of $z$ in $G^{(0)}$ with $G_z^V$ relatively compact we have
\[
\underset{\scriptstyle n}{\lim\,\sup}\,\lambda_{x_n}(G^V)\ge R\lambda_z(G^V);\quad\text{and}
\]
\item\label{2nd_circle_thm_5} there exists a basic decreasing sequence of compact neighbourhoods $\{W_m\}$ of $z$ in $G^{(0)}$ such that, for each $m\ge 1$,
\[
\underset{\scriptstyle n}{\lim\,\sup}\,\lambda_{x_n}(G^{W_m})>(k-1)\lambda_z(G^{W_m}).
\]
\end{enumerate}
\begin{proof}
If \eqref{2nd_circle_thm_1} holds then $\M_\L(\L^z,\braces{\L^{x_{n_i}}})\ge k$ by Theorem \ref{circle_thm}, and so
\[
\M_\U(\L^z,\braces{\L^{x_n}}\ge\M_\U(\L^z,\braces{\L^{x_{n_i}}})\ge\M_\L(\L^z,\braces{\L^{x_{n_i}}})\ge k.
\]

If \eqref{2nd_circle_thm_2} holds then by \cite[Lemma~A.1]{Archbold-anHuef2006} there is a subsequence $\braces{x_{n_r}}$ such that $\M_\L(\L^z,\braces{\L^{x_{n_r}}})=\M_\U(\L^z,\braces{\L^{x_n}})$ so that $\M_\L(\L^z,\braces{\L^{x_{n_r}}})\ge k$. Let $V$ be any open neighbourhood of $z$ in $G^{(0)}$ such that $G_z^V$ is relatively compact. Then
\[
\underset{\scriptstyle n}{\lim\,\sup}\,\lambda_{x_n}(G^V)\ge\underset{\scriptstyle r}{\lim\,\sup}\,\lambda_{x_{n_r}}(G^V)\ge \underset{\scriptstyle r}{\lim\,\inf}\,\lambda_{x_{n_r}}(G^V)\ge k\lambda_z(G^V),
\]
using Theorem \ref{circle_thm} for the last step.

That \eqref{2nd_circle_thm_3} implies \eqref{2nd_circle_thm_4} is immediate.

That \eqref{2nd_circle_thm_4} implies \eqref{2nd_circle_thm_5} follows by making references to Remark \ref{AaH_lemma5.1_variant} rather than Lemma \ref{AaH_lemma5.1} in the \eqref{circle_thm_4} implies \eqref{circle_thm_5} component of the proof of Theorem \ref{circle_thm}.

Assume \eqref{2nd_circle_thm_5}. First suppose that $\M_\L(\L^z,\braces{\L^{x_n}})<\infty$. Since $[x_n]\rightarrow[z]$, we can use an argument found at the beginning of the proof of Theorem \ref{M_thm} to obtain an open $G$-invariant neighborhood $Y$ of $z$ in $G^{(0)}$ so that if we define $H:=G|_Y$, there exists a subsequence $\braces{x_{n_i}}$ of $\braces{x_n}$ such that $[x_{n_i}]\rightarrow [z]$ uniquely in $H^{(0)}/H$. Proposition \ref{AaH_prop4_1_2} now shows us that there exists a subsequence $\braces{x_{n_{i_j}}}$ of $\braces{x_{n_i}}$ that converges $k$-times in $H^{(0)}/H$ to $z$. It follows that $\braces{x_{n_{i_j}}}$ converges $k$-times in $G^{(0)}/G$ to $z$.

When $\M_\L(\L^z,\braces{\L^{x_n}})=\infty$, $\braces{x_n}$ converges $k$-times in $G^{(0)}/G$ to $z$  by Corollary \ref{AaH_cor5.6}, establishing \eqref{2nd_circle_thm_1}.
\end{proof}
\end{theorem}

\begin{cor}\label{AaH_cor5.4}
Suppose that $G$ is a second countable, locally compact, Hausdorff, principal groupoid such that all the orbits are locally closed. Let $k\in\PP$ and let $z\in G^{(0)}$. Then the following are equivalent:
\begin{enumerate}\renewcommand{\labelenumi}{(\arabic{enumi})}
\item\label{AaH_cor5.4_1} there exists a sequence $\braces{x_n}$ in $G^{(0)}$ which is $k$-times convergent in $G^{(0)}/G$ to $z$;
\item\label{AaH_cor5.4_2} $\M_\U(\L^z)\ge k$.
\end{enumerate}
\begin{proof}
Assume \eqref{AaH_cor5.4_1}. By the definitions of upper and lower multiplicity, 
\[
\M_\U(\L^z)\ge\M_\U(\L^z,\braces{\L^{x_n}})\ge\M_\L(\L^z,\braces{\L^{x_n}}).
\]
By Theorem \ref{circle_thm} we know that $\M_\L(\L^z,\braces{\L^{x_n}})\ge k$, establishing \eqref{AaH_cor5.4_2}.

Assume \eqref{AaH_cor5.4_2}. By \cite[Lemma~1.2]{Archbold-Kaniuth1999} there exists a sequence $\braces{\pi_n}$ converging to $\L^z$ such that $\M_\L(\L^z,\braces{\pi_n})=\M_\U(\L^z,\braces{\pi_n})=\M_\U(\L^z)$. Since the orbits are locally closed, by \cite[Proposition~5.1]{Clark2007} the mapping $G^{(0)}\rightarrow C^*(G)^\wedge:x\mapsto\L^x$ is a surjection. So there is a sequence $\braces{\L^{x_n}}$ in $C^*(G)^\wedge$ such that $\L^{x_n}$ is unitarily equivalent to $\pi_n$ for each $n$. Then
\[
\M_\L(\L^z,\braces{\L^{x_n}})\ge\M_\L(\L^z,\braces{\pi_n})=\M_\U(\L^z)\ge k,
\]
and it follows from Theorem \ref{circle_thm} that $\braces{x_n}$ is $k$-times convergent in $G^{(0)}/G$ to $z$.
\end{proof}
\end{cor}

\begin{cor}\label{AaH_cor5.7}
Suppose that $G$ is a second countable, locally compact, Hausdorff, principal groupoid with Haar system $\lambda$. Let $z\in G^{(0)}$ and let $\braces{x_n}\subset G^{(0)}$ be a sequence converging to $z$. Assume that $[z]$ is locally closed. Then the following are equivalent:
\begin{enumerate}\renewcommand{\labelenumi}{(\arabic{enumi})}
\item\label{AaH_cor5.7_1} there exists an open neighbourhood $V$ of $z$ such that $G_z^V$ is relatively compact and
\[
\underset{\scriptstyle n}{\lim\,\sup}\,\lambda_{x_n}(G^V)<\infty;
\]
\item\label{AaH_cor5.7_2} $\M_\U(\L^z,\braces{\L^{x_n}})<\infty$.
\end{enumerate}
\begin{proof}
Suppose that \eqref{AaH_cor5.7_1} holds. Since $C^*(G)$ is separable, it follows from \cite[Lemma~A.1]{Archbold-anHuef2006} that there exists a subsequence $\braces{x_{n_j}}$ of $\braces{x_n}$ such that
\[
\M_\L(\L^z,\braces{\L^{x_{n_j}}})=\M_\U(\L^z,\braces{\L^{x_{n_j}}})=\M_\U(\L^z,\braces{\L^{x_n}}).
\]
By \eqref{AaH_cor5.7_1} and Corollary \ref{AaH_cor5.6}, $\M_\L(\L^z,\braces{\L^{x_n}})<\infty$. Hence $\M_\U(\L^z,\braces{\L^{x_n}})<\infty$, as required.

Suppose that \eqref{AaH_cor5.7_1} fails. Let $\braces{V_i}$ be a basic decreasing sequence of open neighbourhoods of $z$ such that $G_z^{V_1}$ is relatively compact (such neighborhoods exist by \cite[Lemma~4.1(1)]{Clark-anHuef2010-preprint}). Then 
\[
\underset{\scriptstyle n}{\lim\,\sup}\,\lambda_{x_n}(G^{V_i})=\infty\quad\text{for each }i
\]
and we may choose a subsequence $\braces{x_{n_i}}$ of $\braces{x_n}$ such that $\lambda_{x_{n_i}}(G^{V_i})\rightarrow\infty$ as $i\rightarrow\infty$.

Let $V$ be any open neighbourhood of $z$ such that $G_z^V$ is relatively compact. There exists $i_0$ such that $V_i\subset V$ for all $i\ge i_0$. Then, for $i\ge i_0$,
\[
\lambda_{x_{n_i}}(G^{V_i})\le\lambda_{x_{n_i}}(G^V).
\]
Thus $\lambda_{x_{n_i}}(G^V)\rightarrow\infty$ as $i\rightarrow\infty$. By Corollary \ref{AaH_cor5.6}, $\M_\L(\L^z,\braces{\L^{x_n}})=\infty$. Hence $\M_\U(\L^z,\braces{\L^{x_n}})=\infty$, that is \eqref{AaH_cor5.7_2} fails.
\end{proof}
\end{cor}

\needspace{6\baselineskip}
\begin{cor}\label{AaH_cor5.8}
Suppose $G$ is a second countable, locally compact, Hausdorff, principal groupoid with Haar system $\lambda$ such that all the orbits are locally closed. Let $z\in G^{(0)}$. Then the following are equivalent:
\begin{enumerate}\renewcommand{\labelenumi}{(\arabic{enumi})}
\item\label{AaH_cor5.8_1} $\M_\U(\L^z)<\infty$;
\item\label{AaH_cor5.8_2} there exists an open neighbourhood $V$ of $z$ such that $G_z^V$ is relatively compact and
\[
\sup_{x\in V}\lambda_x(G^V)<\infty.
\]
\end{enumerate}
\begin{proof}
If \eqref{AaH_cor5.8_2} holds then \eqref{AaH_cor5.8_1} holds by Corollary \ref{AaH_cor3.7}.

Let $\braces{V_i}$ be a basic decreasing sequence of open neighbourhoods of $z$ such that $G_z^{V_1}$ is relatively compact. If \eqref{AaH_cor5.8_2} fails then $\sup_{x\in V_i}\braces{\lambda_x(G^{V_i})}=\infty$ for each $i$ and we may choose a sequence $\braces{x_i}$ such that $x_i\in V_i$ for all $i$ and $\lambda_{x_i}(G^{V_i})\rightarrow\infty$. Since $\braces{V_i}$ is a basic decreasing sequence, $x_i\rightarrow z$.

Let $V$ be an open neighbourhood of $z$ such that $G_z^V$ is relatively compact. There exists $i_0$ such that $V_i\subset V$ for all $i\ge i_0$. Then, for $i\ge i_0$,
\[
\lambda_{x_i}(G^{V_i})\le \lambda_{x_i}(G^V).
\]
Thus $\lambda_{x_i}(G^V)\rightarrow\infty$. By Corollary \ref{AaH_cor5.7}, $\M_\U(\L^z,\braces{\L^{x_i}})=\infty$. Hence $\M_\U(\L^z)=\infty$, and so \eqref{AaH_cor5.8_1} fails.
\end{proof}
\end{cor}

\section{Path groupoid examples}\label{section_graph_algebra_examples}
In this section we will show how Theorems \ref{circle_thm} and \ref{2nd_circle_thm} can be applied to the path groupoids of directed graphs to reveal the multiplicity structure of their $C^*$-algebras. Recall the definition of a directed graph and the path groupoid from Section \ref{sec_path_groupoid}. \notationindex{VFinfty@$[v,f]^\infty$}Recall from page \pageref{def_vfinfty} that when $v$ is a vertex, $f$ is an edge, and there is exactly one infinite path with range $v$ that includes the edge $f$, then this infinite path is denoted by $[v,f]^\infty$. \notationindex{VFstar@$[v,f]^*$}When there is exactly one finite path $\alpha$ with $r(\alpha)=v$ and $\alpha_{|\alpha|}=f$, we denote $\alpha$ by $[v,f]^*$.

\begin{example}[$2$-times convergence in a path groupoid]\label{2-times_convergence_example}
Let $E$ be the graph
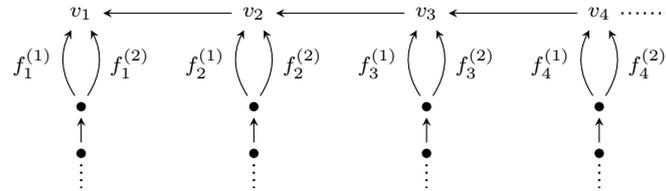
\begin{center}
\begin{tikzpicture}[>=stealth,baseline=(current bounding box.center)] 
\def\cellwidth{5.5};
\clip (-5em,-5.6em) rectangle (3*\cellwidth em + 4.5em,0.3em);

\foreach \x in {1,2,3,4} \foreach \y in {0} \node (x\x y\y) at (\cellwidth*\x em-\cellwidth em,-3*\y em) {$\scriptstyle v_{\x}$};
\foreach \x in {1,2,3,4} \foreach \y in {1} \node (x\x y\y) at (\cellwidth*\x em-\cellwidth em,-3*\y em) {};
\foreach \x in {1,2,3,4} \foreach \y in {2} \node (x\x y\y) at (\cellwidth*\x em-\cellwidth em,-1.5em-1.5*\y em) {};
\foreach \x in {1,2,3,4} \foreach \y in {1,2} \fill[black] (x\x y\y) circle (0.15em);
\foreach \x in {1,2,3,4} \draw [<-, bend left] (x\x y0) to node[anchor=west] {$\scriptstyle f_{\x}^{(2)}$} (x\x y1);
\foreach \x in {1,2,3,4} \draw [<-, bend right] (x\x y0) to node[anchor=east] {$\scriptstyle f_{\x}^{(1)}$} (x\x y1);

\foreach \x / \z in {1/2,2/3,3/4} \draw[black,<-] (x\x y0) to node[anchor=south] {} (x\z y0);

\foreach \x in {1,2,3,4} \draw [<-] (x\x y1) -- (x\x y2);
\foreach \x in {1,2,3,4} {
	\node (endtail\x) at (\cellwidth*\x em-\cellwidth em, -6em) {};
	\draw [dotted,thick] (x\x y2) -- (endtail\x);
}

\node(endtailone) at (3*\cellwidth em + 2.5em,0em) {};
\draw[dotted,thick] (x4y0) -- (endtailone);
\end{tikzpicture}
\end{center}
and let $G$ be the path groupoid. For each $n\ge 1$ define $x^{(n)}:=[v_1, f_n^{(1)}]^\infty$ and let $z$ be the infinite path with range $v_1$ that passes through each $v_n$. Then $\{x^{(n)}\}$ converges $2$-times in $G^{(0)}/G$ to $z$.

\begin{proof}
We will describe two sequences in $G$ as in Definition \ref{def_k-times_convergence}. For each $n\ge 1$ define $\gamma_n^{(1)}:=(x^{(n)},0,x^{(n)})$ and $\gamma_n^{(2)}:=\big([v_1, f_n^{(2)}]^\infty,0,x^{(n)}\big)$. It follows immediately that $s(\gamma_n^{(1)})=x^{(n)}=s(\gamma_n^{(2)})$ for all $n$ and that both $r(\gamma_n^{(1)})$ and $r(\gamma_n^{(2)})$ converge to $z$ as $n\rightarrow\infty$. It remains to show that $\gamma_n^{(2)}(\gamma_n^{(1)})^{-1}\rightarrow\infty$ as $n\rightarrow\infty$.

Let $K$ be a compact subset of $G$. Our goal is to show that $\gamma_n^{(2)}(\gamma_n^{(1)})^{-1}=\gamma_n^{(2)}$ is eventually not in $K$. Since sets of the form $Z(\alpha,\beta)$ for $\alpha,\beta\in E^*$ with $s(\alpha)=s(\beta)$ form a basis for the topology on the path groupoid, for each $\gamma\in K$ there exist $\alpha^{(\gamma)},\beta^{(\gamma)}\in E^*$ with $s(\alpha^{(\gamma)})=s(\beta^{(\gamma)})$ so that $Z(\alpha^{(\gamma)},\beta^{(\gamma)})$ is an open neighbourhood of $\gamma$ in $G$. Thus $\cup_{\gamma\in K}Z(\alpha^{(\gamma)},\beta^{(\gamma)})$ is an open cover of the compact set $K$, and so admits a finite subcover $\cup_{i=1}^I Z(\alpha^{(i)},\beta^{(i)})$. 

We now claim that for any fixed $n\in\PP$, if there exists $i$ with $1\le i\le I$ such that $\gamma_n^{(2)}\in Z(\alpha^{(i)},\beta^{(i)})$, then $\big|[v_1,f_n^{(2)}]^*\big|\le \lvert\alpha^{(i)}\rvert$. Temporarily fix $n\in\PP$ and suppose there exists $i$ with $1\le i\le I$ such that $\gamma_n^{(2)}\in Z(\alpha^{(i)},\beta^{(i)})$. Suppose the converse: that $\lvert\alpha^{(i)}\rvert<\big|[v_1,f_n^{(2)}]^*\big|$. Since $\gamma_n^{(2)}\in Z(\alpha^{(i)},\beta^{(i)})$, it follows that $r(\gamma_n^{(2)})=[v_1, f_n^{(2)}]^\infty\in \alpha^{(i)}E^\infty$, and so $\alpha_p^{(i)}=[v_1, f_n^{(2)}]^\infty_p$ for every $1\le p\le|\alpha^{(i)}|$. By examining the graph we can see that $s\big([v_1, f_n^{(2)}]^\infty_p\big)=v_{p+1}$ for all $1\le p<\big|[v_1,f_n^{(2)}]^*\big|$. Since we also know that $|\alpha^{(i)}|<\big|[v_1,f_n^{(2)}]^*\big|$, we can deduce that $s(\alpha^{(i)})=v_j$ for some $j$. Furthermore since $s(\alpha^{(i)})=s(\beta^{(i)})$, $s(\beta^{(i)})=v_j$. There is only one path with source $v_j$ and range $v_1$, so $\alpha^{(i)}=\beta^{(i)}$.
Note that when $k=|\alpha^{(i)}|-|\beta^{(i)}|$, the set $Z(\alpha^{(i)},\beta^{(i)})$ is by definition equal to
\[
\{(x,k,y):x\in \alpha^{(i)}E^\infty,y\in \beta^{(i)}E^\infty, x_p=y_{p-k} \text{ for }p>|\alpha^{(i)}|\},
\]
so since $\gamma_n^{(2)}\in Z(\alpha^{(i)},\beta^{(i)})$ and $\alpha^{(i)}=\beta^{(i)}$, we can see that $s(\gamma_n^{(2)})_p=r(\gamma_n^{(2)})_p$ for all $p>|\alpha^{(i)}|$. We know $s(\gamma_n^{(2)})=[v_1, f_n^{(1)}]^\infty$ and $r(\gamma_n^{(2)})=[v_1, f_n^{(2)}]^\infty$, so $[v_1, f_n^{(2)}]^\infty_p=[v_1, f_n^{(1)}]^\infty_p$ for all $p>|\alpha^{(i)}|$. In particular, since we assumed that $\big|[v_1,f_n^{(2)}]^*\big|>|\alpha^{(i)}|$, we have 
\[
[v_1, f_n^{(2)}]^\infty_{\big|[v_1,f_n^{(2)}]^*\big|}=[v_1, f_n^{(1)}]^\infty_{\big|[v_1,f_n^{(2)}]^*\big|},\]
so that $f_n^{(2)}=f_n^{(1)}$. But $f_n^{(1)}$ and $f_n^{(2)}$ are distinct, so we have found a contradiction, and we must have $\big|[v_1,f_n^{(2)}]^*\big|\le |\alpha^{(i)}|$.

Our next goal is to show that each $Z(\alpha^{(i)},\beta^{(i)})$ contains at most one $\gamma_n^{(2)}$. Fix $n,m\in\PP$ and suppose that both $\gamma_n^{(2)}$ and $\gamma_m^{(2)}$ are in $Z(\alpha^{(i)},\beta^{(i)})$ for some $i$. We will show that $n=m$. Since $\gamma_n^{(2)}\in Z(\alpha^{(i)},\beta^{(i)})$, $r(\gamma_n^{(2)})=[v_1, f_n^{(2)}]^\infty\in \alpha^{(i)}E^\infty$. Thus there exists $x\in E^\infty $ such that $[v_1,f_n^{(2)}]^*x\in \alpha^{(i)}E^\infty$ and, since $\big|[v_1,f_n^{(2)}]^*\big|\le \lvert\alpha^{(i)}\rvert$, we can crop $x$ to form a finite $\epsilon\in E^*$ such that $[v_1,f_n^{(2)}]^*\epsilon=\alpha^{(i)}$. Similarly there exists $\delta\in E^*$ such that $[v_1,f_m^{(2)}]^*\delta=\alpha^{(i)}$. Then
\[
[v_1,f_n^{(2)}]^*\epsilon=\alpha^{(i)}=[v_1,f_m^{(2)}]^*\delta,
\]
which we can see by looking at the graph is only possible if $n=m$. We have thus shown that if $\gamma_n^{(2)}$ and $\gamma_m^{(2)}$ are in $Z(\alpha^{(i)},\beta^{(i)})$, then $\gamma_n^{(2)}=\gamma_m^{(2)}$.

Let $S=\{n\in\PP:\gamma_n^{(2)}\in K\}$. Since $K\subset \cup_{i=1}^IZ(\alpha^{(i)},\beta^{(i)})$ and since $\gamma_n^{(2)},\gamma_m^{(2)}\in Z(\alpha^{(i)},\beta^{(i)})$ implies $n=m$, $S$ can contain at most $I$ elements. Then $S$ has a maximal element $n_0$ and $\gamma_n^{(2)}\notin K$ provided $n>n_0$. Thus $\gamma_n\rightarrow\infty$ as $n\rightarrow\infty$, and we have shown that $x^{(n)}$ converges $2$-times to $z$ in $G^{(0)}/G$.
\end{proof}
\end{example}

\needspace{4\baselineskip}
\begin{example}[$k$-times convergence in a path groupoid]\label{k-times_convergence_example}
For any fixed positive integer $k$, let $E$ be the graph
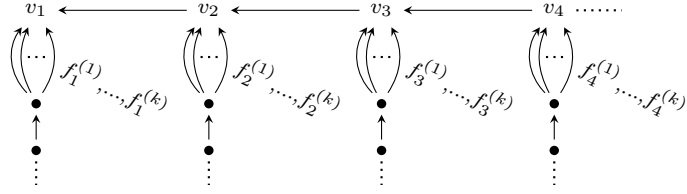
\begin{center}
\begin{tikzpicture}[>=stealth,baseline=(current bounding box.center)] 
\def\cellwidth{5.5};
\clip (-5em,-5.6em) rectangle (3*\cellwidth em + 4.5em,0.3em);

\foreach \x in {1,2,3,4} \foreach \y in {0} \node (x\x y\y) at (\cellwidth*\x em-\cellwidth em,-3*\y em) {$\scriptstyle v_{\x}$};
\foreach \x in {1,2,3,4} \foreach \y in {1} \node (x\x y\y) at (\cellwidth*\x em-\cellwidth em,-3*\y em) {};
\foreach \x in {1,2,3,4} \foreach \y in {2} \node (x\x y\y) at (\cellwidth*\x em-\cellwidth em,-1.5em-1.5*\y em) {};
\foreach \x in {1,2,3,4} \foreach \y in {1,2} \fill[black] (x\x y\y) circle (0.15em);
\foreach \x in {1,2,3,4} \draw [<-, bend right=40] (x\x y0) to node[anchor=west] {} (x\x y1);
\foreach \x in {1,2,3,4} \draw [<-, bend right=25] (x\x y0) to node[anchor=west] {} (x\x y1);
\foreach \x in {1,2,3,4} \draw [<-, bend left] (x\x y0) to node[anchor=west,rotate=-25] {$\scriptstyle f_{\x}^{(1)},\ldots,f_{\x}^{(k)}$} (x\x y1);
\foreach \x in {1,2,3,4} \draw [dotted,thick] ($(x\x y0)!{1/(2*cos(10))}!-10:(x\x y1)$) -- ($(x\x y0)!{1/(2*cos(14))}!14:(x\x y1)$);

\foreach \x / \z in {1/2,2/3,3/4} \draw[black,<-] (x\x y0) to node[anchor=south] {} (x\z y0);

\foreach \x in {1,2,3,4} \draw [<-] (x\x y1) -- (x\x y2);
\foreach \x in {1,2,3,4} {
	\node (endtail\x) at (\cellwidth*\x em-\cellwidth em, -6em) {};
	\draw [dotted,thick] (x\x y2) -- (endtail\x);
}

\node(endtailone) at (3*\cellwidth em + 2.5em,0em) {};
\draw[dotted,thick] (x4y0) -- (endtailone);
\end{tikzpicture}
\end{center}
and let $G$ be the path groupoid. For each $n\ge 1$ define $x^{(n)}:=[v_1, f_n^{(1)}]^\infty$ and let $z$ be the infinite path that passes through each $v_n$. Then the sequence $\{x^{(n)}\}$ converges $k$-times in $G^{(0)}/G$ to $z$.
\begin{proof}
After defining $\gamma_n^{(i)}:=\big([v_1, f_n^{(i)}]^\infty,0,x^{(n)}\big)$ for each $1\le i\le k$, an argument similar to that in Example \ref{2-times_convergence_example} establishes the $k$-times convergence.
\end{proof}
\end{example}

\begin{example}\label{ML2_MU3_example}[Lower multiplicity 2 and upper multiplicity 3]
Consider the graph $E$ described by
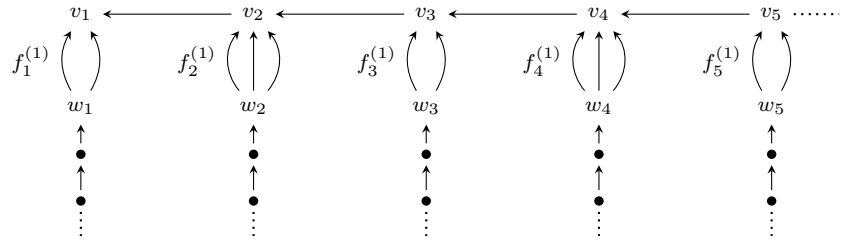
\begin{center}
\begin{tikzpicture}[>=stealth,baseline=(current bounding box.center)] 
\def\cellwidth{5.5};
\clip (-2.5em,-7.1em) rectangle (4*\cellwidth em + 2.5em,0.3em);

\foreach \x in {1,2,3,4,5} \foreach \y in {0} \node (x\x y\y) at (\cellwidth*\x em-\cellwidth em,-3*\y em) {$\scriptstyle v_{\x}$};
\foreach \x in {1,2,3,4,5} \foreach \y in {1} \node (x\x y\y) at (\cellwidth*\x em-\cellwidth em,-3*\y em) {$\scriptstyle w_{\x}$};
\foreach \x in {1,2,3,4,5} \foreach \y in {2,3} \node (x\x y\y) at (\cellwidth*\x em-\cellwidth em,-1.5em-1.5*\y em) {};
\foreach \x in {1,2,3,4,5} \foreach \y in {2,3} \fill[black] (x\x y\y) circle (0.15em);

\foreach \x in {1,3,5} \draw [<-, bend left] (x\x y0) to node[anchor=west] {} (x\x y1);
\foreach \x in {1,3,5} \draw [<-, bend right] (x\x y0) to node[anchor=east] {$\scriptstyle f_{\x}^{(1)}$} (x\x y1);
\foreach \x in {2,4} \draw [<-, bend left=40] (x\x y0) to node[anchor=west] {} (x\x y1);
\foreach \x in {2,4} \draw [<-, bend right=40] (x\x y0) to node[anchor=east] {$\scriptstyle f_{\x}^{(1)}$} (x\x y1);
\foreach \x in {2,4} \draw [<-] (x\x y0) to node {} (x\x y1);

\foreach \x / \z in {1/2,2/3,3/4,4/5} \draw[black,<-] (x\x y0) to node[anchor=south] {} (x\z y0);

\foreach \x in {1,2,3,4,5} \draw [<-] (x\x y1) -- (x\x y2);
\foreach \x in {1,2,3,4,5} \draw [<-] (x\x y2) -- (x\x y3);
\foreach \x in {1,2,3,4,5} {
	\node (endtail\x) at (\cellwidth*\x em-\cellwidth em, -7.5em) {};
	\draw [dotted,thick] (x\x y3) -- (endtail\x);
}

\node(endtailone) at (4*\cellwidth em + 2.5em,0em) {};
\draw[dotted,thick] (x5y0) -- (endtailone);
\end{tikzpicture}
\end{center}
where for each odd $n\ge 1$ there are exactly two paths $f_n^{(1)},f_n^{(2)}$ with source $w_n$ and range $v_n$, and for each even $n\ge 2$ there are exactly three paths $f_n^{(1)},f_n^{(2)},f_n^{(3)}$ with source $w_n$ and range $v_n$. Let $G$ be the path groupoid, define $x^{(n)}:=[v_1, f_n^{(1)}]^\infty$ for every $n\ge 1$, and let $z$ be the infinite path that meets every vertex $v_n$ (so $z$ has range $v_1$). Then
\[
\M_\L(\L^z,\braces{\L^{x^{(n)}}})=2\quad\text{and}\quad\M_\U(\L^z,\braces{\L^{x^{(n)}}})=3.
\]
\begin{proof}
We know that $\braces{x^{(n)}}$ converges $2$-times to $z$ in $G^{(0)}/G$ by the argument in Example \ref{2-times_convergence_example}, so we can apply Theorem \ref{circle_thm} to see that $\M_\L(\L^z,\braces{\L^{x^{(n)}}})\ge 2$. We can see that the subsequence $\braces{x^{(2n)}}$ of $\braces{x^{(n)}}$ converges $3$-times to $z$ in $G^{(0)}/G$ by Example \ref{k-times_convergence_example}. Theorem \ref{2nd_circle_thm} now tells us that $\M_\U(\L^z,\braces{\L^{x^{(n)}}})\ge 3$.

Now suppose $\M_\L(\L^z,\braces{\L^{x^{(n)}}})\ge 3$. Then by Theorem \ref{circle_thm}, $\braces{x^{(n)}}$ converges $3$-times to $z$ in $G^{(0)}/G$, so there must exist three sequences $\braces{\gamma_n^{(1)}}$,$\braces{\gamma_n^{(2)}}$, and $\braces{\gamma_n^{(3)}}$ as in the definition of $k$-times convergence (Definition \ref{def_k-times_convergence}). For each odd $n$, there are only two elements in $G$ with source $x^{(n)}$, so there must exist $1\le i<j\le 3$ such that $\gamma_n^{(i)}=\gamma_n^{(j)}$ frequently. Then $\gamma_n^{(j)}(\gamma_n^{(i)})^{-1}=r(\gamma_n^{(i)})$ frequently and, since $r(\gamma_n^{(i)})\rightarrow z$, $\braces{\gamma_n^{(j)}(\gamma_n^{(i)})^{-1}}$ admits a convergent subsequence. Thus $\gamma_n^{(j)}(\gamma_n^{(i)})^{-1}\nrightarrow\infty$, contradicting the definition of $k$-times convergence.

If $\M_\U(\L^z,\braces{\L^{x^{(n)}}})\ge 4$, then by Theorem \ref{2nd_circle_thm} there is a subsequence of $\braces{x^{(n)}}$ that converges $4$-times to $z$ in $G^{(0)}/G$. A similar argument to that in the preceding paragraph shows that this is not possible since there are at most 3 edges between any $v_n$ and $w_n$. It follows that $\M_\L(\L^z,\braces{\L^{x^{(n)}}})=2$ and $\M_\U(\L^z,\braces{\L^{x^{(n)}}})=3$.
\end{proof}
\end{example}

\begin{lemma}\label{upper_and_lower_multiplicities_lemma}
In Example \ref{2-times_convergence_example},
\[
\M_\L(\L^z,\braces{\L^{x^{(n)}}})=\M_\U(\L^z,\braces{\L^{x^{(n)}}})=2;
\]
and in Example \ref{k-times_convergence_example},
\[
\M_\L(\L^z,\braces{\L^{x^{(n)}}})=\M_\U(\L^z,\braces{\L^{x^{(n)}}})=k.
\]
\begin{proof}
The same argument as that found in Example \ref{ML2_MU3_example} can be used to demonstrate this lemma. The explicit proof was given for Example \ref{ML2_MU3_example} since it covers the case where the upper and lower multiplicities are distinct.
\end{proof}
\end{lemma}

In the next example we will add some structure to the graph from Example \ref{2-times_convergence_example} to create a path groupoid $G$ with non-Hausdorff orbit space that continues to exhibit $2$-times convergence.  
\begin{example}\label{example_with_non-Hausdorff_orbit_space}
Let $E$ be the directed graph
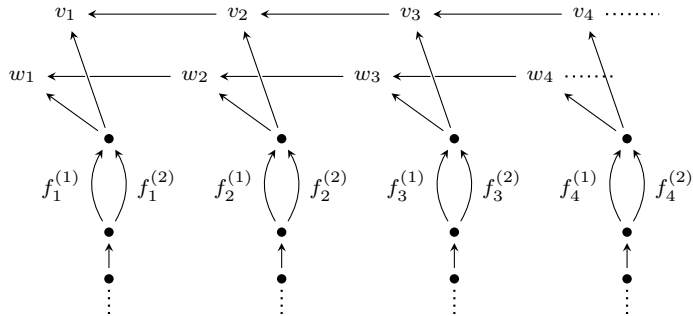
\begin{center}
\begin{tikzpicture}[>=stealth,baseline=(current bounding box.center)] 
\def\cellwidth{5.5};
\clip (-4em,-5.6em) rectangle (3*\cellwidth em + 3em,4.3em);

\foreach \x in {1,2,3,4} \foreach \y in {0,1} \node (x\x y\y) at (\cellwidth*\x em-\cellwidth em,-3*\y em) {};
\foreach \x in {1,2,3,4} \foreach \y in {2} \node (x\x y\y) at (\cellwidth*\x em-\cellwidth em,-1.5em-1.5*\y em) {};
\foreach \x in {1,2,3,4} \foreach \y in {0,1,2} \fill[black] (x\x y\y) circle (0.15em);

\foreach \x in {1,2,3,4} {
	\node (x\x z1) at (\cellwidth*\x em-1.25*\cellwidth em,4em) {$\scriptstyle v_\x$};
	\node (x\x z2) at (\cellwidth*\x em-1.5*\cellwidth em,2em) {$\scriptstyle w_\x$};
}

\foreach \x in {1,2,3,4} \draw [<-, bend left] (x\x y0) to node[anchor=west] {$\scriptstyle f_{\x}^{(2)}$} (x\x y1);
\foreach \x in {1,2,3,4} \draw [<-, bend right] (x\x y0) to node[anchor=east] {$\scriptstyle f_{\x}^{(1)}$} (x\x y1);

\foreach \x / \z in {1/2,2/3,3/4} \draw[black,<-] (x\x z1) to node[anchor=south] {} (x\z z1);
\foreach \x / \z in {1/2,2/3,3/4} \draw[black,<-] (x\x z2) to node[anchor=south] {} (x\z z2);

\foreach \x in {1,2,3,4} {
	\draw[white,ultra thick,-] (x\x z1) -- (x\x y0);
	\draw[<-] (x\x z1) -- (x\x y0);
	\draw[<-] (x\x z2) -- (x\x y0);
}

\foreach \x in {1,2,3,4} \draw [<-] (x\x y1) -- (x\x y2);
\foreach \x in {1,2,3,4} {
	\node (endtail\x) at (\cellwidth*\x em-\cellwidth em, -6em) {};
	\draw [dotted,thick] (x\x y2) -- (endtail\x);
}

\node (endtailone) at ($(x4z1) + 0.5*(\cellwidth em, 0 em)$) {};
\node (endtailtwo) at ($(x4z2) + 0.5*(\cellwidth em, 0 em)$) {};
\draw[dotted,thick] (x4z1) -- (endtailone);
\draw[dotted,thick] (x4z2) -- (endtailtwo);
\end{tikzpicture}
\end{center}
and let $G$ be the path groupoid. For every $n\ge 1$ let $x^{(n)}$ be the infinite path $[v_1, f_n^{(1)}]^\infty$. Let $x$ be the infinite path with range $v_1$ that passes through each $v_n$ and let $y$ be the infinite path with range $w_1$ that passes through each $w_n$. Then the orbit space $G^{(0)}/G$ is not Hausdorff and $\braces{x^{(n)}}$ converges $2$-times in $G^{(0)}/G$ to both $x$ and $y$.

\begin{proof}
To see that $\braces{x^{(n)}}$ converges $2$-times to $x$ in $G^{(0)}/G$, consider the sequences $\braces{([v_1, f_n^{(2)}]^\infty,0,x^{(n)})}$ and $\braces{(x^{(n)},0,x^{(n)})}$ and follow the argument as in Example \ref{2-times_convergence_example}. To see that $\braces{x^{(n)}}$ converges $2$-times to $y$ in $G^{(0)}/G$, consider the sequences $\braces{([w_1,f_n^{(1)}]^\infty,0,x^{(n)})}$ and $\braces{([w_1,f_n^{(2)}]^\infty,0,x^{(n)})}$. While it is tempting to think that this example exhibits $4$-times convergence (or even $3$-times convergence), this is not the case (see Example~\ref{ML2_MU3_example} for an argument demonstrating this). We know $x^{(n)}$ converges $2$-times to $x$ in $G^{(0)}/G$, so $[x^{(n)}]\rightarrow [x]$ in $G^{(0)}/G$, and similarly $[x^{(n)}]\rightarrow [y]$ in $G^{(0)}/G$. It follows that $G^{(0)}/G$ is not Hausdorff since $[x]\ne [y]$.
\end{proof}
\end{example}

In all of the examples above, the orbits in  $G^{(0)}$ are closed and hence $C^*(G)^\wedge$ and $G^{(0)}/G$ are homeomorphic by \cite[Proposition~5.1]{Clark2007}.  By combining the features of the graphs in Examples~\ref{ML2_MU3_example} and~\ref{example_with_non-Hausdorff_orbit_space} we obtain a principal groupoid whose $C^*$-algebra has non-Hausdorff spectrum and  distinct upper and lower multiplicities among its irreducible representations.

\chapter{Classes of \texorpdfstring{$C^*$}{\it C*}-algebras}\label{sec_algebra_classes}
This chapter introduces various classes of $C^*$-algebras, beginning with postliminal and liminal $C^*$-algebras in Section \ref{section_liminal_and_postliminal_overview} and proceeding onto bounded trace, Fell and continuous trace $C^*$-algebras in Section \ref{bounded-trace_fell_cts-trace_overview}. We will also describe theorems that characterise groupoids with $C^*$-algebras in these classes.

\section{Liminal and postliminal \texorpdfstring{$C^*$}{\it C*}-algebras}\label{section_liminal_and_postliminal_overview}
\begin{definition}
A $C^*$-algebra is {\em liminal}\index{liminal} (or {\em CCR}\index{CCR}) if the image of each irreducible representation $\pi$ is exactly $\Kk(\Hh_\pi)$. A $C^*$-algebra is {\em postliminal}\index{postliminal} (or {\em GCR}\index{GCR}) if the image of each irreducible representation $\pi$ contains $\Kk(\Hh_\pi)$.
\end{definition}

Pedersen in \cite[p. 191]{Pedersen1979} points out that $C^*$-algebras are either extremely well behaved (Type I) or totally misbehaved. Sakai in \cite[Theorem~2]{Sakai1966} showed that these well-behaved Type I algebras are precisely the postliminal $C^*$-algebras. 

A locally compact group $H$ is said to be {\em liminal} (or respectively {\em postliminal}) if $C^*(H)$ is liminal (or respectively postliminal) \cite[13.9.4]{Dixmier1977}. The two following theorems are by Clark.
\begin{theorem}[{\cite[Theorem~7.1]{Clark2007-2}}]\label{thm_groupoid_algebra_postliminal}
Suppose $G$ is second countable, locally compact, Hausdorff groupoid with a Haar system $\lambda$ in which all of the stability subgroups are amenable. Then $C^*(G,\lambda)$ is postliminal if and only if $G^{(0)}/G$ is $T_0$ and all of the stability subgroups are postliminal.
\end{theorem}

\begin{theorem}[{\cite[Theorem~6.1]{Clark2007-2}}]\label{thm_groupoid_algebra_liminal}
Suppose $G$ is second countable, locally compact, Hausdorff groupoid with a Haar system $\lambda$ in which all of the stability subgroups are amenable. Then $C^*(G,\lambda)$ is liminal if and only if $G^{(0)}/G$ is $T_1$ and all of the stability subgroups are liminal.
\end{theorem}
\begin{remark}\label{remark_T1_iff_orbits_closed}
Recall that a topological space $X$ is $T_1$ if and only if for each $x\in X$, the singleton set $\{x\}$ is closed. It follows that if $G$ is a groupoid, then the orbit space $G^{(0)}/G$ is $T_1$ if and only if the orbits in $G^{(0)}$ are closed subsets of $G^{(0)}$. 
\end{remark}

Remark \ref{remark_T1_iff_orbits_closed} will be used with Theorem \ref{thm_groupoid_algebra_liminal} for the results in Chapter \ref{chapter_categorising_higher-rank_graphs} characterising the higher-rank graphs with liminal $C^*$-algebras. The rest of this section demonstrates the analogue of Remark \ref{remark_T1_iff_orbits_closed} for use with Theorem \ref{thm_groupoid_algebra_postliminal}. This is that the orbit space $G^{(0)}/G$ is $T_0$ if and only if the orbits in $G^{(0)}$ are locally closed.

A topological space is {\em $\sigma$-compact}\index{sigma-compact@$\sigma$-compact} if it is a union of compact sets. The following result is well known.
\begin{lemma}\label{lemma_second_countable_locally_compact_is_sigma_compact}
Every second countable, locally compact space is $\sigma$-compact.
\begin{proof}
Let $X$ be a second countable, locally compact topological space. Since $X$ is second countable there is a countable base $\{U_i\}$ for the topology on $X$. As $X$ is locally compact, each $x\in X$ has a compact neighbourhood $K_x$. Since $\{U_i\}$ is a base for the topology on $X$, for each $x\in X$ there exists $i_x\in\NN$ such that $x\in U_{i_x}\subset K_x$. Let $\{U_{i_j}\}$ be the subsequence of $\{U_i\}$ obtained by removing the elements that are not in $\{U_{i_x}: x\in X\}$. Since each $x$ has $U_{i_x}\subset K_x$, for each $j$ there exists a compact set $K_j$ with $U_{i_j}\subset K_j$. Then $X=\cup_{j} K_j$, so $X$ is $\sigma$-compact.
\end{proof}
\end{lemma}

A completely metrisable, second countable, topological space is called a {\em Polish space}\index{Polish space} \cite[p.~175]{Williams2007}. As with Lemma \ref{lemma_second_countable_locally_compact_is_sigma_compact}, the following result is well known.
\begin{lemma}[{\cite[Lemma~6.5]{Williams2007}}]\label{lemma_second_countable_locally_compact_hausdorff_is_polish}
Every second countable, locally compact, Hausdorff space is Polish. 
\end{lemma}

An {\em $F_\sigma$} space is a topological space that is a countable union of closed sets. The next result is by Clark.
\begin{lemma}[{\cite[p.~258]{Clark2007}}]\label{lemma_clark}
If $G$ is a $\sigma$-compact, Hausdorff groupoid then the relation on $G^{(0)}$ induced by $G$ is an $F_\sigma$ subset of $G^{(0)}\times G^{(0)}$.
\end{lemma}
The following theorem is a cut-down version of Ramsay's \cite[Theorem~2.1]{Ramsay1990}.
\begin{theorem}[{\cite[Theorem~2.1]{Ramsay1990}}]\label{theorem_ramsay_thm2.1}
Let $G$ be a Polish groupoid and suppose that the relation on $G^{(0)}$ induced by $G$ is an $F_\sigma$ subset of $G^{(0)}\times G^{(0)}$. The following are equivalent:
\begin{enumerate}
\item $G^{(0)}/G$ is $T_0$;
\item each orbit in $G^{(0)}$ is a locally closed subset of $G^{(0)}$;
\item $G^{(0)}/G$ is almost Hausdorff; and
\item $G^{(0)}/G$ is a standard Borel space.
\end{enumerate}
\end{theorem}
The following corollary is, from the perspective of this thesis, the key immediate consequence of Lemmas \ref{lemma_second_countable_locally_compact_is_sigma_compact}, \ref{lemma_second_countable_locally_compact_hausdorff_is_polish} and \ref{lemma_clark} and of Theorem \ref{theorem_ramsay_thm2.1}. It will be used with Theorem \ref{thm_groupoid_algebra_postliminal}.
\begin{cor}\label{cor_ramsay_orbits_locally_closed}
Suppose $G$ is a second countable, locally compact, Hausdorff groupoid. The space $G^{(0)}/G$ is $T_0$ if and only if the orbits in $G^{(0)}$ are locally closed.
\end{cor}

\section{Bounded-trace, Fell and continuous trace \texorpdfstring{$C^*$}{\it C*}-algebras}\label{bounded-trace_fell_cts-trace_overview}
Recall that an element $a\in A$ is {\em positive} if $a=b^*b$ for some $b\in A$ and that $A^+$ is the set of all positive elements in $A$. It follows that if $a\in A^+$ and $\pi$ is a representation (which recall is a $\ast$-homomorphism) of $A$ on a Hilbert space $\Hh_\pi$, then $\pi(a)$ is a positive element of $B(\Hh)$. The trace of $\pi(a)$ is then given by
\[
\mathrm{Tr}\big(\pi(a)\big)=\sum_i\big(\pi(a)\xi_i\, |\, \xi_i\big)\in [0,\infty]
\]
for any orthonormal basis $\{\xi_i\}$ of $\Hh_\pi$. If $\pi$ and $\sigma$ are unitarily equivalent representations of $A$, then $\mathrm{Tr}\big(\pi(a)\big)=\mathrm{Tr}\big(\sigma(a)\big)$, so for each $a\in A^+$ there is a map $\hat{a}:\widehat{A}\rightarrow [0,\infty]$ such that $\hat{a}([\pi])=\mathrm{Tr}\big(\pi(a)\big)$ for every irreducible representation $\pi$ of $A$.


\begin{definition}[{\cite[pp.~10--11]{Milicic1973}}]\index{bounded trace}
Suppose $A$ is a $C^*$-algebra. An element $a\in A^+$ has {\em bounded trace} if $\hat{a}$ is bounded. We say that $A$ has {\em bounded trace} if the span of all bounded-trace elements is a dense subset of $A$.
\end{definition}
Archbold, Somerset and Spielberg in \cite[Theorem~2.6]{Archbold-Somerset-Spielberg1997} show that the bounded trace $C^*$-algebras are precisely the {\em uniformly liminal} $C^*$-algebras from \cite[2.1]{Perdrizet1971} and \cite[4.7.11]{Dixmier1969}. Mili\v{c}i\'c in \cite[Theorem~1]{Milicic1973} showed that every bounded-trace $C^*$-algebra is liminal.

\begin{definition}[{\cite[pp.~104--106]{Dixmier1977}}]\index{continuous trace}
Suppose $A$ is a $C^*$-algebra. An element $a\in A^+$ has {\em continuous trace} if $\hat{a}$ is bounded and continuous. We say that $A$ has {\em continuous trace} if the span of all continuous-trace elements is a dense subset of $A$.
\end{definition}
Note that it follows immediately from these definitions that every continuous-trace $C^*$-algebra has bounded trace. 
\begin{prop}[{\cite[Proposition~4.5.4]{Dixmier1977}}]\label{Dixmier4_5_4}
A $C^*$-algebra $A$ has continuous trace if and only if $\widehat{A}$ is Hausdorff and for every $\pi\in\widehat{A}$ there exist $a\in A^+$ and a neighbourhood $V$ of $\pi$ such that $\sigma(a)$ is a rank-one projection for every $\sigma\in V$.
\end{prop}
Pedersen in \cite[Theorem~6.2.11]{Pedersen1979} shows that postliminal (well behaved) $C^*$-algebras can be thought of as being built out of continuous-trace $C^*$-algebras. Proposition \ref{Dixmier4_5_4} shows that there is a natural generalisation of continuous-trace $C^*$-algebras:
\begin{definition}[{\cite[p.~92]{Archbold-Somerset1993}}]\index{Fell algebra}
A $C^*$-algebra $A$ is {\em Fell} if for every $\pi\in\widehat{A}$ there exist $a\in A^+$ and a neighbourhood $V$ of $\pi$ such that $\sigma(a)$ is a rank-one projection for every $\sigma\in V$.
\end{definition}
In \cite[Theorem~3.3]{anHuef-Kumjian-Sims2011} it is shown by an Huef, Kumjian and Sims that the Fell $C^*$-algebras are precisely the Type $\mathrm{I}_0$ $C^*$-algebras from \cite[p.~191]{Pedersen1979}. In Section \ref{sec_upper_lower_multiplicities} we will introduce the upper and lower multiplicity numbers of the irreducible representations of a $C^*$-algebra. Archbold in \cite[Theorem~4.6]{Archbold1994} showed that a $C^*$-algebra is Fell if and only if the upper multiplicity number of every irreducible representation is 1. In \cite[Theorem~2.6]{Archbold-Somerset-Spielberg1997}, Archbold, Somerset and Spielberg showed that a $C^*$-algebra has bounded trace if and only if the upper multiplicity number of every irreducible representation is finite, so it follows that every Fell $C^*$-algebra has bounded trace.

The following theorems are by Clark, an Huef, Muhly and Williams.
\begin{theorem}[{\cite[Theorem~4.4]{Clark-anHuef2008}}]\label{thm_groupoid_algebra_bounded_trace}
Suppose $G$ is a second countable, locally compact, Hausdorff, principal groupoid with a Haar system $\lambda$. Then $G$ is integrable if and only if $C^*(G,\lambda)$ has bounded trace.
\end{theorem}
\begin{theorem}[{\cite[Theorem~7.9]{Clark2007}}]\label{thm_groupoid_algebra_Fell}
Suppose $G$ is a second countable, locally compact, Hausdorff principal groupoid with a Haar system $\lambda$. Then $G$ is Cartan if and only if $C^*(G,\lambda)$ is Fell.
\end{theorem}
\begin{theorem}[{\cite[Theorem~2.3]{Muhly-Williams1990}}]\label{thm_groupoid_algebra_continuous_trace}
Suppose $G$ is a second countable, locally compact, Hausdorff principal groupoid with a Haar system $\lambda$. Then $G$ is proper if and only if $C^*(G,\lambda)$ has continuous trace.
\end{theorem}

\chapter{An introduction to graph algebras}\label{chapter_graph_algebra_intro}
In this chapter we will provide a brief overview of the theory of graph algebras, concentrating on the components of the theory relevant to the results in the following chapters. For a broader and more in depth introduction to graph algebras, see Raeburn's book \cite{Raeburn2005}.  Section \ref{section_algebras_of_directed_graphs} provides a brief overview of the theory of the $C^*$-algebras associated to directed graphs. This includes the definition of the Cuntz-Krieger $C^*$-algebra of a row-finite directed graph, the gauge-invariant uniqueness theorem, the Cuntz-Krieger uniqueness theorem and the desourcification of a row-finite directed graph. Section \ref{sec_higher-rank_graphs} provides an overview of the theory of the $C^*$-algebras associated to higher-rank graphs. This includes the definition of a higher-rank graph, the path groupoid of a row-finite higher-rank graph, and the Cuntz-Krieger $C^*$-algebra of a row-finite higher-rank graph without sources. The Cuntz-Krieger $C^*$-algebras of row-finite higher-rank graphs that may have sources are more tricky and will be covered in Section \ref{sec_higher-rank_graphs_2}. The desourcification of a higher-rank graph is an important concept in this thesis, however it is not necessary for readers only interested groupoids, so the definition of the desourcification of a higher-rank graph will instead be introduced in Section \ref{sec_boundary_paths_and_desourcification} of the following chapter before its use in Section \ref{sec_desourcification}.

\section{The Cuntz-Krieger \texorpdfstring{$C^*$}{\it C*}-algebras of directed graphs}\label{section_algebras_of_directed_graphs}
Recall from Section \ref{sec_path_groupoid} the definition of a directed graph and the definition and composition of the associated finite and infinite paths. Also recall that a directed graph is row-finite if $r^{-1}(v)$ is finite for every vertex $v$. For $n\in\NN$, let $E^n$ denote the set of all paths of length $n$. For any $A\subset E^*$ and $X\subset E^*\cup E^\infty$, define 
\[AX:=\braces{\alpha x:\alpha\in A, x\in X, s(\alpha)=r(x)}.\]
If in addition $\alpha\in E^*$, define $\alpha X:=\braces{\alpha}X$, noting that this is consistent with the definition of $\alpha E^\infty$ on page \pageref{def_alphaEinfty}. The following definition from Raeburn's book \cite{Raeburn2005} is based on a definition in Kumjian, Pask and Raeburn's \cite{Kumjian-Pask-Raeburn1998}.
\needspace{7\baselineskip}
\begin{definition}\label{def_directed_graph_algebra}
Let $E$ be a row-finite directed graph. A {\em Cuntz-Krieger $E$-family $\braces{S,P}$} consists of a set $\braces{P_v:v\in E^0}$ of mutually orthogonal projections and a set $\braces{S_f:f\in E^1}$ of partial isometries satisfying
\begin{enumerate}\renewcommand{\theenumi}{\roman{enumi}}
\item\label{def_directed_graph_algebra_1} $S_f^*S_f=P_{s(f)}$ for all $f\in E^1$; and
\item\label{def_directed_graph_algebra_2} $P_v=\sum_{f\in vE^1}S_fS_f^*$ for every $v\in E^1$ that is not a source.
\end{enumerate}
The {\em Cuntz-Krieger directed-graph $C^*$-algebra}, denoted $C^*(E)$, is defined to be the universal $C^*$-algebra generated by a Cuntz-Krieger $E$-family $\braces{s,p}$.
\end{definition}
The existence of such a universal $C^*$-algebra is guaranteed by \cite[Proposition~1.21]{Raeburn2005}, which tells us that if $\{T,Q\}$ is a Cuntz-Krieger $E$-family in a $C^*$-algebra $B$, then there is a $\ast$-homomorphism $\pi_{T,Q}:C^*(E)\rightarrow B$ such that $\pi_{T,Q}(s_f)=T_f$ for every $f\in E^1$ and $\pi_{T,Q}(p_v)=Q_v$ for every $v\in E^0$. We will now state the two standard uniqueness theorems for $C^*(E)$ before applying the first of these to the $C^*$-algebras of the path groupoids from Section \ref{sec_path_groupoid}. The first of these uniqueness theorems is Bates, Pask, Raeburn and Szyma{\'n}ski's adaptation \cite[Theorem~2.1]{Bates-Pask-Raeburn-Szymanski2000} to directed graphs of an Huef and Raeburn's \cite[Theorem~2.3]{anHuef-Raeburn1997}.
\index{gauge-invariant uniqueness theorem}
\begin{theorem}[The gauge-invariant uniqueness theorem]\label{thm_gauge_invariant_uniqueness}
Let $E$ be a row-finite directed graph, and suppose that $\{T,Q\}$ is a Cuntz-Krieger $E$-family in a $C^*$-algebra $B$ with each $Q_v\ne 0$. If there is a continuous action $\beta:\TT\rightarrow\mathrm{Aut}\, B$ such that $\beta_z(T_f)=zT_f$ for every $f\in E^1$ and $\beta_z(Q_v)=Q_v$ for every $v\in E^0$, then $\pi_{T,Q}$ is an isomorphism of $C^*(E)$ onto a $C^*$-subalgebra of $B$.
\end{theorem}

The next theorem is Bates, Pask, Raeburn and Szyma{\'n}ski's refinement \cite[Theorem~3.1]{Bates-Pask-Raeburn-Szymanski2000} of Kumjian, Pask and Raeburn's \cite[Theorem~3.7]{Kumjian-Pask-Raeburn1998} version of the Cuntz-Krieger uniqueness theorem. An {\em entry}\index{entry!to a cycle} to a cycle $\alpha$ is an edge $f$ with $r(f)=r(\alpha_i)$ and $f\ne\alpha_i$ for some $i\le \lvert\alpha\rvert$.
\index{Cuntz-Krieger uniqueness theorem}
\begin{theorem}[The Cuntz-Krieger uniqueness theorem]
Suppose $E$ is a row-finite directed graph in which every cycle has an entry, and $\{T,Q\}$ is a Cuntz-Krieger $E$-family in a $C^*$-algebra $B$ such that $Q_v\ne 0$ for every $v\in E^0$. Then $\pi_{T,Q}$ is an isomorphism of $C^*(E)$ onto a $C^*$-subalgebra of $B$. 
\end{theorem}

A {\em source}\index{source!in a directed graph} in a directed graph $E$ is a vertex $v\in E^0$ with $r^{-1}(v)=\emptyset$. Recall the definition of the path groupoid $G_E$ from Section \ref{sec_path_groupoid}.
\begin{prop}\label{prop_algebras_isomorphic}
Suppose $E$ is a row-finite directed graph without sources. Then $C^*(E)$ is isomorphic to $C^*(G_E)$.
\begin{proof}
By \cite[Proposition~4.1]{kprr1997} there is a Cuntz-Krieger $E$-family $\{T,Q\}$ such that $C^*(G_E)$ is generated by $\{T,Q\}$ and for each $f\in E^1$, the partial isometry $T_f$ is the characteristic function $1_{Z(f,s(f))}$ of $Z\big(f,s(f)\big)$. In the proof of \cite[Proposition~4.1]{kprr1997} it is shown that $T_fT_f^*=1_{Z(f,f)}$ for every $f\in E^1$, so each $Q_v$ is non-zero by Definition \ref{def_directed_graph_algebra}\eqref{def_directed_graph_algebra_2} since $E$ has no sources. Since $C^*(G_E)$ is generated by $\{T,Q\}$, the map $\pi_{T,Q}$ maps $C^*(E)$ onto $C^*(G_E)$. By \cite[Corollary~4.8]{kprr1997} there is a continuous action $\beta:\TT\rightarrow\mathrm{Aut}\, C^*(G_E)$ such that $\beta_z(T_e)=zT_e$ for every $e\in E^1$ and $\beta_z(Q_v)=Q_v$ for every $v\in E^0$. Then $C^*(E)$ is isomorphic to $C^*(G_E)$ by Theorem \ref{thm_gauge_invariant_uniqueness}.
\end{proof}
\end{prop}
The next example shows that $C^*(E)$ may not be isomorphic to $C^*(G_E)$ when $E$ has a source.
\begin{example}\label{example_directed_graph_algebras_not_isomorphic}
Suppose $E$ is the following directed graph.
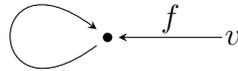
\begin{center}
\begin{tikzpicture}[>=stealth,baseline=(current bounding box.center)]
\clip (-6.2em,-1.1em) rectangle (1.2em,1.2em);
\node (u) at (-3em, 0em) {};
\node (v) at (1em,0em) {};
\node [anchor=center] at (v) {$v$};
\fill[black] (u) circle (0.15em);
\draw [<-, black] (u) to node[above=-0.2em] {$f$} (v);
\draw [<-,black] (u) .. controls (-7em,3em) and (-7em,-3em) .. (u);
\end{tikzpicture}
\end{center}
Then $C^*(G_E)$ is liminal while $C^*(E)$ is not. (This is proved in Example \ref{example_directed_graph_algebras_not_isomorphic_proof} on page \pageref{example_directed_graph_algebras_not_isomorphic_proof}.)
\end{example}
It should be noted that in \cite{Paterson2002}, Paterson shows how to construct a groupoid from an arbitrary directed graph such that the groupoid $C^*$-algebra is isomorphic to the directed graph $C^*$-algebra. While these groupoids benefit from this additional generality, here we focus on path groupoids due to their relative simplicity.

The key point to the directed-graph $C^*$-algebras is that many properties of the $C^*$-algebra are reflected in the directed graph. Recall that Lemma \ref{prop_principal_iff_no_cycles} says that if $E$ is a row-finite directed graph, then the path groupoid $G_E$ is principal if and only if $E$ has no cycles. Kumjian, Pask and Raeburn showed in \cite{Kumjian-Pask-Raeburn1998} showed that this `no cycles' condition is equivalent to $C^*(E)$ being approximately finite-dimensional:
\begin{theorem}[{\cite[Theorem~2.4]{Kumjian-Pask-Raeburn1998}}]
Suppose $E$ is a row-finite directed graph. Then $C^*(E)$ is approximately finite-dimensional (AF) if and only if $E$ has no cycles.
\end{theorem}
Before providing another example of how a property of a directed-graph Cuntz-Krieger $C^*$-algebra is reflected in the directed graph, we will introduce some standard definitions. For a directed graph $E$, let $E^{\le\infty}$\notationindex{Eleinfty@$E^{\le\infty}$} be the union of $E^\infty$ with the set of all $x\in E^*$ where $s(x)$ is a source. For $x\in E^*\cup E^\infty$ define $x(0):=r(x)$ and for $p\in\PP$ with $p\le |x|$, define \notationindex{XP@$x(p)$}$x(p):=s(x_p)$. A directed graph $E$ is {\em cofinal}\index{cofinal directed graph} if for every $x\in E^{\le\infty}$ and $v\in E^0$ there exists $\alpha\in vE^*$ such that  $s(\alpha)$ is in the path $x$ (or, in other words, $s(\alpha)=x(p)$ for some $p\le \lvert x\rvert$).
\needspace{4\baselineskip}
\begin{example}[{\cite[Example~4.1]{Raeburn2005}}]
The directed graph
\begin{center}
\begin{tikzpicture}[>=stealth,baseline=(current bounding box.center)]
\clip (-6.2em,-1.1em) rectangle (1.2em,1.2em);
\node (u) at (-3em, 0em) {};
\node (v) at (1em,0em) {};
\fill[black] (u) circle (0.15em);
\fill[black] (v) circle (0.15em);
\draw [<-, black] (v) to node[above=-0.2em] {} (u);
\draw [<-,black] (u) .. controls (-7em,3em) and (-7em,-3em) .. (u);
\end{tikzpicture}
\end{center}
is cofinal whereas the directed graph in Example \ref{example_directed_graph_algebras_not_isomorphic} is not: if $g$ is the cycle from Example \ref{example_directed_graph_algebras_not_isomorphic} then for the path $ggg\cdots$ in $E^{\le\infty}$, there is no path $\alpha$ with $r(\alpha)=v$ and $s(\alpha)$ in the path $ggg\cdots$.
\end{example}
Recall that a $C^*$-algebra $A$ is {\em simple}\index{simple $C^*$-algebra} if it has no closed ideals other than $0$ and $A$. The following result by Raeburn generalises Bates, Pask, Raeburn and Szyma{\'n}ski's \cite[Proposition~5.1]{Bates-Pask-Raeburn-Szymanski2000}, which required the directed graph to have no sources.
\begin{theorem}[{\cite[Theorem~4.14]{Raeburn2005}}]
Suppose $E$ is a row-finite directed graph. Then $C^*(E)$ is simple if and only if every cycle in $E$ has an entry and $E$ is cofinal.
\end{theorem}
Recall from Proposition \ref{prop_algebras_isomorphic} that the path groupoids can be used as a model for studying row-finite directed graphs without sources. A natural question to ask is: if these groupoids are used to establish a relationship between the row-finite directed graphs without sources and their associated $C^*$-algebras, can this relationship tell us anything about the $C^*$-algebras of directed graphs with sources? We will now describe an answer to this question provided by Bates, Pask, Raeburn and Szyma{\'n}ski's \cite[Lemma~1.2]{Bates-Pask-Raeburn-Szymanski2000}. 
\begin{definition}
\index{00Etilde@$\widetilde{E}$}Suppose $E$ is a row-finite directed graph. The directed graph obtained by adding a head
\begin{center}
\begin{tikzpicture}[>=stealth,baseline=(current bounding box.center)] 
\def\cellwidth{5.5};
\clip (-2em,-0.3em) rectangle (4*\cellwidth em + 2em,0.3em);

\node (x1y0) at (0 em,0 em) {$v$};
\foreach \x in {2,3,4,5} \foreach \y in {0} \node (x\x y\y) at (\cellwidth*\x em-\cellwidth em,-3*\y em) {};
\foreach \x in {2,3,4,5} \foreach \y in {0} \fill[black] (x\x y\y) circle (0.15em);

\foreach \x / \z in {1/2,2/3,3/4,4/5} \draw[black,<-] (x\x y0) to (x\z y0);

\node(endtailone) at (4*\cellwidth em + 2em,0em) {};
\draw[dotted,thick] (x5y0) -- (endtailone);
\end{tikzpicture}
\end{center}
to every source $v$ in $E$ is called the {\em desourcification}\index{desourcification!of a directed graph} of $E$ and is denoted by $\widetilde{E}$.
\end{definition}
\begin{prop}[{Corollary of \cite[Lemma~1.2]{Bates-Pask-Raeburn-Szymanski2000}}]
Suppose $E$ is a row-finite directed graph. Then the desourcification $\widetilde{E}$ of $E$ is a row-finite directed graph without sources and $C^*(\widetilde{E})$ is Morita equivalent to $C^*(E)$.
\end{prop}
The Morita equivalence that the previous proposition establishes is particularly useful because of the following theorem that combines results from Zettl in \cite{Zettl1982} and an Huef, Raeburn and Williams in \cite{anHuef-Raeburn-Williams2007}.
\begin{theorem}[Results from \cite{Zettl1982,anHuef-Raeburn-Williams2007}]\label{thm_Morita_equivalence_results}
Suppose $A$ and $B$ are Morita equivalent $C^*$-algebras. Then
{\allowdisplaybreaks\begin{enumerate}\renewcommand{\theenumi}{\roman{enumi}}
\item $A$ has continuous trace if and only if $B$ has continuous trace;
\item $A$ is Fell if and only if $B$ is Fell;
\item $A$ has bounded trace if and only if $B$ has bounded trace;
\item $A$ is liminal if and only if $B$ is liminal; and
\item $A$ is postliminal if and only if $B$ is postliminal.
\end{enumerate}}
\end{theorem} 


\section{Higher-rank graphs and their \texorpdfstring{$C^*$}{\it C*}-algebras and groupoids}\label{sec_higher-rank_graphs}
In this section we will define the notion of a higher-rank graph and the associated path groupoid before showing the relationship between higher-rank graphs and directed graphs. This section will conclude with the definition of the Cuntz-Krieger $C^*$-algebra of a row-finite higher-rank graph without sources.

A {\em countable category}\index{category} $\Cc=(\Cc^0,\Cc^*,r_\Cc,s_\Cc,\circ,\id)$\notationindex{C@$\Cc,\Cc^0,\Cc^*$} consists of two countable sets $\Cc^0$ (objects\index{object!in a category}) and $\Cc^*$ (morphisms\index{morphism!in a category}), two functions $r_\Cc,s_\Cc:\Cc^*\rightarrow \Cc^0$ (range and source maps)\index{range!in a category}\index{source!in a category}\notationindex{RC@$r_\Cc$}\notationindex{SC@$s_\Cc$}, a partially defined product (composition)\index{composition!in a category} $(f,g)\mapsto f\circ g$ from 
\[\braces{(f,g)\in \Cc^*\times\Cc^*:s_\Cc(f)=r_\Cc(g)}\]
to $\Cc^*$, and an injective map $\id:\Cc^0\rightarrow \Cc^*$, such that the following are satisfied for $f,g,h\in\Cc^*$ and $v\in\Cc^0$:
\begin{itemize}
 \item $r_\Cc(f\circ g)=r_\Cc(f)$ and $s_\Cc(f\circ g)=s_\Cc(g)$;
\item $(f\circ g)\circ h=f\circ (g\circ h)$ when $s_\Cc(f)=r_\Cc(g)$ and $s_\Cc(g)=r_\Cc(h)$;
\item $r_\Cc\big(\id(v)\big)=v=s_\Cc\big(\id(v)\big)$; and
\item $\id(v) \circ f=f$, $g\circ\id(v)=g$ when $r_\Cc(f)=v$ and $s_\Cc(g)=v$.
\end{itemize}
Given two countable categories $\Cc$ and $\Dd$, a {\em functor}\index{functor} $F:\Cc\rightarrow\Dd$ is a pair of maps $F^0:\Cc^0\rightarrow \Dd^0$ and $F^*:\Cc^*\rightarrow\Dd^*$ that respect the range and source maps and composition\footnote{In other words, for every $g,h\in\Cc$ we have $F^0\big(r_\Cc(g)\big)=r_\Cc\big(F^*(g)\big)$, $F^0\big(s_\Cc(f)\big)=s_\Cc\big(F^*(g)\big)$ and $F^*(g\circ h)=F^*(g)\circ F^*(h)$.}, and satisfy $\id\big(F^0(v)\big)=F^0\big(\id(v)\big)$ for every $v\in\Cc^0$. When the category is clear from the context, we write $r$ for $r_\Cc$, $s$ for $s_\Cc$ and $fg$ for $f\circ g$.

\begin{example}
Fix $k\in\PP$. We can construct a countable category from $\NN^k$: define $\Nn_k^0:=\braces{v}$ and $\Nn_k^*:=\NN^k$. For each $m,n\in\NN^k$ define $r(n):=v$, $s(n):=v$ and $m\circ n:=m+n$. Let $\id(v):=0$. Then $(\Nn_k^0,\Nn_k^*,r,s,\circ,\id)$ is a countable category. With the exception of the morphisms this category has a very basic structure, so the category will be denoted by $\NN^k$.
\end{example}

\begin{definition}[{\cite[Definitions~1.1]{Kumjian-Pask2000}}]\label{def_degree_map}
Suppose $k\in\PP$. A {\em $k$-graph} (or {\em higher-rank graph})\index{higher-rank graph}\index{$k$-graph} is a countable category $\Lambda=(\Lambda^0,\Lambda^*,r,s,\id)$\index{$\Lambda$} endowed with a functor $d:\Lambda\rightarrow\NN^k$\index{$d:\Lambda\rightarrow\NN^k$} called the {\em degree map}\label{old_def_degree_map}\index{degree map}, which satisfies the {\em factorisation property}\index{factorisation property}: for every $\alpha\in\Lambda^*$ and $m,n\in\NN^k$ with $d(\alpha)=m+n$, there exist unique $\mu,\nu\in\Lambda^*$ such that $\alpha=\mu\nu$, $d(\mu)=m$ and $d(\nu)=n$.
\end{definition}
\begin{example}\label{example_LambdaE}
For a directed graph $E$ there is a natural category $\Lambda_E$\notationindex{LambdaE@$\Lambda_E$} that can be constructed from $E$ with $E^0$ as the objects and $E^*$ as the morphisms. The map $d:E^*\rightarrow\NN$ such that $d(\alpha)=|\alpha|$ for every $\alpha\in E^*$ satisfies the factorisation property, so $\Lambda_E$ equipped with $d$ is a $1$-graph.
\end{example}
\notationindex{MleN@$m\le n$ ($m,n\in\NN^k$)}Recall that for $k\in\PP$ and $m,n\in\NN^k$, we write $m\le n$ if $m_i\le n_i$ for $i=1,2,\ldots,k$.
\begin{example}\label{example_Omega_k_m}\notationindex{Omegakm@$\Omega_{k,m}$}
Fix $k\in\PP$ and $m\in(\NN\cup\braces\infty)^k$. Define $\Omega_{k,m}$ to be the category $(\mathrm{Obj}(\Omega_{k,m}),\mathrm{Mor}(\Omega_{k,m}),r,s,\id,\circ)$ such that
\begin{itemize}
\item $\mathrm{Obj}(\Omega_{k,m})=\braces{p\in\NN^k:p\le m}$;
\item $\mathrm{Mor}(\Omega_{k,m})=\braces{(p,q)\in \Omega_{k,m}^0\times \Omega_{k,m}^0:p\le q}$;
\item $r(p,q)=p$ and $s(p,q)=q$ for every $(p,q)\in\mathrm{Obj}(\Omega_{k,m})$;
\item $\id(p)=(p,p)$ for every $p\in\mathrm{Obj}(\Omega_{k,m})$; and
\item $(p,q)\circ(q,t)=(p,t)$ for every $(p,q),(q,t)\in\mathrm{Mor}(\Omega_{k,m})$.
\end{itemize}
We equip the category $\Omega_{k,m}$ with the degree map $d:\mathrm{Mor}(\Omega_{k,m})\rightarrow\NN^k$ where $d(p,q)=q-p$ in order to make $\Omega_{k,m}$ a $k$-graph. If $m=(\infty)^k$, we simplify notation by writing $\Omega_k$ for $\Omega_{k,m}$. 
\end{example}
Suppose $\Lambda$ is a $k$-graph. For any $m\in\NN^k$, we define $\Lambda^m:=\braces{\alpha\in\mathrm{Mor}(\Lambda):d(\alpha)=m}$\index{$\Lambda^m$}. Recall that the map $\id:\mathrm{Obj}(\Lambda)\rightarrow \mathrm{Mor}(\Lambda)$ is injective. By the factorisation property $\id$ is a bijection onto $d^{-1}(0)=\Lambda^0$. We implicitly identify an object $v\in\mathrm{Obj}(\Lambda)$ with the associated morphism $\id(v)\in \Lambda^0$, so we end up considering $\Lambda^0$ to be $\mathrm{Obj}(\Lambda)$. We simplify notation further by writing $\Lambda$ for $\mathrm{Mor}(\Lambda)$.


A {\em $k$-coloured graph}\index{coloured graph}\index{k-coloured graph@$k$-coloured graph} $E=(E^0,E^1,r,s,c)$ is a directed graph $(E^0,E^1,r,s)$ with a {\em colour map}\index{colour map} $c:E^1\rightarrow\{c_1,c_2,\ldots,c_k\}$\index{$c:E^1\rightarrow\{c_1,c_2,\ldots,c_k\}$} for some distinct $c_1,c_2,\ldots,c_k$. We denote the usual basis for $\NN^k$ by $\braces{e_i}$\notationindex{EzI@$e_i$}. Suppose $\Lambda$ is a $k$-graph. The {\em skeleton}\index{skeleton} of $\Lambda$ is the $k$-coloured graph $E=(E^0,E^1,r,s,c)$ where $E^0=\Lambda^0$, $E^1=\cup_{i=1}^k \Lambda^{e_i}$, range and source are inherited from $\Lambda$, and for each $i$ ($1\le i\le k$) and $f\in\Lambda^{e_i}$, $c(f):=c_i$.
\begin{example}
The skeletons of the $k$-graphs $\Omega_{k,m}$ are relatively easy to picture. For example, the skeleton of the $2$-graph $\Omega_{2,(3,2)}$ can be illustrated by the following picture.
\begin{center}
\begin{tikzpicture}[>=stealth,baseline=(current bounding box.center)] 
\def\cellwidth{5.5};

\clip (-0.9em,-0.4em) rectangle (3*\cellwidth em + 0.9em,6.4em);

\foreach \x in {0,1,2,3} \foreach \y in {0,1,2} \node (x\x y\y) at (\cellwidth*\x em,3*\y em) {$\scriptstyle (\x,\y)$};
\foreach \x / \z in {0/1,1/2,2/3} \foreach \y in {0,1,2} \draw[black,<-] (x\x y\y) to (x\z y\y);
\foreach \x in {0,1,2,3} \foreach \y / \z in {0/1,1/2} \draw[dashed,black,<-] (x\x y\y) to (x\x y\z);

\end{tikzpicture}
\end{center}
\end{example}

A {\em $k$-graph morphism}\index{$k$-graph morphism} is a degree-preserving functor. A consequence of the factorisation property is that if $\Lambda$ is a $k$-graph and $\alpha\in\Lambda^m$ for some $m\in\NN^k$, then there exists a unique $k$-graph morphism $x_\alpha:\Omega_{k,m}\rightarrow\Lambda$ such that $x_\alpha(0,m)=\alpha$. It follows that for any $m\in\NN^k$ we can identify $\Lambda^m$ with
\[
 \braces{x:\Omega_{k,m}\rightarrow\Lambda:x\text{ is a $k$-graph morphism}}.
\]
When $\alpha\in\Lambda$ and $p,q\in\NN^k$ with $p\le q\le d(\alpha)$, we write $\alpha(p,q)$ for $x_\alpha(p,q)$\index{$\alpha(p,q)$ where $\alpha\in\Lambda$} and $\alpha(p)$\index{$\alpha(p)$ where $\alpha\in\Lambda$} for $x_\alpha(p)$.

Kumjian and Pask in \cite[p. 7]{Kumjian-Pask2000} defined the {\em infinite path space}\index{infinite path space} of a $k$-graph $\Lambda$ to be the set
\[
\Lambda^\infty=\braces{x:\Omega_k\rightarrow\Lambda: x\text{ is a $k$-graph morphism}}.\index{$\Lambda^\infty$}
\]
For each $x\in\Lambda^\infty$ define $r(x):=x(0)$. For each $p\in\NN^k$, define $\sigma^p:\Lambda^\infty\rightarrow\Lambda^\infty$\index{$\sigma^p(x)$} by $\sigma^p(x)(m,n)=x(m+p,n+p)$\label{def_shift_map} for $x\in\Lambda^\infty$ and $(m,n)\in\Omega_k$.

In \cite{Kumjian-Pask2000}, Kumjian and Pask have the standing assumption that each of the higher-rank graphs are row-finite and have no sources (these concepts are defined immediately following Proposition \ref{prop_finite_infinite_composition}). The point here is that these assumptions are not used in the proof of the following result from \cite{Kumjian-Pask2000} and so they are left out here.\begin{prop}[{\cite[Proposition~2.3]{Kumjian-Pask2000}}]\label{prop_finite_infinite_composition}
Suppose $\Lambda$ is a $k$-graph. For any $\alpha\in\Lambda$ and $x\in\Lambda^\infty$ with $r(x)=s(\alpha)$, there is a unique $y\in\Lambda^\infty$ such that $x=\sigma^{d(\alpha)}(y)$ and $\alpha=y\big(0,d(\alpha)\big)$.
\end{prop}

For any $\alpha$, $x$ and $y$ as in Proposition \ref{prop_finite_infinite_composition}, we write $\alpha x$\notationindex{ALPHAX@$\alpha x$} for $y$. For $\alpha\in\Lambda$ and $S\subset\Lambda\cup\Lambda^\infty$, define
\[\alpha S:=\braces{\alpha x\in\Lambda\cup\Lambda^\infty:x\in S, r(x)=s(\alpha)}\index{$\alpha S$}.\]

We say $\Lambda$ is {\em row-finite}\index{row-finite!higher-rank graph} if for each $v\in\Lambda^0$ and $m\in\NN^k$, the set $v\Lambda^m$ is finite; a {\em source}\index{source!in a higher-rank graph} in $\Lambda$ is an element $v\in\Lambda^0$ such that $v\Lambda^{e_i}=\emptyset$ for some $i$ ($1\le i\le k$). It follows that $\Lambda$ has no sources if and only if $v\Lambda^m\ne\emptyset$ for all $v\in\Lambda^0$ and $m\in\NN^k$. Note that these notions of `row-finite' and `source' are natural generalisations of `row-finite' and `source' in the context of directed graphs; for a directed graph $E$, the $1$-graph $\Lambda_E$ has no sources if and only if $E$ has no sources and $\Lambda_E$ is row-finite if and only if $E$ is row-finite.

After assuming that $\Lambda$ is row-finite, Kumjian and Pask endowed $\Lambda^\infty$ with the topology with basis $\braces{\alpha\Lambda^\infty:\alpha\in\Lambda}$ before showing in \cite[Lemma~2.6]{Kumjian-Pask2000} that $\alpha\Lambda^\infty$ is compact for every $\alpha\in\Lambda$.

We now describe the higher-rank graph analogue of the notion of shift equivalence in a directed graph from \cite{kprr1997}. Suppose $\Lambda$ is a $k$-graph. For each $n\in\ZZ^k$ define a relation $\sim_n$\index{$\sim_n$} on $\Lambda^\infty$ by $x\sim_n y$ if there exists $p\in\NN^k$ such that $\sigma^p(x)=\sigma^{p-n}(y)$. 
\begin{lemma}
Suppose $\Lambda$ is a $k$-graph. The relations $\sim_n$ form an equivalence relation on $\Lambda^\infty$ in that $x\sim_0 x$, $x\sim_n y$ implies $y\sim_{-n} x$, and $x\sim_n y,y\sim_m z$ implies $x\sim_{n+m} z$.
\begin{proof}
Suppose $x,y,z\in\Lambda^\infty$ and $n,m\in\NN^k$ satisfy $x\sim_n y$ and $y\sim_m z$. The first claim that $x\sim_0 x$ is trivial. To see $y\sim_{-n} x$, first note that since $x\sim_n y$, there exists $p\in\NN^k$ such that $\sigma^p(x)=\sigma^{p-n}(y)$. Then $\sigma^{p-n}(y)=\sigma^p(x)=\sigma^{(p-n)+n}(x)$, so $y\sim_{-n} x$. To see that $x\sim_{n+m} z$, first note that since $y\sim_m z$, there exists $q\in\NN^k$ such that $\sigma^q(y)=\sigma^{q-m}(z)$. Increase $p$ and $q$ if necessary so that $p-n\ge q$. Then $\sigma^p(x)=\sigma^{p-n}(y)=\sigma^{(p-n)-m}(z)$, so $x\sim_{n+m} y$.
\end{proof}
\end{lemma}
If $x\sim_n y$ for some $x,y\in\Lambda^\infty$ and $n\in\ZZ^k$, we say $x$ is {\em shift equivalent}\index{shift equivalence!in a higher-rank graph} to $y$ with {\em lag}\index{lag (shift equivalence)} $n$. We write $x\sim y$ when the lag is not important so that $\sim$ is an equivalence relation. For every $x\in\Lambda^\infty$ let $[x]:=\braces{y\in\Lambda^\infty:x\sim y}$\index{$[x]\subset\Lambda^\infty$}. Unfortunately\label{shift_equivalence_different} this notion of shift equivalence is slightly different to Kumjian, Pask, Raeburn and Renault's notion in \cite{kprr1997} of shift equivalent infinite paths in directed graphs: shift equivalence with lag $n$ in the $1$-graph case is the same as shift equivalence with lag $-n$ in their case. The change has been made to simplify the description of the groupoid in \cite{Kumjian-Pask2000}.

Kumjian and Pask in \cite{Kumjian-Pask2000} defined the following groupoid given a row-finite $k$-graph $\Lambda$. Let
\[
G_{\Lambda}=\braces{(x,n,y)\in\Lambda^\infty\times\ZZ\times\Lambda^\infty:x\sim_n y}\index{$G_{\Lambda}$}.
\]
Define range and source maps $r,s:G_{\Lambda}\rightarrow\Lambda^\infty$ by $r(x,n,y)=(x,0,x)$, $s(x,n,y)=(y,0,y)$. For $(x,n,y),(y,l,z)\in G_{\Lambda}$ define $(x,n,y)(y,l,z)=(x,n+l,z)$ and set $(x,n,y)^{-1}=(y,-n,x)$; the set $G_\Lambda$ with these operations is a groupoid called the {\em path groupoid}\index{path groupoid!of a higher-rank graph} of $\Lambda$.

\begin{remark}\label{remark_path_groupoid_1}
For any $k$-graph $\Lambda$ there is a natural isomorphism between $\Lambda^\infty$ and $G_\Lambda^{(0)}$ given by $x\mapsto (x,0,x)$. Recall the orbit equivalence relation on the unit space of a groupoid: for $a,b\in G_\Lambda^{(0)}$ we write $a \approx b$ if there exists $\gamma\in G_\Lambda$ such that $r(\gamma)=a$ and $s(\gamma)=b$. It follows immediately from the definition of the path groupoid that for $x,y\in\Lambda^\infty$ we have $x\sim y$ in $\Lambda^\infty$ if and only if $(x,0,x)\approx(y,0,y)$ in $G_\Lambda^{(0)}$. These relations are the same under the isomorphism between $\Lambda^\infty$ and $G_\Lambda^{(0)}$, so both equivalence relations will simply be denoted by $\sim$. Since the equivalence class of $x\in\Lambda^\infty$ under $\sim$ is denoted by $[x]$ we will denote the equivalence class of $(x,0,x)\in G_\Lambda^{(0)}$ under $\sim$ by $[(x,0,x)]$\index{$[(x,0,x)]\subset G_\Lambda^{(0)}$}. 

Proposition \ref{prop_path_groupoid} describes a topology on $G_\Lambda$ and describes a homeomorphism between $G_\Lambda^{(0)}$ and $\Lambda^\infty$. We will frequently abuse notation by implicitly using this homeomorphism. In particular the sets $\alpha\Lambda^\infty$ will often be considered to be subsets of $G_\Lambda^{(0)}$. There is only one issue when implicitly using this homeomorphism: the range and source of $x\in\Lambda^\infty$ are elements of $\Lambda^0$ whereas the range and source of $(x,0,x)\in G_\Lambda^{(0)}$ is $(x,0,x)\in G_\Lambda^{(0)}$. We will thus be careful and use appropriate subscripts with the range and source maps when the context allows for confusion.
\end{remark}

Let $\Lambda$ be a $k$-graph. For $(\alpha,\beta)\in\Lambda\ast_s\Lambda$ define
\[
Z(\alpha,\beta)=\braces{(\alpha x, d(\alpha)-d(\beta),\beta x)\in G_\Lambda:x\in s(\alpha)\Lambda^\infty}.\index{$Z(\alpha,\beta)$}
\]
\begin{prop}[{extension on \cite[Proposition~2.8]{Kumjian-Pask2000}}]\label{prop_path_groupoid}
Let $\Lambda$ be a row-finite $k$-graph. The sets $\braces{Z(\mu,\nu): (\mu,\nu)\in\Lambda\ast_s\Lambda}$ form a basis for a locally compact Hausdorff topology on $G_\Lambda$. With this topology $G_\Lambda$ is a second countable, $r$-discrete groupoid in which each $Z(\mu,\nu)$ is compact and open. The map from $\Lambda^\infty$ onto $G_\Lambda^{(0)}$ given by $x\mapsto (x,0,x)$ is a homeomorphism and the counting measures form a Haar system for $G_\Lambda$.
\end{prop}
The only part of this result not in \cite[Proposition~2.8]{Kumjian-Pask2000} is that the counting measures form a Haar system for $G_\Lambda$. While the proof that the counting measures form a Haar system should be the same as in \cite[Proposition~2.6]{kprr1997}, the proof in \cite{kprr1997} noted that the range and source maps are local homeomorphisms before using \cite[Proposition~I.2.7, Proposition~I.2.8]{Renault1980} to deduce that the counting measures form a Haar system. There is, however, a problem with \cite[Proposition~I.2.7, Proposition~I.2.8]{Renault1980} (see Remark \ref{remark_problems_Renault_results}). To get around this, we can use Proposition \ref{Renault_I_2_8i_sub} with the observation in \cite[p. 511]{kprr1997} that for each $\alpha,\beta\in\Lambda$ with $s(\alpha)=s(\beta)$, the maps $r|_{Z(\alpha,\beta)}$ and $s|_{Z(\alpha,\beta)}$ are homeomorphisms onto $\alpha\Lambda^\infty$ and $\beta\Lambda^\infty$, respectively.

The following is an observation by Kumjian and Pask.
\begin{lemma}[{\cite[p.~8]{Kumjian-Pask2000}}]\label{path_groupoids_isomorphic}
Suppose $E$ is a row-finite directed graph without sources. Then the path groupoid associated to $E$ is isomorphic to the path groupoid associated to $\Lambda_E$.
\end{lemma}

\begin{definition}[{\cite[Definitions 1.5]{Kumjian-Pask2000}}]\label{def_row-finite_no_sources_algebra}
Let $\Lambda$ be a row-finite $k$-graph without sources. A {\em Cuntz-Krieger $\Lambda$-family} is a collection $\braces{t_\alpha:\alpha\in\Lambda}$ of partial isometries satisfying:
\begin{enumerate}\renewcommand{\theenumi}{CK\arabic{enumi}}
\item$\braces{t_v:v\in\Lambda^0}$ is a family of mutually orthogonal projections;
\item $t_{\mu\nu}=t_\mu t_\nu$ for all $\mu,\nu\in\Lambda$ such that $s(\mu)=r(\nu)$;
\item $t_\mu^*t_\mu=t_{s(\mu)}$ for all $\mu\in\Lambda$; and
\item\label{no_sources_CK4} for all $v\in\Lambda^0$ and $n\in\NN^k$ we have $t_v=\sum_{\alpha\in v\Lambda^n}t_\alpha t_\alpha^*$.
\end{enumerate}
The {\em Cuntz-Krieger higher-rank graph $C^*$-algebra}, denoted $C^*(\Lambda)$, is defined to be the universal $C^*$-algebra generated by a Cuntz-Krieger $\Lambda$-family $\braces{s_\alpha:\alpha\in\Lambda}$.
\end{definition}
Kumjian and Pask showed that the Cuntz-Krieger $C^*$-algebras of row-finite higher-rank graphs without sources are a generalisation of the Cuntz-Krieger $C^*$-algebras of row-finite directed graphs without sources\footnote{This result is in turn generalised to remove the `no sources' condition in Lemma \ref{lemma_RSY_LambdaE}.}:
\begin{lemma}[{\cite[Examples~1.7(i)]{Kumjian-Pask2000}}]\label{lemma_KP_LambdaE}
Suppose $E$ is a row-finite directed graph without sources and let $\Lambda_E$ be the $1$-graph from Example \ref{example_LambdaE}. Then $C^*(E)$ is isomorphic to $C^*(\Lambda_E)$.
\end{lemma}
Kumjian and Pask also showed that the path groupoid $C^*$-algebras are faithful models for the Cuntz-Krieger $C^*$-algebras of row-finite higher-rank graphs without sources:
\begin{lemma}[{\cite[Corollary~3.5]{Kumjian-Pask2000}}]\label{lemma_higher_rank_graph_algebras_isomorphic}
Suppose $\Lambda$ is a row-finite $k$-graph without sources. Then $C^*(\Lambda)$ is isomorphic to $C^*(G_\Lambda)$.
\end{lemma}

\section{Row-finite higher-rank graphs that may have sources}\label{sec_higher-rank_graphs_2}
In this section we will recall Raeburn, Sims and Yeend's construction in \cite{rsy2004} of the Cuntz-Krieger $C^*$-algebra of a row-finite higher-rank graph that may have sources. This  is unnecessary for readers who are exclusively interested in groupoid results, however results for this more general class of higher-rank graph algebra will be proved in Section \ref{sec_desourcification}.

For $k\in\PP$ and $m,n \in \NN^k$, we denote by $m \vee n$\notationindex{MvN@$m\vee n, m\wedge n$} the coordinate-wise maximum of $m$ and $n$. In other words, if $l=m\vee n$, then $l\in\NN^k$ with $l_i=\max\{m_i,n_i\}$ for each $1\le i\le k$. Similarly the coordinate-wise minimum of $m$ and $n$ is denoted by $m\wedge n$.

Suppose $\Lambda$ is a $k$-graph. For $\mu,\nu\in\Lambda$, define
\[
\Lambda^{\text{min}}(\mu,\nu)=\braces{(\alpha,\beta)\in\Lambda\times\Lambda:\mu\alpha=\nu\beta, d(\mu\alpha)=d(\mu)\vee d(\nu)}.
\]
The elements of $\Lambda^{\text{min}}(\mu,\nu)$ are called the {\em minimal extensions} of $(\mu,\nu)$. Suppose $v\in\Lambda^0$. A subset $D\subset v\Lambda$ of $\Lambda$ is {\em exhaustive} if for every $\mu\in v\Lambda$ there exists $\nu\in D$ such that $\Lambda^{\text{min}}(\mu,\nu)\ne\emptyset$. The set of all {\em finite exhaustive subsets} of $\Lambda$ are denoted by $\Ff\Ee(\Lambda)$, and we let $v\Ff\Ee(\Lambda)$ denote the set $\braces{D\in\Ff\Ee(\Lambda):D\subset v\Lambda}$. 

The $C^*$-algebra of a row-finite higher-rank graph without sources was defined in Definition \ref{def_row-finite_no_sources_algebra}. When $\Lambda$ has sources, condition \eqref{no_sources_CK4} from Definition \ref{def_row-finite_no_sources_algebra} causes problems since $v\Lambda^n$ may be non-empty for some values of $n$ and empty for others. The following defines the Cuntz-Krieger $C^*$-algebras of row-finite higher-rank graphs that may have sources. Conditions \eqref{CK3} and \eqref{CK4} change substantially between these definitions so that the new relations are ``the right relations for generating tractable Cuntz-Krieger algebras for which a homomorphism is injective on the core if and only if it is nonzero at each vertex projection'' \cite[Remark 2.6]{rsy2004}.
\begin{definition}[{\cite[Definition~2.5]{rsy2004}}\footnote{This is presented in \cite{rsy2004} for the more general `finitely aligned' class of higher-rank graphs.}]\label{def_finitely-aligned_algebra}
Let $\Lambda$ be a row-finite $k$-graph. A {\em Cuntz-Krieger $\Lambda$-family} is a collection $\braces{t_\alpha:\alpha\in\Lambda}$ of partial isometries satisfying:
\begin{enumerate}\renewcommand{\theenumi}{CK\arabic{enumi}}
\item$\braces{t_v:v\in\Lambda^0}$ is a family of mutually orthogonal projections;
\item $t_{\mu\nu}=t_\mu t_\nu$ for all $\mu,\nu\in\Lambda$ such that $s(\mu)=r(\nu)$;
\item\label{CK3} $t_\mu^*t_\nu=\sum_{(\alpha,\beta)\in\Lambda^{\mathrm{min}}(\mu,\nu)}t_\alpha t_\beta^*$ for all $\mu,\nu\in\Lambda$; and
\item\label{CK4} $\prod_{\mu\in D}(t_v-t_\mu t_\mu^*)=0$ for all $v\in\Lambda^0$ and $D\in v\Ff\Ee(\Lambda)$.
\end{enumerate}
The {\em Cuntz-Krieger higher-rank graph $C^*$-algebra}, denoted $C^*(\Lambda)$, is defined to be the universal $C^*$-algebra generated by a Cuntz-Krieger $\Lambda$-family $\braces{s_\alpha:\alpha\in\Lambda}$.
\end{definition}
Raeburn Sims and Yeend in \cite[Proposition~B.1]{rsy2004} show that if $\Lambda$ is a row-finite $k$-graph without sources then the Cuntz-Krieger relations from Definitions \ref{def_row-finite_no_sources_algebra} and \ref{def_finitely-aligned_algebra} coincide\footnote{They actually showed that the Cuntz-Krieger relations from Definition \ref{def_finitely-aligned_algebra} coincide with those relations from \cite[Definition~3.3]{rsy2003}, however this in turn coincides with the relations in Definition \ref{def_row-finite_no_sources_algebra} when $\Lambda$ has no sources.}, so that the definitions of $C^*(\Lambda)$ in Definitions \ref{def_row-finite_no_sources_algebra} and \ref{def_finitely-aligned_algebra} are consistent. The next lemma generalises Lemma \ref{lemma_KP_LambdaE}.
\begin{lemma}[{\cite[Proposition~B.1]{rsy2004}}]\label{lemma_RSY_LambdaE}
Suppose $E$ is a row-finite directed graph and let $\Lambda_E$ be the $1$-graph from Example \ref{example_LambdaE}. Then $C^*(E)$ is isomorphic to $C^*(\Lambda_E)$.
\end{lemma}
The following example shows that it is possible for $C^*(G_\Lambda)$ to be substantially different to $C^*(\Lambda)$ when $\Lambda$ has sources.
\begin{example}\label{example_groupoid_model_requires_no_sources}
Suppose $\Lambda$ is the $1$-graph with the directed graph from Example \ref{example_directed_graph_algebras_not_isomorphic} as the skeleton. Then $C^*(G_\Lambda)$ is liminal while $C^*(\Lambda)$ is not.
\begin{proof}
In Example \ref{example_directed_graph_algebras_not_isomorphic_proof} it is shown that if $E$ is the directed graph from Example \ref{example_directed_graph_algebras_not_isomorphic}, then $C^*(G_E)$ is liminal while $C^*(E)$ is not. The result follows since $G_E\cong G_\Lambda$ by Lemma \ref{path_groupoids_isomorphic} and $C^*(E)\cong C^*(\Lambda)$ by Lemma \ref{lemma_RSY_LambdaE}.
\end{proof}
\end{example}
Raeburn, Sims and Yeend in \cite{rsy2004} describe versions of the gauge-invariant uniqueness theorem and the Cuntz-Krieger uniqueness theorem for row-finite higher-rank graphs that may have sources.

\chapter{Categorising higher-rank graphs using the path groupoid}\label{chapter_categorising_higher-rank_graphs}

\section{Liminal and postliminal higher-rank graph algebras}\label{sec_liminal_postliminal}
This section establishes characterisations of the row-finite higher-rank graphs without sources that have liminal and postliminal $C^*$-algebras. In Section \ref{sec_desourcification} we will will relax the hypothesis to permit higher-rank graphs with sources.

A characterisation of the directed graphs with liminal $C^*$-algebras has been developed by Ephrem in \cite[Theorem~5.5]{Ephrem2004} and characterisations of the directed graphs with postliminal $C^*$-algebras have been developed by Ephrem in \cite[Theorem~7.3]{Ephrem2004} and by Deicke, Hong and Szyma\'{n}ski in \cite[Theorem~2.1]{dhs2003}. To achieve \cite[Theorem~5.5]{Ephrem2004}, Ephrem began by developing a characterisation of the row-finite directed graphs without sources that have liminal $C^*$-algebras and then used the Drinen-Tomforde desingularisation\index{Drinen-Tomforde desingularisation} (see \cite{Drinen-Tomforde2005}) to generalise this result to the arbitrary directed graph case. To characterise the higher-rank graphs with liminal $C^*$-algebras we follow a similar approach: Theorem \ref{k-graph_liminal_thm1} characterises the row-finite $k$-graphs without sources and with liminal $C^*$-algebras and Theorem \ref{k-graph_liminal_thm2} generalises this characterisation to row-finite $k$-graphs that may have sources. The generalisation process uses the Farthing-Webster higher-rank graph desourcification that was developed by Webster in \cite{Webster2011} based on Farthing's \cite{Farthing2008} (see Section \ref{sec_boundary_paths_and_desourcification}). At the time of writing no full desingularisation process exists for higher-rank graphs so our theorems will only apply to row-finite higher-rank graphs.

Ephrem's characterisation \cite[Theorem~7.3]{Ephrem2004} of the directed graphs with postliminal $C^*$-algebras followed a similar approach to \cite[Theorem~5.5]{Ephrem2004} by first developing a characterisation for the row-finite directed graphs without sources and then using the Drinen-Tomforde desingularisation to generalise the result to remove the `row-finite' requirement and permit sources. Again we follow a similar approach: Theorem \ref{k-graph_postliminal_thm1} characterises the row-finite $k$-graphs without sources and with liminal $C^*$-algebras and Theorem \ref{k-graph_postliminal_thm2} uses the Farthing-Webster generalisation to generalise this characterisation to row-finite $k$-graphs that may have sources.

In Section \ref{section_directed_graphs} we will show how, given a row-finite directed graph $E$, we may apply Theorems \ref{k-graph_liminal_thm1}, \ref{k-graph_postliminal_thm1}, \ref{k-graph_liminal_thm2} and \ref{k-graph_postliminal_thm2} to the $1$-graph $\Lambda_E$ to obtain characterisations of the row-finite directed graphs with liminal and postliminal $C^*$-algebras. We will also show how these characterisations coincide with Ephrem's from \cite{Ephrem2004}.

\begin{definition}\label{def_frequently_divertable_infinite_paths}\index{frequently divertable!infinite paths in a $k$-graph}
Suppose $\Lambda$ is a $k$-graph and $x,y\in\Lambda^\infty$. We say {\em $x$ is frequently divertable to $[y]$} if for every $n\in\NN^k$ there exists $z\in x(n)\Lambda^\infty$ such that $z$ is shift equivalent to $y$.
\end{definition}

Suppose $\Lambda$ is a $1$-graph, $v\in\Lambda^0$ and $f\in\Lambda^{e_1}$. If there is exactly one path $x$ in $\Lambda^\infty$ with $x(0)=v$ and $x(n,n+1)=f$ for some $n\in\NN$, we denote $x$ by $\uniqueinfinitepath{v,f}$\index{$\uniqueinfinitepath{v,f}\in\Lambda^\infty$}.
\begin{example}\label{example_frequently_divertable}
Suppose $\Lambda$ is the $1$-graph with skeleton described by the following picture.
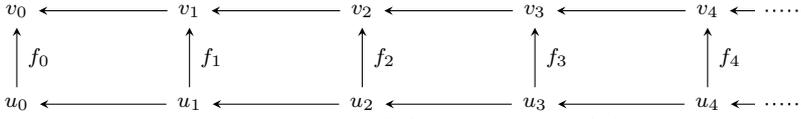
\begin{center}
\begin{tikzpicture}[>=stealth,baseline=(current bounding box.center)] 
\def\cellwidth{5.5};
\clip (-0.7em,-3.2em) rectangle (4*\cellwidth em + 3.5em,0.2em);
\node(endtailone) at (4*\cellwidth em + 3.4em,0em) {};
\node(startendtail) at (4*\cellwidth em + 1.5em,0em) {};
\draw [dotted, thick] (startendtail) -- (endtailone);

\node(endtail2) at (4*\cellwidth em + 3.4em,-3em) {};
\node(startendtail2) at (4*\cellwidth em + 1.5em,-3em) {};
\draw [dotted, thick] (startendtail2) -- (endtail2);

\clip (-0.7em,-3.2em) rectangle (4*\cellwidth em + 1.5em,0.2em);

\foreach \x in {0,1,2,3,4,5,6} \foreach \y in {0} \node (x\x y\y) at (\cellwidth*\x em,-3*\y em) {$\scriptstyle v_\x$};
\foreach \x in {0,1,2,3,4,5,6} \foreach \y in {1} \node (x\x y\y) at (\cellwidth*\x em,-3*\y em) {$\scriptstyle u_\x$};

\foreach \x / \z in {0/1,1/2,2/3,3/4,4/5} \draw[black,<-] (x\x y0) to (x\z y0);
\foreach \x / \z in {0/1,1/2,2/3,3/4,4/5} \draw[black,<-] (x\x y1) to (x\z y1);

\foreach \x / \z in {0,1,2,3,4} \draw[black,<-] (x\x y0) to node[anchor=west] {$\scriptstyle f_\x$} (x\x y1) ;

\end{tikzpicture}
\end{center}
Let $x,y$ be the paths in $\Lambda^\infty$ such that $x(n)=v_n$ and $y(n)=u_n$ for all $n\in\NN$. Then $x$ is frequently divertable to $[y]$ but $y$ is not frequently divertable to $[x]$.

To see that $x$ is frequently divertable to $[y]$, fix $n\in\NN$ and note that $\uniqueinfinitepath{v_n,f_n}\in x(n)\Lambda^\infty\cap [y]$. To see that $y$ is not frequently divertable to $[x]$, observe that the only path in $y(0)\Lambda^\infty=u_0\Lambda^\infty$ is $y$, which is not shift equivalent to $x$.
\end{example}

\begin{lemma}\label{k-graph_closure_lemma}
Suppose $\Lambda$ is a row-finite $k$-graph and that $x,y\in \Lambda^\infty$. Then $x$ is frequently divertable to $[y]$ if and only if $x\in \overline{[y]}$.
\begin{proof}
Suppose $x\in\overline{[y]}$ and fix $n\in\NN^k$. Then every open neighbourhood $U$ of $x$ intersects $[y]$. Since $x(0,n)\Lambda^\infty$ is an open neighbourhood of $x$, it follows that $x(0,n)\Lambda^\infty$ intersects $[y]$. Suppose $z\in x(0,n)\Lambda^\infty\cap [y]$. Then $\sigma^n(z)\in x(n)\Lambda^\infty$ and $\sigma^n(z)$ is shift equivalent to $y$. Since $n$ was fixed arbitrarily, it follows that $x$ is frequently divertable to $[y]$.

Now suppose that $x$ is frequently divertable to $[y]$. Let $U$ be an open neighbourhood of $x$ in $\Lambda^\infty$. Since the sets $\alpha\Lambda^\infty$ form a basis of open sets for $\Lambda^\infty$, there exists $m\in\NN^k$ such that $x(0,m)\Lambda^\infty\subset U$. Since $x$ is frequently divertable to $[y]$, there exists $z\in x(m)\Lambda^\infty$ with $z$ shift equivalent to $y$. Thus
\[
 x(0,m)z\in x(0,m)\Lambda^\infty\subset U,
\]
and as $z\in [y]$, we know that $x(0,m)z\in [y]$, so $[y]$ intersects $U$. It follows that $x\in\overline{[y]}$ since $U$ is an arbitrarily fixed neighbourhood of $x$.
\end{proof}
\end{lemma}

\begin{theorem}\label{k-graph_liminal_thm1}
Suppose $\Lambda$ is a row-finite $k$-graph without sources and let $G$ be the associated path groupoid. Then the following are equivalent:
\begin{enumerate}
\item\label{k-graph_liminal_thm1_1} every orbit in $G^{(0)}$ is closed;
\item\label{k-graph_liminal_thm1_2} for every $x\in\Lambda^\infty$, every path in $\Lambda^\infty$ that is frequently divertable to $[x]$ is shift equivalent to $x$; and
\item\label{k-graph_liminal_thm1_3} $C^*(\Lambda)\cong C^*(G)$ is liminal.
\end{enumerate}
\begin{proof}
\eqref{k-graph_liminal_thm1_1} $\implies$ \eqref{k-graph_liminal_thm1_2}. Suppose \eqref{k-graph_liminal_thm1_1}, fix $x\in\Lambda^\infty$ and suppose $y$ is frequently divertable to $[x]$. Then $y\in\overline{[x]}$ by Lemma~\ref{k-graph_closure_lemma}. Since $[x]$ is closed, $[x]=\overline{[x]}$, so $y\in [x]$, establishing \eqref{k-graph_liminal_thm1_2}.

\eqref{k-graph_liminal_thm1_2} $\implies$ \eqref{k-graph_liminal_thm1_1}. Fix $x\in\Lambda^\infty$ and choose $y\in\overline{[x]}$. Then $y$ is frequently divertable to $[x]$ by Lemma~\ref{k-graph_closure_lemma} and it follows from \eqref{k-graph_liminal_thm1_2} that $y\in [x]$. Thus $[x]=\overline{[x]}$ and \eqref{k-graph_liminal_thm1_1} follows by considering the homeomorphism between $\Lambda^\infty$ and $G^{(0)}$ from Proposition \ref{prop_path_groupoid}.

\eqref{k-graph_liminal_thm1_1} $\iff$ \eqref{k-graph_liminal_thm1_3}. Note that $C^*(G)$ is isomorphic to $C^*(\Lambda)$ by Lemma \ref{lemma_higher_rank_graph_algebras_isomorphic}. Since the stability subgroups $G|_x$ are abelian, they are amenable and $C^*(G|_x)\cong C_0(\widehat{G|_x})$ is liminal. We can now apply Theorem \ref{thm_groupoid_algebra_liminal} and Remark \ref{thm_groupoid_algebra_liminal} which says that $C^*(G)$ is liminal if and only if every orbit in $G^{(0)}$ is closed.
\end{proof}
\end{theorem}

\begin{example}\label{example_not_liminal}
Let $\Lambda$ be the $1$-graph from Example \ref{example_frequently_divertable}. Let $x,y$ be the associated paths in $\Lambda^\infty$, noting that $x$ is not shift equivalent to $y$. It was shown in Example \ref{example_frequently_divertable} that $x$ is frequently divertable to $[y]$, so condition \eqref{k-graph_liminal_thm1_2} from Theorem \ref{k-graph_liminal_thm1} does not hold and $C^*(\Lambda)$ is not liminal.
\end{example}

\begin{lemma}\label{lemma_equivalence_class_frequently_divertable}
Suppose $\Lambda$ is a $k$-graph and $x\in\Lambda^\infty$. Suppose there exists $n\in\NN^k$ such that every element of $x(n)\Lambda^\infty$ that is frequently divertable to $[x]$ is shift equivalent to $x$. Then for every $y\in [x]$, there exists $m\in\NN^k$ such that every element of $y(m)\Lambda^\infty$ that is frequently divertable to $[x]$ is shift equivalent to $x$.
\begin{proof}
Fix $y\in [x]$. Then $x\sim_a y$ for some $a\in\ZZ^k$, so there exists $l\in\NN^k$ such that $\sigma^l(x)=\sigma^{l-a}(y)$. We may increase $n$ if necessary to ensure that $n\ge l$. We now have
\[
 x(n)=x\big((n-l)+l\big)=y\big((n-l)+(l-a)\big)=y(n-a)
\]
and the result follows since $x(n)\Lambda^\infty=y(n-a)\Lambda^\infty$.
\end{proof}
\end{lemma}

\begin{theorem}\label{k-graph_postliminal_thm1}
Suppose $\Lambda$ is a row-finite $k$-graph without sources and let $G$ be the associated path groupoid. Then the following are equivalent:
\begin{enumerate}
\item\label{k-graph_postliminal_thm1_1} every orbit in $G^{(0)}$ is locally closed;
\item\label{k-graph_postliminal_thm1_2} for every path $x\in\Lambda^\infty$ there exists $n\in\NN^k$ such that every path in $x(n)\Lambda^\infty$ that is frequently divertable to $[x]$ is shift equivalent to $x$; and
\item\label{k-graph_postliminal_thm1_3} $C^*(\Lambda)\cong C^*(G)$ is postliminal.
\end{enumerate}
\begin{proof}
\eqref{k-graph_postliminal_thm1_1} $\implies$ \eqref{k-graph_postliminal_thm1_2}. Suppose \eqref{k-graph_postliminal_thm1_1} and fix $x\in\Lambda^\infty$. By \eqref{k-graph_postliminal_thm1_1} there is an open $U\subset\Lambda^\infty$ such that $[x]= U\cap\overline{[x]}$. Then $x\in U$ so by the topology on $\Lambda^\infty$ there exists $n\in\NN^k$ such that $x(0,n)\Lambda^\infty\subset U$. Suppose $y$ is a path in $x(n)\Lambda^\infty$ that is frequently divertable to $[x]$. Then $x(0,n)y\in\overline{[x]}$ by Lemma \ref{k-graph_closure_lemma} so $x(0,n)y\in U \cap \overline{[x]}=[x]$ and it follows that $y\in [x]$, establishing \eqref{k-graph_postliminal_thm1_2}.

\eqref{k-graph_postliminal_thm1_2} $\implies$ \eqref{k-graph_postliminal_thm1_1}. Fix $x\in\Lambda^\infty$ and suppose there exists $n\in\NN^k$ as in \eqref{k-graph_postliminal_thm1_2}. It follows from Lemma \ref{lemma_equivalence_class_frequently_divertable} that for every $y\in [x]$, there exists $m^{(y)}\in\NN^k$ such that every path in $y(m^{(y)})\Lambda^\infty$ that is frequently divertable to $[x]$ is shift equivalent to $x$. Let $V=\cup_{y\in [x]}y(0,m^{(y)})\Lambda^\infty$ and note that $[x]\subset V$ and that every path in $V$ that is frequently divertable to $[x]$ is shift equivalent to $x$. We can now use Lemma \ref{k-graph_closure_lemma} to see that $V\cap\overline{[x]}=[x]$. Finally, note that $V$ is open since each $y(0,m^{(y)})\Lambda^\infty$ is, so $[x]$ is locally closed.

\eqref{k-graph_postliminal_thm1_1} $\iff$ \eqref{k-graph_postliminal_thm1_3}. The \eqref{k-graph_postliminal_thm1_1} $\iff$ \eqref{k-graph_postliminal_thm1_3} argument from Theorem \ref{k-graph_liminal_thm1} works by using Theorem \ref{thm_groupoid_algebra_postliminal} and Corollary \ref{cor_ramsay_orbits_locally_closed} in place of Theorem \ref{thm_groupoid_algebra_liminal} and Remark \ref{remark_T1_iff_orbits_closed}.
\end{proof}
\end{theorem}

\begin{example}\label{example_postliminal_not_liminal}
Let $\Lambda$ be the $1$-graph from Example \ref{example_frequently_divertable}. Then $C^*(\Lambda)$ is postliminal.
\begin{proof}
Fix $a\in \Lambda^\infty$ and let $x,y$ be the paths in $\Lambda^\infty$ from Example \ref{example_frequently_divertable}. We aim to establish condition \eqref{k-graph_postliminal_thm1_2} from Theorem \ref{k-graph_postliminal_thm1}. That is, for each $a\in\Lambda^\infty$ we need to find $n\in\NN$ such that every path in $a(n)\Lambda^\infty$ that is frequently divertable to $[a]$ is shift equivalent to $a$. We will consider three cases. The first case is when $a=\sigma^m(y)$ for some $m\in\NN$. The only path in $a(0)\Lambda^\infty=u_m \Lambda^\infty$ is $a$, so $a(0) \Lambda^\infty\backslash [a]=\emptyset$, and $n=0$ will do.

The second case is when $a=\sigma^m(x)$ for some $m\in\NN$. Then 
\[a(0)\Lambda^\infty\backslash[a]=v_m \Lambda^\infty\backslash [a]=\braces{\uniqueinfinitepath{v_m,f_{m+l}}:l\in\NN}\]
and none of these paths are frequently divertable to $[a]$, so $n=0$ will do.

The final case is when $a=\uniqueinfinitepath{v_m,f_{m+l}}$ for some $m,l\in\NN$. Then $a(l,l+1)=f_{m+l}$ and by observing the graph it can be seen that $b:=\sigma^{l+1}(a)$ is the only path in $a(l+1)\Lambda^\infty$. Since $b$ is shift equivalent to $a$, $a(l+1)\Lambda^\infty\backslash [a]=\emptyset$, so $n=l+1$ will do.

It follows from our consideration of these three cases that condition \eqref{k-graph_postliminal_thm1_2} from Theorem \ref{k-graph_postliminal_thm1} holds and so $C^*(\Lambda)$ is postliminal.
\end{proof}
\end{example}

\begin{example}
Suppose $\Lambda$ is the $1$-graph with skeleton described by the following picture. Then $C^*(\Lambda)$ is not postliminal.
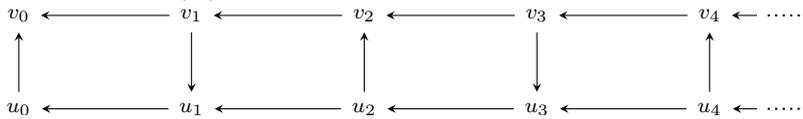
\begin{center}
\begin{tikzpicture}[>=stealth,baseline=(current bounding box.center)] 
\def\cellwidth{5.5};
\clip (-0.7em,-3.2em) rectangle (4*\cellwidth em + 3.5em,0.2em);
\node(endtailone) at (4*\cellwidth em + 3.4em,0em) {};
\node(startendtail) at (4*\cellwidth em + 1.5em,0em) {};
\draw [dotted, thick] (startendtail) -- (endtailone);

\node(endtail2) at (4*\cellwidth em + 3.4em,-3em) {};
\node(startendtail2) at (4*\cellwidth em + 1.5em,-3em) {};
\draw[dotted, thick] (startendtail2) -- (endtail2);

\clip (-0.7em,-3.2em) rectangle (4*\cellwidth em + 1.5em,0.2em);

\foreach \x in {0,1,2,3,4,5} \foreach \y in {0} \node (x\x y\y) at (\cellwidth*\x em,-3*\y em) {$\scriptstyle v_\x$};
\foreach \x in {0,1,2,3,4,5} \foreach \y in {1} \node (x\x y\y) at (\cellwidth*\x em,-3*\y em) {$\scriptstyle u_\x$};

\foreach \x / \z in {0/1,1/2,2/3,3/4,4/5} \draw[black,<-] (x\x y0) to (x\z y0);
\foreach \x / \z in {0/1,1/2,2/3,3/4,4/5} \draw[black,<-] (x\x y1) to (x\z y1);

\foreach \x / \z in {0,2,4} \draw[black,<-] (x\x y0) to (x\x y1);
\foreach \x / \z in {1,3} \draw[black,->] (x\x y0) to (x\x y1);

\end{tikzpicture}
\end{center}
\begin{proof}
Let $x$ and $y$ be the paths in $\Lambda^\infty$ such that $x(n)=v_n$ and $y(n)=u_n$ for all $n\in\NN$. We claim that condition \eqref{k-graph_postliminal_thm1_2} from Theorem \ref{k-graph_postliminal_thm1} does not hold. To see this, note that for every $n\in\NN$, there are paths in $x(n)\Lambda^\infty=v_n\Lambda^\infty$ that are shift equivalent to $y$. Each of these paths in $x(n)\Lambda^\infty\cap [y]$ are frequently divertable to $[x]$ but not shift equivalent to $x$. Thus condition \eqref{k-graph_postliminal_thm1_2} from Theorem \ref{k-graph_postliminal_thm1} does not hold and so $C^*(\Lambda)$ is not postliminal.
\end{proof}
\end{example}

\section{Bounded-trace, Fell, and continuous-trace graph algebras}\label{sec_bded-trace_cts-trace_Fell}

This section establishes characterisations of row-finite higher-rank graphs without sources that have integrable, Cartan and proper path groupoids. Assuming that the path groupoids are principal enables us to use the theorems in Section \ref{bounded-trace_fell_cts-trace_overview} by Clark, an Huef, Muhly and Williams to provide necessary and sufficient conditions for the $C^*$-algebras of these higher rank graphs to have bounded trace, to be Fell and to have continuous trace. Not much generality is lost when we assume that path groupoids are principal since every integrable path groupoid is principal (see Lemma \ref{lemma_integrable_implies_principal}). A later remark, Remark \ref{remark_directed_graphs_and_cycles}, discusses what the requirement of having a principal path groupoid means in the directed graph context.

\begin{lemma}\label{lemma_principal_iff_no_semi-periodic_paths}
Suppose $\Lambda$ is a row-finite $k$-graph without sources. The path groupoid associated to $\Lambda$ is principal if and only if $n=0$ whenever a path in $\Lambda^\infty$ is shift equivalent to itself with lag $n$.

\begin{proof}
The stability subgroups in $G$ are $S_{(x,0,x)}=\braces{(x,n,x): x\sim_n x}$. Since a groupoid is principal if and only if all the stability subgroups are trivial, the result follows.
\end{proof}
\end{lemma}
It follows from this lemma that $G_\Lambda$ is principal if and only if $\Lambda$ has no periodic paths (as defined in \cite{Kumjian-Pask2000}). We avoid this notion of `periodic paths' because it can be misleading when extended to boundary paths as would be required in Section \ref{sec_desourcification}.

\begin{example}\label{example_2_graph_not_principal_no_cycles}
Let $\Lambda$ be the unique $2$-graph with the following skeleton.
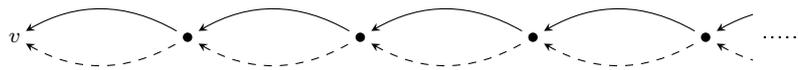
\begin{center}
\begin{tikzpicture}[>=stealth,baseline=(current bounding box.center)] 
\def\cellwidth{5.5};
\clip (-0.5em,-1em) rectangle (4*\cellwidth em + 3.5em,1em);
\node(endtailone) at (4*\cellwidth em + 3.4em,0em) {};
\node(startendtail) at (4*\cellwidth em + 1.5em,0em) {};
\draw[dotted, thick] (startendtail) -- (endtailone);

\clip (-0.5em,-1em) rectangle (4*\cellwidth em + 1.5em,1em);

\foreach \x in {1,2,3,4,5,6} \foreach \y in {0} \node (x\x y\y) at (\cellwidth*\x em-\cellwidth em,-3*\y em) {};
\foreach \x in {2,3,4,5,6} \foreach \y in {0} \fill[black] (x\x y\y) circle (0.15em);

\node at (x1y0) {$\scriptstyle v$};

\foreach \x / \z in {1/2,2/3,3/4,4/5,5/6} \draw[black,<-, bend left] (x\x y0) to node[anchor=south] {} (x\z y0);
\foreach \x / \z in {1/2,2/3,3/4,4/5,5/6} \draw[black,dashed,<-, bend right] (x\x y0) to node[anchor=south] {} (x\z y0);

\end{tikzpicture}
\end{center}
The path groupoid associated to $\Lambda$ is not principal: suppose the morphisms in $\Lambda$ corresponding to the solid edges have degree $(1,0)$ and the morphisms corresponding to the dashed edges have degree $(0,1)$. Let $x$ be the unique path in $\Lambda^\infty$ with range $v$. Then $\sigma^{(1,0)}(x)=\sigma^{(0,1)}(x)$, so $x$ is shift equivalent to itself with lag $(1,-1)$. Lemma \ref{lemma_principal_iff_no_semi-periodic_paths} now applies to show that the path groupoid associated to $\Lambda$ is not principal.
\end{example}

\begin{remark}\label{remark_no_cycles_principal}
A {\em cycle} in a $k$-graph is a path $\alpha\in\Lambda$ with $d(\alpha)\ne 0$ and, for distinct $m,n\le d(\alpha)$, we have $\alpha(m)=\alpha(n)$ if and only if $m,n\in\{0,d(\alpha)\}$. A corollary of Lemma \ref{lemma_principal_iff_no_semi-periodic_paths} is that the path groupoid of a $1$-graph is principal if and only if the $1$-graph has no cycles. Example \ref{example_2_graph_not_principal_no_cycles} demonstrates that the same is not true for $2$-graphs since no path in that $2$-graph has both non-zero degree and equal range and source.
\end{remark}

\begin{lemma}\label{lemma_integrable_implies_principal}
Suppose $\Lambda$ is a row-finite $k$-graph without sources. If $G_\Lambda$ is integrable, then $G_\Lambda$ is principal.
\begin{proof}
Suppose $G=G_\Lambda$ is integrable and that $x\in\Lambda^\infty$ satisfies $x\sim_n x$ for some $n\in\NN^k$. Then there exists $p\in\NN^k$ such that $\sigma^p(x)=\sigma^{p-n}(x)$. Fix $i\in\NN$ and let $q=p\vee (p+n)\vee\cdots\vee (p+in)$. Then $q-jn\ge p$ for every $j\le i$ so $\sigma^q(x)=\sigma^{q-n}(x)=\cdots=\sigma^{q-in}(x)$ and $x\sim_{in} x$. We now have $\{(x,jn,x):j\in\NN\}\subset G_x^{r(x)\Lambda^\infty}$, but this set must be finite since $\lambda_x$ is the counting measure on $G_x$ and $G$ integrable implies $\lambda_x(G^{r(x)\Lambda^\infty})<\infty$. Thus $n=0$ and so $G$ is principal by Lemma \ref{lemma_principal_iff_no_semi-periodic_paths}.
\end{proof}
\end{lemma}

\begin{lemma}\label{lemma_description_lambda_x_GV}
Suppose $\Lambda$ is a row-finite $k$-graph without sources. If $G_\Lambda$ is principal then for every $x\in\Lambda^\infty$ and $V\subset\Lambda^\infty$, the quantity $\lambda_x(G^V)$ is the number of paths in $V$ that are shift equivalent to $x$.
\begin{proof}
Suppose $G=G_\Lambda$ is principal and note that $\lambda_x(G^V)$ is the number of elements in $\{(y,n)\in V\times\ZZ^k: y\sim_n x\}$. If $y\sim_m x$ and $y\sim_n x$ for some $m,n\in\ZZ^k$, then $(y,m,x)$ and $(y,n,x)$ are elements of the principal groupoid $G$ with equal range and source, so $m=n$. Then $\lambda_x(G^V)$ is the number of elements in $\{y\in V: y\sim x\}$, as required.
\end{proof}
\end{lemma}

\begin{theorem}\label{thm_integrable_bounded-trace}
Suppose $\Lambda$ is a row-finite $k$-graph without sources. The following are equivalent:
\begin{enumerate}
 \item\label{k-thm_integrable1} $G_\Lambda$ is integrable;
\item\label{k-thm_integrable2} $G_\Lambda$ is principal, and for every $v\in\Lambda^0$ there exists $M\in\NN$ such that for any $x\in\Lambda^\infty$ there are at most $M$ paths in $\Lambda^\infty$ with range $v$ that are shift equivalent to $x$; and
 \item\label{k-thm_integrable3} $G_\Lambda$ is principal, and $C^*(\Lambda)\cong C^*(G_\Lambda)$ has bounded trace.
\end{enumerate}

\begin{proof}
Let $G=G_\Lambda$. \eqref{k-thm_integrable1} $\implies$ \eqref{k-thm_integrable2}. Suppose \eqref{k-thm_integrable1} and note that $G$ is principal by Lemma \ref{lemma_integrable_implies_principal}. Fix $v\in\Lambda^0$. Since $G$ is integrable, $\sup_{x\in v\Lambda^\infty}\lambda_x(G^{v\Lambda^\infty})<\infty$, so there exists $M\in\NN$ such that $\lambda_x(G^{v\Lambda^\infty})\le M$ for every $x\in v\Lambda^\infty$. Lemma \ref{lemma_description_lambda_x_GV} tells us that the number of elements in $v\Lambda^\infty$ that are shift equivalent to $x$ is equal to $\lambda_x(G^{v\Lambda^\infty})$. We just showed this is less than or equal to $M$ when $x\in v\Lambda^\infty$. When $x\notin v\Lambda^\infty$ the number of elements in $v\Lambda^\infty$ that are shift equivalent to $x$ is still less than or equal to $M$ since shift equivalence is an equivalence relation: if $x$ is shift equivalent to $y\in v\Lambda^\infty$, every $z\in v\Lambda^\infty$ that is shift equivalent to $x$ must also be shift equivalent to $y$, and we have already shown that there are at most $M$ of these paths, establishing \eqref{k-thm_integrable2}.

\eqref{k-thm_integrable2} $\implies$ \eqref{k-thm_integrable1}. Suppose \eqref{k-thm_integrable2} and let $N$ be a compact subset of $G^{(0)}$. Since there is a natural homeomorphism between $G^{(0)}$ and $\Lambda^\infty$ given by $(x,0,x)\mapsto x$, we can also consider $N$ to be a compact subset of $\Lambda^\infty$. We will affix subscripts in an attempt to minimise potential confusion between the range map of $\Lambda^\infty$ and the range map of $G$.

For every $x\in N$ the cylinder set $r_\Lambda(x)\Lambda^\infty$ is an open neighbourhood of $x$ in $\Lambda^\infty$. Then $N\subset r_\Lambda(N)\Lambda^\infty=\cup_{v\in r_\Lambda(N)}v\Lambda^\infty$ so, since $N$ is compact and each $v\Lambda^\infty$ is open, there exists a finite subset $V=\braces{v_1,\ldots,v_n}$ of $r_\Lambda(N)$ so that $N\subset V\Lambda^\infty$. For every $v\in V$ let $M_v$ be the integer $M$ as in \eqref{k-thm_integrable2}. Fix $x\in N$ and note that Lemma \ref{lemma_description_lambda_x_GV} tells us that $\lambda_x(G^{v_i\Lambda^\infty})$ is the number of elements in $v_i\Lambda^\infty$ that are shift equivalent to $x$. It follows that $\lambda_x(G^{v_i\Lambda^\infty})\le M_{v_i}$ for each $1\le i\le n$. Now
\[\lambda_x(G^N)\le \lambda_x(G^{V\Lambda^\infty})=\lambda_x\Big(\bigcup_{i=1}^n G^{v_i\Lambda^\infty}\Big)\le\sum_{i=1}^n\lambda_x (G^{v_i\Lambda^\infty})\le \sum_{i=1}^n M_{v_i}.\]
We can see that $G$ is integrable since $\sum_{i=1}^n M_{v_i}$ is independent of our choice of $x\in N$.

\eqref{k-thm_integrable1} $\iff$ \eqref{k-thm_integrable3}. Lemma \ref{lemma_integrable_implies_principal} says that every integrable path groupoid must be principal. When $G$ is principal, Theorem \ref{thm_groupoid_algebra_bounded_trace} says that $C^*(G)$ has bounded trace if and only if $G$ is integrable. The result follows since $C^*(G)$ is isomorphic to $C^*(\Lambda)$ by Lemma \ref{lemma_higher_rank_graph_algebras_isomorphic}.
\end{proof}
\end{theorem}

The directed graph in the following example occured earlier in Example \ref{2-times_convergence_example} and was chosen so that the path groupoid would be principal, integrable and not Cartan. The search for examples of directed graphs with these properties provided the initial motivation for Chapter \ref{chapter_categorising_higher-rank_graphs}.
\needspace{6\baselineskip}
\begin{example}\label{example_integrable_not_cartan}
Suppose $\Lambda$ is the $1$-graph with the following skeleton.
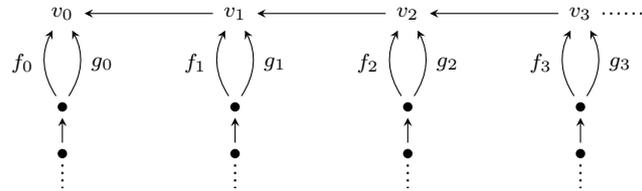
\begin{center}
\begin{tikzpicture}[>=stealth,baseline=(current bounding box.center)] 
\def\cellwidth{5.5};
\clip (-5em,-5.6em) rectangle (3*\cellwidth em + 4.3em,0.3em);

\foreach \x in {0,1,2,3} \foreach \y in {0} \node (x\x y\y) at (\cellwidth*\x em,-3*\y em) {$\scriptstyle v_{\x}$};
\foreach \x in {0,1,2,3} \foreach \y in {1} \node (x\x y\y) at (\cellwidth*\x em,-3*\y em) {};
\foreach \x in {0,1,2,3} \foreach \y in {2} \node (x\x y\y) at (\cellwidth*\x em,-1.5em-1.5*\y em) {};
\foreach \x in {0,1,2,3} \foreach \y in {1,2} \fill[black] (x\x y\y) circle (0.15em);

\foreach \x in {0,1,2,3} \draw [<-, bend left] (x\x y0) to node[anchor=west] {$\scriptstyle g_{\x}$} (x\x y1);
\foreach \x in {0,1,2,3} \draw [<-, bend right] (x\x y0) to node[anchor=east] {$\scriptstyle f_{\x}$} (x\x y1);
\foreach \x / \z in {0/1,1/2,2/3} \draw[black,<-] (x\x y0) to node[anchor=south] {} (x\z y0);

\foreach \x in {0,1,2,3} \draw [<-] (x\x y1) -- (x\x y2);
\foreach \x in {0,1,2,3} {
	\node (endtail\x) at (\cellwidth*\x em, -6em) {};
	\draw  [dotted, thick] (x\x y2) -- (endtail\x);
}

\node(endtailone) at (3*\cellwidth em + 2.3em,0em) {};
\draw [dotted, thick] (x3y0) -- (endtailone);
\end{tikzpicture}
\end{center}
Recall that this directed graph previously appeared in Example \ref{2-times_convergence_example}, where it was shown that the path groupoid associated to this directed graph exhibits $2$-times convergence. Here we claim that the path groupoid associated to $\Lambda$ is principal and integrable and that $C^*(\Lambda)$ has bounded trace. First observe that since $\Lambda$ has no cycles, the path groupoid is principal by Remark \ref{remark_no_cycles_principal}. To see that the rest of condition \eqref{k-thm_integrable2} from Theorem \ref{thm_integrable_bounded-trace} holds, note that for any $v\in\Lambda^0$, there are at most $2$ paths in $v\Lambda^\infty$ that are shift equivalent to each other. Theorem \ref{thm_integrable_bounded-trace} now shows that the path groupoid is integrable and $C^*(\Lambda)$ has bounded trace.
\end{example}

We define the {\em monolithic extension} of $(\alpha,\beta)\in\Lambda\ast_s\Lambda$ by $x\in s(\alpha)(\Lambda\cup\Lambda^\infty)$ to be the pair $(\alpha x,\beta x)$ in $(\Lambda\ast_s\Lambda)\cup(\Lambda^\infty\times\Lambda^\infty)$. For $S\subset\Lambda$, define 
\[S\Lambda^\infty:=\bigcup_{\alpha\in S}\alpha\Lambda^\infty=\braces{\alpha x\in\Lambda^\infty: \alpha\in S, x\in s(\alpha)\Lambda^\infty},\] and note that $S\Lambda^\infty$ is compact whenever $S$ is finite.

\begin{lemma}\label{lemma_compactness}
Suppose $\Lambda$ is a row-finite $k$-graph without sources and suppose that the path groupoid $G$ is principal. Let $V$ be a subset of $\Lambda^0$. The following are equivalent:
\begin{enumerate}
\item\label{lemma_compactness_1} $G|_{V\Lambda^\infty}$ is compact;
\item\label{lemma_compactness_2} there exists a finite $F\subset\Lambda$ such that, for every $x,y\in V\Lambda^\infty$ with $x\sim y$, the pair $(x,y)$ is a monolithic extension of a pair in $F\times F$; and
\item\label{lemma_compactness_3} there exists a finite $F\subset\Lambda$ such that, for every $(\alpha,\beta)\in\Lambda\ast_s\Lambda$ with $r(\alpha),r(\beta)\in V$ and every $\gamma\in s(\alpha)\Lambda$ such that $d(\gamma)=\vee_{\eta\in F}d(\eta)$, the pair $(\alpha\gamma,\beta\gamma)$ is a monolithic extension of a pair in $F\times F$.
\end{enumerate}
\begin{proof}
\eqref{lemma_compactness_1}$\implies$\eqref{lemma_compactness_2}. Suppose \eqref{lemma_compactness_1}. For each $\phi\in G|_{V\Lambda^\infty}$ there exists $(\eta^{(\phi)},\zeta^{(\phi)})\in\Lambda\ast_s\Lambda$ such that $\phi\in Z(\eta^{(\phi)},\zeta^{(\phi)})$. Then $\braces{Z(\eta^{(\phi)},\zeta^{(\phi)}):\phi\in G|_{V\Lambda^\infty}}$ is an open cover of the compact set $G|_{V\Lambda^\infty}$ and so admits a finite subcover 
\[\braces{Z(\eta^{(1)},\zeta^{(1)}),Z(\eta^{(2)},\zeta^{(2)}),\ldots,Z(\eta^{(p)},\zeta^{(p)})}.\]
Define $F:=\braces{\eta^{(1)},\zeta^{(1)},\eta^{(2)},\zeta^{(2)},\ldots,\eta^{(p)},\zeta^{(p)}}$ and fix $x,y\in V\Lambda^\infty$ such that $x\sim_n y$ for some $n\in\ZZ^k$. Then $(x,n,y)\in G|_{V\Lambda^\infty}$ so there exists $q$ such that $(x,n,y)\in Z(\eta^{(q)},\zeta^{(q)})$. It follows that $(x,y)$ is a monolithic extension of $(\eta^{(q)},\zeta^{(q)})\in F\times F$, establishing \eqref{lemma_compactness_2}.

\eqref{lemma_compactness_2}$\implies$\eqref{lemma_compactness_3}. Suppose $F$ is a finite subset of $\Lambda$ as in \eqref{lemma_compactness_2}. Fix $(\alpha,\beta)\in\Lambda\ast_s\Lambda$ with $r(\alpha),r(\beta)\in V$ and fix $\gamma\in s(\alpha)\Lambda$ such that $d(\gamma)=\vee_{\eta\in F}d(\eta)$. Suppose $z\in s(\gamma)\Lambda^{\infty}$ so that $\alpha\gamma z\sim_{d(\alpha)-d(\beta)}\beta\gamma z$. By \eqref{lemma_compactness_2} there exist $\eta,\zeta \in F$ such that $(\alpha\gamma z,\beta\gamma z)$ is a monolithic extension of $(\eta,\zeta)$, so there exists $y\in s(\eta)\Lambda^\infty$ such that $(\alpha\gamma z,\beta\gamma z)=(\eta y,\zeta y)$. Note that $d(\eta),d(\zeta)\le d(\gamma)$, so we can let $\phi=y\big(0,d(\alpha\gamma)-d(\eta)\big)$ and $\psi=y\big(0,d(\beta\gamma)-d(\zeta)\big)$ and observe that $(\alpha\gamma,\beta\gamma)=(\eta\phi,\zeta\psi)$.

We claim that $\phi=\psi$. To see this, note that since $(\alpha\gamma z,\beta\gamma z)=(\eta y,\zeta y)$, we have $\alpha\gamma z\sim_{d(\eta)-d(\zeta)} \beta\gamma z$. We earlier noted that $\alpha\gamma z\sim_{d(\alpha)-d(\beta)} \beta\gamma z$, so both $(\alpha\gamma z, d(\eta)-d(\zeta),\beta\gamma z)$ and $(\alpha\gamma z, d(\alpha)-d(\beta),\beta\gamma z)$ are in $G$. But $G$ is principal so these elements must be equal since they have equal range and source. Then $d(\eta)-d(\zeta)=d(\alpha)-d(\beta)$, which implies 
$d(\alpha)+d(\gamma)-d(\eta)=d(\beta)+d(\gamma)-d(\zeta)$, so $d(\alpha\gamma)-d(\eta)=d(\beta\gamma)-d(\zeta)$. It now follows that $\phi=\psi$ by their definition. Then $(\alpha\gamma,\beta\gamma)=(\eta\phi,\zeta\phi)$, so $(\alpha\gamma,\beta\gamma)$ is a monolithic extension of $(\eta,\zeta)\in F\times F$, establishing \eqref{lemma_compactness_3}.

\eqref{lemma_compactness_3}$\implies$\eqref{lemma_compactness_1}. Suppose $F$ is a finite subset of $\Lambda$ as in \eqref{lemma_compactness_3} and fix $(x,n,y)\in G|_{V\Lambda^\infty}$. Then there exists $m\in\NN^k$ such that $\sigma^m(x)=\sigma^{m-n}(y)$. Let $\alpha=x(0,m)$, $\beta=y(0,m-n)$, $\gamma=x\big(m,m+\vee_{\eta\in F}d(\eta)\big)$ and $z=\sigma^{d(\alpha)+d(\gamma)}(x)$, so that $(x,y)=(\alpha\gamma z,\beta\gamma z)$. By \eqref{lemma_compactness_3} the pair $(\alpha\gamma,\beta\gamma)$ is a monolithic extension of some $(\mu,\nu)\in F\times F$ so there exists $\delta\in\Lambda$ such that $(\alpha\gamma,\beta\gamma)=(\mu\delta,\nu\delta)$. Then
\[
(x,n,y)\in Z(\alpha\gamma,\beta\gamma)=Z(\mu\delta,\nu\delta)\subset Z(\mu,\nu)\subset \bigcup_{(\eta,\zeta)\in F\ast_s F} Z(\eta,\zeta),
\]
so that $G|_{V\Lambda^\infty}\subset \cup_{(\eta,\zeta)\in F\ast_s F} Z(\eta,\zeta)$. Since each $Z(\eta,\zeta)$ is compact and $F$ is finite, the set $\cup_{(\eta,\zeta)\in F\ast_s F}Z(\eta,\zeta)$ is compact. Since $V$ is finite, the set $V\Lambda^\infty$ is closed, so $G|_{V\Lambda^\infty}=r^{-1}(V\Lambda^\infty)\cap s^{-1}(V\Lambda^\infty)$ is a closed subset of the compact set $\cup_{(\eta,\zeta)\in F\ast_s F}Z(\eta,\zeta)$, establishing \eqref{lemma_compactness_1}.
\end{proof}
\end{lemma}

\begin{remark}\label{homeomorphism_remark}
It was observed in \cite[Remarks 2.5]{Kumjian-Pask2000} that for a fixed $\alpha\in\Lambda$, the map given by $\alpha x\mapsto x$ for $x\in r^{-1}\big(s(\alpha)\big)$ induces a homeomorphism between $\alpha\Lambda^\infty$ and $s(\alpha)\Lambda^\infty$. By similar reasoning it can be seen that the map from $G|_{s(\alpha)\Lambda^\infty}$ to $G|_{\alpha\Lambda^\infty}$ where $(x,n,y)\mapsto (\alpha x,n,\alpha y)$ is a homeomorphism.
\end{remark}

\begin{theorem}\label{thm_Cartan_Fell}
Suppose $\Lambda$ is a row-finite $k$-graph without sources. The following are equivalent:
\begin{enumerate}\renewcommand{\labelenumi}{(\arabic{enumi})}
\item\label{thm_Cartan_Fell_1} $G_\Lambda$ is Cartan;
\item\label{thm_Cartan_Fell_2} $G_\Lambda$ is principal, and for every $z\in\Lambda^\infty$ there exist $p\in\NN^k$ and a finite $F\subset\Lambda$ such that for every $x,y\in z(p)\Lambda^\infty$ with $x\sim y$, the pair $(x,y)$ is a monolithic extension of a pair in $F\times F$;
\needspace{4\baselineskip}\item\label{thm_Cartan_Fell_3} $G_\Lambda$ is principal, and for every $z\in\Lambda^\infty$ there exist $p\in\NN^k$ and a finite $F\subset\Lambda$ such that for every $(\alpha,\beta)\in \Lambda\ast_s\Lambda$ with $r(\alpha)=r(\beta)=z(p)$ and every $\gamma\in s(\alpha)\Lambda$ with $d(\gamma)=\vee_{\eta\in F}d(\eta)$, the pair $(\alpha\gamma,\beta\gamma)$ is a monolithic extension of a pair in $F\times F$; and
\item\label{thm_Cartan_Fell_4} $G_\Lambda$ is principal, and $C^*(\Lambda)\cong C^*(G_\Lambda)$ is Fell.
\end{enumerate}
\begin{proof}[Proof of Theorem \ref{thm_Cartan_Fell}]
Let $G=G_\Lambda$. \eqref{thm_Cartan_Fell_1} $\implies$ \eqref{thm_Cartan_Fell_2}. Suppose \eqref{thm_Cartan_Fell_1}. Then $G$ is integrable and thus principal by \cite[Proposition~3.11]{Clark-anHuef2008}, \cite[Lemma~5.1]{Clark-anHuef2012} and Lemma \ref{lemma_integrable_implies_principal}. Since $G$ is Cartan, for every $z\in G^{(0)}$ there exists a neighbourhood $N$ of $z$ in $G^{(0)}$ such that $G|_N$ is relatively compact. By considering $z$ to be in $\Lambda^\infty$ and $N\subset\Lambda^\infty$, the sets $z(0,n)\Lambda^\infty$ for $n\in\NN^k$ form a neighbourhood basis for $z$ in $\Lambda^\infty$, so there exists $p\in\NN^k$ such that $z\in z(0,p)\Lambda^\infty\subset N$. Since $G|_{z(0,p)\Lambda^\infty}$ is a closed subset of the compact set $\overline{G|_N}$, it follows that $G|_{z(0,p)\Lambda^\infty}$ is compact, and so $G|_{z(p)\Lambda^\infty}$ is compact by Remark \ref{homeomorphism_remark}. We now apply Lemma \ref{lemma_compactness} with $V=\braces{z(p)}$ to establish \eqref{thm_Cartan_Fell_2}.

\eqref{thm_Cartan_Fell_2} $\implies$ \eqref{thm_Cartan_Fell_3}. Fix $z\in\Lambda^\infty$ and let $p\in\NN^k$ and $F\subset\Lambda$ be as in \eqref{thm_Cartan_Fell_2}. We can now apply Lemma \ref{lemma_compactness} with $V=\braces{z(p)}$ to establish \eqref{thm_Cartan_Fell_3}.

\eqref{thm_Cartan_Fell_3} $\implies$ \eqref{thm_Cartan_Fell_1}. Fix $z\in\Lambda^\infty$ and let $p$ be the element of $\NN^k$ as in the statement of \eqref{thm_Cartan_Fell_3}. By Lemma \ref{lemma_compactness} we can see that $G|_{z(p)\Lambda^\infty}$ is compact, so $G|_{z(0,p)\Lambda^\infty}$ is compact by Remark \ref{homeomorphism_remark}, and $z(0,p)\Lambda^\infty$ is a wandering neighbourhood of $x$ in $G^{(0)}$, establishing \eqref{thm_Cartan_Fell_1}.

\eqref{thm_Cartan_Fell_1} $\iff$ \eqref{thm_Cartan_Fell_4}. If $G$ is Cartan then $G$ is integrable and thus principal by \cite[Proposition~3.11]{Clark-anHuef2008}, \cite[Lemma~5.1]{Clark-anHuef2012} and Lemma \ref{lemma_integrable_implies_principal}. When $G$ is principal, Theorem \ref{thm_groupoid_algebra_Fell} says that $C^*(G)$ is Fell if and only if $G$ is Cartan. The result follows since $C^*(G)$ is isomorphic to $C^*(\Lambda)$ by Lemma \ref{lemma_higher_rank_graph_algebras_isomorphic}.
\end{proof}
\end{theorem}

\begin{example}\label{example_integrable_not_Cartan}
Let $\Lambda$ be the $1$-graph from Example \ref{example_integrable_not_cartan}. In Example \ref{2-times_convergence_example} it is shown that $G_\Lambda$ exhibits a sequence with $2$-times convergence; if a groupoid exhibits such a sequence, then it is not Cartan by \cite[Lemma~5.1]{Clark-anHuef2012}. Since $G_\Lambda$ is principal (Example \ref{example_integrable_not_cartan}), Theorem \ref{thm_Cartan_Fell} can be used to deduce that $C^*(\Lambda)$ is not Fell. A later result, Theorem \ref{thm_directed_graph_Cartan_Fell}, is a simplification of Theorem \ref{thm_Cartan_Fell} for the case of directed graphs that can be used to make the same claims about $G_\Lambda$ and $C^*(\Lambda)$ with a quick observation of the skeleton.
\end{example}

In the next example we will show how, for the $1$-graph from Example \ref{example_integrable_not_cartan}, conditions \eqref{thm_Cartan_Fell_2} and \eqref{thm_Cartan_Fell_3} from Theorem \ref{thm_Cartan_Fell} can be used to deduce that $C^*(\Lambda)$ is not a Fell algebra. First note that since we are dealing with a $1$-graph, we simplify notation by referring to $\NN$ instead of $\NN^1$; then $e_1$ will just be $1$. \notationindex{VFstar@$[v,f]^*$}As in the directed graph case, if for $v\in \Lambda^0$ and $f\in\Lambda^{1}$ there is a unique $\alpha\in v\Lambda$ with $\alpha\big(d(\alpha)-1,d(\alpha)\big)=f$, we denote $\alpha$ by $[v,f]^*$.
\begin{example}\label{example_integrable_not_Cartan_2}
Suppose that $\Lambda$ is the $1$-graph from Example \ref{example_integrable_not_cartan}. Then $G_\Lambda$ is not Cartan and $C^*(\Lambda)$ is not Fell.
\begin{proof}
We proceed by contradiction. Suppose $G_\Lambda$ is Cartan and let $z$ be the path in $\Lambda^\infty$ such that $z(q)=v_q$ for every $q\in\NN$. Then there exist $p\in\NN$ and a finite $F\subset\Lambda$ as in condition \eqref{thm_Cartan_Fell_3} of Theorem \ref{thm_Cartan_Fell}. For each $q\ge p$, let $\alpha^{(q)}=[v_p,f_q]^*$ and $\beta^{(q)}=[v_p,g_q]^*$. Then $s(\alpha^{(q)})=s(f_q)=s(g_q)=s(\beta^{(q)})$, so $(\alpha^{(q)},\beta^{(q)})\in\Lambda\ast_s\Lambda$. Note that $d(\alpha^{(q)})=d(\beta^{(q)})$. Let $\gamma^{(q)}$ be the path in $s(\alpha^{(q)})\Lambda$ with $d(\gamma^{(q)})=\vee_{\eta\in F}d(\eta)$.

Fix $q\ge p$. Then $(\alpha^{(q)}\gamma^{(q)},\beta^{(q)}\gamma^{(q)})$ is a monolithic extension of a pair $(\eta,\zeta)\in F\times F$ so there exists $\mu\in\Lambda$ such that $\alpha^{(q)}\gamma^{(q)}=\eta\mu$ and $\beta^{(q)}\gamma^{(q)}=\zeta\mu$. Now
\[
(\eta\mu)\big(d(\alpha^{(q)})-1,d(\alpha^{(q)})\big)=\alpha^{(q)}\big(d(\alpha^{(q)})-1,d(\alpha^{(q)})=f_q
\]
and
\begin{align*}
(\zeta\mu)\big(d(\alpha^{(q)})-1,d(\alpha^{(q)})\big)&=(\zeta\mu)\big(d(\beta^{(q)})-1,d(\beta^{(q)})\big)\\
&=\beta^{(q)}\big(d(\beta^{(q)})-1,d(\beta^{(q)})=g_q,
\end{align*}
so we must have $d(\eta)\ge d(\alpha)$ and $d(\zeta)\ge d(\beta)$. Then $(\eta,\zeta)$ must be a monolithic extension of $(\alpha^{(q)},\beta^{(a)})$. Since $q$ was fixed arbitrarily, $F$ must contain a monolithic extension of $(\alpha^{(q)},\beta^{(q)})$ for every $q\ge p$. This contradicts our assumption that $F$ is finite, so it follows that $G_\Lambda$ is not Cartan. Recall from Example \ref{example_integrable_not_cartan} that $G_\Lambda$ is principal. Since in addition $G_\Lambda$ is not Cartan, we can use Theorem \ref{thm_Cartan_Fell} to see that $C^*(\Lambda)$ is not Fell.
\end{proof}
\end{example}

\begin{theorem}\label{thm_proper_cts-trace}
Suppose $\Lambda$ is a row-finite $k$-graph without sources. The following are equivalent:
\begin{enumerate}\renewcommand{\labelenumi}{(\arabic{enumi})}
\item\label{thm_proper_cts-trace_1} $G_\Lambda$ is proper;
\item\label{thm_proper_cts-trace_2} $G_\Lambda$ is principal, and for every finite subset $V$ of $\Lambda^0$, there exists a finite $F\subset\Lambda$ such that, for every $x,y\in V\Lambda^\infty$ with $x\sim_n y$, the pair $(x,y)$ is a monolithic extension of a pair in $F\times F$;
\item\label{thm_proper_cts-trace_3} $G_\Lambda$ is principal, and for every finite subset $V$ of $\Lambda^0$, there exists a finite $F\subset\Lambda$ such that, for every $(\alpha,\beta)\in\Lambda\ast_s\Lambda$ with $r(\alpha),r(\beta)\in V$, and every $\gamma\in s(\alpha)\Lambda$ with $d(\gamma)=\vee_{\eta\in F} d(\eta)$,
the pair $(\alpha\gamma,\beta\gamma)$ is a monolithic extension of a pair in $F\times F$; and
\item\label{thm_proper_cts-trace_4} $G_\Lambda$ is principal, and $C^*(\Lambda)\cong C^*(G_\Lambda)$ has continuous trace.
\end{enumerate}
\begin{proof}
Let $G=G_\Lambda$. \eqref{thm_proper_cts-trace_1} $\implies$ \eqref{thm_proper_cts-trace_2}.
Suppose \eqref{thm_proper_cts-trace_1}. It follows immediately from the definitions of proper and Cartan that $G$ is Cartan. Then $G$ is integrable and thus principal by \cite[Proposition~3.11]{Clark-anHuef2008}, \cite[Lemma~5.1]{Clark-anHuef2012} and Lemma \ref{lemma_integrable_implies_principal}. Let $V$ be a finite subset of $\Lambda^0$. Then $V\Lambda^\infty$ is compact since $V$ is finite. Since $G$ is proper, it follows that $G|_{V\Lambda^\infty}$ is compact. We can now apply Lemma \ref{lemma_compactness} to establish \eqref{thm_proper_cts-trace_2}.

\eqref{thm_proper_cts-trace_2} $\implies$ \eqref{thm_proper_cts-trace_3} follows immediately from Lemma \ref{lemma_compactness}.

\eqref{thm_proper_cts-trace_3} $\implies$  \eqref{thm_proper_cts-trace_1}. Suppose \eqref{thm_proper_cts-trace_3} and let $K$ be a compact subset of $\Lambda^\infty$. For every $z\in K$, the set $r(z)\Lambda^\infty$ is an open neighbourhood of $z$ in $\Lambda^\infty$. Thus $\braces{r(z)\Lambda^\infty:z\in K}$ is an open cover of the compact set $K$, so there exists a finite set $V\subset r(K)$ such that $K\subset V\Lambda^\infty$. By \eqref{thm_proper_cts-trace_3} and Lemma \ref{lemma_compactness} we can see that $G|_{V\Lambda^\infty}$ is compact. Since $K$ is compact and $G$ is Hausdorff, $K$ is a closed subset of $V\Lambda^\infty$. Thus $G|_{K}=s_G^{-1}(K)\cap r_G^{-1}(K)$ is a closed subset of the compact set $G|_{V\Lambda^\infty}$, so $G|_K$ is compact, establishing \eqref{thm_proper_cts-trace_1}.

\eqref{thm_proper_cts-trace_1} $\iff$ \eqref{thm_proper_cts-trace_4}. If $G$ is proper then $G$ is principal (see \eqref{thm_proper_cts-trace_1} $\implies$ \eqref{thm_proper_cts-trace_2}). When $G$ is principal, Theorem \ref{thm_groupoid_algebra_continuous_trace} tells us that $C^*(G)$ has continuous trace if and only if $G$ is proper. The result follows since $C^*(G)$ is isomorphic to $C^*(\Lambda)$ by \cite[Corollary~3.5]{Kumjian-Pask2000}.
\end{proof}
\end{theorem}

\section{Boundary paths and the Farthing-Webster desourcification}\label{sec_boundary_paths_and_desourcification}
This section begins with Farthing, Muhly and Yeend's definition of a boundary path in a higher-rank graph. Given a locally-convex higher-rank graph $\Lambda$, Farthing in \cite{Farthing2008} showed how to construct a locally-convex higher-rank graph with no sources such that the Cuntz-Krieger $C^*$-algebra is Morita equivalent to $C^*(\Lambda)$. For any row-finite higher-rank graph $\Lambda$, Webster in \cite{Webster2011} modified Farthing's construction to produce a row-finite higher-rank graph $\widetilde\Lambda$ with no sources and with $C^*(\widetilde\Lambda)$ Morita equivalent to $C^*(\Lambda)$. We call the higher-rank graph $\widetilde\Lambda$ the {\em Farthing-Webster desourcification} of $\Lambda$ and we will describe its construction in this section. In Section \ref{sec_desourcification} we will generalise results from Sections \ref{sec_liminal_postliminal} and \ref{sec_bded-trace_cts-trace_Fell} using the Farthing-Webster desourcification and that the liminal, postliminal, Fell, bounded- and continuous- trace properties of $C^*$-algebras are preserved under Morita equivalence (Theorem \ref{thm_Morita_equivalence_results}).

For a $k$-graph $\Lambda$, define
\[
W_\Lambda:=\braces{x:\Omega_{k,m}\rightarrow\Lambda: m\in (\NN\cup\braces{\infty})^k, x \text{ is a }k\text{-graph morphism}}.
\]
For every map $x:\Omega_{k,m}\rightarrow\Lambda$ in $W_\Lambda$, define $d(x):=m$. This map $d:W_\Lambda\rightarrow(\NN\cup\{\infty\})^k$ is called the degree map on $W_\Lambda$ since it is a natural generalisation of the degree map $d:\Lambda\rightarrow \NN^{\infty}$. Recall that on page \pageref{def_shift_map} we defined a shift map $\sigma^p:\Lambda^\infty\rightarrow \Lambda^\infty$ for every $p\in\NN^k$. There is a natural extension: for each $p\in\NN^k$ and $x\in W_\Lambda$ with $d(x)\ge p$, define $\sigma^p(x)\in W_\Lambda$ by $\sigma^p(x)(m,n)=x(m+p,n+p)$ for every $(m,n)\in \Omega_{k,d(x)-p}$. 
\begin{definition}[{\cite[Definition~5.10]{Farthing-Muhly-Yeend2005}}]
An element $x\in W_\Lambda$ is a {\em boundary path} if for all $n\in\NN^k$ with $n\le d(x)$ and for all $D\in x(n)\Ff\Ee(\Lambda)$ there exists $m\in\NN^k$ such that $x(n,n+m)\in D$. The set of all boundary paths is denoted by $\partial\Lambda$ and $v\partial\Lambda$ is defined to be $\braces{x\in\partial\Lambda:r(x)=v}$.
\end{definition}
In \cite[Examples~5.11]{Farthing-Muhly-Yeend2005}, Farthing, Muhly and Yeend point out that if $\Lambda$ is a row-finite $k$-graph with no sources, then $\partial\Lambda=\Lambda^\infty$\label{observation_partial_Lambda_equals_Lambda_infty}. We can roughly think of the boundary paths as the paths that are as close as possible to being infinite. Examples \ref{example_boundary_paths_in_Omega_2_32}, \ref{omega_2_infty_2} and \ref{example_annoying_k-graph} describe the boundary paths in various $2$-graphs.

\begin{example}\label{example_boundary_paths_in_Omega_2_32}
Consider the $2$-graph $\Omega_{2,(3,2)}$ from Example \ref{example_Omega_k_m}. Then
\[
\partial\Omega_{2,(3,2)}=\braces{\alpha\in\Omega_{2,(3,2)}:s(\alpha)=(3,2)}.
\]
\begin{proof}
Fix $x\in\partial\Omega_{2,(3,2)}$ and let $\gamma$ be the path in $\Omega_{2,(3,2)}$ with range $r(x)$ and source $(3,2)$. Then $\braces{\gamma}\in x(0)\Ff\Ee(\Omega_{2,(3,2)})$, so since $x$ is a boundary path there exists $m\in\NN^k$ such that $x(0,m)=\gamma$. Since $(3,2)$ does not receive any edges, we must have $d(x)=m$. Then $s(x)=x(m)=s(\gamma)=(3,2)$.

Fix $\alpha\in\Omega_{2,(3,2)}$ with $s(\alpha)=(3,2)$. Fix $n\le d(\alpha)$ and $D\in\alpha(n)\Ff\Ee(\Omega_{2,(3,2)})$. Since $D$ is exhaustive there exists $\beta\in D$ such that $\Lambda^{\min}\big(\sigma^n(\alpha),\beta\big)$ is non-empty. Let $(\mu,\nu)\in\Lambda^{\min}\big(\sigma^n(\alpha),\beta\big)$, so that $\sigma^n(\alpha)\mu=\beta\nu$. We must have $\mu=s(\alpha)$ since $(3,2)$ receives no edges, so $\sigma^n(\alpha)=\beta\nu$. Then $\alpha\big(n,n+d(\beta)\big)=\beta\in D$, and it follows that $\alpha\in\partial\Omega_{2,(3,2)}$.
\end{proof}
\end{example}

\begin{example}\label{omega_2_infty_2}
The $2$-graph $\Omega_{2,(\infty,2)}$ has skeleton
\begin{center}
\begin{tikzpicture}[>=stealth,baseline=(current bounding box.center)] 
\def\cellwidth{5.5};
\clip (-0.9em,-0.4em) rectangle (3*\cellwidth em + 4.1em,6.4em);

\foreach \x in {0,1,2,3} \foreach \y in {0,1,2} \node (x\x y\y) at (\cellwidth*\x em,3*\y em) {$\scriptstyle (\x,\y)$};
\foreach \y in {0,1,2} \node (x4y\y) at (\cellwidth*3 em+2.5em,3*\y em) {};
\foreach \y in {0,1,2} \node (x5y\y) at (\cellwidth*3 em+4.5em,3*\y em) {};

\foreach \x / \z in {0/1,1/2,2/3} \foreach \y in {0,1,2} \draw[black,<-] (x\x y\y) to (x\z y\y);
\foreach \x in {0,1,2,3} \foreach \y / \z in {0/1,1/2} \draw[dashed,black,<-] (x\x y\y) to (x\x y\z);
\foreach \y in {0,1,2} \draw[black,<-] (x3y\y) to (x4y\y);
\foreach \y in {0,1,2} \draw[black,dotted, thick] (x4y\y) to (x5y\y);
\end{tikzpicture}
\end{center}
and $\partial\Omega_{2,(\infty,2)}$ is the set of all $2$-graph morphisms $x^{(p)}:\Omega_{2,(\infty,2-p_2)}\rightarrow\Omega_{2,(\infty,2)}$ for $p\in\NN^2$ with $p_2\le 2$ such that $x^{(p)}(m,n)=(m+p,n+p)$ for all $m\le n\in\NN^k$ with $n_2\le 2-p_2$.
\begin{proof}
Let $S=\braces{x^{(p)}:p\in \NN^2, p_2\le 2}$ and fix $x^{(p)}\in S$. Fix $n\in\NN^2$ with $n\le d(x^{(p)})$ (i.e. with $n_2\le 2-p_2$), fix $D\in x^{(p)}(n)\Ff\Ee(\Omega_{2,(\infty,2)})$ and fix $\alpha\in D$. By examining the skeleton we can see for any $v\in\Omega_{2,(\infty,2)}^0$ and any $l\in\NN^2$ there is at most one path in $v\Omega_{2,(\infty,2)}$ with degree $l$. It follows that $\alpha=x^{(p)}\big(n,n+d(\alpha)\big)$, so $x^{(p)}\in\partial\Omega_{2,(\infty,2)}$.

Now fix $y\in\partial\Omega_{2,(\infty,2)}$ and let $p=y(0)\in\NN^2$. Fix $m,n\in\NN^2$ with $m\le n$ and $n_2\le 2-p_2$. Let $\beta=(m+p,n+p)\in\Omega_{2,(\infty,2)}$. Then $\braces{\beta}\in y(m)\Ff\Ee(\Omega_{2,(\infty,2)})$, so $y\big(m,m+d(\beta)\big)=y(m,n)=\beta=(m+p,n+p)$, and it follows that $y=x^{(p)}\in S$.
\end{proof}
\end{example}

Suppose $\Lambda$ is a row-finite $k$-graph that may have sources. We will now describe the Farthing-Webster desourcification \index{Farthing-Webster desourcification} $\widetilde\Lambda$ of $\Lambda$. Recall from the first paragraph of this section that $\widetilde\Lambda$ was first described by Webster in \cite{Webster2011} based on Farthing's \cite{Farthing2008}. The key property is that $C^*(\widetilde\Lambda)$ is Morita equivalent to $C^*(\Lambda)$. Define a relation $\thickapprox$ on 
\[V_\Lambda:=\braces{(x;m):x\in\partial\Lambda,m\in\NN^k}\]
by: $(x;m)\thickapprox (y;p)$ if and only if
\begin{enumerate}\renewcommand{\theenumi}{V\arabic{enumi}}
\item\label{V1} $x\big(m\wedge d(x)\big) =y\big( p\wedge d(y)\big)$; and
\item\label{V2} $m-m\wedge d(x)=p-p\wedge d(y)$.
\end{enumerate}
Define a relation $\thicksim$ on $P_\Lambda:=\big\{\big(x;(m,n)\big):x\in\partial\Lambda,m\le n\in\NN^k\big\}$ by: $\big(x;(m,n)\big)\thicksim \big(y;(p,q)\big)$ if and only if
\begin{enumerate}\renewcommand{\theenumi}{P\arabic{enumi}}
\item\label{P1} $x\big(m\wedge d(x),n\wedge d(x)\big)=y\big(p\wedge d(y),q\wedge d(y)\big)$;
\item\label{P2} $m-m\wedge d(x)=p-p\wedge d(y)$; and
\item\label{P3} $n-m=q-p$.
\end{enumerate}
Note that $\thickapprox$ and $\thicksim$ are equivalence relations. Let $\widetilde{V_\Lambda}:=V_{\Lambda}/\thickapprox$ and $\widetilde{P_\Lambda}:=P_{\Lambda}/\thicksim$. The class in $\widetilde{V_\Lambda}$ of $(x;m)\in V_\Lambda$ is denoted by $\desourceclass{x;m}$ and the class in $\widetilde{P_\Lambda}$ of $\big(x;(m,n)\big)\in V_\Lambda$ is denoted by $\desourceclass{x;(m,n)}$. Define $\tilde{r},\tilde{s}:\widetilde{P_\Lambda}\rightarrow \widetilde{V_\Lambda}$ by $\tilde{r}\big(\desourceclass{x;(m,n)}\big)=\desourceclass{x;m}$ and $\tilde{s}\big(\desourceclass{x;(m,n)}\big)=\desourceclass{x;n}$. The following proposition will be used to define the composition of elements of $\widetilde{P_\Lambda}$.

\begin{prop}[{\cite[Proposition~4.6]{Webster2011}}]\label{prop_webster_3.6}
Suppose that $\Lambda$ is a row-finite $k$-graph and let $\desourceclass{x;(m,n)}$ and $\desourceclass{y;(p,q)}$ be elements of $\widetilde{P_\Lambda}$ satisfying $\desourceclass{x;n}=\desourceclass{y;p}$. Let $z:=x\big(0,n\wedge d(x)\big)\sigma^{p\wedge d(y)}y$. Then
\begin{enumerate}
\item $z\in\partial\Lambda$;
\item $m\wedge d(x)=m\wedge d(z)$ and $n\wedge d(x)=n\wedge d(z)$; and
\item $x\big(m\wedge d(x),n\wedge d(x)\big)=z\big(m\wedge d(z),n\wedge d(z)\big)$ and $y\big(p\wedge d(y),q\wedge d(y)\big)=z\big(n\wedge d(z),(n+q-p)\wedge d(z)\big)$.
\end{enumerate}
\end{prop}
For $\desourceclass{x;(m,n)},\desourceclass{y;(p,q)}\in \widetilde{P_\Lambda}$ with $\tilde{s}\big(\desourceclass{x;(m,n)}\big)=\tilde{r}\big(\desourceclass{y;(p,q)}\big)$ and for $z\in\partial\Lambda$ as in Proposition \ref{prop_webster_3.6}, the equation
\[
\desourceclass{x;(m,n)}\circ\desourceclass{y;(p,q)}=\desourceclass{z;(m,n+q-p)}
\]
determines a well defined composition $\circ$. Define $\mathrm{id}:\widetilde{V_\Lambda}\rightarrow \widetilde{P_\Lambda}$ by $\mathrm{id}\big(\desourceclass{x;m}\big)=\desourceclass{x;(m,m)}$. Then $(\widetilde{P_\Lambda},\widetilde{V_\Lambda},\tilde{r},\tilde{s},\circ,\mathrm{id})$ is a category by \cite[Lemma~2.19]{Farthing2008}. Define a map $d:\widetilde{P_\Lambda}\rightarrow \NN^k$ by $d\big(\desourceclass{x;(m,n)}\big)=n-m$. Then $\widetilde{\Lambda}:=(\widetilde{P_\Lambda},\widetilde{V_\Lambda},\tilde{r},\tilde{s},\circ,\mathrm{id},d)$ is a $k$-graph without sources by \cite[Theorem~2.22]{Farthing2008}. We call $\widetilde\Lambda$ the {\em Farthing-Webster desourcification} (or just the desourcification) of $\Lambda$ and usually simplify notation by writing $r$ for $\tilde{r}$, $s$ for $\tilde{s}$ and omitting the $\circ$ and $\mathrm{id}$ notation. \label{describing_morita_equivalence_results}Webster in \cite[Theorem~6.3]{Webster2011} shows that if $\Lambda$ is row-finite, then $\widetilde\Lambda$ is row-finite and $C^*(\widetilde{\Lambda})$ is Morita equivalent to $C^*(\Lambda)$. The Morita equivalence of $C^*$-algebras will be particularly useful to us since Zettl in \cite{Zettl1982} and an Huef, Raeburn and Williams in \cite{anHuef-Raeburn-Williams2007} show that the liminal, postliminal, continous- and bounded-trace, and Fell properties are preserved by Morita equivalence.

\begin{prop}[{\cite[Proposition~4.13]{Webster2011}}]
Suppose that $\Lambda$ is a row-finite $k$-graph, and that $\alpha\in\Lambda$. Then $s(\alpha)\partial\Lambda\ne\emptyset$. If $x,y\in s(\alpha)\partial\Lambda$, then $\alpha x,\alpha y\in\partial\Lambda$ and $\desourceclass{\alpha x;(0,d(\alpha))}=\desourceclass{\alpha y;(0,d(\alpha))}$. There is an injective $k$-graph morphism $\iota:\Lambda\rightarrow\widetilde\Lambda$ such that $\iota(\alpha)=\bigdesourceclass{\alpha x;\big(0,d(\alpha)\big)}$ for any $\alpha\in\Lambda$ and $x\in s(\alpha)\partial\Lambda$.
\end{prop}

Define $\kappa:\partial\Lambda\rightarrow\iota(\Lambda^0)\widetilde{\Lambda}^\infty$ by $\kappa(x)(m,n)=\desourceclass{x;(m,n)}$ for every $x\in\partial\Lambda$ and $m,n\in\NN^k$ with $m\le n$. Since $\widetilde\Lambda$ has no sources, we know $\partial\widetilde\Lambda=\widetilde\Lambda^\infty$, so $\kappa$ maps the boundary paths in $\Lambda$ into the boundary paths $\widetilde\Lambda$. The following results show that every path in $\partial\widetilde\Lambda$ is shift equivalent to a path in $\kappa(\partial\Lambda)$ and that this map preserves notions of `shift equivalence' and `frequently divertable'.

\begin{lemma}\label{lemma_kappa_bijection}
Suppose $\Lambda$ is a row-finite $k$-graph. Then $\kappa:\partial\Lambda\rightarrow\iota(\Lambda^0)\widetilde{\Lambda}^\infty$ is a bijection.
\begin{proof}
Suppose $\kappa(x)=\kappa(y)$ for some $x,y\in\partial\Lambda$ and fix $n\in\NN^k$. Then $\kappa(x)(0,n)=\kappa(y)(0,n)$, and so $\desourceclass{x;(0,n)}=\desourceclass{y;(0,n)}$. By \eqref{P1} we have $x\big(0,n\wedge d(x)\big)=y\big(0,n\wedge d(y)\big)$ and it follows that $x=y$ since $n$ was chosen arbitrarily.

Now fix $a\in \iota(\Lambda^0)\widetilde{\Lambda}^\infty$. In \cite[p.~170]{Webster2011}, Webster describes a map $\pi:\iota(\Lambda^0)\widetilde{\Lambda}^\infty\rightarrow\iota(\partial\Lambda)$. Since $\pi$ maps into $\iota(\partial\Lambda)$, there exists $x\in\partial\Lambda$ such that $\pi(a)=\iota(x)$. Now \cite[Lemma~5.3]{Webster2011} shows that $a(0,n)=\desourceclass{x;(0,n)}$ for every $n\in\NN^k$, so $a=\kappa(x)$, establishing the surjectivity of $\kappa$.
\end{proof}
\end{lemma}

\begin{lemma}\label{lemma_boundary_path_associated_to_infinite_path}
Suppose $\Lambda$ is a row-finite $k$-graph. If $a\in\widetilde{\Lambda}^\infty$ then there exists $x\in\partial\Lambda$ and $m\in\NN^k$ such that $a=\sigma^m\big(\kappa(x)\big)$.
\begin{proof}
Fix $a\in\widetilde{\Lambda}^\infty$. By the construction of $\widetilde{\Lambda}$ there exists $y\in\partial\Lambda$ and $m\in\NN^k$ such that $r(a)=\desourceclass{y;m}$. Then $\desourceclass{y;(0,m)}$ is a path in $\widetilde{\Lambda}$ with source $r(a)$ and range in $\iota(\Lambda^0)$, so $\desourceclass{y;(0,m)}a\in\iota(\Lambda^0)\widetilde{\Lambda}^\infty$. Since $\kappa$ is a bijection onto $\iota(\Lambda^0)\widetilde\Lambda^\infty$ (Lemma \ref{lemma_kappa_bijection}) there exists $x\in\partial\Lambda$ such that $\kappa(x)=\desourceclass{y;(0,m)}a$. We conclude that $\sigma^m\big(\kappa(x)\big)=a$ since $d\big(\desourceclass{y,(0,m)}\big)=m$.
\end{proof}
\end{lemma}

\begin{lemma}\label{lemma_sigma_kappa}
Suppose $\Lambda$ is a row-finite $k$-graph, $x\in\partial\Lambda$ and $n\in\NN^k$ satisfies $n\le d(x)$. Then $\sigma^n\big(\kappa(x)\big)=\kappa\big(\sigma^n(x)\big)$.
\begin{proof}
Fix $m\in\NN^k$. It suffices to show that $\sigma^n\big(\kappa(x)\big)(0,m)=\kappa\big(\sigma^n(x)\big)(0,m)$. Now \[\sigma^n\big(\kappa(x)\big)(0,m)=\kappa(x)(n,n+m)=\desourceclass{x;(n,n+m)}\] and $\kappa\big(\sigma^n(x)\big)(0,m)=\desourceclass{\sigma^n(x);(0,m)}$, so we need to show that $\desourceclass{x;(n,n+m)}=\desourceclass{\sigma^n(x);(0,m)}$. For \eqref{P1}, consider
{\allowdisplaybreaks
\begin{align*}
\sigma^n(x)\Big(0\wedge d\big(\sigma^n(x)\big),m\wedge d\big(\sigma^n(x)\big)\Big) &=\sigma^n(x)\Big(0,m\wedge\big(d(x)-n\big)\Big)\\
&=x\Big(n,n+m\wedge\big(d(x)-n\big)\Big)\\
&=x\big(n\wedge d(x),(n+m)\wedge d(x)\big).
\end{align*}}
Axiom \eqref{P2} follows from $n\le d(x)$ and axiom \eqref{P3} is immediate, completing the proof.
\end{proof}
\end{lemma}

Suppose $\Lambda$ is a $k$-graph. For $x,y\in \partial\Lambda$ and $n\in\ZZ^k$, write $x\sim_n y$\index{$\sim_n$} if there exists $m\in\NN^k$ with $m\le d(x)\wedge \big(d(y)+n\big)$ such that $\sigma^m(x)=\sigma^{m-n}(y)$.
\begin{lemma}\label{lemma_shift_equivalence2}
Suppose $\Lambda$ is a $k$-graph. The relations $\sim_n$ form an equivalence relation on $\partial\Lambda$ in that $x\sim_0 x$, $x\sim_n y\implies y\sim_{-n} x$ and $x\sim_n y, y\sim_l z\implies x\sim_{n+l} z$ for any $x,y,z\in\partial\Lambda$ and $n,l\in\ZZ^k$.
\begin{proof}
The assertion that $x\sim_0 x$ for any $x\in\partial\Lambda$ is trivial. Suppose $x\sim_n y$ and $y\sim_l z$ for some $x,y,z\in \partial\Lambda$ and $n,l\in\NN^k$. Then there exists $p\in\NN^k$ with $p\le d(x)\wedge \big(d(y)+n\big)$ such that $\sigma^p(x)=\sigma^{p-n}(y)$. Then $p-n\le \big(d(x)-n\big)\wedge \big(d(y)+n-n\big)=d(y)\wedge\big(d(x)-n\big)$ and $\sigma^{p-n}(y)=\sigma^p(x)=\sigma^{(p-n)+n}(x)$, so $y\sim_{-n} x$.

To see that $x\sim_{n+l}z$, first observe that there exists $q\in\NN^k$ with $q\le d(y)\wedge \big(d(z)+l\big)$ such that $\sigma^q(y)=\sigma^{q-l}(z)$. Let $t=p\vee (q+n)$. Since $d(x)=d(y)+n$, $d(y)=d(z)+l$, $p\le d(x)$ and $q\le d(y)$, we have $t\le d(x)\wedge\big(d(z)+n+l\big)$. Now $p\le t\le d(x)$, so $\sigma^t(x)=\sigma^{t-n}(y)$. Similarly $q\le t-n\le d(y)$ gives us $\sigma^{t-n}(y)=\sigma^{t-n-l}(z)$. Thus $\sigma^t(x)=\sigma^{t-(n+l)}(z)$ and it follows that $x\sim_{n+l}z$.
\end{proof}
\end{lemma}

The next lemma shows how shift equivalence passes from a row-finite $k$-graph to its desourcification via the map $\kappa$.
\begin{lemma}\label{lemma_desource_shift_equivalence_preserved}
Suppose $\Lambda$ is a row-finite $k$-graph, $x,y\in\partial\Lambda$ and $n\in \ZZ^k$. Then $\kappa(x)\sim_n\kappa(y)$ if and only if $x\sim_n y$.
\begin{proof}
Suppose $x\sim_n y$. Then there exists $m\in\NN^k$ with $m\le d(x)\wedge \big(d(y)+n\big)$ such that $\sigma^m(x)=\sigma^{m-n}(y)$ and so $\kappa\big(\sigma^m(x)\big)=\kappa\big(\sigma^{m-n}(y)\big)$. It follows by Lemma \ref{lemma_sigma_kappa} that $\sigma^m\big(\kappa(x)\big)=\sigma^{m-n}\big(\kappa(y)\big)$, so $\kappa(x)\sim_n\kappa(y)$.

We will now prove the converse. Suppose $\kappa(x)\sim_n\kappa(y)$ so that there exists $m\in\NN^k$ such that $\sigma^m\big(\kappa(x)\big)=\sigma^{m-n}\big(\kappa(y)\big)$. Then $\kappa(x)(m)=\kappa(y)(m-n)$ so \eqref{V2} tells us that $m-m\wedge d(x)=(m-n)-(m-n)\wedge d(y)$. Let $p=m-m\wedge d(x)$. We now claim that $\big(x;(m-p,m)\big)\thicksim \big(y;(m-n-p,m-n)\big)$. To show this, we need to show that \eqref{P1}, \eqref{P2} and \eqref{P3} are satisfied. To see \eqref{P1}, note that
{\allowdisplaybreaks\begin{align*}
x\big((m-p)\wedge d(x),m\wedge d(x)\big)&=x\Big(\big(m\wedge d(x)\big)\wedge d(x),m\wedge d(x)\Big)\\
&=x\big(m\wedge d(x),m\wedge d(x)\big)\\
&=y\big((m-n)\wedge d(y),(m-n)\wedge d(y)\big)\quad\text{(by \eqref{V1}})\\
&=y\Big(\big((m-n)\wedge d(y)\big)\wedge d(y), (m-n)\wedge d(y)\Big)\\
&=y\big((m-n-p)\wedge d(y), (m-n)\wedge d(y)\big).
\end{align*}}
To see \eqref{P2}, note that
\[
 (m-p)-(m-p)\wedge d(x)=(m-p)-\big(m\wedge d(x)\big)\wedge d(x)=m-p-m\wedge d(x)=0
\]
and similarly $(m-n-p)-(m-n-p)\wedge d(y)=0$. Axiom \eqref{P3} follows immediately.

We now know that $\big(x;(m-p,m)\big)\thicksim \big(y;(m-n-p,m-n)\big)$, so $\kappa(x)(m-p,m)=\kappa(y)(m-n-p,m-n)$. Since $\sigma^m\big(\kappa(x)\big)=\sigma^{m-n}\big(\kappa(y)\big)$ and $\kappa(x)(m-p,m)=\kappa(y)(m-n-p,m-n)$, we must have $\sigma^{m-p}\big(\kappa(x)\big)=\sigma^{m-n-p}\big(\kappa(y)\big)$. Since $m-p=m\wedge d(x)\le d(x)$ and $m-n-p=(m-n)\wedge d(y)\le d(y)$, we can apply Lemma \ref{lemma_sigma_kappa} to see that
\[
 \kappa\big(\sigma^{m-p}(x)\big)=\sigma^{m-p}\big(\kappa(x)\big)=\sigma^{m-n-p}\big(\kappa(y)\big)=\kappa\big(\sigma^{m-n-p}(y)\big).
\]
Since $\kappa$ is injective by Lemma \ref{lemma_kappa_bijection}, it follows that $\sigma^{m-p}(x)=\sigma^{m-n-p}(y)$, so $x\sim_n y$, as required.
\end{proof}
\end{lemma}

\begin{definition}[{\cite[Definition~3.9]{rsy2003}}]\label{def_locally_convex}\index{locally convex}
A $k$-graph $\Lambda$ is {\em locally convex} if, for all $v\in\Lambda^0$, $i,j\in\braces{1,\ldots,k}$ with $i\ne j$, $\alpha\in v\Lambda^{e_i}$ and $\beta\in v\Lambda^{e_j}$, the sets $s(\alpha)\Lambda^{e_j}$ and $s(\beta)\Lambda^{e_i}$ are non-empty.
\end{definition}
In \cite[p.~159]{Webster2011} Webster points out that $\Lambda$ being locally convex equates to its skeleton not containing any subgraph resembling
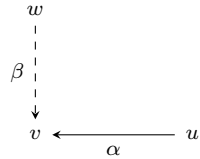
\begin{center}
\begin{tikzpicture}[>=stealth,baseline=(current bounding box.center)] 
\clip (-0.9em,-0.8em) rectangle (5.3 em,4.2em);
\node (salpha) at (5em, 0em) {$\scriptstyle u$};
\node (sbeta) at (0em, 4em) {$\scriptstyle w$};
\node (v) at (0em,0em) {$\scriptstyle v$};
\draw[black,->] (salpha) to node[below] {$\scriptstyle \alpha$} (v);
\draw[black,dashed,->] (sbeta) to node[left] {$\scriptstyle \beta$} (v);
\end{tikzpicture}
\end{center}
where $w$ receives no solid edges and $u$ receives no dashed edges. The $2$-graph $\Omega_{2,(3,2)}$ from Example \ref{example_Omega_k_m} is locally convex whereas the $2$-graphs from Examples \ref{example_not_locally_convex_with_same_boundary_paths} and \ref{example_annoying_k-graph} are not locally convex.
\begin{remark}\label{remark_every_1-graph_locally_convex}
It follows immediately from Definition \ref{def_locally_convex} that every $1$-graph is locally convex.
\end{remark}
\begin{definition}[{\cite[Definition~2.8]{rsy2004}}]
\label{def_Lambda_leinfty}\notationindex{LAMBDALEINFTY@$\Lambda^{\le\infty}$}
For a $k$-graph $\Lambda$, define $\Lambda^{\le\infty}$ be the set of all $x\in W_\Lambda$ where there exists $m\in\NN^k$ with $m\le d(x)$ such that for every $n\in\NN^k$ with $m\le n\le d(x)$ and every $1\le i\le k$ with $n_i=d(x)_i$, the set $x(n)\Lambda^{e_i}$ is empty. 
\end{definition}
The set $\Lambda^{\le\infty}$ from Definition \ref{def_Lambda_leinfty} is a slight modification of a smaller set with the same denotation first defined by Raeburn, Sims and Yeend in \cite[Definition~3.14]{rsy2003}. This smaller set was defined to be the set of all $x\in W_\Lambda$ such that, for every $n\in\NN^k$ with $n\le d(x)$ and every $1\le i\le k$ with $n_i=d(x)_i$, the set $x(n)\Lambda^{e_i}$ is empty. These two definitions of $\Lambda^{\le\infty}$ coincide when $\Lambda$ is locally convex. Given that the main results from \cite{rsy2004} require the assumption that $\Lambda$ is locally convex, Raeburn, Sims and Yeend in \cite{rsy2004} used this new definition of $\Lambda^{\le\infty}$ as an alternative definition to that found in \cite{rsy2003}. It turns out that using the \cite{rsy2004} definition of $\Lambda^{\le\infty}$ strengthens some of the main theorems in this thesis. The two definitions of $\Lambda^{\le\infty}$ do not coincide in the $2$-graph $\Lambda$ from Example \ref{example_not_locally_convex_with_same_boundary_paths}.
\begin{example}
Consider the $2$-graph $\Omega_{2,(3,2)}$ from Example \ref{example_Omega_k_m}. Then
\[
\Omega_{2,(3,2)}^{\le\infty}=\partial\Omega_{2,(3,2)}=\braces{\alpha\in\Omega_{2,(3,2)}:s(\alpha)=(3,2)}.
\]
\begin{proof}
We saw that $\partial\Omega_{2,(3,2)}=\braces{\alpha\in\Omega_{2,(3,2)}:s(\alpha)=(3,2)}$ in Example \ref{example_boundary_paths_in_Omega_2_32}. Fix $x\in\Omega_{2,(3,2)}^{\le\infty}$ and note that $x\big(d(x)\big)\Omega_{2,(3,2)}^{e_i}=s(x)\Omega_{2,(3,2)}^{e_i}$ is empty for $i=1,2$. In other words, $s(x)$ does not receive any edges. Since the only vertex in $\Omega_{2,(3,2)}$ that does not receive any edges is $(3,2)$, it follows that $s(x)=(3,2)$.

Now suppose that $\alpha\in\Omega_{2,(3,2)}$ with $s(\alpha)=(3,2)$. Since $(3,2)$ does not receive any edges,
\[
(3,2)\Omega_{2,(3,2)}^{e_i}=s(\alpha)\Omega_{2,(3,2)}^{e_i}=\alpha\big(d(\alpha)\big)\Omega_{2,(3,2)}^{e_i}
\]
is empty for $i=1,2$. We can now take $m=d(\alpha)$ to see that $\alpha\in\Omega_{2,(3,2)}^{\le\infty}$.
\end{proof}
\end{example}

The following proposition by Webster shows that it is not a coincidence that $\partial\Omega_{2,(3,2)}=\Omega_{2,(3,2)}^{\le\infty}$. Indeed, $\partial\Lambda$ was chosen as a kind of generalisation of $\Lambda^{\le\infty}$ and in \cite{rsy2003} Raeburn, Sims and Yeend refer to elements of $\Lambda^{\le\infty}$ as `boundary paths', although here we reserve this term for elements of $\partial\Lambda$. 

\begin{prop}[{\cite[Proposition~2.12]{Webster2011}}]\label{prop_Lambda_leinfty_equals_partial_lambda}
If $\Lambda$ is a row-finite $k$-graph, then $\Lambda^{\le\infty}\subset\partial\Lambda$. If in addition $\Lambda$ is locally convex, then $\Lambda^{\le\infty}=\partial\Lambda$.
\end{prop}
It is possible to have $\Lambda^{\le\infty}=\partial\Lambda$ without $\Lambda$ being locally convex:
\begin{example}\label{example_not_locally_convex_with_same_boundary_paths}
If $\Lambda$ is the $2$-graph with the following skeleton, then $\Lambda$ is not locally convex and $\partial\Lambda=\Lambda^{\le\infty}$.
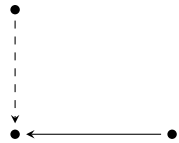
\begin{center}
\begin{tikzpicture}[>=stealth,baseline=(current bounding box.center)] 
\clip (-0.9em,-0.8em) rectangle (5.3 em,4.2em);
\fill[black] (salpha) circle (0.15em);
\fill[black] (sbeta) circle (0.15em);
\fill[black] (v) circle (0.15em);
\node (salpha) at (5em, 0em) {};
\node (sbeta) at (0em, 4em) {};
\node (v) at (0em,0em) {};
\draw[black,->] (salpha) to (v);
\draw[black,dashed,->] (sbeta) to (v);
\end{tikzpicture}
\end{center}
\end{example}

The condition $\partial\Lambda=\Lambda^{\le\infty}$ turns out to be very useful even though the stronger (in the row-finite context) locally convex condition is often easier to check. The next example first appeared in Robertson's honours thesis \cite{RobertsonHonours} and is the standard example of a path that is in $\partial\Lambda$ but not in $\Lambda^{\le\infty}$.
\begin{example}\label{example_annoying_k-graph}
Let $\Lambda$ be the $2$-graph with the following skeleton.
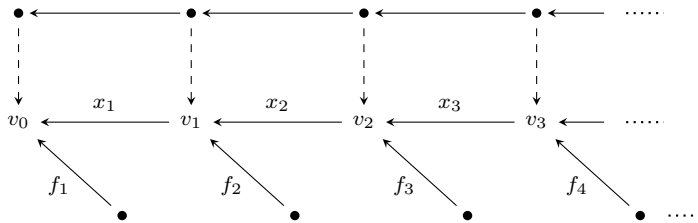
\begin{center}
\begin{tikzpicture}[>=stealth,baseline=(current bounding box.center)] 
\def\cellwidth{5.5};
\foreach \x in {0,1,2,3} \node (x\x y1) at (\cellwidth*\x em,3.5em) {};
\foreach \x in {0,1,2,3} \fill[black] (x\x y1) circle (0.15em);
\foreach \x in {0,1,2,3} \node (x\x y0) at (\cellwidth*\x em,0em) {$\scriptstyle v_\x$};

\foreach \x / \z in {0/1,1/2,2/3,3/4} {
\node (a\x) at (\cellwidth*\x em+0.6*\cellwidth em,-3em) {};
\fill[black] (a\x) circle (0.15em);
\draw[black,<-] (x\x y0) to node[below left=-0.3em] {$\scriptstyle f_\z$} (a\x);
}

\foreach \y in {0,1} \node (x4y\y) at (\cellwidth*3 em+2.5em,3.5*\y em) {};
\foreach \y in {0,1} \node (x5y\y) at (\cellwidth*3 em+4.4em,3.5*\y em) {};

\foreach \x / \z in {0/1,1/2,2/3} \draw[black,<-] (x\x y0) to node[above] {$\scriptstyle x_\z$} (x\z y0);
\foreach \x / \z in {0/1,1/2,2/3} \draw[black,<-] (x\x y1) to (x\z y1);

\foreach \x in {0,1,2,3} \foreach \y / \z in {0/1} \draw[dashed,black,<-] (x\x y\y) to (x\x y\z);
\foreach \y in {0,1} \draw[black,<-] (x3y\y) to (x4y\y);
\foreach \y in {0,1} \draw[black,dotted, thick] (x4y\y) to (x5y\y);

\node (a4) at (3.7*\cellwidth em,-3em) {};
\node (a5) at (3.7*\cellwidth em + 1.9em,-3em) {};
\draw[black,dotted, thick] (a4) to (a5);
\end{tikzpicture}
\end{center}
The path $x=x_1x_2x_3\cdots$ is in $\partial\Lambda$ but not in $\Lambda^{\le\infty}$.
\begin{proof}
First suppose that the solid edges have degree $(1,0)$ and the dashed edges have degree $(0,1)$. Observe that $x\notin\Lambda^{\le\infty}$: for every $m\in\NN^k$ with $m\le d(x)=(\infty,0)$, we have $x(m)\Lambda^{e_2}=v_{m_1}\Lambda^{e_2}$, which is non-empty since $v_{m_1}$ receives an edge of degree $(0,1)$.

To see that $x\in\partial\Lambda$, fix $n\in\NN^k$ with $n\le d(x)$ and fix $D\in x(n)\Ff\Ee(\Lambda)$. Since $n\le d(x)$ and $d(x)=(\infty,0)$, we have $n_2=0$, so $D\in v_{n_1}\Ff\Ee(\Lambda)$. Suppose there does not exist $m\in\NN^k$ such that $x(n,n+m)\in D$. For each integer $p> n_1$, let $\alpha^{(p)}=x_{n_1}x_{n_1+1}\cdots x_{p-1}f_p$. Since $D$ is exhaustive, for each $\alpha^{(p)}$ there exists $\beta^{(p)}\in D$ such that $\Lambda^{\min}(\alpha^{(p)},\beta^{(p)})$ is non-empty. 
Now let $(\mu,\nu)\in\Lambda^{\min}(\alpha^{(p)},\beta^{(p)})$ for some $p$, so that $\alpha^{(p)}\mu=\beta^{(p)}\nu$. Since $s(\alpha^{(p)})=s(f_p)$ receives no edges, we have $\mu=s(\alpha^{(p)})$, and so $\alpha^{(p)}=\beta^{(p)}\nu$. In order for $\beta^{(p)}$ to not be of the form $x(n,n+m)$ for some $m$, we must also have $\nu=s(\beta^{(p)})$, so that $\alpha^{(p)}=\beta^{(p)}$. It follows that $\alpha^{(p)}$ is in $D$ for every $p$, so $D$ is not finite, contradicting $D\in\Ff\Ee(\Lambda)$. Thus our assumption is incorrect, so $x(n,n+m)\in D$ for some $m\in\NN^k$, and $x\in\partial\Lambda$.
\end{proof}
\end{example}
The desourcification of the $2$-graph $\Lambda$ from Example \ref{example_annoying_k-graph} will be described in Example \ref{example_annoying_k-graph_is_liminal}. In Example \ref{examples_cts-trace} it will be shown that $C^*(\Lambda)$ has continuous trace.

\needspace{4\baselineskip}
\begin{lemma}\label{lemma_property_of_old_boundary_path}
Suppose $\Lambda$ is a row-finite $k$-graph. If $x\in\Lambda^{\le\infty}$, there exists $n\in\NN^k$ with $n\le d(x)$ such that for any $m\in\NN^k$ with $m\ge n$, $x\big(m\wedge d(x)\big)\Lambda^{e_i}=\emptyset$ for every $i$ with $d(x)_i<\infty$. If $y\in x\big(m\wedge d(x)\big)\partial\Lambda$, then
\[
\kappa(x)\big(m\wedge d(x), m\big)=\kappa(y)\big(0,m-m\wedge d(x)\big).
\]
\begin{proof}
Since $x\in\Lambda^{\le\infty}$, there exists $n\in\NN^k$ with $n\le d(x)$ such that for any $l\in\NN^k$ and $1\le i\le k$ such that $n\le l\le d(x)$ and $l_i=d(x)_i$, the set $x(l)\Lambda^{e_i}$ is empty. Increase $n$ so that $n_i=d(x)_i$ for every $i$ with $d(x)_i<\infty$. Then for every $m\in\NN^k$ with $m\ge n$, the set $x\big(m\wedge d(x)\big)\Lambda^{e_i}$ is empty for every $i$ with $d(x)_i<\infty$.

Fix $m\in\NN^k$ such that $m\ge n$ and suppose $y\in x\big(m\wedge d(x)\big)\partial\Lambda$. Let $z=x\big(0,m\wedge d(x)\big)y$. In order to show that $\kappa(x)(0,m)=\kappa(z)(0,m)$, it suffices to show that $x\big(0,m\wedge d(x)\big)=z\big(0,m\wedge d(z)\big)$. By the definition of $z$, we know $z\big(0,m\wedge d(x)\big)=x\big(0,m\wedge d(x)\big)$, so it will suffice to show that $m\wedge d(z)=m\wedge d(x)$.

Suppose $m\wedge d(z)\ne m\wedge d(x)$. We know by the construction of $z$ that $d(z)\ge m\wedge d(x)$, so there exists $i$ such that
\begin{equation}\label{frequently_divertable_eqn1}
m\wedge d(z)\ge m\wedge d(x)+e_i.
\end{equation}
Since $\big(m\wedge d(x)\big)_i$ equals $m_i$ if $d(x)_i=\infty$ and equals $d(x)_i$ if $d(x)_i<\infty$, equation \eqref{frequently_divertable_eqn1} shows that we must have $d(x)_i<\infty$. Now
\[
 z\big(m\wedge d(x),m\wedge d(x)+e_i\big)\in x\big(m\wedge d(x)\big)\Lambda^{e_i},
\]
but no such path can exist since we know $m\ge n$ and $\big(m\wedge d(x)\big)_j=d(x)_j$ implies $x\big(m\wedge d(x)\big)\Lambda^{e_j}=\emptyset$.

This contradiction shows that $n\wedge d(z)=n\wedge d(x)$, so it follows from the construction of $z$ that $z\big(0,d(z)\big)=x\big(0,d(x)\big)$, and we can use this to see that $\kappa(z)(0,n)=\kappa(x)(0,n)$.
Now $y=\sigma^{m\wedge d(x)}(z)$, so 
\begin{align*}
\kappa(y)\big(0,m-m\wedge d(x)\big)&=\kappa\big(\sigma^{m\wedge d(x)}(z)\big)\big(0,m-m\wedge d(x)\big)\\
&=\sigma^{m\wedge d(x)}\big(\kappa(z)\big)\big(0,m-m\wedge d(x)\big)\quad\text{(by Lemma \ref{lemma_sigma_kappa})}\\
&=\kappa(z)\big(m\wedge d(x),m\big)=\kappa(x)\big(m\wedge d(x),m\big).\qedhere
\end{align*}
\end{proof}
\end{lemma}

\begin{definition}\label{def_frequently_divertable_boundary_paths}\index{frequently divertable!boundary paths in a $k$-graph}
Suppose $\Lambda$ is a $k$-graph. For $x,y\in \partial\Lambda$, we say $x$ is {\em frequently divertable} to $[y]$ if for every $n\in\NN^k$ with $n\le d(x)$ there is a path in $x(n)\partial\Lambda$ that is shift equivalent to $y$.
\end{definition}
The frequently divertable property is preserved by $\kappa$:
\needspace{4\baselineskip}
\begin{lemma}\label{lemma_frequently_divertable_preservation}
Suppose $\Lambda$ is a row-finite $k$-graph and $x,y\in\partial\Lambda$. Consider
\begin{enumerate}
\item\label{lemma_frequently_divertable_preservation1} $x$ is frequently divertable to $[y]$; and
\item\label{lemma_frequently_divertable_preservation2} $\kappa(x)$ is frequently divertable to $[\kappa(y)]$.
\end{enumerate}
Then \eqref{lemma_frequently_divertable_preservation2} implies \eqref{lemma_frequently_divertable_preservation1} and, if in addition $x\in\Lambda^{\le\infty}$, then \eqref{lemma_frequently_divertable_preservation1} implies \eqref{lemma_frequently_divertable_preservation2}.
\begin{proof}
\eqref{lemma_frequently_divertable_preservation2}$\implies$\eqref{lemma_frequently_divertable_preservation1}. Suppose \eqref{lemma_frequently_divertable_preservation2} and fix $n\le d(x)$. Then there exists $a\in \kappa(x)(n)\widetilde\Lambda^\infty$ such that $a$ is shift equivalent to $\kappa(y)$. Then $r(a)\in\iota(\Lambda^0)$ since $r(a)=\kappa(x)(n)$ and $n\le d(x)$. Since $\kappa$ is a bijection onto $\iota(\Lambda^0)\widetilde\Lambda^\infty$ (Lemma \ref{lemma_kappa_bijection}), there exists $z\in\partial\Lambda$ such that $a=\kappa(z)$. Now
$r\big(\kappa(z)\big)=r(a)=\kappa(x)(n)$ implies $\kappa(z)(0)=\kappa(x)(n)$ and since $\kappa$ is injective, $r(z)=z(0)=x(n)$. Thus $z\in x(n)\partial\Lambda$. Since $\kappa(z)\in [\kappa(y)]$, $\kappa(z)\sim\kappa(y)$ and so we know $z\sim y$ by Lemma \ref{lemma_desource_shift_equivalence_preserved}. Thus $z\in x(n)\partial\Lambda\cap[y]$ and \eqref{lemma_frequently_divertable_preservation1} follows since $n$ was chosen arbitrarily.

\eqref{lemma_frequently_divertable_preservation1}$\implies$\eqref{lemma_frequently_divertable_preservation2} assuming $x\in\Lambda^{\le\infty}$. Suppose \eqref{lemma_frequently_divertable_preservation1}. Since $x\in\Lambda^{\le\infty}$, we can let $n$ be the element of $\NN^k$ as in Lemma \ref{lemma_property_of_old_boundary_path}. Fix $m\in\NN^k$ so that $m\ge n$. Since  $x$ is frequently divertable to $[y]$, there exists $z\in x\big(m\wedge d(x)\big)\partial\Lambda\cap [y]$. Now
\[
 \kappa(z)\big(0,m-m\wedge d(x)\big)=\kappa(x)\big(m\wedge d(x),m\big),
\]
so $\sigma^{m-m\wedge d(x)}\big(\kappa(z)\big)\in \kappa(x)(m)\widetilde{\Lambda}^\infty$. Since $z\sim y$, $\sigma^{m-m\wedge d(x)}\big(\kappa(z)\big)$ is shift equivalent to $\kappa(y)$ by Lemma \ref{lemma_desource_shift_equivalence_preserved}. Condition \eqref{lemma_frequently_divertable_preservation2} follows since $m$ was chosen to be an arbitrary element of $\NN^k$ with $m\ge n$.
\end{proof}
\end{lemma}

\begin{remark}\label{remark_annoying_example_nearly_breaks_lemma}
Statement \eqref{lemma_frequently_divertable_preservation1} does not imply \eqref{lemma_frequently_divertable_preservation2} in general. To see this, consider Example \ref{example_annoying_k-graph}. Let $y$ be the path of degree $(\infty,1)$ with range $v_0$. Then $x$ and $y$ are both boundary paths and $x$ is frequently divertable to $[y]$. But by examining the desourcification presented in Figure \ref{figure_desourcification} it can be seen that $\kappa(x)$ is not frequently divertable to $[\kappa(y)]$ since there is no path in $\kappa(x)(0,1)\widetilde{\Lambda}^\infty$ which is shift equivalent to $\kappa(y)$.
\end{remark}

\section{Generalising results with the Farthing-Webster desourcification}\label{sec_desourcification}
Recall from the previous section that if $\widetilde\Lambda$ is the Farthing-Webster desourcification of a row-finite $k$-graph $\Lambda$, then $C^*(\widetilde\Lambda)$ is Morita equivalent to $C^*(\Lambda)$. Since the liminal, postliminal, Fell, bounded- and continuous- trace properties of $C^*$-algebras are preserved under Morita equivalence (Theorem \ref{thm_Morita_equivalence_results}), the theorems in Sections \ref{sec_liminal_postliminal} and \ref{sec_bded-trace_cts-trace_Fell}, which require the $k$-graph to have no sources, can be applied to $\widetilde\Lambda$ rather than $\Lambda$ in order to determine which of the the liminal, postliminal, Fell, bounded- and continuous-trace properties are satisfied by $C^*(\Lambda)$ (the latter three properties also requiring that the path groupoid associated to $\widetilde\Lambda$ is principal). Rather than having to consider the desourcification of each $k$-graph to apply these theorems, the theorems in this section provide conditions on $\Lambda$ that are necessary and sufficient for the conditions in the theorems in Sections \ref{sec_liminal_postliminal} and \ref{sec_bded-trace_cts-trace_Fell} to hold on $\widetilde\Lambda$, thus providing characterisations for row-finite $k$-graphs that may have sources.

\begin{theorem}\label{k-graph_liminal_thm2}
Suppose $\Lambda$ is a row-finite $k$-graph and consider the statements
\begin{enumerate}
 \item\label{k-graph_liminal_thm2_1} for every $x\in\partial\Lambda$, every path in $\partial\Lambda$ that is frequently divertable to $[x]$ is shift equivalent to $x$; and
\item\label{k-graph_liminal_thm2_2} $C^*(\Lambda)$ is liminal.
\end{enumerate}
Then \eqref{k-graph_liminal_thm2_1} implies \eqref{k-graph_liminal_thm2_2} and, if in addition $\partial\Lambda=\Lambda^{\le\infty}$, then \eqref{k-graph_liminal_thm2_2} implies \eqref{k-graph_liminal_thm2_1}.
\begin{proof}
Suppose \eqref{k-graph_liminal_thm2_1}. We begin by showing that $\widetilde{\Lambda}$ satisfies condition \eqref{k-graph_liminal_thm1_2} from Theorem \ref{k-graph_liminal_thm1}. Fix $a\in\widetilde{\Lambda}^\infty$ and suppose that $b\in\widetilde{\Lambda}^\infty$ is frequently divertable to $[a]$. By Lemma \ref{lemma_boundary_path_associated_to_infinite_path} there exists $x,y\in\partial\Lambda$ and $n,m\in\NN^k$ such that $a=\sigma^n\big(\kappa(x)\big)$ and $b=\sigma^m\big(\kappa(y)\big)$. Then $\kappa(y)$ is frequently divertable to $[\kappa(x)]$ and Lemma \ref{lemma_frequently_divertable_preservation} shows that $y$ is frequently divertable to $[x]$. Then $y$ is shift equivalent to $x$ by \eqref{k-graph_liminal_thm2_1}. Lemma \ref{lemma_desource_shift_equivalence_preserved} now shows us that $\kappa(y)$ is shift equivalent to $\kappa(x)$ and it follows that $b$ is shift equivalent to $a$. We have thus shown that $\widetilde{\Lambda}$ satisfies condition \eqref{k-graph_liminal_thm1_2} from Theorem \ref{k-graph_liminal_thm1}, so $C^*(\widetilde{\Lambda})$ is liminal. Then $C^*(\Lambda)$ is liminal by Morita equivalence (p. \pageref{describing_morita_equivalence_results}, \cite[Theorem~6.3]{Webster2011} and \cite[Proposition~2]{anHuef-Raeburn-Williams2007}).

Suppose $\partial\Lambda=\Lambda^{\le\infty}$ and that $C^*(\Lambda)$ is liminal. Fix $x\in\partial\Lambda$ and suppose $y\in\partial\Lambda$ is frequently divertable to $[x]$. Then $\kappa(y)$ is frequently divertable to $[\kappa(x)]$ by Lemma \ref{lemma_frequently_divertable_preservation}. We know that $\widetilde{\Lambda}$ is row-finite and has no sources so, since in addition $C^*(\widetilde{\Lambda})$ is liminal by Morita equivalence (p. \pageref{describing_morita_equivalence_results},  \cite[Theorem~6.3]{Webster2011} and \cite[Proposition~2]{anHuef-Raeburn-Williams2007}), we can apply Theorem \ref{k-graph_liminal_thm1} to see that $\kappa(x)$ is shift equivalent to $\kappa(y)$. Lemma \ref{lemma_desource_shift_equivalence_preserved} now shows that $x$ is shift equivalent to $y$, establishing \eqref{k-graph_liminal_thm2_1}.
\end{proof}
\end{theorem}

\begin{example}
The $2$-graph $C^*$-algebra $C^*(\Omega_{2,(3,2)})$ is liminal: in Example \ref{example_boundary_paths_in_Omega_2_32}
we saw that $\partial\Omega_{2,(3,2)}=\braces{\alpha\in\Omega_{2,(3,2)}:s(\alpha)=(3,2)}$, so every path in $\partial\Omega_{2,(3,2)}$ is shift equivalent to every other path in $\partial\Omega_{2,(3,2)}$. It follows that $\Omega_{2,(3,2)}$ satisfies condition \eqref{k-graph_liminal_thm2_2} from Theorem \ref{k-graph_liminal_thm2} and so $C^*(\Omega_{2,(3,2)})$ is liminal.
\end{example}

\begin{example}\label{example_annoying_k-graph_is_liminal}
Suppose $\Lambda$ is the $2$-graph from Example \ref{example_annoying_k-graph}. Let $x$ be the path mentioned in Example \ref{example_annoying_k-graph} and, as in Remark \ref{remark_annoying_example_nearly_breaks_lemma}, let $y$ be the unique path in $v_0\partial\Lambda$ with degree $(\infty,1)$. Since $x$ is frequently divertable to $y$ with $x$ not shift equivalent to $y$, it follows that $\Lambda$ does not satisfy condition \eqref{k-graph_liminal_thm2_1} of Theorem \ref{k-graph_liminal_thm2}. The key point is that $\partial\Lambda\ne\Lambda^\infty$, so we cannot use Theorem \ref{k-graph_liminal_thm2} to deduce that $C^*(\Lambda)$ is not liminal. In fact $C^*(\Lambda)$ is liminal. To see this, a careful inspection of $\Lambda$ reveals that the skeleton of the desourcification $\widetilde\Lambda$ can be pictured by Figure \ref{figure_desourcification}.
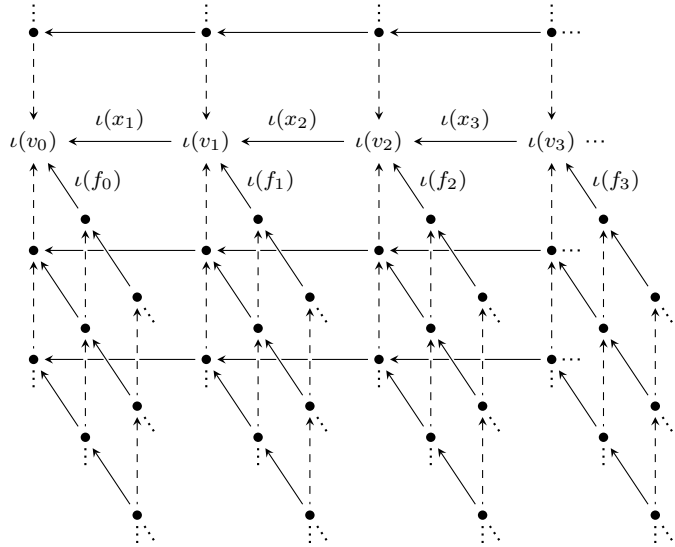
\begin{figure}
\begin{center}
\begin{tikzpicture}[>=stealth,baseline=(current bounding box.center)] 
\def\cellwidth{5.5};
\clip (-0.9em,-13.1em) rectangle (20.5em, 4.6em);
\fill [white] (-20em,-20em) rectangle (80em,80em);
\foreach \x in {0,1,2,3} \foreach \y in {-2,-1,1} {
\node (x\x y\y) at (\cellwidth*\x em,\y*3.5em) {};
\fill[black] (x\x y\y) circle (0.15em);
}
\foreach \x in {0,1,2,3} \node (x\x y0) at (\cellwidth*\x em,0em) {$\scriptstyle \iota(v_\x)$};
\foreach \x / \z in {0/1,1/2,2/3} \draw[black,<-] (x\x y0) to node[above] {$\scriptstyle \iota(x_\z)$} (x\z y0);
\foreach \x / \z in {0/1,1/2,2/3} \foreach \y in {1,-1,-2} \draw[black,<-] (x\x y\y) to (x\z y\y);
\foreach \x in {0,1,2,3} \foreach \y / \z in {0/1,0/-1,-1/-2} \draw[dashed,black,<-] (x\x y\y) to (x\x y\z);

\foreach \x in {0,1,2,3} \foreach \a / \y in {0/0,1/-1,2/-2} {
	\node(x\x a\a) at (\cellwidth*\x em+0.3*\cellwidth em,-2.5em-3.5*\a em) {};
	\fill[black] (x\x a\a) circle (0.15em);
	\draw [<-] (x\x y\y) to (x\x a\a);
}
\foreach \x in {0,1,2,3} \node [above right=-0.3em] at ($ (x\x y0)!.7!(x\x a0) $) {$\scriptstyle \iota(f_\x)$};
\foreach \x in {0,1,2,3} \foreach \b / \y in {0/0,1/-1,2/-2} {
	\node (x\x b\b) at ($ (x\x y\y)!2!(x\x a\b) $) {};
	\fill[black] (x\x b\b) circle (0.15em);
	\draw [white,ultra thick] (x\x a\b) to (x\x b\b);
	\draw [<-] (x\x a\b) to (x\x b\b);
}
\foreach \x in {0,1,2,3} \foreach \c / \y in {0/0,1/-1,2/-2} {
	\node (x\x c\c) at ($ (x\x y\y)!2.50!(x\x a\c) $) {};
	\draw [dotted, thick] (x\x b\c) to (x\x c\c);
}
\foreach \x in {0,1,2,3} \foreach \a / \d in {0/1,1/2} {
	\draw [white, ultra thick] (x\x a\a) to (x\x a\d);
	\draw [<-,dashed] (x\x a\a) to (x\x a\d);
	\draw [white, ultra thick] (x\x b\a) to (x\x b\d);
	\draw [<-,dashed] (x\x b\a) to (x\x b\d);
}
\foreach \x in {0,1,2,3} {
	\node (x\x aend) at ($ (x\x a0)!2.4!(x\x a1) $) {};
	\node (x\x bend) at ($ (x\x b0)!2.4!(x\x b1) $) {};
	\node (x\x ypend) at ($ (x\x y0)!1.4!(x\x y1) $) {};
	\node (x\x ynend) at ($ (x\x y0)!2.4!(x\x y-1) $) {};
	
	\draw [dotted, thick] (x\x a2) to (x\x aend);
	\draw [dotted, thick] (x\x b2) to (x\x bend);
	\draw [dotted, thick] (x\x y1) to (x\x ypend);
	\draw [dotted, thick] (x\x y-2) to (x\x ynend);
}
\foreach \y in {1,0,-1,-2} {
	\node (xendy\y) at ($ (x3y\y.east) + (1em,0em) $) {};
	\draw [dotted, thick] (x3y\y) to (xendy\y);
}
\end{tikzpicture}
\end{center}
\vspace{-1em}
\addtocounter{theorem}{1}
\caption{The desourcification of the $2$-graph from Example \ref{example_annoying_k-graph}.}
\label{figure_desourcification}
\end{figure}
We can see that $\widetilde\Lambda$ satisfies condition \eqref{k-graph_liminal_thm2_1} from Theorem \ref{k-graph_liminal_thm2} or condition \eqref{k-graph_liminal_thm1_2} from Theorem \ref{k-graph_liminal_thm1}, so $C^*(\widetilde\Lambda)$ is liminal. Then $C^*(\Lambda)$ is liminal by Morita equivalence (p. \pageref{describing_morita_equivalence_results},  \cite[Theorem~6.3]{Webster2011} and \cite[Proposition~2]{anHuef-Raeburn-Williams2007}). In fact, $C^*(\Lambda)$ has continuous trace, and is thus liminal, by Theorem \ref{thm_cts-trace}.\end{example}

\begin{theorem}\label{k-graph_postliminal_thm2}
Suppose $\Lambda$ is a row-finite $k$-graph and consider the statements
\begin{enumerate}
 \item\label{k-graph_postliminal_thm2_1} for every $x\in\partial\Lambda$ there exists $n\in\NN^k$ with $n\le d(x)$ such that every path in $x(n)\partial\Lambda$ that is frequently divertable to $[x]$ is shift equivalent to $x$; and
\item\label{k-graph_postliminal_thm2_2} $C^*(\Lambda)$ is postliminal.
\end{enumerate}
Then \eqref{k-graph_postliminal_thm2_1} implies \eqref{k-graph_postliminal_thm2_2} and, if in addition $\partial\Lambda=\Lambda^{\le\infty}$, then \eqref{k-graph_postliminal_thm2_2} implies \eqref{k-graph_postliminal_thm2_1}.

\begin{proof}
Suppose \eqref{k-graph_postliminal_thm2_1}. We will first show that condition \eqref{k-graph_postliminal_thm1_2} from Theorem \ref{k-graph_postliminal_thm1} holds for the desourcification $\widetilde{\Lambda}$ of $\Lambda$. Fix $a\in\widetilde{\Lambda}^\infty$. By Lemma \ref{lemma_boundary_path_associated_to_infinite_path} there exists $x\in\partial\Lambda$ and $m\in\NN^k$ such that $a=\sigma^m\big(\kappa(x)\big)$. By \eqref{k-graph_postliminal_thm2_1} there exists $n\in\NN^k$ such that every path in $x(n)\partial\Lambda$ that is frequently divertable to $[x]$ is shift equivalent to $x$.

Now suppose $b\in a(n)\widetilde{\Lambda}^\infty$ is frequently divertable to $[a]$. Note that $a(n)=\kappa(x)(m+n)$. Since $n\le d(x)$, $\kappa(x)(n)\in\kappa(\Lambda^0)$, and there exists $y\in x(n)\partial\Lambda$ such that $\kappa(y)=\kappa(x)(n,n+m)b$. Note that $\sigma^m\big(\kappa(y)\big)=b$. Since $b$ is frequently divertable to $[a]$, $\sigma^m\big(\kappa(y)\big)$ is frequently divertable to $[\kappa(x)]$ and it follows that $\kappa(y)$ is frequently divertable to $[\kappa(x)]$. Now $y$ is frequently divertable to $[x]$ by Lemma \ref{lemma_frequently_divertable_preservation} and, since in addition $y\in x(n)\partial\Lambda$, $y$ is shift equivalent to $x$. Lemma \ref{lemma_desource_shift_equivalence_preserved} now shows that $\kappa(y)$ is shift equivalent to $\kappa(x)$, so $b$ is shift equivalent to $a$, and condition \eqref{k-graph_postliminal_thm1_2} from Theorem \ref{k-graph_postliminal_thm1} holds for $\widetilde{\Lambda}$. We can now apply Theorem \ref{k-graph_postliminal_thm1} to see that $C^*(\widetilde{\Lambda})$ is postliminal and so $C^*(\Lambda)$ is postliminal by Morita equivalence (p. \pageref{describing_morita_equivalence_results},  \cite[Theorem~6.3]{Webster2011} and \cite[Proposition~2]{anHuef-Raeburn-Williams2007}).

Now suppose $\partial\Lambda=\Lambda^{\le\infty}$, suppose that \eqref{k-graph_postliminal_thm2_2} holds, and fix $x\in\partial\Lambda$. Let $n$ be the element of $\NN^k$ as in Lemma \ref{lemma_property_of_old_boundary_path}. We know $C^*(\widetilde{\Lambda})$ is postliminal by Morita equivalence (p. \pageref{describing_morita_equivalence_results},  \cite[Theorem~6.3]{Webster2011} and \cite[Proposition~2]{anHuef-Raeburn-Williams2007}), so we can apply Theorem \ref{k-graph_postliminal_thm1} to see that there exists $m\in\NN^k$ such that every path in $\kappa(x)(m)\widetilde{\Lambda}^\infty$ that is frequently divertable to $[\kappa(x)]$ is shift equivalent to $\kappa(x)$. We may increase $m$ to ensure that $m\ge n$ while preserving this defining property. Suppose $y\in x\big(m\wedge d(x)\big)\partial\Lambda$ is frequently divertable to $[x]$. Then $\kappa(y)$, and hence $\sigma^{m-m\wedge d(x)}\big(\kappa(y)\big)$, is frequently divertable to $[\kappa(x)]$ by Lemma \ref{lemma_frequently_divertable_preservation}. Since $\kappa(y)\in \kappa(x)\big(m\wedge d(x)\big)\widetilde{\Lambda}^\infty$, it follows from Lemma \ref{lemma_property_of_old_boundary_path} that
\[
 \sigma^{m-m\wedge d(x)}\big(\kappa(y)\big)\in\kappa(x)(m)\widetilde{\Lambda}^\infty,
\]
and so $\sigma^{m-m\wedge d(x)}\big(\kappa(y)\big)$ is shift equivalent to $\kappa(x)$ by our choice of $m$. Then $\kappa(y)$ is shift equivalent to $\kappa(x)$, so $y$ is shift equivalent to $x$ by Lemma \ref{lemma_desource_shift_equivalence_preserved}, establishing \eqref{k-graph_postliminal_thm2_1}.
\end{proof}
\end{theorem}

We will now use the results from Section \ref{sec_bded-trace_cts-trace_Fell} to describe when row-finite $k$-graphs which may have sources have $C^*$-algebras that have bounded trace, are Fell, and have continous trace, respectively. The results in Section \ref{sec_bded-trace_cts-trace_Fell} require the path groupoid to be principal, but here we will instead require that the path groupoid associated to the desourcification is principal. The following lemma describes precisely when this is the case.

\begin{lemma}
Suppose $\Lambda$ is a row-finite $k$-graph. The path groupoid associated to the desourcification $\widetilde{\Lambda}$ of $\Lambda$ is principal if and only if $m=0$ whenever a path in $\partial\Lambda$ is shift equivalent to itself with lag $m$.

\begin{proof}
Suppose the path groupoid associated to $\widetilde{\Lambda}$ is principal and suppose $x\sim_m x$ for some $x\in\partial\Lambda$ and $m\in\ZZ^k$. Then $\kappa(x)\sim_m\kappa(x)$ by Lemma \ref{lemma_desource_shift_equivalence_preserved}, and hence $m=0$.

Now suppose that $m=0$ whenever a path in $\partial\Lambda$ is shift equivalent to itself with lag $m$. Suppose $a\in\widetilde{\Lambda}^\infty$ satisfies $a\sim_m a$ for some $m\in\ZZ^k$. By Lemma \ref{lemma_boundary_path_associated_to_infinite_path} there exists $l\in\NN^k$ and $x\in\partial\Lambda$ such that $a=\sigma^l\big(\kappa(x)\big)$. Then $\sigma^l\big(\kappa(x)\big)\sim_m \sigma^l\big(\kappa(x)\big)$, so $\kappa(x)\sim_m\kappa(x)$ and $x\sim_m x$ by Lemma \ref{lemma_desource_shift_equivalence_preserved}. Then $m=0$ by our assumption so path groupoid associated to $\widetilde{\Lambda}$ is principal by Lemma \ref{lemma_principal_iff_no_semi-periodic_paths}.
\end{proof}
\end{lemma}

\needspace{4\baselineskip}
\begin{theorem}\label{thm_bounded_trace}
Suppose $\Lambda$ is a row-finite $k$-graph such that $n=0$ whenever a path in $\partial\Lambda$ is shift equivalent to itself with lag $n$. Then $C^*(\Lambda)$ has bounded trace if and only if for every $v\in\Lambda^0$ there exists $M\in\NN$ such that for every $x\in\partial\Lambda$ there are at most $M$ paths in $\partial\Lambda$ with range $v$ that are shift equivalent to $x$.

\begin{proof}
Suppose $C^*(\Lambda)$ has bounded trace and fix $v\in\Lambda^0$. Then $C^*(\widetilde{\Lambda})$ has bounded trace since this property is preserved by Morita equivalence (see \cite[Proposition~7]{anHuef-Raeburn-Williams2007}). Since in addition $\widetilde\Lambda$ has no sources, by Theorem \ref{thm_integrable_bounded-trace} there exists $M\in\NN^k$ such that for every $a\in\widetilde{\Lambda}^\infty$ there are at most $M$ paths in $\widetilde{\Lambda}^\infty$ with range $\iota(v)$ that are shift equivalent to $a$. Fix $x\in\partial\Lambda$ and let $S$ be the set of all paths in $\partial\Lambda$ that are shift equivalent to $x$ and have range $v$. Then  every path in $\kappa(S)$ is shift equivalent to $\kappa(x)$ by Lemma \ref{lemma_desource_shift_equivalence_preserved}. In addition, for every $y\in S$ we have $r(y)=v$, so $r\big(\kappa(y)\big)=\iota(v)$, and $\kappa(S)$ has at most $M$ elements by our earlier application of Theorem \ref{thm_integrable_bounded-trace}. Since $\kappa$ is injective by Lemma \ref{lemma_kappa_bijection}, we may conclude that $S$ has at most $M$ elements.

We now prove the converse. Suppose that for every $v\in\Lambda^0$ there exists $M_v\in\NN$ as in the statement of the theorem. Fix $u\in\widetilde{\Lambda}^0$. By the construction of $\widetilde\Lambda$, there exists $y\in\partial\Lambda$ and $n\in\NN^k$ such that $u=\kappa(y)(n)$. Fix $a\in\widetilde{\Lambda}^\infty$. By Lemma \ref{lemma_boundary_path_associated_to_infinite_path} we may assume without loss of generality that $a=\kappa(x)$ for some $x\in\partial\Lambda$. Let $S$ be the set of all paths in $\partial\Lambda$ with range $r(y)$ that are shift equivalent to $x$ and let $T$ be the set of all paths in $\widetilde{\Lambda}^\infty$ with range $u$ that are shift equivalent to $\kappa(x)$. Fix $b\in T$. Then $\kappa(y)(0,n)b\in\iota(\Lambda^0)\widetilde{\Lambda}^\infty$ so, since $\kappa$ is a bijection onto $\iota(\Lambda^0)\widetilde\Lambda^\infty$ (Lemma \ref{lemma_kappa_bijection}), there exists $z\in\partial\Lambda$ such that $\kappa(y)(0,n)b=\kappa(z)$. Furthermore since $b\in T$, we have $b\sim\kappa(x)$, so both $\kappa(y)(0,m)b$ and $\kappa(z)$ are shift equivalent to $\kappa(x)$. Now $r\big(\kappa(y)(0,m)b\big)=\iota\big(r(y)\big)$, so $\braces{\kappa(y)(0,m)b:b\in T}\subset \kappa(S)$. Since $T$ and $\braces{\kappa(y)(0,m)b:b\in T}$ have the same number of elements, and since $\kappa$ is injective (Lemma \ref{lemma_kappa_bijection}), we have
$\#T\le\#\kappa(S)=\# S\le M_{r(y)}$. We arbitrarily fixed $u\in\widetilde{\Lambda}^0$, so condition \eqref{k-thm_integrable2} of Theorem \ref{thm_integrable_bounded-trace} holds by taking $M=M_{r(y)}$. Since $\widetilde\Lambda$ has no sources, Theorem \ref{thm_integrable_bounded-trace} now implies that $C^*(\widetilde{\Lambda})$ has bounded trace, and it follows that $C^*(\Lambda)$ has bounded trace by Morita equivalence (p. \pageref{describing_morita_equivalence_results},  \cite[Theorem~6.3]{Webster2011} and \cite[Proposition~7]{anHuef-Raeburn-Williams2007}).
\end{proof}
\end{theorem}

\begin{lemma}\label{lemma_cts-trace_Fell_desourcification}
Suppose $\Lambda$ is a row-finite $k$-graph and $W\subset\widetilde\Lambda^0$ is finite. For each $w\in W$, suppose $z^{(w)}\in\partial\Lambda$ and $n^{(w)}\in\NN^k$ satisfy $w=\kappa(z^{(w)})(n^{(w)})$. Let $V=\braces{r(z^{(w)}):w\in W}$. If there exists a finite $F\subset\Lambda$ such that for every $x,y\in V\partial\Lambda$ with $x\sim y$, the pair $(x,y)$ is a monolithic extension of a pair in $F\times F$, then there exists a finite $E\subset\widetilde\Lambda$ such that for every $(a,b)\in W\widetilde\Lambda^\infty$ with $a\sim b$, the pair $(a,b)$ is a monolithic extension of a pair in $E\times E$.
\begin{proof}
Suppose $F$ is a finite subset of $\Lambda$ as in the statement of the Lemma. Let $N=\vee_{w\in W}n^{(w)}$ and define
\[
E=\Big\{\big(\iota(\eta)\delta\big)\Big(n^{(w)},d\big(\iota(\eta)\delta\big)\Big):\eta\in F, w\in W, \delta\in s\big(\iota(\eta)\big)\widetilde\Lambda^N\Big\},
\]
noting that $E$ is finite since $F$ and $W$ are finite and $\widetilde\Lambda$ is row-finite. Fix $a,b\in W\widetilde\Lambda^\infty$ with $a\sim b$, let 
\[\phi=\kappa(z^{(r(a))})(0,n^{(r(a))})\quad\text{and}\quad\psi=\kappa(z^{(r(b))})(0,n^{(r(b))}).\]
Then $\phi a,\psi b\in\iota(V)\widetilde\Lambda^\infty$ so, since $\kappa$ is a bijection onto $\iota(\Lambda^0)\widetilde\Lambda^\infty$ (Lemma \ref{lemma_kappa_bijection}), there exist $x,y\in V\partial\Lambda$ such that $\kappa(x)=\phi a$ and $\kappa(y)=\psi b$. Since $a\sim b$, we have $x\sim y$ by Lemma \ref{lemma_desource_shift_equivalence_preserved}. By our assumption there exist $\eta,\zeta\in F$ such that $(x,y)$ is a monolithic extension of $(\eta,\zeta)$. Then $\big(\kappa(x),\kappa(y)\big)$ is a monolithic extension of $\big(\iota(\eta),\iota(\zeta)\big)$. Let $\delta=\kappa(x)\big(d(\eta),d(\eta)+N\big)$, so that $\big(\kappa(x),\kappa(y)\big)=(\phi a,\psi b)$ is a monolithic extension of $\big(\iota(\eta)\delta,\iota(\zeta)\delta\big)$ with $d\big(\iota(\eta)\delta\big)\ge n^{(r(a))}=d(\phi)$ and $d\big(\iota(\zeta)\delta\big)\ge n^{(r(b))}=d(\psi)$. Then $(a,b)$ is a monolithic extension of
\[
\bigg(\big(\iota(\eta)\delta\big)\Big(n^{(r(a))},d\big(\iota(\eta)\delta\big)\Big),\big(\iota(\zeta)\delta\big)\Big(n^{(r(b))},d\big(\iota(\zeta)\delta\big)\Big)\bigg)\in E\times E.\qedhere
\]
\end{proof}
\end{lemma}

\begin{lemma}\label{lemma_transfer_ME_from_desourcification}
Let $\Lambda$ be a row-finite $k$-graph and suppose that $\big(\kappa(x),\kappa(y)\big)$ is a monolithic extension of $\big(\kappa(f)(0,m),\kappa(g)(0,n)\big)$ for some $x,y,f,g\in\partial\Lambda$ and $m,n\in\NN^k$. Then $(x,y)$ is a monolithic extension of 
$\big(f(0,m\wedge d(f)),g(0,n\wedge d(g))\big)$.
\begin{proof}
We have $\sigma^m\big(\kappa(x)\big)=\sigma^n\big(\kappa(y)\big)$, so $\kappa(x)(m)=\kappa(y)(n)$ and by \eqref{V1} and \eqref{V2} we can see that $x\big(m\wedge d(x)\big)=y\big(n\wedge d(y)\big)$ and $m-m\wedge d(x)=n-n\wedge d(y)$. This is enough to show that $\kappa(x)\big(m\wedge d(x),m\big)=\kappa(y)\big(n\wedge d(y),n\big)$ by considering \eqref{P1}, \eqref{P2} and \eqref{P3}. Thus $\sigma^{m\wedge d(x)}\big(\kappa(x)\big)=\sigma^{n\wedge d(y)}\big(\kappa(y)\big)$, and so $\kappa\big(\sigma^{m\wedge d(x)}(x)\big)=\kappa\big(\sigma^{n\wedge d(y)}(y)\big)$ by Lemma \ref{lemma_sigma_kappa}. For every $p\in\NN^k$ we now have $\kappa\big(\sigma^{m\wedge d(x)}(x)\big)(0,p)=\kappa\big(\sigma^{n\wedge d(y)}(y)\big)(0,p)$, so 
\[
\sigma^{m\wedge d(x)}(x)\Big(0,p\wedge d\big(\sigma^{m\wedge d(x)}(x)\big)\Big)=\sigma^{n\wedge d(y)}(y)\Big(0,p\wedge d\big(\sigma^{n\wedge d(y)}(y)\big)\Big),
\]
and it follows that $\sigma^{m\wedge d(x)}(x)=\sigma^{n\wedge d(y)}(y)$. Since $\kappa(x)(0,m)=\kappa(f)(0,m)$, we have $x\big(0,m\wedge d(x)\big)=f\big(0,m\wedge d(f)\big)$ by \eqref{P1} and similarly $y\big(0,n\wedge d(y)\big)=g\big(0,n\wedge d(g)\big)$. The result follows.
\end{proof}
\end{lemma}

\begin{theorem}\label{thm_Fell}
Suppose $\Lambda$ is a row-finite $k$-graph such that $m=0$ whenever a path in $\partial\Lambda$ is shift equivalent to itself with lag $m$. Then $C^*(\Lambda)$ is Fell if for every $z\in\partial\Lambda$ there exist $p\in\NN^k$ with $p\le d(z)$ and a finite $F\subset\Lambda$ such that for any $x,y\in z(p)\partial\Lambda$ with $x\sim y$, the pair $(x,y)$ is a monolithic extension of a pair in $F\times F$. The converse is true if $\partial\Lambda=\Lambda^{\le\infty}$.
\begin{proof}
We will first show that $C^*(\widetilde\Lambda)$ is Fell by showing that $\widetilde\Lambda$ satisfies condition \eqref{thm_Cartan_Fell_2} of Theorem \ref{thm_Cartan_Fell}. Fix $a\in\widetilde\Lambda^\infty$. By Lemma \ref{lemma_boundary_path_associated_to_infinite_path} there exist $z\in\partial\Lambda$ and $m\in\NN^k$ such that $a=\sigma^m\big(\kappa(z)\big)$. Let $p\in\NN^k$ and $F\subset\Lambda$ be as in the hypothesis when applied to $z$. Let $W=\big\{\kappa(z)(m\vee p)\big\}$ and $V=\braces{z(p)}$. By noting that $W=\big\{\kappa\big(\sigma^p(z)\big)(m\vee p -p)\big\}$ by Lemma \ref{lemma_sigma_kappa} and that $V=\big\{r\big(\sigma^p(z)\big)\big\}$, we can apply Lemma \ref{lemma_cts-trace_Fell_desourcification} to see that there exists a finite $E\subset\widetilde\Lambda$ such that for any $b,c\in W\widetilde\Lambda^\infty$ with $b\sim c$, the pair $(b,c)$ is a monolithic extension of a pair in $E\times E$. Since Lemma \ref{lemma_sigma_kappa} gives us
\[
\kappa\big(\sigma^p(z)\big)(m\vee p-p)=\kappa(z)(m\vee p)=\sigma^m\big(\kappa(z)\big)(m\vee p-m)=a(m\vee p-m),
\]
we can see that $W=\braces{a(m\vee p-m)}$ and thus that $\widetilde\Lambda$ satisfies condition \eqref{thm_Cartan_Fell_2} of Theorem \ref{thm_Cartan_Fell}. As we also know that $\widetilde\Lambda$ has no sources, $C^*(\widetilde\Lambda)$ is Fell by Theorem \ref{thm_Cartan_Fell}. Then $C^*(\Lambda)$ is Fell by Morita equivalence (p. \pageref{describing_morita_equivalence_results},  \cite[Theorem~6.3]{Webster2011} and \cite[Corollary~14]{anHuef-Raeburn-Williams2007}).

We now prove the converse. Suppose $\partial\Lambda=\Lambda^{\le\infty}$, $C^*(\Lambda)$ is Fell and fix $z\in\partial\Lambda$. Since $\partial\Lambda=\Lambda^{\le\infty}$, by Lemma \ref{lemma_property_of_old_boundary_path} there exists $p\in\NN^k$ with $p\le d(z)$ such that for any $q\in\NN^k$ with $q\ge p$ and $x\in z\big(q\wedge d(x)\big)\partial\Lambda$,
\begin{equation}\label{Fell_eqn1}
\kappa(z)\big(q\wedge d(z),q\big)=\kappa(x)\big(0,q-q\wedge d(z)\big).
\end{equation}
We know $C^*(\widetilde{\Lambda})$ is Fell by Morita equivalence (p. \pageref{describing_morita_equivalence_results},  \cite[Theorem~6.3]{Webster2011} and \cite[Corollary~14]{anHuef-Raeburn-Williams2007}). Since in addition $\widetilde\Lambda$ has no sources, Theorem \ref{thm_Cartan_Fell} can be used to see that for the path $\sigma^{p}\big(\kappa(z)\big)\in\widetilde\Lambda^\infty$ there exists $l\in\NN^k$ and a finite $E\subset\widetilde\Lambda$ that satisfy condition \eqref{thm_Cartan_Fell_2} from Theorem \ref{thm_Cartan_Fell}. Let $m=p+l$. Then for every $a,b\in \kappa(z)(m)\widetilde\Lambda^\infty$ with $a\sim b$, the pair $(a,b)$ is a monolithic extension of a pair in $E\times E$.
Let $n=m\wedge d(z)$ and define $A=\braces{\kappa(z)(n,m)\eta:\eta\in E}$, noting that $A\subset\Lambda$. By the constuction of $\widetilde\Lambda$, for every $\gamma\in\iota(\Lambda^0)$ we can choose $t^{(\gamma)}\in\partial\Lambda$ such that $\gamma=\kappa(t^{(\gamma)})\big(0,d(\gamma)\big)$. Let
\[
F=\Big\{t^{(\phi)}\big(0,d(\phi)\wedge d(t^{(\phi)})\big):\phi\in A\Big\}.
\]

Fix $x,y\in z(n)\partial\Lambda$ such that $x\sim y$. Then $\kappa(x),\kappa(y)\in\kappa(z)(n)\widetilde\Lambda^\infty$ with $\kappa(x)\sim \kappa(y)$ by Lemma \ref{lemma_desource_shift_equivalence_preserved}, so $\sigma^{m-n}\big(\kappa(x)\big),\sigma^{m-n}\big(\kappa(y)\big)\in \kappa(z)(m)\widetilde\Lambda^\infty$ with $\sigma^{m-n}\big(\kappa(x)\big)\sim\sigma^{m-n}\big(\kappa(y)\big)$. By the earlier use of Theorem \ref{thm_Cartan_Fell}, there exists $\eta,\zeta\in E$ such that $\Big(\sigma^{m-n}\big(\kappa(x)\big),\sigma^{m-n}\big(\kappa(y)\big)\Big)$ is a monolithic extension of $(\eta,\zeta)$. Let $\gamma=\kappa(z)(n,m)\eta$ and $\delta=\kappa(z)(n,m)\zeta$, so that $\gamma,\delta\in A$. Since $n=m\wedge d(z)$, by \eqref{Fell_eqn1} we have 
\[\kappa(x)(0,m-n)=\kappa(y)(0,m-n)=\kappa(z)(n,m),\]
so $\big(\kappa(x),\kappa(y)\big)$ is a monolithic extension of $(\gamma,\delta)$. Recall that $\gamma=\kappa(z^{(\gamma)})\big(0,d(\gamma)\big)$ and $\delta=\kappa(z^{(\delta)})\big(0,d(\delta)\big)$. Lemma \ref{lemma_transfer_ME_from_desourcification} can now be applied to see that $(x,y)$ is a monolithic extension of
\[
\Big(z^{(\gamma)}\big(0,d(\gamma)\wedge d(z^{(\gamma)}),z^{(\delta)}\big(0,d(\delta)\wedge d(z^{(\delta)})\big)\Big)\in F\times F.\qedhere
\]
\end{proof}
\end{theorem}

\begin{theorem}\label{thm_cts-trace}
Suppose $\Lambda$ is a row-finite $k$-graph such that $m=0$ whenever a path in $\partial\Lambda$ is shift equivalent to itself with lag $m$. Then $C^*(\Lambda)$ has continuous trace if and only if for every finite subset $V$ of $\Lambda^0$ there exists a finite $F\subset\Lambda$ such that for every $x,y\in V\partial\Lambda$ with $x\sim y$, the pair $(x,y)$ is a monolithic extension of a pair in $F\times F$.
\begin{proof}
Suppose that for every finite $V\subset\Lambda^0$ there exists a finite $F\subset\Lambda$ as in the statement of the theorem. Let $W$ be a finite subset of $\widetilde{\Lambda}^0$. For every $w\in W$ there exist $y^{(w)}\in\partial\Lambda$ and $m^{(w)}\in\NN^k$ such that $w=\kappa(y^{(w)})(m^{(w)})$. Let $V=\braces{r(y^{(w)}):w\in W}$. By our assumption there exists a finite $F\subset\Lambda$ such that for every $x,y\in V\partial\Lambda$ with $x\sim y$, the pair $(x,y)$ is a monolithic extension of a pair in $F\times F$. Lemma \ref{lemma_cts-trace_Fell_desourcification} can now be applied to see that there exists a finite $E\subset\widetilde\Lambda$ such that for every $a,b\in W\widetilde\Lambda^\infty$ with $a\sim b$, the pair $(a,b)$ is a monolithic extension of a pair in $E\times E$. Then condition \eqref{thm_proper_cts-trace_2} of Theorem \ref{thm_proper_cts-trace} holds so, since in addition $\widetilde\Lambda$ has no sources, $C^*(\widetilde\Lambda)$ has continuous trace by Theorem \ref{thm_proper_cts-trace}. It follows that $C^*(\Lambda)$ has continuous trace by Morita equivalence (p. \pageref{describing_morita_equivalence_results},  \cite[Theorem~6.3]{Webster2011} and \cite[Proposition~7]{anHuef-Raeburn-Williams2007}).

Now suppose $C^*(\Lambda)$ has continuous trace and choose a finite $V\subset\Lambda^0$. Then $C^*(\widetilde{\Lambda})$ has continuous trace by Morita equivalence (p. \pageref{describing_morita_equivalence_results},  \cite[Theorem~6.3]{Webster2011} and \cite[Proposition~7]{anHuef-Raeburn-Williams2007}). We also know $\widetilde\Lambda$ has no sources, so by Theorem \ref{thm_proper_cts-trace} there exists a finite $E\subset\widetilde\Lambda$ such that for any $a,b\in \iota(V)\widetilde\Lambda^\infty$ with $a\sim b$, the pair $(a,b)$ is a monolithic extension of a pair in $E\times E$.
For every $\gamma\in \iota(\Lambda^0)\widetilde\Lambda$ we choose $z^{(\gamma)}\in\partial\Lambda$ such that $\gamma=\kappa(z^{(\gamma)})\big(0,d(\gamma)\big)$. Let 
\[
F=\big\{z^{(\mu)}\big(0,d(\mu)\wedge d(z^{(\mu)})\big):\mu\in E\big\}
\]

Fix $x,y\in V\partial\Lambda$ with $x\sim y$. Then $\kappa(x),\kappa(y)\in \iota(V)\widetilde\Lambda^\infty$ with $\kappa(x)\sim\kappa(y)$ by Lemma \ref{lemma_desource_shift_equivalence_preserved}, so $\big(\kappa(x),\kappa(y)\big)$ is a monolithic extension of some $(\mu,\nu)\in E\times E$. Since $\mu=\kappa(z^{(\mu)})\big(0,d(\mu)\big)$ and $\nu=\kappa(z^{(\nu)})\big(0,d(\nu)\big)$, we can apply Lemma \ref{lemma_transfer_ME_from_desourcification} to see that $(x,y)$ is a monolithic extension of
\[\Big(z^{(\mu)}\big(0,d(\mu)\wedge d(z^{(\mu)})\big),z^{(\nu)}\big(0,d(\nu)\wedge d(z^{(\nu)})\big)\Big)\in F\times F.\qedhere\]
\end{proof}
\end{theorem}

\begin{example}\label{examples_cts-trace}
The Cuntz-Krieger $C^*$-algebras generated by the $2$-graphs from Examples \ref{omega_2_infty_2}, \ref{example_not_locally_convex_with_same_boundary_paths} and \ref{example_annoying_k-graph} have continuous trace by Theorem \ref{thm_cts-trace}. To see this, in each of these cases the condition in the Theorem can be satisfied by taking $F=V$ for every finite subset $V$ of $\Lambda^0$.
\end{example}

\section{Special case: directed graphs}\label{section_directed_graphs}
In this section we will state the main results from Sections \ref{sec_liminal_postliminal}, \ref{sec_bded-trace_cts-trace_Fell} and \ref{sec_desourcification} for the special case of directed graphs. We will also state alternative conditions that characterise the row-finite directed graphs that are postliminal, liminal, Fell, and have continuous trace.

A characterisation of liminal directed graph $C^*$-algebras has previously been developed by Ephrem in \cite[Theorem~5.5]{Ephrem2004}. Similarly, characterisations of postliminal directed graph $C^*$-algebras have been developed by Ephrem in \cite[Theorem~7.3]{Ephrem2004} and by Deicke, Hong and Szyma{\'n}ski in \cite[Theorem~2.1]{dhs2003}. We reconcile our results with Ephrem's in Remark \ref{remark_Ephrems_liminal_condition_coincides} and Remark \ref{remark_Ephrems_postliminal_condition_coincides} below. Unlike the results presented here, both Ephrem's and Deicke, Hong and Szyma{\'n}ski's results may be applied to directed graphs that are not row-finite. This shortcoming will be largely overcome by showing that the conditions presented in this section are equivalent to Ephrem's.

In developing his characterisation of directed graphs with liminal $C^*$-algebras, Ephrem started by developing a characterisation \cite[Theorem~3.10]{Ephrem2004} of the row-finite directed graphs without sources and with liminal $C^*$-algebras before using the Drinen-Tomforde desingularisation to remove the row-finite and no sources hypothesis in \cite[Theorem~5.5]{Ephrem2004}. Our approach is similar through the use of the higher-rank graph Farthing-Webster desourcification in Section \ref{sec_desourcification}. The characterisations of directed graphs with postliminal $C^*$-algebras in Theorem \ref{thm_directed_graph_postliminal} and \cite[Theorem~7.3]{Ephrem2004} follow similar approaches.

Recall the definition of a directed graph and the associated path groupoid from Section \ref{sec_path_groupoid} and the overview of the Cuntz-Krieger $C^*$-algebras of row-finite directed graphs from Section \ref{section_algebras_of_directed_graphs}. For a directed graph $E$ we can consider $E^\infty=\Lambda_E^\infty$. Recall the definition of the boundary paths $\partial\Lambda$ of a $k$-graph $\Lambda$ from Section \ref{sec_boundary_paths_and_desourcification}. Since every $1$-graph is locally convex (Remark \ref{remark_every_1-graph_locally_convex}), after assuming that $E$ is row-finite we have $\partial\Lambda_E=\Lambda_E^{\le\infty}$ by Proposition \ref{prop_Lambda_leinfty_equals_partial_lambda}. By the definition of $\Lambda^{\le\infty}$ (Definition \ref{def_Lambda_leinfty}) we can consider $E^{\le\infty}=\Lambda_E^{\le\infty}$, so that $\partial\Lambda_E=E^{\le\infty}$. Recall that Lemmas \ref{path_groupoids_isomorphic} and \ref{lemma_KP_LambdaE} tell us that $G_E\cong G_{\Lambda_E}$ and $C^*(E)\cong C^*(\Lambda_E)$, respectively.

Suppose $\eta$ is a path in a directed graph with $r(\eta)=s(\eta)$. For each $p\in\PP$ define \notationindex{ETAPINFTY@$\eta^p,\eta^\infty$ where $\eta\in E^*, p\in\PP$}
\[
\eta^p:=\underbrace{\eta\eta\cdots\eta}_{p\mathrm{-times}}
\]
and define $\eta^\infty$ to be the infinite path $\eta\eta\eta\cdots$. For any path $x\in E^*\cup E^\infty$, define\notationindex{FPX@$\mathcal{FP}(x)$}
\[
\mathcal{FP}(x):=\{\gamma\in E^*: \gamma=x(m,n)\text{ for some }m,n\in\NN\},
\]
so that $\mathcal{FP}(x)$ can be thought of as the set of all finite paths that appear in $x$. We will frequently abuse notation by talking about paths as if they are sets; when we refer to elements of $x$, we really mean elements of $\mathcal{FP}(x)$. For example, by the statement ``there are no cycles in $x$'', we really mean ``there are no cycles in $\mathcal{FP}(x)$''.

\begin{definition}
For a directed graph $E$ and for $x,y\in E^{\le\infty}$, we say $x$ is {\em frequently divertable}\index{frequently divertable!path in a directed graph} to $[y]$ if for every $n\in\NN$ with $n\le |x|$ there is a path in $x(n)E^{\le\infty}$ that is shift equivalent to $y$.
\end{definition}
From the above observations that $E^\infty=\Lambda_E^\infty$ and $E^{\le\infty}=\partial\Lambda_E$, we can see that this notion of a frequently divertable path in $E^{\le\infty}$ is consistent with the notions of frequently divertable infinite paths and boundary paths in higher-rank graphs as introduced in Definitions \ref{def_frequently_divertable_infinite_paths} and \ref{def_frequently_divertable_boundary_paths}.

\begin{theorem}\label{thm_directed_graph_liminal}
Suppose $E$ is a row-finite directed graph. The following are equivalent:
\begin{enumerate}
\item\label{thm_directed_graph_liminal_1} for every $x\in E^{\le\infty}$, every path in $E^\infty$ that is frequently divertable to $[x]$ is shift equivalent to $x$;
\item\label{thm_directed_graph_liminal_2} for every $x\in E^{\le\infty}$ there are finitely many paths in $r(x)E^{\le\infty}$ that are shift equivalent to $x$; and
\item\label{thm_directed_graph_liminal_3} $C^*(E)$ is liminal.
\end{enumerate}
If in addition $E$ has no sources, the following are also equivalent to the preceding:
\begin{enumerate}\setcounter{enumi}{3}
\item\label{thm_directed_graph_liminal_4} every orbit in $G_E^{(0)}$ is closed; and
\item\label{thm_directed_graph_liminal_5} $C^*(G_E)$ is liminal.
\end{enumerate}
\end{theorem}
\begin{remark}\label{remark_Ephrems_liminal_condition_coincides}
Condition \eqref{thm_directed_graph_liminal_2} is very similar to Ephrem's condition from \cite[Theorem~5.5]{Ephrem2004}. It is a straightforward task to show that these conditions are equivalent. 
\end{remark}
The following two lemmas are used in the proof of Theorem \ref{thm_directed_graph_liminal} to establish the equivalence of \eqref{thm_directed_graph_liminal_1} and \eqref{thm_directed_graph_liminal_2}.
\begin{lemma}\label{lemma_FD_not_SE}
Suppose $E$ is a directed graph and that $x\in E^\infty$ is a path that does not contain a cycle and that has at least two elements in $x(n)E^\infty\cap [x]$ for each $n\in\NN$. Then for each $m\in\NN$ there is a path in $x(m)E^\infty$ that is frequently divertable to $[x]$ but not shift equivalent to $x$.
\begin{proof}
Fix $m\in\NN$ and let $p_0=0$. Choose $z^{(1)}\in x(m)E^\infty\cap[x]$ with $z^{(1)}\ne x$. Since $z^{(1)}\sim x$ there exist $p_1,q_1\in\NN$ with $\sigma^{p_1}(x)=\sigma^{q_1}(z^{(1)})$. Let $\gamma^{(1)}=z^{(1)}(0,q_1)$. Since $z^{(1)}\ne x$ we must have $\gamma^{(1)}\ne x(0,p_1)$. Now choose $z^{(2)}\in x(p_1)E^\infty\cap[x]$ with $z^{(2)}\ne\sigma^{p_1}(x)$. Since $z^{(2)}\sim\sigma^{p_1}(x)$ there exist $p_2,q_2\in\NN$ with $p_2>p_1$ such that $\sigma^{p_2}(x)=\sigma^{q_2}(z^{(2)})$. Let $\gamma^{(2)}=z^{(2)}(0,q_2)$. Since $z^{(2)}\ne\sigma^{p_2}(x)$ we must have $\gamma^{(2)}\ne x(p_1,p_2)$. By continuing in this manner we can see that there is a strictly increasing sequence $p_0,p_1,p_2,\ldots$ in $\NN$ and a sequence $\gamma^{(1)},\gamma^{(2)},\ldots$ in $E^*$ such that $r(\gamma^{(i)})=x(p_{i-1})$, $s(\gamma^{(i)})=x(p_i)$ and $\gamma^{(i)}\ne x(p_{i-1},p_i)$ for $i=1,2,\ldots$. 

Let $y$ be the path $\gamma^{(1)}\gamma^{(2)}\cdots$ in $x(m)E^\infty$. To show that $y\nsim x$, first fix $a,b\in\NN$. We need to show that $\sigma^a(x)\ne \sigma^b(y)$. Since $y=\gamma^{(1)}\gamma^{(2)}\cdots$, we may increase $a,b$ if necessary to ensure that $\sigma^b(y)=\gamma^{(j)}\gamma^{(j+1)}\cdots$ for some $j\in\PP$. If $x(a)\ne y(b)$ it follows immediately that $\sigma^a(x)\ne\sigma^b(y)$ so we may assume that $x(a)=y(b)$. Then $x(a)=y(b)=r(\gamma^{(j)})=x(p_{j-1})$, so $a=p_{j-1}$ since $x$ does not contain a cycle.

If $p_j-p_{j-1}=|\gamma^{(j)}|$, then 
\begin{align*}
\sigma^a(x)(0,p_j-&p_{j-1})=\sigma^{p_{j-1}}(x)(0,p_j-p_{j-1})=x(p_{j-1},p_j)\\
&\ne\gamma^{(j)}=\sigma^b(y)(0,|\gamma^{(j)}|)=\sigma^b(y)(0,p_j-p_{j-1}),
\end{align*}
so $\sigma^b(y)\ne\sigma^a(x)$. If $p_j-p_{j-1}\ne|\gamma^{(j)}|$ and $\sigma^b(y)=\sigma^a(x)$, we must have 
\[x(p_{j-1}+|\gamma^{(j)}|)=x(a+|\gamma^{(j)}|)=y(b+|\gamma^{(j)}|)=s(\gamma^{(j)})=x(p_j).\]
This is not possible when $p_j-p_{j-1}\ne|\gamma^{(j)}|$ since $x$ does not contain a cycle. We have thus  shown that $y$ is not shift equivalent to $x$. That $y$ is frequently divertable to $[x]$ follows immediately from $r(\gamma^{(i)})$ being in $x$ for each $i$.
\end{proof}
\end{lemma}

\begin{lemma}\label{lemma_cycle_with_entry_implies_not_liminal}
Let $E$ be a directed graph. Suppose that for every $x\in E^{\le\infty}$, every path in $E^\infty$ that is frequently divertable to $[x]$ is shift equivalent to $x$. Then no cycle in $E$ has an entry.
\begin{proof}
Suppose $\alpha$ is a cycle in $E$ with an entry $f\in E^1$. We may assume that $r(\alpha)=s(\alpha)=r(f)$. If there is a path $\gamma\in E^*$ with $r(\gamma)=s(f)$ and $s(\gamma)=r(f)$ then let $x=(\gamma f)^\infty$; otherwise let $x$ be any path in $s(f)E^{\le\infty}$. The path $\alpha^\infty$ in $E^\infty$ is then frequently divertable to $[x]$ but not shift equivalent to $x$.
\end{proof}
\end{lemma}

\begin{proof}[Proof of Theorem \ref{thm_directed_graph_liminal}]
\eqref{thm_directed_graph_liminal_1}$\implies$\eqref{thm_directed_graph_liminal_2}. Suppose \eqref{thm_directed_graph_liminal_2} does not hold; we aim to show that \eqref{thm_directed_graph_liminal_1} does not hold. If $E$ contains a cycle with an entry then \eqref{thm_directed_graph_liminal_1} does not hold by Lemma \ref{lemma_cycle_with_entry_implies_not_liminal}; from here on we may assume that no cycle in $E$ has an entry. Since \eqref{thm_directed_graph_liminal_2} does not hold there is a path $x\in E^{\le\infty}$ such that $r(x)E^{\le\infty}\cap[x]$ is infinite. Since $E$ is row-finite, there are only finitely many edges with range $r(x)$ and so there exists $y_1\in r(x)E^1$ such that $s(y_1)E^{\le\infty}\cap [x]$ is infinite. There then exists $y_2\in s(y_1)E^1$ such that $s(y_2)E^{\le\infty}\cap[x]$ is infinite. By repeating this observation we can see that there is a path $y=y_1y_2y_3\cdots$ in $E^\infty$ such that $y(n)E^{\le\infty}\cap [x]$ is infinite for each $n\in\NN$. Note that $y$ is frequently divertable to $[x]$ so that if $y$ is not shift equivalent to $x$ then \eqref{thm_directed_graph_liminal_1} does not hold and we are done.

Suppose $y$ is shift equivalent to $x$. Then $y(n)E^{\le\infty}\cap [x]=y(n)E^{\le\infty}\cap [y]$ is infinite for all $n\in\NN$. We know $y\in E^\infty$, so $E^{\le\infty}\cap [y] = E^\infty\cap [y]$ and thus $y(n)E^{\le\infty}\cap [y] = y(n)E^\infty\cap [y]$ for each $n\in\NN$. If $y$ contains a cycle $\alpha$ then there exists $m\in\NN$ with $y(m,m+|\alpha|)=\alpha$ and since no cycle in $E$ has an entry, it follows that $y=y(0,m)\alpha^\infty$. Then the only path in $y(m)E^\infty$ is $\alpha^\infty$, so $y(m)E^\infty\cap[y]$ is finite. It follows from this contradiction that $y$ does not contain a cycle. By applying Lemma \ref{lemma_FD_not_SE} to $y$ we can see that there is a path $z$ in $y(0)E^\infty$ that is frequently divertable to $[y]$ but not shift equivalent to $[y]$. Since $x$ is shift equivalent to $y$, it follows that $z$ is frequently divertable to $[x]$ but not shift equivalent to $x$, so \eqref{thm_directed_graph_liminal_1} does not hold.

\eqref{thm_directed_graph_liminal_2}$\implies$\eqref{thm_directed_graph_liminal_1}. Suppose \eqref{thm_directed_graph_liminal_2}, fix $x\in E^{\le\infty}$ and suppose $y\in E^\infty$ is frequently divertable to $[x]$. Then for every $n\in\NN$ there exists $z^{(n)}\in y(n)E^{\le\infty}$ with $z^{(n)}\sim x$. By \eqref{thm_directed_graph_liminal_2} there are finitely many paths in $r(z^{(0)})E^{\le\infty}=y(0)E^{\le\infty}$ that are shift equivalent to $z^{(0)}$ (and thus shift equivalent to $x$). Then $\{y(0,n)z^{(n)}:n\in\NN\}$ is finite so there must exist a strictly increasing sequence $n_1,n_2,\ldots$ such that $y(0,n_1)z^{(n_1)}=y(0,n_2)z^{(n_2)}=\ldots$. But then each of these paths $y(0,n_i)z^{(n_i)}$ must equal $y$ and so $y$ is shift equivalent to $x$, establishing \eqref{thm_directed_graph_liminal_1}.

\eqref{thm_directed_graph_liminal_1}$\implies$\eqref{thm_directed_graph_liminal_3}. Suppose \eqref{thm_directed_graph_liminal_1} and recall that $E^{\le\infty}=\Lambda_E^{\le\infty}=\partial\Lambda_E$. Applying Theorem \ref{k-graph_liminal_thm2} to the $1$-graph $\Lambda_E$ shows that $C^*(E)$ is liminal if
\begin{statement}\label{condition_directed_graph_liminal_thm2}
for every $x\in E^{\le\infty}$, every path in $E^{\le\infty}$ that is frequently divertable to $[x]$ is shift equivalent to $x$.
\end{statement}
To show that \ref{condition_directed_graph_liminal_thm2} holds, first fix $x\in E^{\le\infty}$ and suppose that $y\in E^{\le\infty}$ is frequently divertable to $[x]$. If $y\in E^\infty$ then it follows immediately from \eqref{thm_directed_graph_liminal_1} that $y\sim x$. Suppose $y\in E^{\le\infty}\backslash E^\infty$. Since $y$ is frequently divertable to $[x]$, the only path in $y(|y|)E^{\le\infty}$, namely $y(|y|)$, must be shift equivalent to $x$. Then $y$ is shift equivalent to $x$, so \eqref{condition_directed_graph_liminal_thm2} holds and $C^*(E)$ is liminal by Theorem \ref{k-graph_liminal_thm2}. 

\eqref{thm_directed_graph_liminal_3}$\implies$\eqref{thm_directed_graph_liminal_1}. Suppose \eqref{thm_directed_graph_liminal_3}. Since $C^*(E)$ is liminal, Theorem \ref{k-graph_liminal_thm2} tell us that \eqref{condition_directed_graph_liminal_thm2} holds and \eqref{thm_directed_graph_liminal_1} follows immediately.

Now suppose that $E$ has no sources. Then \eqref{thm_directed_graph_liminal_3} and \eqref{thm_directed_graph_liminal_5} are equivalent since $C^*(E)$ is isomorphic to $C^*(G_E)$ by Proposition \ref{prop_algebras_isomorphic}. Since $G_E$ is isomorphic to $G_{\Lambda_E}$ by Lemma \ref{path_groupoids_isomorphic}, we can apply Theorem \ref{k-graph_liminal_thm1} to the $1$-graph $\Lambda_E$ to establish the equivalence of \eqref{thm_directed_graph_liminal_4} and \eqref{thm_directed_graph_liminal_5}.
\end{proof}
The following is an immediate corollary of Theorem \ref{thm_directed_graph_liminal} and Lemma \ref{lemma_cycle_with_entry_implies_not_liminal}.
\begin{cor}\label{cor_liminal_implies_no_entries}
Suppose $E$ is a row-finite directed graph. If $C^*(E)$ is liminal then no cycle in $E$ has an entry.
\end{cor}
\begin{example}\label{example_directed_graph_algebras_not_isomorphic_proof}
Recall that in Example \ref{example_directed_graph_algebras_not_isomorphic} it was claimed that if $E$ is the following directed graph, then $C^*(G_E)$ is liminal while $C^*(E)$ is not.
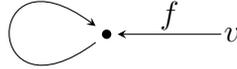
\begin{center}
\begin{tikzpicture}[>=stealth,baseline=(current bounding box.center)]
\clip (-6.2em,-1.1em) rectangle (1.2em,1.2em);
\node (u) at (-3em, 0em) {};
\node (v) at (1em,0em) {};
\node [anchor=center] at (v) {$v$};
\fill[black] (u) circle (0.15em);
\draw [<-, black] (u) to node[above=-0.2em] {$f$} (v);
\draw [<-,black] (u) .. controls (-7em,3em) and (-7em,-3em) .. (u);
\end{tikzpicture}
\end{center}
To see this, first note that $C^*(E)$ is not liminal by Corollary \ref{cor_liminal_implies_no_entries}. Let $F$ be the directed graph obtained by removing $f$ and $v$ from $E$. Then $G_E=G_F$ and, since $F$ has no sources, $C^*(G_F)=C^*(G_E)$ is liminal by Theorem \ref{thm_directed_graph_liminal}.
\end{example}

\begin{theorem}\label{thm_directed_graph_postliminal}
Suppose $E$ is a row-finite directed graph. The following are equivalent:
\begin{enumerate}
\item\label{thm_directed_graph_postliminal_1} for every path $x\in E^{\infty}$ there exists $n\in\NN$ such that every path in $x(n) E^\infty$ that is frequently divertable to $[x]$ is shift equivalent to $x$;
\item\label{thm_directed_graph_postliminal_2} for every $x\in E^\infty$ there exists $n\in\NN$ such that the only path in $x(n)E^\infty$ that is shift equivalent to $x$ is $\sigma^n(x)$; and
\item\label{thm_directed_graph_postliminal_3} $C^*(E)$ is postliminal.
\end{enumerate}
If in addition $E$ has no sources, the following are also equivalent to the preceding:
\begin{enumerate}\setcounter{enumi}{3}
\item\label{thm_directed_graph_postliminal_4} every orbit in $G_E^{(0)}$ is locally closed; and
\item\label{thm_directed_graph_postliminal_5} $C^*(G_E)$ is postliminal.
\end{enumerate}
\end{theorem}
Condition \eqref{thm_directed_graph_postliminal_2} is similar to the condition in Ephrem's \cite[Theorem~7.3]{Ephrem2004}. We will show that these conditions are equivalent in Remark \ref{remark_Ephrems_postliminal_condition_coincides}. The proof of Theorem \ref{thm_directed_graph_postliminal} uses the following three lemmas to establish the equivalence of \eqref{thm_directed_graph_postliminal_1} and \eqref{thm_directed_graph_postliminal_2}.
\begin{lemma}\label{lemma_entries_not_part_of_cycles}
Let $E$ be a directed graph. Suppose that for every $x\in E^\infty$ there exists $n\in\NN$ such that every path in $x(n)E^\infty$ that is frequently divertable to $[x]$ is shift equivalent to $x$. Then no entry to a cycle is in a cycle.
\begin{proof}
Suppose $f\in E^1$ is an entry to a cycle $\alpha$ and that $f$ is in a cycle $\beta$; we may assume that $\beta_1=f$. Let $x=\alpha^\infty$ and fix $n\in\NN$. Since $r(f)$ is a vertex in $\alpha$, there is a path $\eta\in E^*$ with $r(\eta)=x(n)$ and $s(\eta)=r(f)$. Let $y$ be the path $\eta\beta^\infty$ in $x(n)E^\infty$. Since the edge $f$ is in $y$ but not in $x$, the path $y$ cannot equal $\sigma^n(x)$ and the result follows.
\end{proof}
\end{lemma}

\begin{lemma}\label{lemma_cycle_fun}
Suppose $E$ is a directed graph such that no entry to a cycle is in a cycle. If $\eta\in E^*$ is of non-zero length and has $r(\eta)=s(\eta)$, then there is a cycle $\alpha$ such that $\eta=\alpha^p$ for some $p\in\PP$.
\begin{proof}
Let $\beta$ be a shortest path in $\eta$ with $r(\beta)=s(\beta)$. Suppose $s(\eta)$ is not in $\beta$. Then there exists $q\in\PP$ such that $\eta_q$ is an entry to the cycle $\beta$. But then $\eta(q,|\eta|)\eta(0,q-1)$ is a path in $E^*$ with range $s(\eta_q)$ and source $r(\eta_q)$, so $\eta_q$ must be in a cycle. This contradicts the hypothesis, so no such entry $\eta_q$ exists and $s(\eta)$ must be in $\beta$. Since $s(\eta)$ is in $\beta$ there exists $n\in\NN$ such that $\beta(n)=s(\eta)$. Let $\alpha$ be the cycle $\beta(n,|\beta|)\beta(0,n)$, noting that $r(\alpha)=s(\alpha)=s(\eta)=r(\eta)$.

Now choose $p\in\PP$ with $p\ge |\eta|/|\alpha|$. We now claim that $\alpha^p(0,|\eta|)=\eta$. Suppose not. Then there is a smallest $q\in\PP$ with $\alpha^p_q\ne\eta_q$. Then $\eta_q$ is an entry to the cycle $\alpha$, which we earlier showed is not possible by the hypothesis since $\eta_q$ is in a cycle. Thus $\alpha^p(0,|\eta|)=\eta$, establishing the claim. Since $s(\alpha)=s(\eta)=\alpha^p(|\eta|)$ and $\alpha(n)=s(\alpha)$ only when $n=|\alpha|$ or $n=0$, we can reduce $p$ if necessary to conclude that $\eta=\alpha^p$.
\end{proof}
\end{lemma}

\begin{lemma}\label{lemma_postliminal_removing_cycles_from_path}
Suppose $E$ is a directed graph such that no entry to a cycle is in a cycle. Suppose further that there is a path $x\in E^\infty$ such that there is more than one element in $x(n)E^\infty\cap [x]$ for each $n\in\NN$. Then there is a path $y\in E^\infty$ that does not contain a cycle and has more than one element in $y(n)E^\infty\cap [y]$ for each $n\in\NN$.
\begin{proof}
Suppose there exists $v\in E^0$ such that $v=x(n)$ for infinitely many $n\in\NN$. If $m$ is an integer with $x(m)=v$, then by Lemma \ref{lemma_cycle_fun} there is a cycle $\alpha$ in $E$ such that $x=x(0,m)\alpha^\infty$. Then $x(m)E^\infty\cap [x]$ can only have one element since otherwise there would need to be an entry to $\alpha$ that is on a path that is shift equivalent to $\alpha^\infty$, which is not possible by the hypothesis. It follows that, for every $v\in E^0$, there are only finitely many $n\in\NN$ with $x(n)=v$.

If there exists $n\in\NN$ such that $\sigma^n(x)$ contains no cycles then we can take $y=\sigma^n(x)$. For the remainder of this proof we may assume that no such $n$ exists. We will construct a path $y\in E^\infty$ as in the statement of this lemma by removing the cycles from $x$. We begin by defining two sequences $\{m_i\},\{n_i\}$ in $\NN$. Let $m_1$ be the smallest integer such that $x(m_1)=x(n)$ for some $n>m_1$. Then let $n_1$ be the largest integer such that $x(n_1)=x(m_1)$. Define $m_2,n_2$ similarly: let $m_2$ be the largest integer greater than $n_1$ such that $x(m_2)=x(n)$ for some $n>m_2$, and let $n_2$ be the greatest integer with $x(n_2)=x(m_2)$. In this manner we may continue defining $m_i$ and $n_i$ for $i=3,4,\ldots$ so that each $m_i$ is the smallest integer greater than $n_{i-1}$ with $x(m_i)=x(n)$ for some $n>m_i$, and each $n_i$ is the largest integer with $x(n_i)=x(m_i)$.

Let $y=x(n_1,m_2)x(n_2,m_3)x(n_3,m_4)\cdots$. We claim that $y$ contains no cycles. Fix $p\in\NN$. It suffices to show that $y(p)\ne y(q)$ for each $q>p$. By the definition of $y$ there must exist $l\in\NN$ and $i\in\PP$ such that $n_{i-1}\le l<m_i$ and $x(l)=y(p)$. We consider two cases: when $n_{i-1}<l<m_i$ and when $n_{i-1}=l$. First suppose $n_{i-1}<l<m_i$. Since $m_i$ is the smallest integer greater than $n_{i-1}$ with $x(m_i)=x(n)$ for some $n>m_i$, we must have $y(p)=x(l)\ne x(n)$ for every $n>l$, and so $y(p)\ne y(q)$ for all $q>p$. For the second case suppose $n_{i-1}=l$. Since $n_{i-1}$ is the largest integer with $x(n_{i-1})=x(m_{i-1})$, it follows that $x(n)\ne x(n_{i-1})$ for each $n>n_{i-1}=l$. Thus $y(p)\ne y(q)$ for each $q>p$, establishing the claim that $y$ contains no cycles.

Recall that $p$ was fixed arbitrarily in $\NN$. It remains to show that $y(p)E^\infty\cap[y]$ has more than one element. Choose $l\in\NN$ such that $x(l)=y(p)$ and choose distinct $z^{(1)},z^{(2)}\in x(l)E^\infty\cap [x]$. Since these paths are shift equivalent to $x$ there exist $a,b\in\NN$ with $\sigma^a(z^{(1)})=\sigma^b(z^{(2)})$ and we may increase $a,b$ if necessary to ensure that $z^{(1)}(a)=z^{(2)}(b)=x(m_j)$ for some $j\in\PP$. Then $z^{(1)}(0,a)\ne z^{(2)}(0,b)$ and $x(m_j)=y(q)$ for some $q\in\NN$, so the paths $z^{(1)}(0,a)\sigma^q(y)$ and $z^{(2)}(0,b)\sigma^q(y)$ are distinct elements of $y(p)E^\infty\cap [y]$, as required.
\end{proof}
\end{lemma}

\begin{proof}[Proof of Theorem \ref{thm_directed_graph_postliminal}]
\eqref{thm_directed_graph_postliminal_1}$\implies$\eqref{thm_directed_graph_postliminal_2}. Suppose \eqref{thm_directed_graph_postliminal_2} does not hold. 
We may assume that no entry to a cycle is in a cycle as otherwise it follows immediately from Lemma \ref{lemma_entries_not_part_of_cycles} that \eqref{thm_directed_graph_postliminal_1} does not hold. Since \eqref{thm_directed_graph_postliminal_2} does not hold there is a path $x\in E^\infty$ such that $x(n)E^\infty\cap [x]$ has at least two elements for each $n\in\NN$. By Lemma \ref{lemma_postliminal_removing_cycles_from_path} there is a path $y\in E^\infty$ that does not contain cycles such that $y(m)E^\infty\cap[y]$ has at least two elements for each $m\in\NN$. That \eqref{thm_directed_graph_postliminal_1} does not hold now follows by applying Lemma \ref{lemma_FD_not_SE} to $y$.

\eqref{thm_directed_graph_postliminal_2}$\implies$\eqref{thm_directed_graph_postliminal_1}. Suppose \eqref{thm_directed_graph_postliminal_2} and fix $x\in E^\infty$. Then there exists $n\in\NN$ such that $\sigma^n(x)$ is the only path in $x(n)E^\infty\cap[x]$. Suppose $y\in x(n)E^\infty$ is frequently divertable to $[x]$. It follows that for every $m\in\NN$ there is a path $z^{(m)}\in y(m)E^\infty\cap [x]$. Then, for each $m\in\NN$, the path $y(0,m)z^{(m)}$ is shift equivalent to $x$, so $y(0,m)z^{(m)}$ is in $x(n)E^\infty\cap [x]$. By our choice of $n$ it follows that $y(0,m)z^{(m)}=\sigma^n(x)$ for each $m\in\NN$. Then $y=\sigma^n(x)$, so $y$ is shift equivalent to $x$, establishing \eqref{thm_directed_graph_postliminal_1}.

\eqref{thm_directed_graph_postliminal_1}$\implies$\eqref{thm_directed_graph_postliminal_3}. Suppose \eqref{thm_directed_graph_postliminal_1}. Recall from above that $E^{\le\infty}=\Lambda_E^{\le\infty}=\partial\Lambda_E$. Applying Theorem \ref{k-graph_postliminal_thm2} to the $1$-graph $\Lambda_E$ shows that $C^*(E)$ is postliminal if
\begin{statement}\label{condition_directed_graph_postliminal_thm2}
for every $x\in E^{\le\infty}$ there exists $n\in\NN$ with $n\le |x|$ such that every path in $x(n)E^{\le\infty}$ that is frequently divertable to $[x]$ is shift equivalent to $x$.
\end{statement}
To see that \eqref{condition_directed_graph_postliminal_thm2} holds, first fix $x\in E^{\le\infty}$. When $x$ is finite, the only path in $x(|x|)E^{\le\infty}=s(x)E^{\le\infty}$ is $s(x)$, which is frequently divertable to $[x]$ and shift equivalent to $x$. Suppose $x$ is infinite. By \eqref{thm_directed_graph_postliminal_1} there exists $n\in\NN$ such that every path in $x(n)E^\infty$ that is frequently divertable to $[x]$ is shift equivalent to $x$. If $y\in x(n)E^{\le\infty}\cap E^*$ is frequently divertable to $[x]$, then there must be a path in $s(y)E^\infty$ that is shift equivalent to $x$. This is not possible since $s(y)$ is a source and $x\in E^\infty$, so no path in $x(n)E^{\le\infty}\cap E^*$ is frequently divertable to $[x]$. It follows that condition \eqref{condition_directed_graph_postliminal_thm2} holds and so $C^*(E)$ is postliminal by applying Theorem \ref{k-graph_postliminal_thm2} to the $1$-graph $\Lambda_E$.

\eqref{thm_directed_graph_postliminal_3}$\implies$\eqref{thm_directed_graph_postliminal_1}. Suppose \eqref{thm_directed_graph_postliminal_3}. Since $C^*(E)$ is postliminal, Theorem \ref{k-graph_postliminal_thm2} tell us that \eqref{condition_directed_graph_postliminal_thm2} holds and \eqref{thm_directed_graph_postliminal_1} follows immediately.

Now suppose that $E$ has no sources. Then \eqref{thm_directed_graph_postliminal_3} and \eqref{thm_directed_graph_postliminal_5} are equivalent since $C^*(E)$ is isomorphic to $C^*(G_E)$ by Proposition \ref{prop_algebras_isomorphic}. Since $G_E$ is isomorphic to $G_{\Lambda_E}$ by Lemma \ref{path_groupoids_isomorphic}, we can apply Theorem \ref{k-graph_postliminal_thm1} to the $1$-graph $\Lambda_E$ to establish the equivalence of \eqref{thm_directed_graph_postliminal_4} and \eqref{thm_directed_graph_postliminal_5}.
\end{proof}
\begin{remark}\label{remark_Ephrems_postliminal_condition_coincides}
For a directed graph $E$, Ephrem in \cite[Theorem~7.3]{Ephrem2004} showed that $C^*(E)$ is postliminal if and only if
\begin{enumerate}\renewcommand{\theenumi}{\roman{enumi}}
\item\label{Ephrem_postliminal_1} no entry to a cycle is in a cycle, and
\item\label{Ephrem_postliminal_2} for any $x\in E^\infty$, there are finitely many vertices in $x$ that receive multiple edges that are in paths that are shift equivalent to $x$.
\end{enumerate}
To see that this is equivalent to condition \eqref{thm_directed_graph_postliminal_2} from Theorem \ref{thm_directed_graph_postliminal}, first suppose \eqref{Ephrem_postliminal_1} and \eqref{Ephrem_postliminal_2} hold and fix $x\in E^\infty$. By \eqref{Ephrem_postliminal_2} we can let $v_1,v_2,\ldots,v_p$ be the vertices in $x$ that receive multiple edges that are in paths that are shift equivalent to $x$. Suppose there exists an integer $q$ ($1\le q\le p$) such that $x(m)=v_q$ for infinitely many $m\in\NN$. Let $n$ be any integer such that $x(n)=v_q$. By \eqref{Ephrem_postliminal_1} and Lemma \ref{lemma_cycle_fun} there is a cycle $\alpha$ such that $x=x(0,n)\alpha^\infty$. Since no entry to a cycle is a in a cycle, the only path in $x(n)E^\infty\cap[x]$ is $\sigma^n(x)=\alpha^\infty$, as required.

Now suppose that for each $v_i$ there are finitely many $m\in\NN$ with $x(m)=v_i$. Then there exists $n\in\NN$ such that, for every $m\ge n$, $x(m)$ does not equal any $v_i$. Then the only path in $\sigma^n(x)E^\infty\cap[x]$ is $\sigma^n(x)$, as required.

To prove the converse, suppose $E$ satisfies condition \eqref{thm_directed_graph_postliminal_2} from Theorem \ref{thm_directed_graph_postliminal} and note that \eqref{Ephrem_postliminal_1} follows from Lemma \ref{lemma_entries_not_part_of_cycles} and the \eqref{thm_directed_graph_postliminal_2}$\implies$\eqref{thm_directed_graph_postliminal_1} component of Theorem \ref{thm_directed_graph_postliminal} (the proof of which does not require $E$ to be row-finite). Fix $x\in E^\infty$. Then there exists $n\in\NN$ such that the only path in $x(n)E^\infty\cap[x]$ is $\sigma^n(x)$. It follows that the only vertices in $x$ that may receive multiple edges that are in paths that are shift equivalent to $x$ are the vertices of the form $x(m)$ for $m<n$. Condition \eqref{Ephrem_postliminal_2} follows and we can conclude that Ephrem's conditions \eqref{Ephrem_postliminal_1} and \eqref{Ephrem_postliminal_2} are equivalent to condition \eqref{thm_directed_graph_postliminal_2} from Theorem \ref{thm_directed_graph_postliminal}.
\end{remark}
The following is an immediate corollary of Theorem \ref{thm_directed_graph_postliminal} and Lemma \ref{lemma_entries_not_part_of_cycles}.
\begin{cor}
Suppose that $E$ is a row-finite directed graph. If $C^*(E)$ is postliminal then no entry to a cycle in $E$ is in a cycle.
\end{cor}

We will now move on to the directed graph case of the theorems from Section \ref{sec_bded-trace_cts-trace_Fell}.
\begin{remark}\label{remark_directed_graphs_and_cycles}
Recall that the theorems characterising the higher-rank graphs with $C^*$-algebras that have bounded trace, are Fell and have continuous trace (Theorems \ref{thm_bounded_trace}, \ref{thm_Fell} and \ref{thm_cts-trace}) only apply to higher-rank graphs with principal path groupoids. By Proposition \ref{prop_principal_iff_no_cycles} this means that these theorems can only be applied to directed graphs that do not contain cycles. Corollary \ref{cor_liminal_implies_no_entries} tells us that if a directed graph $E$ contains a cycle with an entry, then $C^*(E)$ is not liminal (and so does not have bounded or continuous trace and is not Fell). It follows that the only directed graphs where our theorems cannot be directly applied are those that have cycles, none of which have an entry. 

Clark and an Huef in \cite{Clark-anHuef2012} showed that if a directed graph $E$ satisfies a property concerning the stability subgroups of the path groupoid $G_E$, then $E$ can be modified to remove cycles while preserving the bounded trace and Fell properties of $C^*(E)$. In \cite[Example~7.1]{Clark-anHuef2012} Clark and an Huef describe a directed graph that contains cycles, none of which has an entry. They then demonstrate the removal of the cycles from that directed graph before using the next theorem, Theorem \ref{directed_graph_integrable_thm}, to deduce that the associated $C^*$-algebra has bounded trace.
\end{remark}

The following result is obtained by applying Theorem \ref{thm_integrable_bounded-trace} to the $1$-graph $\Lambda_E$.
\begin{theorem}\label{directed_graph_integrable_thm}
Suppose $E$ is a row-finite directed graph without sources. The following are equivalent:
\begin{enumerate}
 \item\label{directed_graph_integrable_thm_1} $G_E$ is integrable;
\item\label{directed_graph_integrable_thm_2} $E$ has no cycles, and for every $v\in E^0$ there exists $M\in\NN$ such that for any $x\in E^\infty$ there are at most $M$ paths in $E^\infty$ with range $v$ that are shift equivalent to $x$; and
 \item\label{directed_graph_integrable_thm_3} $E$ has no cycles, and $C^*(E)\cong C^*(G_E)$ has bounded trace.
\end{enumerate}
\end{theorem}
The next result is obtained by applying Theorem \ref{thm_bounded_trace} to the $1$-graph $\Lambda_E$.
\begin{theorem}
Suppose $E$ is a row-finite directed graph without cycles. Then $C^*(E)$ has bounded trace if and only if for every $v\in E^0$ there exists $M\in\NN$ such that for any $x\in E^{\le\infty}$ there are at most $M$ paths in $E^{\le\infty}$ with range $v$ that are shift equivalent to $x$.
\end{theorem}

A pair of edges $\braces{f,g}$ in a directed graph is {\em splitting}\index{splitting!pair of edges} if $s(f)=s(g)$ and $f\ne g$. Suppose $S\subset E^*\cup E^\infty$. We define $E|_S$ to be the subgraph of $E$ such that
\[
E|_S^1=\braces{e\in E:e=x_p\text{ for some }x\in S, p\in\PP}
\]
and $E_S^0=r(E_S^1)\cup s(E_S^1)$.

\begin{lemma}\label{lemma_finite_connections}
Let $E$ be a directed graph that contains no cycles and has finitely many splitting pairs of edges. For any vertices $v$ and $w$ in $E$ there are at most finitely many finite paths with source $w$ and range $v$.
\begin{proof}
We prove by contradiction. Let $A$ be an infinite set of finite paths with source $w$ and range $v$. Let $B$ be a finite subset of $A$. Then $E|_B$ has finitely many splitting pairs of edges, say, $M$. We claim that we can add only finitely many paths from $A$ to $B$ before increasing the number of splitting pairs in $E|_B$ to at least $M+1$.

Let $B^+$ be the set of all finite paths in $E|_B$ with source $w$ and range $v$. Since $E$ has no cycles, $B^+$ is finite and $B\subset B^+\subsetneq A$. Now suppose $\alpha\in A\backslash B^+$. Then there must be a greatest $p\in\PP$ such that $\alpha_p\notin E|_B^1$. Note that $s(\alpha_p)\in E|_B^0$ and that since $r(\alpha)=v$ and $E$ contains no cycles, we cannot have $s(\alpha_p)=v$. Then there exist $\beta\in B$ and $q\in\PP$ such that $s(\alpha_p)=s(\beta_q)$. It follows that $\{\alpha_p,\beta_q\}$ is a splitting pair in $E|_{B^+\cup\{\alpha\}}$ but not in $E|_B$, so $E|_{B^+\cup\{\alpha\}}$ must contain at least $M+1$ splitting pairs.

Thus given the finite set $B$ with $M$ splitting pairs of edges we can only add finitely many paths from $A$ to $B$ (forming $B^+$) before adding an additional path will increase the number of splitting pairs of edges in $E|_B$ to at least $M+1$. Then $A$ must contain infinitely many splitting pairs of edges, contradicting that $E$ has finitely many splitting pairs of edges.
\end{proof}
\end{lemma}

\begin{lemma}\label{lemma_directed_graph_simplification}
Suppose $E$ is a row-finite directed graph without cycles and $V$ is a finite subset of $E^0$. Then the following are equivalent:
\begin{enumerate}
\item\label{lemma_directed_graph_simplification_1} the directed graph $E|_{VE^{\le\infty}}$ has finitely many splitting pairs of edges; and
\item\label{lemma_directed_graph_simplification_2} there exists a finite $F\subset E^*$ such that such that for any $x,y\in VE^{\le\infty}$ with $x\sim y$, the pair $(x,y)$ is a monolithic extension of a pair in $F\times F$.
\end{enumerate}
\begin{proof}
Suppose \eqref{lemma_directed_graph_simplification_1}. Let $A$ be the set of all edges that occur in a splitting pair of edges in $E|_{VE^{\le\infty}}$ and let $F=VE^*V\cup VE^*A$. Then $F$ is finite by Lemma \ref{lemma_finite_connections} since $A$ and $V$ are finite. Suppose $x,y\in VE^{\le\infty}$ satisfy $x\sim_n y$. Then there exists a smallest $m\in\NN$ such that $\sigma^m(x)=\sigma^{m-n}(y)$. If $m=0$ or $m-n=0$, then $r\big(\sigma^m(x)\big)$ and $r\big(\sigma^{m-n}(y)\big)$ are in $V$ so $(x,y)$ is a monolithic extension of a pair in $F\times F$. Suppose $m\ne 0$ and $m-n\ne 0$. Then $\braces{x_m,y_{m-n}}$ is a splitting pair of edges in $E|_{VE^{\le\infty}}$ so there exists $\alpha,\beta\in VE^*$ such that $(x,y)$ is a monolithic extension of $(\alpha x_m,\beta y_{m-n})\in F\times F$, establishing \eqref{lemma_directed_graph_simplification_2}.

Let $V$ be a finite subset of $E^0$. Suppose there exists a finite $F\subset E^*$ as in \eqref{lemma_directed_graph_simplification_2}. Let $A=\braces{\alpha_m\in E^1:\alpha\in F, m\le |\alpha|}$, noting that $A$ is finite since $F$ is. Suppose $\braces{f,g}$ is a splitting pair of edges in $E|_{VE^{\le\infty}}$. We claim that $f,g\in A$. To see this, first note that since $\braces{f,g}$ is a splitting pair in $E|_{VE^{\le\infty}}$, there exist $\alpha,\beta\in VE^*$ such that $s(\alpha)=r(f)$ and $s(\beta)=r(g)$. Suppose $x\in s(f)E^{\le\infty}$. Then $\alpha f x\sim\beta gx$ so there exists $(\eta,\zeta)\in F\times F$ such that $(\alpha fx,\beta gx)$ is a monolithic extension of $(\eta,\zeta)$. Let $y\in E^{\le\infty}$ be the path such that $(\eta y, \zeta y) = (\alpha fx, \beta gx)$. Then $\eta y$ is shift equivalent to $\zeta y$ with lags $|\beta|-|\alpha|$ and $|\zeta|-|\eta|$. It follows that there exists $l\in\NN$ such that $\sigma^l(\eta y)=\sigma^{l-|\beta|+|\alpha|}(\zeta y) = \sigma ^{l-|\zeta|+|\eta|}(\zeta y)$. Since $E$ has no cycles we must have $|\beta|-|\alpha|=|\zeta|-|\eta|$.

Now $f=(\alpha fx)_{|\alpha|+1}=(\eta y)_{|\alpha|+1}$ and $g=(\beta gx)_{|\beta|+1}=(\zeta y)_{|\beta|+1}$. If $|\alpha f|>|\eta|$, by noting that $|\alpha f|=|\alpha|+1$ and $|\alpha|-|\eta|=|\beta|-|\zeta|$, we have
\[
f=(\eta y)_{|\alpha|+1} = y_{|\alpha|-|\eta|+1} = y_{|\beta|-|\zeta|+1}=(\zeta y)_{|\beta|+1}=g,\]
which is not possible since $f\ne g$ as $\{f,g\}$ is a splitting pair. Thus $|\alpha f|\le |\eta|$, so $f=\eta_{|\alpha f|}\in A$ and $g=\zeta_{|\beta g|}\in A$. Since $A$ is finite and $\braces{f,g}$ is an arbitrary splitting pair of edges in $E|_{VE^{\le\infty}}$, \eqref{lemma_directed_graph_simplification_1} follows.
\end{proof}
\end{lemma}
We can now proceed to the improved versions of Theorems \ref{thm_Cartan_Fell}, \ref{thm_proper_cts-trace}, \ref{thm_Fell} and \ref{thm_cts-trace} for the directed graph case. The first of these follows immediately from Lemma \ref{lemma_directed_graph_simplification} and Theorem \ref{thm_Cartan_Fell}.
\begin{theorem}\label{thm_directed_graph_Cartan_Fell}
Suppose $E$ is a row-finite directed graph without sources. The following are equivalent:
\begin{enumerate}\renewcommand{\labelenumi}{(\arabic{enumi})}
\item\label{thm_directed_graph_Cartan_Fell_1} $G_E$ is Cartan;
\item $E$ has no cycles, and for every $z\in E^\infty$ there exists $n\in\NN$ such that $E|_{z(n)E^{\infty}}$ has finitely many splitting pairs of edges;
\item\label{thm_directed_graph_Cartan_Fell_2} $E$ has no cycles, and for every $z\in E^\infty$ there exist $p\in\NN$ and a finite $F\subset E^*$ such that for every $x,y\in z(p)E^\infty$ with $x\sim y$, the pair $(x,y)$ is a monolithic extension of a pair in $F\times F$;
\item\label{thm_directed_graph_Cartan_Fell_3} $E$ has no cycles, and for every $z\in E^\infty$ there exist $p\in\NN$ and a finite $F\subset E^*$ such that for every $(\alpha,\beta)\in E^*\ast_s E^*$ with $r(\alpha)=r(\beta)=z(p)$ and every $\gamma\in s(\alpha)E^*$ with $|\gamma|=\max\{|\eta|:\eta\in F\}$, the pair $(\alpha\gamma,\beta\gamma)$ is a monolithic extension of a pair in $F\times F$; and
\item\label{thm_directed_graph_Cartan_Fell_4} $E$ has no cycles, and $C^*(E)\cong C^*(G_E)$ is Fell.
\end{enumerate}
\end{theorem}

\needspace{4\baselineskip}
\begin{theorem}\label{thm_directed_graph_Fell}
Suppose $E$ is a row-finite directed graph without cycles. The following are equivalent:
\begin{enumerate}
\item\label{thm_directed_graph_Fell_1} $C^*(E)$ is Fell;
\item\label{thm_directed_graph_Fell_2} for every $z\in E^\infty$ there exists $n\in\NN$ such that $E|_{z(n)E^{\le\infty}}$ has finitely many splitting pairs of edges; and
\item\label{thm_directed_graph_Fell_3} for every $z\in E^\infty$ there exist $n\in\NN$ and a finite $F\subset E^*$ such that for any $x,y\in z(n)E^{\le\infty}$ with $x\sim y$, the pair $(x,y)$ is a monolithic extension of a pair in $F\times F$.
\end{enumerate}
\begin{proof}
First note that \eqref{thm_directed_graph_Fell_2} and \eqref{thm_directed_graph_Fell_3} are equivalent by Lemma \ref{lemma_directed_graph_simplification}. By applying Theorem \ref{thm_Fell} to the $1$-graph $\Lambda_E$, we can see that \eqref{thm_directed_graph_Fell_1} is equivalent to
\begin{statement}\label{statement_directed_graph_Fell}
for every $z\in E^{\le\infty}$ there exist $n\in\NN$ with $n\le |z|$ and a finite $F\subset E^*$ such that for any $x,y\in z(n)E^{\le\infty}$ with $x\sim y$, the pair $(x,y)$ is a monolithic extension of a pair in $F\times F$.
\end{statement}
If \eqref{statement_directed_graph_Fell} holds, then \eqref{thm_directed_graph_Fell_3} must hold. Conversely, suppose $z\in E^{\le\infty}$ is finite and let $F=\{s(z)\}$. The only path in $z(|z|)^{\le\infty}$ is $z(|z|)$ (which equals $s(z)$) so, if $x,y\in z(|z|)E^{\le\infty}$, then we must have $(x,y)=\big(s(z),s(z)\big)$, which is trivially a monolithic extension of $\big(s(z),s(z)\big)\in F\times F$. Thus \eqref{thm_directed_graph_Fell_3} implies \eqref{statement_directed_graph_Fell}, and so \eqref{thm_directed_graph_Fell_3} is equivalent to \eqref{statement_directed_graph_Fell}.
\end{proof}
\end{theorem}

The next theorem follows immediately from Lemma \ref{lemma_directed_graph_simplification} and Theorem \ref{thm_proper_cts-trace}.
\needspace{6\baselineskip}
\begin{theorem}\label{thm_directed_graph_proper_cts-trace}
Suppose $E$ is a row-finite directed graph without sources. The following are equivalent:
\begin{enumerate}\renewcommand{\labelenumi}{(\arabic{enumi})}
\item $G_E$ is proper;
\item $E$ has no cycles, and for every finite $V\subset E^0$ the directed graph $E|_{VE^\infty}$ has finitely many splitting pairs of edges;
\item $E$ has no cycles, and for every finite subset $V$ of $E^0$, there exists a finite $F\subset E^*$ such that, for every $x,y\in VE^\infty$ with $x\sim_n y$, the pair $(x,y)$ is a monolithic extension of a pair in $F\times F$;
\item $E$ has no cycles, and for every finite subset $V$ of $E^0$, there exists a finite $F\subset E^*$ such that, for every $(\alpha,\beta)\in E^*\ast_s E^*$ with $r(\alpha),r(\beta)\in V$, and every $\gamma\in s(\alpha)E^*$ with $|\gamma|=\max\{|\eta|:\eta\in F\}$, the pair $(\alpha\gamma,\beta\gamma)$ is a monolithic extension of a pair in $F\times F$; and
\item $E$ has no cycles, and $C^*(E)\cong C^*(G_E)$ has continuous trace.
\end{enumerate}
\end{theorem}

Our final theorem follows immediately from Lemma \ref{lemma_directed_graph_simplification} and Theorem \ref{thm_cts-trace}.

\needspace{6\baselineskip}
\begin{theorem}
Suppose $E$ is a row-finite directed graph without cycles. The following are equivalent:
\begin{enumerate}
\item $C^*(E)$ has continuous trace;
\item for every finite $V\subset E^0$ the directed graph $E|_{VE^{\le\infty}}$ has finitely many splitting pairs of edges; and
\item for every finite $V\subset E^0$ there exists a finite $F\subset E^*$ such that for any $x,y\in VE^{\le\infty}$ with $x\sim y$, the pair $(x,y)$ is a monolithic extension of a pair in $F\times F$.
\end{enumerate}
\end{theorem}

\backmatter
\linespread{1}

\printindex
\end{document}